\documentclass[10pt,a4paper]{article}

\usepackage{amsfonts,amsmath,amssymb,amstext}
\usepackage{amsthm}
\usepackage{graphicx}
\usepackage{subcaption}
\usepackage{appendix}

\usepackage{fullpage}

\usepackage[ruled,vlined]{algorithm2e}

\usepackage{tikz}

\usepackage[pdftex,bookmarks,colorlinks,breaklinks]{hyperref}  
\usepackage{color}

\definecolor{dullmagenta}{rgb}{0.4,0,0.4}   
\definecolor{darkblue}{rgb}{0,0,0.4}
\hypersetup{linkcolor=red,citecolor=blue,filecolor=dullmagenta,urlcolor=darkblue} 

\graphicspath{{figs/}}

\newtheorem{theorem}{Theorem}[section]
\newtheorem{lemma}[theorem]{Lemma}
\newtheorem{definition}[theorem]{Definition}
\newtheorem{corol}[theorem]{Corollary}

\newenvironment{remark}%
  {\par\medbreak\refstepcounter{theorem}%
    \noindent\textbf{Remark~\thetheorem. }}%
  {\par\medskip}

\newcommand{\vz}[1]{\ensuremath{\mathbb{#1}}}

\newcommand{\R}{{\vz R}}
\newcommand{\N}{{\vz N}}

\newcommand{\dvg}{\text{div}\,}

\newcommand{\TV}{\text{TV}}

\newcommand{\argmin}{\mathrm{argmin}}
\DeclareMathOperator{\supp}{supp}

\def\vol#1{\mathrm{vol}\left(#1\right)}

\let\e\varepsilon

\newcounter{hdps}
\newenvironment{asm}[1]
  {\refstepcounter{hdps}%
    \begin{algorithm}\renewcommand\thealgocf{#1}}
  {\end{algorithm}\addtocounter{algocf}{-1}}

\makeatletter
\newcommand*\bigcdot{\mathpalette\bigcdot@{.5}}
\newcommand*\bigcdot@[2]{\mathbin{\vcenter{\hbox{\scalebox{#2}{$\m@th#1\bullet$}}}}}
\makeatother

\title{An MBO scheme for minimizing the graph Ohta-Kawasaki functional}

\author{Yves van Gennip}

\begin{document}

\maketitle

\begin{abstract}
We study a graph based version of the Ohta-Kawasaki functional, which was originally introduced in a continuum setting to model pattern formation in diblock copolymer melts and has been studied extensively as a paradigmatic example of a variational model for pattern formation.

Graph based problems inspired by partial differential equations (PDEs) and varational methods have been the subject of many recent papers in the mathematical literature, because of their applications in areas such as image processing and data classification. This paper extends the area of PDE inspired graph based problems to pattern forming models, while continuing in the tradition of recent papers in the field.

We introduce a mass conserving Merriman-Bence-Osher (MBO) scheme for minimizing the graph Ohta-Kawasaki functional with a mass constraint. We present three main results: (1) the Lyapunov functionals associated with this MBO scheme $\Gamma$-converge to the Ohta-Kawasaki functional (which includes the standard graph based MBO scheme and total variation as a special case); (2) there is a class of graphs on which the Ohta-Kawasaki MBO scheme corresponds to a standard MBO scheme on a transformed graph and for which generalized comparison principles hold; (3) this MBO scheme allows for the numerical computation of (approximate) minimizers of the graph Ohta-Kawasaki functional with a mass constraint.
\end{abstract}



{\let\thefootnote\relax\footnote{Abbreviated title: MBO for minimizing graph Ohta-Kawasaki}}

\section{Introduction}

In this paper we study the minimization problem
\[
\underset{u}\min\, \TV(u) + \frac\gamma2 \|u-\mathcal{A}(u)\|_{H^{-1}}^2
\]
on undirected graphs. Here $\TV$ and $\|\cdot\|_{H^{-1}}^2$ are graph based analogues of the continuum total variation seminorm and continuum $H^{-1}$ Sobolev norm, respectively, and $u$ is allowed to vary over the set of node functions with prescrbed mass. These concepts will be made precise later in the paper, culminating in formulation \eqref{eq:minimprobsF0} of the minimization problem. The main contributions of this paper are the introduction of the graph Ohta-Kawasaki functional into the literature, the development of an algorithm to produce (approximate) minimizers, and the study of that algorithm, which leads to, among other results, further insight into the connection between the graph Merriman-Bence-Osher (MBO) method and the graph total variation, following on from initial investigations in \cite{vanGennipGuillenOstingBertozzi14}.

There are various reasons to study this minimization problem. First of all, it is the graph based analogue of the continuum Ohta-Kawasaki variational model \cite{OhtaKawasaki86,KawasakiOhtaKohrogui88}. This model was originally introduced as a model for pattern formation in diblock copolymer systems and has become a paradigmatic example of a variational model which exhibits pattern formation. It spawned a large mathematical literature which explores its properties analytically and computationally. A complete literature overview for this area is outside the scope of this paper; for a sample of mathematical papers on this topic, see for example \cite{RenWei00,RenWei02,RenWei03a,RenWei03b,ChoksiRen03,ChoksiRen05,ChoksiSternberg06,vanGennipPeletier08,vanGennipPeletier09,ChoksiPeletierWilliams09,glasner2009coarsening,le2010convergence,peletier2010stripe,choksi2010small,choksi2011small,ChoksiMarasWilliams11,vanGennipPeletier11,choksi2012global,bourne2014hexagonal,ren2014double,ren2017spectrum,glasner2017multilayered}. For a brief overview of the continuum Ohta-Kawasaki model, see Appendix~\ref{sec:continuumOK}. The problem studied in this paper thus follows in the footsteps of a rich mathematical heritage, but at the same time, being the graph analogue of the continuum functional, connects with the recent interest in discrete PDE inspired problems.

Recently there has been a growing enthusiasm in the mathematical literature for graph based variational methods and graph based dynamics which mimic continuum based variational methods and partial differential equations (PDEs), respectively. This is partly driven by novel applications of such methods in data science and image analysis \cite{ta2011nonlocal,elmoataz2012non,BertozziFlenner12,merkurjev2013anmboscheme,hu2013method,CCCMulticlass2014,calatroni2017graph,bosch2016generalizing,merkurjevmodified17,elmoataz_desquesnes_toutain_2017} and partly by theoretical interest in the new connections between graph theory and PDEs \cite{vanGennipBertozzi12,vanGennipGuillenOstingBertozzi14,GarciaTrillosSlepcev16}. Broadly speaking these studies fall into one (or more) of three categories: papers connecting graph problems with continuum problems, for example through a limiting process \cite{
vanGennipBertozzi12,GarciaTrillosSlepcev16,trillos2016consistency}, papers adapting a PDE approach to a graph context in order to tackle a graph problem such as graph clustering and classification \cite{bertozzi2016diffuse,bresson2014incremental,merkurjev2016semi}, maximum cut computations \cite{KeetchvanGennip17}, and bipartite matching \cite{caracciolo2014scaling,caraccioloscaling2015}, and papers studying the graph analogue of a PDE or variational problem that has interesting properties in the continuum, to explore what (potentially similar) properties are present in the graph based version of the problem \cite{vanGennipGuillenOstingBertozzi14,Luo2017,ELMOATAZ2017177}. This paper mostly falls in the latter category.

The study of the graph based Ohta-Kawasaki model is also of interest, because it connects with graph methods, concepts, and questions that have recently attracted attention, such as the graph MBO method (also known as threshold dynamics), graph curvature, and the question how these concepts relate to each other. 

The MBO scheme was originally introduced (in a continuum setting) to approximate motion by mean curvature \cite{MBO1992,MBO1993,MerrimanBenceOsher94}. It is an iterative scheme, which alternates between a short-time diffusion step and a threshold step. Not only have these dynamics been proven to converge to motion by mean curvature \cite{Evans93,BarlesGeorgelin95,swartz2017convergence}, but they have been a very useful basis for numerical schemes as well, both in the continuum and on graphs. Without aiming for completeness, we mention some of the papers that investigate or use the MBO scheme: \cite{Mascarenhas92,ruuth1998,ruuth1998b,chambolle2006convergence,Esedoglu2008,Esedoglu2008,ruuth2010diffusion,hu2013method,merkurjev2013anmboscheme,merkurjev2014graph,hu2015multi,CPA:CPA21527}. 

In this paper we study two different MBO schemes, \ref{alg:OKMBO} and \ref{alg:massOKMBO}. The former is an extension of the standard graph MBO scheme \cite{vanGennipGuillenOstingBertozzi14} in the sense that it replaces the diffusion step in the scheme with a step whose dynamics are related to the Ohta-Kawasaki model and reduce to diffusion in the special case when $\gamma=0$. The latter uses the same dynamics as the former in the first step, but incorporates mass conservation in the threshold step. The \ref{alg:massOKMBO} scheme produces approximate graph Ohta-Kawasaki minimizers and is the one we use in our simulations which are presented in Section~\ref{sec:numerical}. The scheme \ref{alg:OKMBO} is of interest both as a precursor to \ref{alg:massOKMBO} and as an extension of the standard graph MBO scheme. In \cite{vanGennipGuillenOstingBertozzi14} it was conjectured that the standard graph MBO scheme is related to graph mean curvature flow and minimizers of the graph total variation functional. This paper furthers the study of that conjecture (but does not provide a definitive answer): in Section~\ref{sec:Gammaconvergence} it is shown that the Lyapunov functionals associated with the \ref{alg:OKMBO} $\Gamma$-converge to the the graph Ohta-Kawasaki functional (which reduces to the total variation functional in the case when $\gamma=0$). Moreover, in Section~\ref{sec:specialclasses} we introduce a special class of graphs, $\mathcal{C}_\gamma$, dependent on $\gamma$. For graphs from this class the \ref{alg:OKMBO} scheme can be interpreted as the standard graph MBO scheme on a transformed graph. For such graphs we extend existing elliptic and parabolic comparison princpiples for the graph Laplacian and graph diffusion equation to our new Ohta-Kawasaki operator and dynamics (Lemmas~\ref{lem:compellgeneralization} and~\ref{lem:comprincII}).

A significant role in the analysis presented in this paper is played by the equilibrium measures associated to a given node subset \cite{BENDITO2000155,BenditoCarmonaEncinas03}, especially in the construction of the aforementioned class $\mathcal{C}^0$. In Section~\ref{sec:DirPois} we study these equilibrium measures and the role they play in constructing Green's functions for the graph Dirichlet and Poisson problems. The Poisson problem in particular, is an important ingredient in the definition of the graph $H^{-1}$ norm and the graph Ohta-Kawasaki functional as they are introduced in Section~\ref{sec:graphOKfunctional}. Both the equilibrium measures and the Ohta-Kawasaki functional itself are related to the graph curvature, which was introduced in \cite{vanGennipGuillenOstingBertozzi14}, as is shown in Lemma~\ref{lem:nuandkappa} and Corollary~\ref{cor:OKexpressions}, respectively.

The structure of the paper is as follows. In Section~\ref{sec:setup} we define our general setting. Section~\ref{sec:DirPois} introduces the equilibrium measures from \cite{BenditoCarmonaEncinas03} into the paper (the terminology is derived from potential theory; see e.g. \cite{simon2007equilibrium} and references therein) and uses them to study the Dirichlet and Poisson problems on graphs, generalizing some results from \cite{BenditoCarmonaEncinas03}. In Section~\ref{sec:graphOKfunctional} we define the $H^{-1}$ inner product and norm and use those to construct the object at the centre of our paper: the (sharp interface) Ohta-Kawasaki functional on graphs, $F_0$. We also briefly consider $F_\e$, a diffuse interface version of the Ohta-Kawasaki functional and its relationship with $F_0$. Moreover, in this section we start using tools from spectral analysis to study $F_0$. These tools will be one of the main ingredients in the remainder of the paper. In Section~\ref{sec:MBO} the algorithms \ref{alg:OKMBO} and \ref{alg:massOKMBO} are introduced and analysed. It is shown that both these algorithms have an associated Lyapunov functional (which extends a result from \cite{vanGennipGuillenOstingBertozzi14} and that these functionals $\Gamma$-converge to $F_0$ in the limit when $\tau$ (the time parameter associated with the first step in the MBO iteration) goes to zero. We introduce the class $\mathcal{C}^0$ in Section~\ref{sec:specialclasses} and prove that the Ohta-Kawasaki dynamics (i.e. the dynamics used in the first steps of both \ref{alg:OKMBO} and \ref{alg:massOKMBO}) on graphs from this class corresponds to diffusion on a transformed graph. We also prove comparison principles for these graphs. In Section~\ref{sec:numerical} we then use \ref{alg:massOKMBO} to numerically construct (approximate) minimizers of $F_0$, before ending with a discussion of potential future research directions in Section~\ref{sec:discussthefuture}. Throughout the paper we will use the example of an unweighted star graph to illustrate many of the concepts and results that are introduced and proven.

\section{Setup}\label{sec:setup}

In this paper we consider graphs $G\in \mathcal{G}$, where
$\mathcal{G}$ is the set consisting of all finite, simple\footnote{By `simple' we mean here `without multiple edges between the same pair of vertices and without self-loops'}, connected, undirected, edge-weighted graphs $(V,E,\omega)$ with $n:= |V| \geq 2$ nodes. Here $E\subset V\times V$ and $\omega: E\to (0,\infty)$. 
Because $G\in\mathcal{G}$ is undirected, we identify $(i,j)\in E$ with $(j,i)\in E$. If we want to consider an unweighted graph, we view it as a weighted graph with $\omega=1$ on $E$.

Let $\mathcal{V}$ be the set of node functions $u:V\to \R$ and $\mathcal{E}$ the set of skew-symmetric\footnote{In the literature, the condition of skew-symmetry (i.e. $\varphi_{ij} = -\varphi_{ji}$) is often, but not always included in definitions of the edge function space. We follow that convention, but it does not hinder or help us, except for simplifying a few expressions, such as that of the divergence below.} edge functions $\varphi: E\to \R$. For $i\in V$, $u\in \mathcal{V}$, we write $u_i:=u(i)$ and for $(i,j)\in E$, $\varphi\in \mathcal{E}$ we write $\varphi_{ij}:=\varphi((i,j))$. To simplify notation, we extend each $\varphi\in\mathcal{E}$ to a function $\varphi\colon V^2 \to \R$ (without changing notation) by setting $\varphi_{ij} = 0$ if $(i,j)\not\in E$. The condition that $\varphi$ is skew-symmetric means that, for all $i,j\in V$, $\varphi_{ij}=-\varphi_{ji}$. Similarly, for the edge weights
we write $\omega_{ij}:=\omega((i,j))$ 
and we extend $\omega$ (without changing notation) to a function $\omega: V^2 \to [0,\infty)$ by setting $\omega_{ij} = 0$ if and only if $(i,j)\not\in E$. Because $G\in\mathcal{G}$ is undirected, we have for all $i,j\in V$, $\omega_{ij} = \omega_{ji}$. 

The \textit{degree} of node $i\in V$ is $\displaystyle d_i:=\sum_{j\in V} \omega_{ij}$ and the minimum and maximum degrees of the graph are defined as
$\displaystyle d_-:= \underset{1\leq i\leq n}\min d_i$  and $\displaystyle d_+:= \underset{1\leq i\leq n}\max d_i$,
respectively. Because $G\in\mathcal{G}$ is connected and $n\geq 2$, there are no isolated nodes and thus $d_-,d_+>0$.

For a node $i\in V$, we denote the set of its neighbours by
\begin{equation}\label{eq:neighbours}
\mathcal{N}(i) := \{j\in V: \omega_{ij}>0\}.
\end{equation}

For simplicity of notation we will assume that the nodes of a given graph $G\in\mathcal{G}$ are labeled such that $V=\{1, \ldots, n\}$. For definiteness and to avoid confusion we specify that we consider $0\not\in \N$, i.e. $\N=\{1, 2, 3, \ldots \}$, and when using the notation $A\subset B$ we allow for the possibility of $A=B$. For a node set $S\subset V$, we denote its characteristic function (or indicator function) by
\[
(\chi_S)_i := \begin{cases} 1 & \text{if } i\in S,\\ 0, & \text{otherwise}.\end{cases}
\]
If $S = \{i\}$, we can use the Kronecker delta to write\footnote{When beneficial for the readability, we will also write $\delta_{i,j}$ for $\delta_{ij}$.}: $\displaystyle (\chi_{\{i\}})_j = \delta_{ij} := \begin{cases} 1, &\text{if } i=j,\\0, & \text{otherwise.}\end{cases}
$

As justified in earlier work \cite{HeinAudibertvonLuxburg07,vanGennipBertozzi12,vanGennipGuillenOstingBertozzi14} we introduce the following inner products,
\[
\langle u, v \rangle_{\mathcal{V}} := \sum_{i\in V} u_i v_i d_i^r, \quad
\langle \varphi, \psi \rangle_{\mathcal{E}} := \frac12 \sum_{i,j\in V} \varphi_{ij} \psi_{ij} \omega_{ij}^{2q-1},
\]
for parameters $q\in[1/2,1]$ and $r\in[0,1]$\footnote{Note that the powers $2q-1$ and $1-q$ in the $\mathcal{E}$ inner product 
and in the gradient are zero for the admissible choices $q=\frac12$ and $q=1$ respectively. In these cases we define $\omega_{ij}^0=0$ whenever $\omega_{ij}=0$, so as not to make the $\mathcal{E}$ inner product 
(or the gradient) nonlocal on the graph.}. 
We define the gradient $\nabla: \mathcal{V} \to \mathcal{E}$ by, for all $i,j\in V$,
\[
(\nabla u)_{ij} := \begin{cases} \omega_{ij}^{1-q} (u_j-u_i), & \text{if } \omega_{ij}>0,\\ 0, &\text{otherwise,}\end{cases}
\]

Note that $\langle \cdot, \cdot\rangle_{\mathcal{V}}$ is indeed an inner product on $\mathcal{V}$ if $G$ has no isolated nodes (i.e. if $d_i>0$ for all $i\in V$), as is the case for $G\in\mathcal{G}$. Furthermore $\langle \cdot, \cdot \rangle_{\mathcal{E}}$ is an inner product on $\mathcal{E}$ (since functions in $\mathcal{E}$ are either only defined on $E$ or are required to be zero on $V^2\setminus E$, depending on whether we consider them as edge functions or as extended edge functions, as explained above).

Using these building blocks, we define the divergence as the adjoint of the gradient the and (graph) Laplacian as the divergence of the gradient, leading to\footnote{Note that for the divergence we have used the assumption that $\varphi$ is skew-symmetric}, for all $i\in V$,
\[
(\dvg \varphi)_i := \frac1{d_i^r} \sum_{j\in V} \omega_{ij}^q \varphi_{ji}, \quad
(\Delta u)_i :=\left(\dvg (\nabla u)\right)_i =  d_i^{-r} \sum_{j\in V} \omega_{ij} (u_i-u_j),
\]
as well as the following norms:
\begin{align*}
 \|u\|_{\mathcal{V}} &:= \sqrt{\langle u,u\rangle_{\mathcal{V}}}, \quad \|\varphi\|_{\mathcal{E}} := \sqrt{\langle \varphi, \varphi\rangle_{\mathcal{E}}},\\ 
\|u\|_{\mathcal{V},\infty} &:= \max\{|u_i|\colon i\in V\}, \quad \|\varphi\|_{\mathcal{E},\infty} := \max\{|\varphi_{ij}|\colon i,j \in V\}.
\end{align*}
Note that we indeed have, for all $u\in \mathcal{V}$ and all $\psi\in \mathcal{E}$,
\begin{equation}\label{eq:adjoint}
\langle \nabla u,\psi\rangle_{\mathcal{E}} = \langle u, \dvg\psi\rangle_{\mathcal{V}}.
\end{equation}
In \cite[Lemma 2.2]{vanGennipGuillenOstingBertozzi14} it is proven that, for all $u\in \mathcal{V}$,
\begin{equation}\label{eq:normineq}
d_-^{\frac{r}2} \|u\|_{\mathcal{V},\infty} \leq \|u\|_{\mathcal{V}} \leq \sqrt{\vol{V}} \|u\|_{\mathcal{V},\infty}.
\end{equation}

For a function $u \in \mathcal{V}$, we define its \textit{support} as 
$
\supp(u) := \{i\in V: u_i \neq 0\}.
$ 
The \textit{mass} of a function $u\in \mathcal{V}$ is
\[
\mathcal{M}(u):= \langle u, \chi_V\rangle_{\mathcal{V}} = \sum_{i\in V}d_i^r u_i,
\]
and the \textit{volume} of a node set $S\subset V$ is
\[
\vol{S}  := \mathcal{M}(\chi_S) = \| \chi_S \|_{\mathcal{V}}^2 = \sum_{i\in S} d_i^r.
\]
Note that, if $r=1$, then $\vol{S} = |S|$, where $|S|$ denotes the number of elements in $S$. Using \eqref{eq:adjoint}, we find the useful property that, for all $u\in \mathcal{V}$,
\begin{equation}\label{eq:Laplacemass}
\mathcal{M}(\Delta u) = \langle \Delta u, \chi_V\rangle_{\mathcal{V}} = \langle \nabla u, \nabla \chi_V\rangle_{\mathcal{E}} = 0.
\end{equation}

For $u\in \mathcal{V}$, define the \textit{average mass function} of $u$ as
\[
\mathcal{A}(u) := \frac{\mathcal{M}(u)}{\vol{V}} \chi_V.
\]
Note in particular that
\begin{equation}\label{eq:massu-Au}
\mathcal{M}(u-\mathcal{A}(u)) = 0.
\end{equation}

We also define the \textit{Dirichlet energy} of a function $u\in \mathcal{V}$, 
\begin{equation}\label{eq:Dirichletenergy}
\frac12 \|\nabla u\|_{\mathcal{E}}^2 = \frac{1}{4} \sum_{i,j\in V} \omega_{ij}(u_i - u_j)^2,
\end{equation}
and the \textit{total variation} of $u\in \mathcal{V}$,
\[
\TV(u) := \max\left\{\langle \dvg \varphi, u\rangle_{\mathcal{V}}: \varphi\in \mathcal{E}, \, \|\varphi\|_{\mathcal{E},\infty}\leq 1\right\} = \frac12 \sum_{i,j\in V} \omega_{ij}^q |u_i-u_j|.
\]

\begin{remark}
We have introduced two parameters, $q\in [1/2,1]$ and $r\in [0,1]$, in our definitions so far. As we will see later in this paper, the choice $q=1$ is the natural one for our purposes. In those cases where we do not require $q=1$, however, we do keep the parameter $q$ unspecified, because there are papers in the literature in which the choice $q=1/2$ is made, such as \cite{GilboaOsher2009}. One reason for the choice $q=1/2$ is that in that case $\omega_{ij}$ appears in the graph gradient, graph divergence, and graph total variation with the same power ($1/2$), allowing one to think of $\sqrt{\omega_{ij}}$ as analogous to a reciprocal distance. 

The parameter $r$ is the more interesting one of the two, as the choices $r=0$ and $r=1$ lead to two different graph Laplacians that appear in the spectral graph theory literature under the names combinatorial (or unnormalized) graph Laplacian and random walk (or normalized, or non-symmetrically normalized) graph Laplacian, respectively. Many of the results in this paper hold for all $r\in [0,1]$ and we will clearly indicate if and when further assumptions on $r$ are required

We note that, besides the graph Laplacian, also the mass of a function depends on $r$, whereas the total variation of a function does not depend on $r$, but does depend on $q$. The Dirichlet energy depends on neither parameter.

Unless we explicitly mention any further restrictions on $q$ or $r$, only the conditions $q\in [1/2,1]$ and $r\in [0,1]$ are implicitly assumed.
\end{remark}

\begin{definition}
Given a graph $G=(V,E,\omega)\in \mathcal{G}$, we define the following useful subsets of $\mathcal{V}$:
\begin{itemize}
\item the subset of node functions with a given mass $M\in \R$,
\begin{equation}\label{eq:VM}
\mathcal{V}_M := \{u\in \mathcal{V}: \mathcal{M}(u)= M\};
\end{equation}
\item the subset of nonnegative node functions,
\[
\mathcal{V}_+ := \{u \in \mathcal{V}: \forall i\in V\,\, u_i \geq 0\};
\]
\item the subset of $\{0,1\}$-valued binary node functions,
\begin{equation}\label{eq:Vb}
\mathcal{V}^b := \{u\in \mathcal{V}: \forall i\in V \ \ u_i \in \{0,1\} \};
\end{equation}
\item the subset of $\{0,1\}$-valued binary node functions with a given mass $M\geq 0$,
\[
\mathcal{V}^b_M := \mathcal{V}_M \cap \mathcal{V}^b;
\]
\item the subset of $[0,1]$-valued node functions,
\begin{equation}\label{eq:setK}
\mathcal{K} := \{u\in \mathcal{V}: \forall i\in V\ u_i\in [0,1]\},
\end{equation} 
\item the subset of $[0,1]$-valued node functions with a given mass $M\geq 0$,
\begin{equation}\label{eq:setKM}
\mathcal{K}_M :=  \mathcal{V}_M \cap \mathcal{K}.
\end{equation} 
\end{itemize}
\end{definition}
The space of zero mass node functions, $\mathcal{V}_0$ will play an important role, as it is the space of admissible `right hand side' functions in the Poisson problem \eqref{eq:Poisson}. Note that every $u\in \mathcal{V}^b$ is of the form $u=\chi_S$ for some $S\subset V$. 

Observe that for $M>\vol{V}$, $\mathcal{V}^b_M = \emptyset$. In fact, for a given finite graph there are only finitely many $M\in [0,\vol{V}]$ such that $\mathcal{V}_M^b \neq \emptyset$. For a given graph, we define the (finite) set of admissable masses as
\begin{equation}\label{eq:admissmass}
\mathfrak{M} := \{M\in [0,\vol{V}]: \mathcal{V}^b_M \neq \emptyset\}.
\end{equation}

Throughout the paper we will use the example of a star graph to illustrate various ideas, because it is amenable to analytical calculations. We therefore give its definition here and introduce the notation we will be using for it. 

\begin{definition}\label{def:bipartstar}
A (weighted) undirected simple graph $G=(V,E,\omega)$ is \textit{complete} if, for all $i,j\in V$, $i\neq j$ implies $\omega_{ij}>0$.

A (weighted) undirected simple graph $G=(V,E,\omega)$ is \textit{bipartite} if there are disjoint subsets $V_1$ and $V_2$ of $V$ such that $V=V_1 \cup V_2$ and for all $i,j\in V_1$ and for all $k,l\in V_2$, $\omega_{ij}=\omega_{kl} = 0$. In this case we write $V=V_1|V_2$.

A bipartite graph with node set $V=V_1|V_2$ is called a \textit{complete bipartite graph} if, for all $i\in V_1$ and for all $j\in V_2$, $\omega_{ij}>0$.

A (weighted) undirected simple graph $G=(V,E,\omega)$ is a (weighted) \textit{star graph} if it is a complete bipartite graph with $V=V_1|V_2$ and $|V_1|=1$ or $|V_2|=1$. The single node in $V_1$ or $V_2$, respectively, is the \textit{centre node} or \textit{internal node}. The other nodes, in $V_2$ or $V_1$, respectively, are \textit{leave nodes}.

For a (weighted) star graph we will use the notational convention that $1\in V$ is the centre node and $\{2, \ldots, n\}$ is the set of leaves, i.e. for all $i\in V$, $\omega_{ii} = 0$, for all $j\in V\setminus\{1\}$, $\omega_{1j}=1$, and if $i,j\in V\setminus\{1\}$, then $\omega_{ij}=0$.
\end{definition}
See Figure~\ref{fig:stargraph} for an example of a star graph. 

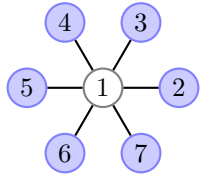
\begin{figure}[b]
\centering
\begin{tikzpicture} 
[inner sep=.8mm,
dot/.style={circle,draw=blue!50,fill=blue!20,thick},
redDot/.style={circle,draw=black!50,thick}] 
\node (1) at (0,0)[redDot] {1};
\node (2) at (1,0)[dot]{2}
	edge[-,thick] node[auto,swap] {} (1);
\node (3) at (0.5,0.866)[dot] {3}
	edge[-,thick] node[auto,swap] {} (1);
\node (4) at (-0.5,0.866)[dot]{4}
	edge[-,thick] node[above left,swap] {} (1);	
\node (5) at (-1,0)[dot]{5}
	edge[-,thick] node[above left,swap] {} (1);
\node (6) at (-0.5,-0.866)[dot]{6}
	edge[-,thick] node[above left,swap] {} (1);
\node (7) at (0.5,-0.866)[dot]{7}
	edge[-,thick] node[above left,swap] {} (1);
\end{tikzpicture}
\caption{An example of a star graph with six leave nodes} \label{fig:stargraph}
\end{figure}

The following lemma describes $\mathfrak{M}$ for the star graph and shows that the mass condition in $\mathcal{V}_M^b$ can be quite restrictive, especially if $r\neq 0$.

\begin{lemma}\label{lem:starmass}
Let $G=(V,E,\omega) \in\mathcal{G}$ be an unweighted star graph as in Definition~\ref{def:bipartstar} with $n\geq 3$ nodes and let $q=1$. Let $\mathfrak{M}$ be the set of admissable masses as in \eqref{eq:admissmass}, then
\[
\mathfrak{M} = \{0, 1, \ldots, n-1, (n-1)^r, (n-1)^r+1, \ldots, (n-1)^r + n -1\}.
\]
\begin{itemize}
\item If $(n-1)^r \not\in \N$ and $M\in \mathfrak{M}\cap \N$, then for all $\chi_S \in \mathcal{V}^b_M$ (with $S\subset V$), $1\not\in S$ and $|S|=M$.
\item If $(n-1)^r \not\in \N$ and $M\in \mathfrak{M}\cap \left(\R\setminus\N\right)$, then for all $\chi_S\in \mathcal{V}^b_M$, $1\in S$ and $|S\setminus\{1\}|=M-1$.
\item If $M \in \mathfrak{M} \cap [0, (n-1)^r)$, then for all $\chi_S\in \mathcal{V}_M^b$, $1\not\in S$.
\item If $M \in \mathfrak{M}\cap (n-1, \vol{V}]$, then for all $\chi_S\in \mathcal{V}_M^b$, $1 \in S$.
\item If $M \in \mathfrak{M}$ and $\chi_S, \chi_{\tilde S} \in \mathcal{V}_M^b$ are such that $\left(\chi_S\right)_1 = \left(\tilde \chi_S\right)_1$, then $|S|=|\tilde S|$.
\end{itemize}
\end{lemma}
\begin{proof}
Let $u\in \mathcal{V}_M^b$, then $\mathcal{M}(u) = (n-1)^r u_1 + \sum_{i=2}^n u_i$, from which the expression for $\mathfrak{M}$ immediately follows. 

If $(n-1)^r \not\in \N$ and $M\in \mathfrak{M}\cap \N$, then $M\in \{0, 1, \ldots, n-1\}$, hence $1\not\in S$ and $|S|=M$. If on the other hand $(n-1)^r \not\in \N$ and $M\in \mathfrak{M}\cap \left(\R\setminus\N\right)$, then $M\in \{(n-1)^r, (n-1)^r+1, \ldots, (n-1)^r+n-1\}$ and thus $1\in S$ and $|S\setminus\{1\}|=M-1$.

If $M\in \mathfrak{M}$ satisfies $M<(n-1)^r$, then $\mathcal{M}(\chi_S) = M$ implies $1\not\in S$. If on the other hand $M>n-1$ and $1 \not \in S$, then $\mathcal{M}(\chi_S) \leq n-1 < M$, hence $\chi_S\not\in \mathcal{V}_M^b$.

If $\mathcal{M}(\chi_S) = \mathcal{M}(\chi_{\tilde S})$ and $\left(\chi_S\right)_1 = \left(\chi_{\tilde S}\right)_1$, then
\[
|S| - |\tilde S| = \left(\chi_S\right)_1 - \sum_{i=2}^n \left(\chi_S\right)_i - \left(\chi_{\tilde S}\right)_1 - \sum_{i=2}^n \left(\chi_{\tilde S}\right)_i =  \sum_{i=2}^n \left(\chi_S\right)_i  - \sum_{i=2}^n \left(\chi_{\tilde S}\right)_i = \mathcal{M}(\chi_S) - \mathcal{M}(\chi_{\tilde S})=0.
\]
\end{proof}

\begin{remark}
If $r=0$, the assumptions in the first four bullet points of Lemma~\ref{lem:starmass} cannot be satisfied and the condition $M\in \mathfrak{M}$ is less restrictive than in the case $r\in (0,1]$.
\end{remark}

\section{Dirichlet and Poisson equations}\label{sec:DirPois}

\subsection{A comparison principle}

\begin{lemma}[Comparison principle I]\label{lem:compell}
Let $G=(V,E,\omega) \in \mathcal{G}$, let $V'$ be a proper subset of $V$, and let $u, v \in \mathcal{V}$ be such that, for all $i\in V'$,
$(\Delta u)_i \geq (\Delta v)_i$ and, for all $i\in V\setminus V'$, $u_i\geq v_i$. Then, for all $i\in V$, $u_i \geq v_i$.
\end{lemma}
\begin{proof}
The result follows as a special case of the comparision principle for uniformly elliptic partial differential equations on graphs with Dirichlet boundary conditions in \cite[Theorem 1]{manfredi2015nonlinear}. For completeness we provide the proof of this special case here. In particular, we will prove that if $w\in \mathcal{V}$ is such that, for all $i\in V'$, $(\Delta w)_i  \geq 0$, and, for all $i\in V\setminus V'$, $w_i\geq 0$, then
then for all $i\in V$, $w_i \geq 0$. Applying this to $w=u-v$ gives the desired result.

If $V'=\emptyset$, the result follows trivially. In what follows we assume that $V'\neq\emptyset$.

Define the set $U := \left\{i\in V:  w_i = \min_{j\in V} w_j\right\}$. Note that $U\neq \emptyset$. For a proof by contradiction, assume $\min_{j\in V} w_j < 0$, then $U\subset V'$. By assumption $V'\neq V$, hence $\emptyset \neq V\setminus V' \subset V\setminus U$. Let $i^*\in V\setminus U$. Since $G$ is connected, there is a path from $U$ to $i^*$\footnote{By a path from $U$ to $i^*$ we mean a finite sequence of nodes $\{i_j\}_{j=1}^k$, such that $i_1\in U$, $i_k = i^*$, and for all $j\in \{2, \ldots, k\}$, $(i_j, i_{j+1})\in E$.} . Fix such a path and let $k^*$ be the first node along this path such that $k^*\in V\setminus U$ and let $j^*\in U$ be the node immediately preceeding $k^*$ in the path. Then, for all $k\in V$, $(\nabla w)_{kj^*}  \leq 0$, and
$
(\nabla w)_{k^*j^*} = \omega_{k^*j^*}^{1-q} (w_{j^*}-w_{k^*}) < 0.
$
Thus
$
d_{j^*}^r (\Delta w)_{j^*} = \sum_{k\in V} \omega_{j^*k}^q (\nabla w)_{kj^*} < 0.
$
Since $j^* \in V'$, this contradicts one of the assumptions on $w$, hence $\min_{i\in V} w_i \geq 0$ and the result is proven.
\end{proof}

We will see a generalization of Lemma~\ref{lem:compell} as well as another comparison principle in Section~\ref{sec:comppinning}, but their proofs require some groundwork which is interesting in its own right as well. That is the topic of Section~\ref{sec:graphtransformation}.

\subsection{Equilibrium measures}

Let $G=(V,E,\omega)\in \mathcal{G}$. Given a proper\footnote{The subset $S$ is proper if $S\neq V$. Note that, by \eqref{eq:Laplacemass}, the equation $\Delta u = f$ on $V$ can only have a solution $u$ if $f$ has zero mass. If $S=V$, this necessary zero mass solvability condition is not satisfied by equation \eqref{eq:Laplacemass}.} subset $S\subset V$, consider the equation
\begin{equation}\label{eq:equilibrium}
\begin{cases}
(\Delta \nu^S)_i = 1, &\text{if } i\in S,\\
\nu^S_i = 0, &\text{if } i \in V\setminus S.
\end{cases}
\end{equation}

We recall some properties that are proven in \cite[Section 2]{BenditoCarmonaEncinas03}.

\begin{lemma}\label{lem:equimeasprops}
Let $G=(V,E,\omega)\in \mathcal{G}$. The following results and properties hold:
\begin{enumerate}
\item\label{item:posdefLapl} The Laplacian $\Delta$ is positive semidefinite on $\mathcal{V}$ and positive definite on $\mathcal{V}_0$.
\item\label{item:maximumprinc} The Laplacian satisfies a maximum principle: for all $u\in \mathcal{V}_+$,
$\displaystyle 
\max_{i\in V} (\Delta u)_i = \max_{i\in \supp(u)} (\Delta u)_i.
$
\item\label{item:suppprop} For each proper subset $S\subset V$, \eqref{eq:equilibrium} has a unique solution in $\mathcal{V}$. If $\nu^S$ is this solution, then $\nu^S \in \mathcal{V}_+$ and $\supp(\nu^S) = S$.
\item\label{item:subsetsnu} If $R \subset S$ are both proper subsets of $V$ and $\nu^R, \nu^S \in \mathcal{V}_+$ are the corresponding solutions of \eqref{eq:equilibrium}, then $\nu^S \geq \nu^R$.
\end{enumerate}
\end{lemma}
\begin{proof}
These results are proven in \cite[Section 2]{BenditoCarmonaEncinas03} for the case $r=0$. The same proofs, mutatis mutandis, work for general $r\in [0,1]$. Because the equilibrium measures play an important role in the current paper, however, we will provide our own proofs here, which deviate slightly from those in \cite[Section 2]{BenditoCarmonaEncinas03} in places.

To prove \ref{item:posdefLapl}, we note that, by \eqref{eq:adjoint}, for all $u\in V$, $\langle \Delta u, u\rangle_{\mathcal{V}} = \|\nabla u\|_{\mathcal{E}}^2 \geq 0$ and thus $\Delta$ is positive semidefinite on $\mathcal{V}$. Moreover, equality is achieved if and only if $\nabla u = 0$. Because $G$ is connected $\nabla u=0$ if and only if $u$ is constant. If $u\in \mathcal{V}_0$ then $u$ is constant if and only if $u=0$. Hence, if $u\in \mathcal{V}_0$ and $u\neq 0$, then $\|\nabla u\|_{\mathcal{E}}^2 > 0$ and thus $\Delta$ is positive definite on $\mathcal{V}_0$.

To prove \ref{item:maximumprinc}, let $u\in \mathcal{V}_+$. We first observe that the result follows trivially if $\supp(u)=V$. Hence we now assume that $\supp(u)\neq V$. 
If $j\in V\setminus\supp(u)$, then $(\Delta u)_j = -d_j^{-r} \sum_{k\in V} \omega_{jk} u_k \leq 0$. 
Hence $\max_{i\in V\setminus\supp(u)} (\Delta u)_i \leq 0$. Now let $l\in \supp(u)$ be such that, for all $k\in \supp(u)$, $u_l \geq u_k$. Then $(\Delta u)_l = d_l^{-r} \sum_{k\in V} \omega_{lk} (u_l - u_k) \geq d_l^{-r} \sum_{k\in V\setminus \supp(u)} \omega_{lk} (u_l-u_k) \geq 0$. Hence $\max_{i\in \supp(u)} (\Delta u)_i \geq 0 \geq \max_{i\in V\setminus\supp(u)} (\Delta u)_i$ and the result follows.

Let $S$ be a proper subset of $V$. To prove the uniqueness claim in \ref{item:suppprop}, assume that $\nu_1^S, \nu_2^S\in \mathcal{V}_+$ are both solutions of \eqref{eq:equilibrium}. Define $\nu := \nu_1^S-\nu_2^S$, then $\Delta \nu = 0$ on $S$ and $\nu=0$ on $S^c$. Let $V'=S$ and apply Lemma~\ref{lem:compell} twice, once with $u=\nu$, $v=0$ and once with $u=0$, $v=\nu$. This shows that $\nu=0$ and thus $\nu_1^S=\nu_2^S$.

Next we show that \eqref{eq:equilibrium} has a solution in $\mathcal{V}_+$. Let $S$ be a proper subset of $V$.
Alll norms on finite dimensional vector spaces are topologically equivalent and if we interpret $u\mapsto \frac12 \|\nabla u\|_{\mathcal{E}}^2$ as a function from the Euclidean space $\R^n$ to $\R$, it is continuous (being a polynomial in $n$ variables). Hence it is also continuous as a functional on $\mathcal{V}$. The set $\mathcal{V}_{+,1}(S) := \{u \in \mathcal{V}_+ \cap \mathcal{V}_1: \supp(u) \subset S\}$ interpreted as subset of $\R^n$ is closed and bounded and thus compact. Hence it is also compact as subset of $\mathcal{V}$. Thus there is a $u^*\in \mathcal{V}_{+,1}(S)$ such that, for all $u\in \mathcal{V}_{+,1}(S)$, $\frac12 \|\nabla u^*\|_{\mathcal{E}}^2\leq \frac12 \|\nabla u\|_{\mathcal{E}}^2$. In other words, $u^*$ is the solution to the minimization problem
\begin{align*}
&\min_{u\in V} \frac12 \|\nabla u\|_{\mathcal{E}}^2\\
&\text{subject to } \forall i\in V\,\, u_i \geq 0, \quad \forall j\in V\setminus S\,\, u_j = 0, \quad \text{and} \quad \mathcal{M}(u)=1.
\end{align*}
Thus $u^*$ satisfies the Karush-Kuhn-Tucker (KKT) optimality conditions \cite[Section 5.5.3]{boyd2004convex} \cite[Theorem 12.1]{NocedalWright99} for this minimization problem, which give us the existence of $\chi_1: V\to [0,\infty)$, $\chi_2: V\setminus S \to \R$, and $\chi_3\in \R$, such that, for all $i\in V$ and all $j\in V\setminus S$,
\begin{align}
d_i^r \left(\Delta u^*\right)_i - \sum_{k\in V} \left(\chi_1\right)_k \delta_{ik} + \sum_{k\in V\setminus S} \left(\chi_2\right)_k \delta_{ik} + \chi_3 d_k^r = 0,\label{eq:KKT}\\
u^*_i \geq 0, \quad \left(\chi_1\right)_i\geq 0, \quad \left(\chi_1\right)_i u^*_i = 0,\notag\\
u^*_j = 0, \quad
\mathcal{M}\left(u^*\right) = 1.\notag
\end{align}
Hence
\begin{align*}
0 &= \sum_{k\in V} \left(\chi_1\right)_k u^*_k = \sum_{k\in S} \left(\chi_1\right)_k u^*_k = \sum_{k\in S} d_k^r \left(\left(\Delta u^*\right)_k + \chi_3\right) u^*_k \\
&= \sum_{k\in V} d_k^r \left(\left(\Delta u^*\right)_k + \chi_3\right) u^*_k = \langle \Delta u^*, u^*\rangle_{\mathcal{V}} + \chi_3 \mathcal{M}\left(u^*\right).
\end{align*}
Thus, using \eqref{eq:adjoint}, we find $\chi_3 = - \|\nabla u^*\|_{\mathcal{E}}^2$.

Assume $i\in S$ and $u^*_i=0$, then $\left(\chi_1\right)_i > 0$. Moreover, $d_i^r \left(\Delta u^*\right)_i = -d_i^r \sum_{k\in V}\omega_{ik} u_k \leq 0$ and $\chi_3 d_i^r = -d_i^r \|\nabla u^*\|_{\mathcal{E}}^2 \leq 0$. Hence, by the first KKT condition above in \eqref{eq:KKT}, $-\left(\chi_1\right)_i \geq 0$, which is a contradiction. Hence, if $i\in S$, then $u^*_i > 0$. In that case the KKT conditions give $\left(\chi_1\right)_i=0$ and thus $\left(\Delta u^*\right)_i = \|\nabla u^*\|_{\mathcal{E}}^2$. We see that $\nu^S := \frac{u^*}{\|\nabla u^*\|_{\mathcal{E}}^2} \in \mathcal{V}_+$ is a solution of \eqref{eq:equilibrium}.

To prove the final statement in \ref{item:suppprop}, assume $\nu^S\in \mathcal{V}_+$ solves \eqref{eq:equilibrium}. Then clearly $\supp(\nu^S) \subset S$. Assume there is a $j\in S$ such that $\nu^S_j = 0$, then (as in the proof of property~\ref{item:maximumprinc} above) $(\Delta \nu^S)_j \leq 0$. which contradicts \eqref{eq:equilibrium}. Hence $S \subset \supp(u)$ and \ref{item:suppprop} is proven.

Note that by \ref{item:maximumprinc} we have that, for all $j\in V\setminus R$, $(\Delta \nu^R)_j \leq \max_{i\in \supp(\nu^R)} (\Delta \nu^R)_i = 1$. To prove \ref{item:subsetsnu}, we define $\tilde \nu := \nu^S-\nu^R$. Then $\tilde \nu = 0$ on $V\setminus S$, $\Delta\tilde \nu = 0$ on $R$ and $\Delta \tilde \nu = 1 - \Delta \nu^R \geq 0$ on $S\setminus R$. Hence, by Lemma~\ref{lem:compell}, we have that $\tilde \nu \geq 0$ on $V$.
\end{proof}

Using property~\ref{item:suppprop} in Lemma~\ref{lem:equimeasprops} we can now define the concept of the equilibrium measure of a node subset $S$.

\begin{definition}\label{def:equilibrium}
Let $G=(V,E,\omega)\in \mathcal{G}$. For any proper subset $S\subset V$, the \textit{equilibrium measure for S}, $\nu^S$, is the unique function in $\mathcal{V}_+$ which satisfies, for all $i\in V$, the equation in \eqref{eq:equilibrium}.
\end{definition}

The following lemmas give examples of explicitly constructed equilibrium measures on a bipartite graph and a star graph.

\begin{lemma}\label{lem:bipartequil}
Let $G=(V,E,\omega) \in \mathcal{G}$ be a bipartite graph with $V=V_1|V_2$. Let $S\subset V_1$ and let $\nu^S$ be the equilibrium measure for $S$, as in Definition~\ref{def:equilibrium}. Then $\nu^S_i = d_i^{r-1} (\chi_S)_i$.
\end{lemma}
\begin{proof}
Since $n\geq 2$ and $G$ is connected, $S$ is a proper subset of $V$. Per definition we have, for all $i\in S^c$, $\nu^S_i = 0$. In particular this holds for all $i\in V_2\subset S^c$, hence, for all $j\in S\subset V_1$ we compute
\[
1=(\Delta \nu^S)_j = d_j^{-r} \sum_{k\in V_2} \omega_{jk} \left(\nu^S_j - \nu^S_k\right) = d_j^{1-r} \nu^S_j.
\]
\end{proof}

\begin{lemma}\label{lem:starequil}
Let $G=(V,E,\omega) \in \mathcal{G}$ be an unweighted star graph with $n\geq 3$ nodes as in Definition~\ref{def:bipartstar}. If $j=1\in V$, then the equilibrium measure for $V\setminus\{1\}$, as defined in Definition~\ref{def:equilibrium}, is given by, for all $i\in V$,
\[
\nu^{V\setminus\{1\}}_i = \begin{cases} 0, &\text{if } i=1,\\ 1, &\text{if } i\neq 1. \end{cases}
\]
If $j\in V\setminus\{1\}$, the equilibrium measure for $V\setminus\{j\}$ is given by
\[
\nu^{V\setminus\{j\}}_i = \begin{cases} 0, &\text{if } i=j,\\ \vol{V}-1 = (n-1)^r+n-2, &\text{if } i=1,\\ \vol{V}=(n-1)^r+n-1, &\text{if } i\neq j \text{ and } i\neq 1. \end{cases}
\]
\end{lemma}
\begin{proof}
In the case where $j=1$, the result follows immediately from Lemma~\ref{lem:bipartequil} with $d_i^{r-1}=1$ for $i\neq 1$ and $r\in [0,1]$.

Next let $j\neq 1$, then
\begin{align*}
\left(\Delta \nu^{V\setminus\{j\}}\right)_1 &= (n-1)^{-r} \left((n-1) \nu^{V\setminus\{j\}}_1 - \nu^{V\setminus\{j\}}_j - \sum_{k\in V\setminus\{1,j\}} \nu^{V\setminus\{j\}}_k\right)\\
 &= (n-1)^{-r} \left((n-1)^{r+1} + (n-1)(n-2) - (n-2)(n-1)^r-(n-2)(n-1)\right)\\ 
&= (n-1) - (n-2) = 1,
\end{align*}
where we used that $d_1=n-1$. Moreover, if $i\neq 1 \neq j$, then
\[
\left(\Delta \nu^{V\setminus\{j\}}\right)_i = d_i^{-r} \left((n-1)^r+n-1-(n-1)^r-n+2\right) = 1,
\]
since $d_i=1$. Thus $\nu^{V\setminus\{j\}}$ solves \eqref{eq:equilibrium} for $S=V\setminus\{j\}$.

Finally we note that 
$
\vol{V} = d_1^r + \sum_{i\in V\setminus\{1\}} d_i^r = (n-1)^r + n-1.
$
\end{proof}

\subsection{Graph curvature}

We recall the concept of graph curvature, which was introduced in \cite[Section 3]{vanGennipGuillenOstingBertozzi14}.

\begin{definition}\label{def:graphcurvature}
Let $G\in \mathcal{G}$ and $S\subset V$. Then we define the {\em graph curvature} of the set $S$ by, for all $i\in V$,
\[
(\kappa_S^{q,r})_i := d_i^{-r} \begin{cases} \sum_{j\in V\setminus S} \omega_{ij}^q, &\text{if } i\in S,\\
-\sum_{j\in S} \omega_{ij}^q, &\text{if } i\in V\setminus S.
\end{cases}
\]
We are mainly interested in the case $q=1$ in this paper and in any given situation, if there are any restrictions on $r\in [0,1]$, they will be clear from the context. Hence for notational simplicity, we will write $\kappa_S := \kappa_S^{1,r}$. 

For future use we also define
\begin{equation}\label{eq:kappaS+}
\kappa_S^+ := \max_{i\in V} \left(\kappa_S\right)_i.
\end{equation}
\end{definition}

The following lemma collects some useful properties of the graph curvature.
\begin{lemma}
Let $G\in \mathcal{G}$, $S\subset V$, and let $\kappa_S^{q,r}$ and $\kappa_S$ be the graph curvatures from Definition~\ref{def:graphcurvature}.
Then
\noindent\begin{minipage}{0.42\linewidth}
\vspace{\abovedisplayskip}
\begin{equation}
\TV(\chi_S) = \langle\kappa_S^{q,r}, \chi_S\rangle_{\mathcal{V}} \label{eq:TVintermsofkappa}\vspace{\belowdisplayskip}
\end{equation}
\end{minipage}
\hspace{0.5cm} and
\begin{minipage}{0.42\linewidth}
\vspace{\abovedisplayskip}
\begin{equation}
\Delta \chi_S = \kappa_S,  \label{eq:Laplcurv}\vspace{\belowdisplayskip}
\end{equation}
\end{minipage}

Moreover, if $\kappa_S^+$ is as in \eqref{eq:kappaS+}, then $\kappa_S^+ = \max_{i\in S} \left(\kappa_S\right)_i$.
\end{lemma}
\begin{proof}
The properties in \eqref{eq:TVintermsofkappa} and \eqref{eq:Laplcurv} are proven in \cite[Section 3]{vanGennipGuillenOstingBertozzi14} and can be checked by a direct computation. Note that the latter requires $q=1$. The property for $\kappa_S^+$ follows from the fact that $\kappa_S$ is nonnegative on $S$ and nonpositive on $S^c$.
\end{proof}

We can use Lemma~\ref{lem:compell} to connect the equilibrium measures from \eqref{eq:equilibrium} with the graph curvature.

\begin{lemma}\label{lem:nuandkappa}
Assume $G=(V,E,\omega)\in \mathcal{G}$ and let $S$ be a proper subset of $V$. Let $\nu^S$ be the equilibrium measure for $S$ from \eqref{eq:equilibrium} and let $\kappa_S$ be the graph cuvature of $S$ (for $q=1$) and $\kappa_S^+$ its maximum value, as in Definition~\ref{def:graphcurvature}. Then, for all $i\in S$, $\nu^S_i \geq \left(\kappa_S^+\right)^{-1}$.
\end{lemma}
\begin{proof}
Define $x:= \min_{i\in S} \left(\kappa_S\right)_i^{-1} = \left(\kappa_S^+\right)^{-1}$. Since $G$ is connected and $S$ is a proper subset of $V$, $\max_{i\in S} \left(\kappa_S\right)_i > 0$, and hence $x$ is well-defined. Using \eqref{eq:Laplcurv}, we compute $\Delta\left(x \chi_S\right) = x \kappa_S \leq 1$ on $V$ (and in particular on $S$). Hence, for $i\in S$, $\left(\Delta\left(x \chi_S\right)\right)_i \leq 1 = \left(\Delta \nu^S_i\right)_i$. Furthermore, for $i\in V\setminus S$, $x \left(\chi_S\right)_i = 0 = \nu^S_i$. Thus, by Lemma~\ref{lem:compell}, for all $i\in S$, $x = x (\chi_S)_i \leq \nu^S_i$.
\end{proof}

\begin{remark}
We can use the equilibrium measure we computed for the bipartite graph and star graph in Lemma~\ref{lem:bipartequil} and Lemma~\ref{lem:starequil}, respectively, to illustrate the result from Lemma~\ref{lem:nuandkappa}. 

For the bipartite graph from Lemma~\ref{lem:bipartequil} and $S\subset V_1$, we compute, for all $i\in S$, $\left(\kappa_S\right)_i = d_i^{-r+1} = \left(\nu^S_i\right)^{-1}$. This shows that the result from Lemma~\ref{lem:nuandkappa} is sharp, in the sense that there is no greater lower bound for $\nu^S$ on $S$ which holds for every $G\in \mathcal{G}$.

For the star graph from Lemma~\ref{lem:starequil} we compute, for $i\in V\setminus\{1\}$, $\left(\kappa_{V\setminus\{1\}}\right)_i = d_i^{-r+1}  = 1$. It is not surprising that this is another occasion in which equality is achieved in the bound from Lemma~\ref{lem:nuandkappa}, as this situation is a special case of the bipartite graph result. The case when $j\neq 1$, however, shows that equality is not always achieved. In this case, if $i\in V\setminus\{j\}$, then $\left(\kappa_{V\setminus\{j\}}\right)_i = d_i^{-r} \omega_{ij}$. Since $\omega_{1j}=1$ and, if $i\neq 1$, $\omega_{ij}=0$, we have $\kappa_{V\setminus\{j\}}^+ = (n-1)^{-r}$, so $\nu^{V\setminus\{j\}} > \left(\kappa_{V\setminus\{j\}}^+\right)^{-1}$ on $V\setminus\{j\}$.
\end{remark}

\subsection{Green's functions}\label{sec:Greensfunctions}

Next we use the equilibrium measures to construct Green's functions for Dirichlet and Poisson problems, following the discussion in \cite[Section 3]{BenditoCarmonaEncinas03}; see also \cite{BenditoCarmonaEncinas00, ChungYau2000}.

In this section all the results assume the context of a given graph $G\in \mathcal{G}$. In this section and in some select places later in the paper we will also denote Green's functions by the symbol $G$. It will always be very clear from the context whether $G$ denotes a graph or a Green's function in any given situation.

\begin{definition}
For a given subset $S\subset V$, we denote by $\mathcal{V}(S)$ the set of all real-valued node functions whose domain is $S$.
Note that $\mathcal{V}(V)=\mathcal{V}$.

Given a nonempty, proper subset $S\subset V$ and a function $f \in \mathcal{V}(S)$, the \textit{(semi-homogeneous) Dirichlet problem} is to find $u\in \mathcal{V}$ such that, for all $i\in V$,
\begin{equation}\label{eq:Dirichlet}
\begin{cases}
(\Delta u)_i = f_i, &\text{if  } i\in S,\\
u_i = 0, &\text{if } i\in V\setminus S.
\end{cases}
\end{equation}

Given $k\in V$ and $f\in \mathcal{V}_0$, the \textit{Poisson problem} is to find $u\in \mathcal{V}$ such that, 
\begin{equation}\label{eq:Poisson}
\begin{cases}
\Delta u = f,\\
u_k = 0.
\end{cases}
\end{equation}
\end{definition}
\begin{remark}
Note that a general Dirichlet problem which prescribes $u=g$ on $V\setminus S$, for some $g\in \mathcal{V}(V\setminus S)$, can be transformed into a semi-homogeneous problem by considering the function $u-\tilde g$, where, for all $i\in S$, $\tilde g_i = 0$ and for all $i\in V\setminus S$, $\tilde g_i = g_i$.
\end{remark}

\begin{lemma}\label{lem:uniquesolution}
Let $S\subset V$ be a nonempty, proper subset, and $f\in \mathcal{V}(S)$. Then the Dirichlet problem \eqref{eq:Dirichlet} has at most one solution. Similarly, given $k\in V$ and $f\in \mathcal{V}_0$, the Poisson problem \eqref{eq:Poisson} has at most one solution.
\end{lemma}
\begin{proof}
Given two solutions $u$ and $v$ to the Dirichlet problem, we have $\Delta (u-v) = 0$ on $S$ and $u-v=0$ on $V\setminus S$. Since the graph is connected, this has as unique solution $u-v=0$ on $V$ (see the uniqueness proof in point~\ref {item:suppprop} of Lemma~\ref{lem:equimeasprops}, which uses the comparison principle of Lemma~\ref{lem:compell}). A similar argument proves the result for the Poisson problem.
\end{proof}

Next we will show that solutions to both the Dirichlet and Poisson problem exist, by explicitly constructing them using Green's functions.  

\begin{definition}
Let $S$ be a nonempty, proper subset of $V$. The function $G: V \times S \to \R$ is a \textit{Green's function for the Dirichlet equation}, \eqref{eq:Dirichlet}, if, for all $f\in \mathcal{V}(S)$, the function $u \in \mathcal{V}$ which is defined by, for all $i\in V$,
\begin{equation}\label{eq:Greensexpansion}
u_i  := \sum_{j\in S} d_j^r G_{ij} f_j,
\end{equation}
satisfies \eqref{eq:Dirichlet}. 

Let $k\in V$. The function $G: V \times V\to \R$ is a \textit{Green's function for the Poisson equation}, \eqref{eq:Poisson}, if, for all $f\in \mathcal{V}_0$, \eqref{eq:Poisson} is satisfied by the function $u\in \mathcal{V}$ which is defined by, for all $i\in V$, 
\begin{equation}\label{eq:Greensexpansion2}
u_i   := \sum_{j\in V} d_j^r G_{ij} f_j = \langle G_{i\bigcdot}, f\rangle_{\mathcal{V}},
\end{equation}
where, for all $i\in V$, $G_{i\cdot}: V \to \R$\footnote{We can rewrite the Green's function for the Dirichlet equation in terms of the $\mathcal{V}$ inner product as well, if we extend either $G_{i\cdot}$ or $f$ to be zero on $V\setminus S$ and extend the other function in any desired way to all of $V$.}.

In either case, for fixed $j\in S$ (Dirichlet) or fixed $j\in V$ (Poisson), we define $G^j: V\to \R$, by, for all $i\in V$,
\begin{equation}\label{eq:fixj}
G^j_i := G_{ij}.
\end{equation}
\end{definition}

\begin{lemma}
Let $S$ be a nonempty, proper subset of $V$ and let $G:V\times S \to \R$. Then $G$ is a Green's function for the Dirichlet equation, \eqref{eq:Dirichlet}, if and only if, for all $i\in V$ and for all $j\in S$,
\begin{equation}\label{eq:GreensproblemDirichlet}
\begin{cases}
(\Delta G^j)_i = d_j^{-r} \delta_{ij}, &\text{if } i\in S,\\
G^j_i = 0, &\text{if } i\in V\setminus S.
\end{cases}
\end{equation}

Let $k\in V$ and $G:V\times V \to \R$. Then $G$ is a Green's function for the Poisson equation, \eqref{eq:Poisson}, if and only if there is a $q\in \mathcal{V}$ which satisfies 
\begin{equation}\label{eq:qcondition}
\mathcal{M}(q) = -1
\end{equation}
and there is a $C\in \R$, such that $G$ satisfies, for all $i,j\in V$,
\begin{equation}\label{eq:GreensproblemPoisson}
\begin{cases}
(\Delta G^j)_i = d_j^{-r} \delta_{ij} + q_i\\
G^j_k = C.
\end{cases}
\end{equation}
\end{lemma}
\begin{proof}
For the Dirichlet case, let $u$ be given by \eqref{eq:Greensexpansion}, then, for all $i\in S$,
$
(\Delta u)_i = \sum_{j\in S} d_j^r (\Delta G^j)_i f_j.
$
If the function $G$ is a Green's function, then, for all $f\in \mathcal{V}(S)$ and for all $i\in S$, $(\Delta u)_i = f_i$. In particular, if we apply this to $f=\chi_{\{j\}}$ for $j\in S$, we find, for all $i, j\in S$,
$
d_j^r (\Delta G^j)_i = \delta_{ij}.
$
Moreover, for all $f\in \mathcal{V}(S)$ and for all $i\in V\setminus S$, $u_i=0$. Applying this again to $f=\chi_{\{j\}}$ for $j\in S$, we find for all $i\in V\setminus S$ that $d_j^r G^j_i = 0$. Hence, for all $i\in V\setminus S$, and for all $j\in S$, $G^j_i = 0$. This gives us \eqref{eq:GreensproblemDirichlet}.

Next assume $G$ satisfies \eqref{eq:GreensproblemDirichlet}. By substituting $G$ into \eqref{eq:Greensexpansion} we find that $u$ satisfies \eqref{eq:Dirichlet} and thus $G$ is a Green's function.

Now we consider the Poisson case and we let $u$ be given by \eqref{eq:Greensexpansion2}. Let $q$ satisfy \eqref{eq:qcondition}. If $G$ is a Green's function, then, for all $f\in \mathcal{V}_0$, $\Delta u = f$. Let $l_1, l_2 \in V$ and apply $\Delta u = f$ to $f=d_{l_1}^r \chi_{\{l_1\}} - d_{l_2}^r \chi_{\{l_2\}}$. It follows that, for all $i\in V$, $\left(\Delta G^{l_1}\right)_i - \left( \Delta G^{l_2}\right)_i = d_{l_1}^{-r} \left((\chi_{\{l_1\}}\right)_i - d_{l_2}^{-r} \left(\chi_{\{l_2\}}\right)_i$. In particular, if $l_1\neq i \leq l_2$ the right hand side in this equality is zero and thus for all $i\in V$, $j\mapsto \left(\Delta G^j\right)_i$ is constant on $V\setminus\{i\}$. In other words, there is a $q\in \mathcal{V}$, such that, for all $i\in V$ and for all $j\in V\setminus \{j\}$, $\left(\Delta G^j\right)_i = q_i$.

Next let $l\in V$ and apply $\Delta u = f$ to the function $f=\chi_{\{l\}} - \mathcal{A}\left(\chi_{\{l\}}\right)$. We compute $\left(\Delta u\right)_i = \left(\Delta G^k\right)_i d_k^r - \frac{d_k^r d_i^r}{\vol V} \left(\Delta G^i\right)_i - \frac{d_k^r}{\vol V} q_i (\vol V - d_i^r)$. Hence, if $l=i$, $\Delta u = f$ reduces to
$
\left(\Delta G^i\right)_i d_i^r \left(1-\frac{d_i^r}{\vol V}\right) - d_i^r q_i + \frac{d_i^{2r}}{\vol V} q_i = 1 - \frac{d_i^r}{\vol V}.
$ 
We solve this for $\left(\Delta G^i\right)_i$ to find $\left(\Delta G^i\right)_i = d_i^{-r} + q_i$. If $l\in V\setminus\{i\}$, $\Delta u = f$ reduces to
$
\left(\Delta G^k\right)_i \left(d_l^r-d_l^r + \frac{d_l^r d_i^r}{\vol V}\right) - \frac{d_l^r d_i^r}{\vol V}  \left(\Delta G^i\right)_i  = 0 - \frac{d_l^r}{\vol V}.
$ 
Using the expression for $\left(\Delta G^i\right)_i$ that we found above, we solve for $\left(\Delta G^k\right)_i$ to find $\left(\Delta G^i\right)_i = q_i$.

Combining the above, we find, for all $i,j\in V$,
$
(\Delta G^j)_i = d_j^{-r} \delta_{ij} + q_i.
$
Now we compute, for each $j \in V$,
\[
0 = \langle \Delta G^j, \chi_V \rangle_{\mathcal{V}} = \langle d_j^{-r} \chi_{\{j\}} + q, \chi_V\rangle_{\mathcal{V}} = 1+ \langle q, \chi_V\rangle_{\mathcal{V}} = 1+\mathcal{M}(q),
\]
thus $\mathcal{M}(q) = -1$.

The `boundary condition' $u_k=0$ for a fixed $k\in V$ in \eqref{eq:Poisson}, holds for all $f\in \mathcal{V}_0$. Applying this again for $f=d_{l_1}^{-r} \chi_{\{l_1\}} - d_{l_2}^{-r} \chi_{\{l_2\}}$ we find $G_k^{l_1} - G_k^{l_2} = 0$. Hence there is a constant $C\in \R$ such that, for all $j\in V$, $G^j_k = C$. This gives us \eqref{eq:GreensproblemPoisson}.

Next assume $G$ satisfies \eqref{eq:GreensproblemPoisson}. By substituting $G$ into \eqref{eq:Greensexpansion2} we find that $u$ satisfies \eqref{eq:Poisson}. In particular, remember that $f\in \mathcal{V}_0$. Thus, since $q$ does not depend on $j$ we have  
$\langle q, f \rangle_{\mathcal{V}} = 0$ and moreover $u_k = C \mathcal{M}(f) = 0$. Thus $G$ is a Green's function.
\end{proof}

\begin{remark}
Any choice of $q$ in \eqref{eq:GreensproblemPoisson} consistent with \eqref{eq:qcondition} will lead to a valid Green's function for the Poisson equation and hence to the same (and only) solution $u$ of the Poisson problem \eqref{eq:Poisson} via \eqref{eq:Greensexpansion2}. We make the following convenient choice: for all $i\in V$,
\begin{equation}\label{eq:qchoice}
q_i = -d_k^{-r} \delta_{ik}.
\end{equation}
In Lemma~\ref{lem:symmetry} below we will see that this choice of $q$ leads to a symmetric Green's function.

Also any choice of $C\in \R$ in \eqref{eq:GreensproblemPoisson} will lead to a valid Green's function for the Poisson equation. The function $\tilde G$ satisfies \eqref{eq:GreensproblemPoisson} with $C=\tilde C \in \R$ if and only if $\tilde G-\tilde C$ satisfies \eqref{eq:GreensproblemPoisson} with $C=0$. Hence in Lemma~\ref{lem:Greensfunctions} we will give a Green's function for the Poisson equation for the choice
\begin{equation}\label{eq:Cchoice}
C=0.
\end{equation}
\end{remark}

\begin{corol}
For a given nonempty, proper subset $S\subset V$, if there is a solution to \eqref{eq:GreensproblemDirichlet}, it is unique. Moreover, for given $k\in V$, $q\in \mathcal{V}_{-1}$, and $C\in \R$, if there is a solution to \eqref{eq:GreensproblemPoisson}, it is unique.
\end{corol}
\begin{proof}
Let $j\in S$ (or $j\in V$). If $G^j$ and $H^j$ both satisfy \eqref{eq:GreensproblemDirichlet} (or \eqref{eq:GreensproblemPoisson}), then $G^j-H^j$ satisfies a Dirichlet (or Poisson) problem of the form \eqref{eq:Dirichlet} (or \eqref{eq:Poisson}). Hence by a similar argument as in the proof of Lemma~\ref{lem:uniquesolution}, $G^j-H^j = 0$.
\end{proof}

\begin{lemma}\label{lem:Greensfunctions}
Let $S$ be a nonempty, proper subset of $V$. The function $G: V\times S \to \R$, defined by, for all $i\in V$ and all $j\in S$,
\begin{equation}\label{eq:GreenDirichlet}
G_{ij} = \frac{\nu^S_j}{\mathcal{M}(\nu^S) - \mathcal{M}(\nu^{S\setminus\{j\}})} \left(\nu^S_i - \nu^{S\setminus\{j\}}_i\right),
\end{equation}
is the Green's function for the Dirichlet equation, satisfying \eqref{eq:GreensproblemDirichlet}.

Let $k\in V$. The function $G: V\times V \to \R$, defined by, for all $i,j\in V$,
\begin{equation}\label{eq:GreenPoisson}
G_{ij} = \frac1{\vol{V}} \left( \nu^{V\setminus\{k\}}_i + \nu^{V\setminus\{j\}}_k - \nu^{V\setminus\{j\}}_i\right),
\end{equation}
is the Green's function for the Poisson equation, satisfying \eqref{eq:GreensproblemPoisson} with \eqref{eq:qchoice} and \eqref{eq:Cchoice}.
\end{lemma}
\begin{proof}
Remember the relation between $G_{ij}$ and $G^j_i$ from \eqref{eq:fixj}.

We start with the Dirichlet case. Let $j\in S$. If $i\in V\setminus S$, then $\nu^S_i = \nu^{S\setminus\{j\}}_i = 0$, hence the boundary condition is satisfied. Next we note that, for $i\in S$,
\begin{equation}\label{eq:DeltaGDirichlet}
\left(\Delta (\nu^S - \nu^{S\setminus\{j\}})\right)_i = \left(1-(\Delta \nu^{S\setminus\{j\}})_j\right) \delta_{ij}.
\end{equation}
Moreover,
\begin{align*}
\mathcal{M}(\nu^{S\setminus\{j\}}) &= \sum_{i\in S} d_i^r \nu^{S\setminus\{j\}}_i = \langle \Delta \nu^S, \nu^{S\setminus\{j\}}\rangle_{\mathcal{V}}\\
&= \langle \nu^S, \Delta \nu^{S\setminus\{j\}}\rangle_{\mathcal{V}} = \langle \nu^S, \chi_V\rangle_{\mathcal{V}} - \langle \nu^S, \chi_V-\Delta \nu^{S\setminus\{j\}}\rangle_{\mathcal{V}}\\
&= \mathcal{M}(\nu^S) - d_j^r \nu^S_j \left(1-(\Delta \nu^{S\setminus\{j\}})_j\right).
\end{align*}
Hence
$
\frac{\nu^S_j}{\mathcal{M}(\nu^S) - \mathcal{M}(\nu^{S\setminus\{j\}})} = d_j^{-r} \left(1-(\Delta \nu^{S\setminus\{j\}})_j\right),
$ 
which, combined with \eqref{eq:DeltaGDirichlet}, shows that, for all $i\in S$, $(\Delta G^j)_i = d_j^{-r}\delta_{ij}$. This proves the desired result in the Dirichlet case.

Next we consider the Poisson case. Since $\nu^{V\setminus\{k\}}_k = 0$, for all $j\in V$ the boundary condition $G_{k,j}=0$ is satisfied.

Let $j\in V$. Using \eqref{eq:Laplacemass}, we compute
$
0 = \langle \Delta \nu^{V\setminus\{j\}}, \chi_V\rangle_{\mathcal{V}} = \sum_{i\in V\setminus\{j\}} d_i^r + d_j^r (\Delta \nu^{V\setminus\{j\}})_j,
$ 
hence, for all $i\in V$,
\[
(\Delta \nu^{V\setminus\{j\}})_i =
\begin{cases}
1, &\text{if } i \neq j,\\
-d_j^{-r} \vol{V\setminus\{j\}}, &\text{if } i = j.
\end{cases}
\]
Note that
\begin{equation}\label{eq:volumeidentity}
d_j^{-r} \vol{V\setminus\{j\}} = d_j^{-r} (\vol{V} - d_j^r) = d_j^{-r} \vol{V} - 1.
\end{equation}

Since $\nu^{V\setminus\{j\}}_k$ is constant with respect to $i$, it does not contribute to $\Delta G^j$. Hence we consider, for all $i\in V$,
\begin{align*}
\left(\Delta \left(\nu^{V\setminus\{k\}} - \nu^{V\setminus\{j\}}\right)\right)_i &= \left(1+d_j^{-r} \vol{V\setminus\{j\}}\right) \delta_{ij} - \left(1+d_k^{-r} \vol{V\setminus\{k\}}\right) \delta_{ik}\\
&= \vol{V} (d_j^{-r} \delta_{ij} - d_k^{-r} \delta_{ik}),
\end{align*}
where we used \eqref{eq:volumeidentity}. This shows that, for all $i\in V$, $(\Delta G^j)_i = d_j^{-r} \delta_{ij} - d_k^{-r} \delta_{ik}$, which proves the result.
\end{proof}

\begin{remark}
Let $G$ be the Green's function from \eqref{eq:GreenPoisson} for the Poisson equation. As shown in Lemma~\ref{lem:Greensfunctions}, $G$ satisfies \eqref{eq:GreensproblemPoisson} with \eqref{eq:qchoice} and \eqref{eq:Cchoice}. Now let us try to find another Green's function satisfying \eqref{eq:GreensproblemPoisson} with \eqref{eq:Cchoice} and with a different choice of $q$. Fix $k\in V$ and define $\tilde q\in \mathcal{V}$, by, for all $i\in V$,
$
\tilde q_i := q_i + d_k^{-r} \delta_{ik}.
$
Then, by \eqref{eq:qcondition}, $\mathcal{M}(\tilde q) = 0$. Hence, using \eqref{eq:Greensexpansion2} with the Green's function $G$, we find a function $v\in \mathcal{V}$ which satisfies,
$\Delta v_i = \tilde q$ 
and 
$v_k = 0$. 
Hence, for all $i,j\in V$,
\[
\begin{cases}
(\Delta (G^j+v))_i = d_j^{-r} \delta_{ij} - d_k^{-r} \delta_{ik} + \tilde q_i,\\
(G^j+v)_k = 0.
\end{cases}
\]
So $G^j+v$ is the new Green's function we are looking for.
\end{remark}

\begin{lemma}\label{lem:symmetry}
Let $S$ be a nonempty, proper subset of $V$. If $G: V \times S \to \R$ is the Green's function for the Dirichlet equation satisfying \eqref{eq:GreensproblemDirichlet}, then $G$ is symmetric on $S\times S$, i.e., for all $i, j \in S$,
$
G_{ij} = G_{ji}.
$

Let $k\in V$. If $G: V \times V \to \R$ is the Green's function for the Poisson equation satisfying \eqref{eq:GreensproblemPoisson} with \eqref{eq:qchoice} (and any choice of $C\in \R$), then $G$ is symmetric, i.e., for all $i, j\in V$,
$
G_{ij} = G_{ji}.
$
\end{lemma}
\begin{proof}
Let $G: V \times S \to \R$ be the Green's function for the Dirichlet equation, satisfying \eqref{eq:GreensproblemDirichlet}. Let $u\in \mathcal{V}$ be such that $u=0$ on $V\setminus S$. Let $i\in V$, then
\[
\langle \Delta G^i, u\rangle_{\mathcal{V}} = \sum_{j, k \in V} \omega_{jk} (G^i_j-G^i_k) u_j = \sum_{j\in S \atop k\in V} \omega_{jk} (G^i_j-G^i_k) u_j = \sum_{j\in S} d_j^r d_i^{-r} \delta_{ji} u_j = u_i,
\]
where the third equality follows from
$
d_i^{-r} \delta_{ji} (\Delta G^i)_j = d_j^{-r} \sum_{k\in V} \omega_{jk} (G^i_j-G^i_k).
$ 
Now let $i,j\in S$ and use the equality above with $u=G^j$ to deduce
\[
G_{ij} = G^j_i = \langle \Delta G^i, G^j\rangle_{\mathcal{V}} = \langle G^i, \Delta G^j\rangle_{\mathcal{V}} = G^i_j = G_{ji}.
\]

Next we consider the Poisson case with Green's function $G: V \times V\to \R$, satisfying \eqref{eq:GreensproblemPoisson} with \eqref{eq:qchoice} and \eqref{eq:Cchoice}. Let $k\in V$ and $u\in \mathcal{V}$ with $u_k=0$. Then, similar to the computation above, for $i\in V$, we find
\[
\langle \Delta G^i, u\rangle_{\mathcal{V}} = \sum_{j\in V} (\Delta G^i)_j u_j d_j^r = \sum_{j\in V} \left(d_i^{-r} \delta_{ji}+q_j\right) u_j d_j^r  = u_i-\sum_{j\in V} d_k^{-r} \delta_{jk}  u_j d_j^r
= u_i - u_k = u_i,
\]
where we used 
\eqref{eq:qchoice}.
If we use the identity above with $u=G^j$, we obtain, for $i,j\in V$,
\[
G_{ij} = G^j_i = \langle \Delta G^i, G^j\rangle_{\mathcal{V}} = \langle G^i, \Delta G^j\rangle_{\mathcal{V}} = G^i_j = G_{ji},
\]
where we have applied that $G^j_k=G^i_k=0$.

Finally, if $\tilde G: V\times V \to \R$ satisfies \eqref{eq:GreensproblemPoisson} with \eqref{eq:qchoice} and with $C\neq 0$, then $\tilde G = G+C$ and hence $\tilde G$ is also symmetric.
\end{proof}

\begin{remark}
The symmetry property of the Green's function for the Dirichlet equation from Lemma~\ref{lem:symmetry} allows us to note that, if we write the equilibrium measure $\nu^S$ which solves the Dirichlet problem in \eqref{eq:equilibrium} in terms of the Green's function for the Dirichlet equation from \eqref{eq:GreenDirichlet}, using \eqref{eq:Greensexpansion}, we find the consistent relationship, for $i\in S$,
\[
\nu^S_i = \sum_{j\in S} d_j^r G_{ij}(\chi_S)_j = \sum_{j\in V} d_j^r G_{ji} (\chi_S)_j  = \sum_{j\in V} d_j^r \nu^S_i \frac{\nu^S_j - \nu^{S\setminus\{i\}}_j}{\mathcal{M}(\nu^S) - \mathcal{M}(\nu^{S\setminus\{i\}})} (\chi_S)_j = \nu_i^S.
\]
For the last equality, we used that, by property~\ref{item:suppprop} from Lemma~\ref{lem:equimeasprops}, $\nu^S_j - \nu^{S\setminus\{i\}}_j=0$ if $j\in V\setminus S$, hence the factor $(\chi_S)_j$ can be replaced by $(\chi_V)_j$ without changing the value of the summation.
\end{remark}

\begin{remark}
Combining property~\ref{item:suppprop} from Lemma~\ref{lem:equimeasprops} regarding the support of the equilibrium measure with \eqref{eq:GreenDirichlet}, we see that we could consistently extent the Green's function for the Dirichlet equation to a function $V\times V$, by setting $G_{ij}=0$ for $j\in V\setminus S$. In that case \eqref{eq:Greensexpansion} takes the same form as in the Poisson case, \eqref{eq:Greensexpansion2}. The defining properties \eqref{eq:GreensproblemDirichlet} will still hold, as will the symmetry property from Lemma~\ref{lem:symmetry} (now for all $i,j\in V$).

In this paper we will stick to the original domain $V\times S$ for the Green's function for the Dirichlet equation.
\end{remark}

In Appendix~\ref{sec:randomwalk} we give a random walk interpretation for the Green's function for the Poisson equation.

\section{The graph Ohta-Kawasaki functional}\label{sec:graphOKfunctional}

\subsection{A negative graph Sobolev norm and Ohta-Kawasaki}

In analogy with the negative $H^{-1}$ Sobolev norm (and underlying inner product) in the continuum (see for example \cite{Evans02,adams2003sobolev,brezis1999analyse}), we introduce the graph $H^{-1}$ inner product and norm.

\begin{definition}\label{def:H-1}
The \textit{$H^{-1}$ inner product} of $u, v\in \mathcal{V}_0$ is given by
\[
\langle u, v \rangle_{H^{-1}} := \langle \nabla \varphi, \nabla \psi \rangle_{\mathcal{E}},
\]
where $\varphi, \psi \in \mathcal{V}$ are any functions such that
$\Delta \varphi = u$ and $\Delta \psi = v$ hold on $V$.
\end{definition}

\begin{remark}
The zero mass conditions on $u$ and $v$ in Definition~\ref{def:H-1} are necessary and sufficient conditions for the solutions $\varphi$ and $\psi$ to the Poisson equations above to exist as we have seen in Section~\ref{sec:Greensfunctions}. These solutions are unique up to an additive constant. Note that the choice of this constant does not influence the value of the inner product. 
\end{remark}

\begin{remark}
It is useful to realise we can rewrite the inner product from Definition~\ref{def:H-1} as
\begin{equation}\label{eq:H-1innerrewrite}
\langle u, v \rangle_{H^{-1}} = \langle \varphi, \Delta \psi\rangle_{\mathcal{V}} = \langle \varphi, v \rangle_{\mathcal{V}} \quad \text{or} \quad \langle \varphi, v \rangle_{\mathcal{V}} = \langle \Delta \varphi, v \rangle_{H^{-1}}.
\end{equation}
\end{remark}

\begin{remark}
Note that for a connected graph the expression in Definition~\ref{def:H-1} indeed defines an inner product on $\mathcal{V}_0$, as $\langle u,u\rangle_{H^{-1}} = 0$ implies that $(\nabla \varphi)_{ij}=0$ for all $i,j\in V$ for which $\omega_{ij}>0$. Hence, by connectivity, $\varphi$ is constant on $V$ and thus $u=\Delta \varphi = 0$ on $V$. 

The $H^{-1}$ inner product then also gives us the \textit{$H^{-1}$ norm}:
\[
\|u\|_{H^{-1}}^2 := \langle u, u \rangle_{H^{-1}} = \|\nabla \varphi\|_{\mathcal{E}}^2 = \langle u, \varphi \rangle_{\mathcal{V}}.
\]
\end{remark}

Let $k\in V$. By \eqref{eq:massu-Au}, if $u\in \mathcal{V}$, then $u-\mathcal{A}(u) \in \mathcal{V}_0$, hence there exists a unique solution to the Poisson problem
\begin{equation}\label{eq:varphiinH-1}
\begin{cases}
\Delta \varphi = u-\mathcal{A}(u),\\
\varphi_k = 0,
\end{cases}
\end{equation}
which can be expressed using the Green's function from \eqref{eq:GreenPoisson}. We say that this solution $\varphi$ \textit{solves \eqref{eq:varphiinH-1} for $u$}. Because the kernel of $\Delta$ contains only the constant functions, the solution $\varphi$ for any other choice of $k$ will only differ by an additive constant. Hence the norm
\[
\|u-\mathcal{A}(u)\|_{H^{-1}}^2 = \|\nabla \varphi\|_{\mathcal{E}}^2 = \frac12 \sum_{i,j\in V} \omega_{ij} (\varphi_i-\varphi_j)^2,
\]
is independent of the choice of $k$. Note also that this norm in general \textit{does} depend on $r$, since $\varphi$ does. Contrast this with the Dirichlet energy in \eqref{eq:Dirichletenergy} which is independent of $r$. The norm does not depend on $q$.

Using the Green's function expansion from \eqref{eq:Greensexpansion2} for $\varphi$, with $G$ being the Green's function for the Poisson equation from \eqref{eq:GreenPoisson}, we can also write
\[
\|u-\mathcal{A}(u)\|_{H^{-1}}^2 = \langle u-\mathcal{A}(u), \varphi\rangle_{\mathcal{V}}^2 = \sum_{i,j\in V} \left(u_i-\mathcal{A}(u)\right) d_i^r G_{ij} d_j^r \left(u_j - \mathcal{A}(u)\right).
\]
Note that this expression seems to depend on the choice of $k$, via $G$, but by the discussion above we know in fact that it does not depend on $k$. This can also be seen as follows. A different choice for $k$, leads to an additive constant change in the function $G$, which leaves the norm unchanged, since $\sum_{i\in V} d_i^r \left(u_i-\mathcal{A}(u)\right) = 0$.

Let $W: \R \to \R$ be the double well potential defined by, for all $x\in \R$,
\begin{equation}\label{eq:doublewell}
W(x):= x^2 (x-1)^2.
\end{equation}
Note that $W$ has wells of equal depth located at $x=0$ and $x=1$. 

\begin{definition}
For $\e>0$, $\gamma\geq 0$ and $u\in\mathcal{V}$, we now define both the \textit{(epsilon) Ohta-Kawasaki functional} (or \textit{diffuse interface Ohta-Kawasaki functional})
\begin{equation}\label{eq:epsOK}
F_\e(u) := \frac12 \|\nabla u\|_{\mathcal{E}}^2 + \frac1\e \sum_{i\in V} W(u_i) + \frac\gamma2 \|u-\mathcal{A}(u)\|_{H^{-1}}^2,
\end{equation}
and the \textit{limit Ohta-Kawasaki functional} (or \textit{sharp interface Ohta-Kawasaki functional})
\begin{equation}\label{eq:limitOK}
F_0(u) := \TV(u) + \frac\gamma2 \|u-\mathcal{A}(u)\|_{H^{-1}}^2.
\end{equation}
\end{definition}
The nomenclature and notation is justified by the fact that $F_0$ (with its domain restricted to $\mathcal{V}^b$, see \eqref{eq:Vb}) is the $\Gamma$-limit of $F_\e$ for $\e\to 0$ (this is shown by a straightforward adaptation of the results and proofs in \cite[Section 3]{vanGennipBertozzi12}; see Appendix~\ref{sec:gammaconvergence}).

There are two minimization problems of interest here:
\noindent\begin{minipage}{0.5\linewidth}
\vspace{\abovedisplayskip}
\begin{equation}
\min_{u\in \mathcal{V}_M}F_e(u),  \label{eq:minimprobsFe}\vspace{\belowdisplayskip}
\end{equation}
\end{minipage}
\begin{minipage}{0.5\linewidth}
\vspace{\abovedisplayskip}
\begin{equation}
\min_{u\in \mathcal{V}^b_M}F_0(u), \label{eq:minimprobsF0}\vspace{\belowdisplayskip}
\end{equation}
\end{minipage}
\vspace{\belowdisplayskip}
for a given $M\in \R$ for the first problem and a given $M\in \mathfrak{M}$ for the second. In this paper we will mostly be concerned with the second problem, \eqref{eq:minimprobsF0}.

The following lemma describes a symmetry in $F_0$ when the underlying graph is the star graph we have seen before; in a sense it is an extension of the last statement in Lemma~\ref{lem:starmass}. It will come in handy later.

\begin{lemma}\label{lem:starmass2}
Let $G=(V,E,\omega)$ be an unweighted star graph as in Definition~\ref{def:bipartstar} with $n\geq 3$ nodes and let $q=1$. Let $\mathfrak{M}$ be the set of admissable masses as in \eqref{eq:admissmass}.
If $S, \tilde S\subset V$ are such that $|S|=|\tilde S|$ and $\left(\chi_S\right)_1 = \left( \chi_{\tilde S}\right)_1$, then $F_0\left(\chi_S\right) = F_0\left(\chi_{\tilde S}\right)$, where $F_0$ is the limit Ohta-Kawasaki functional from \eqref{eq:limitOK}.
\end{lemma}
\begin{proof}
Let $S, \tilde S\subset V$ be such that $|S|=|\tilde S|$ and $\left(\chi_S\right)_1 = \left( \chi_{\tilde S}\right)_1$. Because the unweighted star graph $G$ is symmetric under permutations of its leaves (i.e. the nodes $\{2, \ldots, n\}$) for any $u\in \mathcal{V}^b$ the value of $F_0(u)$ depends only on the value of $u_1$ and the number of leaves $i$ for which $u_i=0$. Hence $F_0(\chi_S) = F_0(\chi_{\tilde S})$.
\end{proof}

\subsection{Ohta-Kawasaki in spectral form}\label{sec:GraLapspectrum}

Because of the role the graph Laplacian plays in the Ohta-Kawasaki energies, it it useful to consider its spectrum. As is well known (see for example \cite{Chung97,Luxburg:2007}), for any $r\in[0,1]$, the eigenvalues of $\Delta$, which we will denote by
\begin{equation}\label{eq:eigenvalues}
0=\lambda_0 \leq \lambda_1 \leq \ldots \leq \lambda_{n-1},
\end{equation}
are real and nonnegative. The multiplicity of $0$ as eigenvalue is equal to the number of connected components of the graph and the corresponding eigenspace is spanned by the indicator functions of those components. If $G\in \mathcal{G}$, then $G$ is connected, and thus, for all $m\in \{1, \ldots, n-1\}$, $\lambda_m>0$. We consider a set of corresponding $\mathcal{V}$-orthonormal eigenfunctions $\phi^m\in \mathcal{V}$, i.e., for all $m,l \in \{0, \ldots, n-1\}$,
\begin{equation}\label{eq:eigenfunctions}
\Delta \phi^m = \lambda_m \phi^m, \quad \text{and} \quad \langle \phi^m, \phi^l\rangle_{\mathcal{V}} = \delta_{ml},
\end{equation}
where $\delta_{ml}$ denotes the Kronecker delta. Note that, since $\Delta$ and $\langle \cdot, \cdot \rangle_{\mathcal{V}}$ depend on $r$, but not on $q$, so do the eigenvalues $\lambda_m$ and the eigenfunctions $\phi^m$.

For definiteness we choose\footnote{As opposed to the equally valid choice $\phi^0 =-(\vol{V})^{-1/2}.$}
\begin{equation}\label{eq:zerotheigvec}
\phi^0 := (\vol{V})^{-1/2} \chi_V.
\end{equation}
The eigenfunctions form an $\mathcal{V}$-orthonormal basis for $\mathcal{V}$, hence, for any $u\in \mathcal{V}$, we have
\begin{equation}\label{eq:expansion}
u = \sum_{m=0}^{n-1} a_m \phi^m, \quad \text{where } a_m := \langle u, \phi^m\rangle_{\mathcal{V}}.
\end{equation}
As an example we revisit the star graph from Definition~\ref{def:bipartstar}.
\begin{lemma}\label{lem:starspectrum}
Let $G=(V,E,\omega)$ be an unweighted star graph as in Definition~\ref{def:bipartstar} with $n\geq 3$ nodes. The eigenvalues are\footnote{If all the edges are given the same weight $\omega>0$ instead of $1$, it is quickly checked that all eigenvalues are multiplied by $\omega^{1-r}$, because the Laplacian is multiplied by the same factor. Since in that case the factor $d_i^r$ in the $\mathcal{V}$-inner product changes by a factor $\omega^r$, the eigenfunctions all acquire a multiplicative factor $\omega^{-r/2}$.}
\[
\lambda_0 = 0, \qquad \lambda_m = 1 \quad (m\in \{1, \ldots, n-2\}), \qquad \lambda_{n-1} = (n-1)^{1-r}+1.
\]
A corresponding $\mathcal{V}$-orthonormal system of eigenfunctions is given by, for $ m\in \{1, \ldots, n-2\}$ and $i\in \{1, \ldots, n\}$,
\begin{align*}
\phi^0 &= \left[(n-1)^{1-r}+n-1\right]^{-1/2} \chi_V,\\
\phi^m_i &= \left[(n-m-1)^2+n-m-1\right]^{-1/2} \begin{cases} 0, &\text{if }1\leq i\leq m,\\ n-m-1, &\text{if } i=m+1,\\ -1, &\text{if } m+2 \leq i \leq n, 
\end{cases}\\
\phi^{n-1}_i &= \left[(n-1)^{2-r}+n-1\right]^{-1/2} \begin{cases} (n-1)^{1-r}, &\text{if } i=1,\\ -1, &\text{if } 2 \leq i\leq n,
\end{cases}
\end{align*}
where the subscript $i$ 
indicates the component of the vector.
\end{lemma}
\begin{proof}
The eigenvalues and eigenvectors were found following a computation similar to that in \cite[Section 6.2]{vanGennipGuillenOstingBertozzi14}, but for this proof a direct computation suffices to show that $\langle \phi^m, \phi^l\rangle_{\mathcal{V}} = \delta_{ml}$ and $\Delta \phi^m = \lambda_m \phi^m$.
\end{proof}

\begin{lemma}\label{lem:varphi}
Let $u\in \mathcal{V}$, $k\in V$, then $\varphi$ satisfies $\Delta \varphi = u-\mathcal{A}(u)$, if and only if
\begin{equation}\label{eq:varphieigenexpansion}
\varphi = \mathcal{A}(\varphi) + \sum_{m=1}^{n-1}\lambda_m^{-1} \langle u, \phi^m\rangle_{\mathcal{V}}\, \phi^m.
\end{equation}
\end{lemma}
\begin{proof}
Let $\varphi$ satisfy $\Delta \varphi = u - \mathcal{A}(u)$. Using expansions as in \eqref{eq:expansion} for $\varphi$ and $u-\mathcal{A}(u)$, we have
\[
\Delta \left(\sum_{m=0}^{n-1} a_m \phi^m\right) = \Delta \varphi = u-\mathcal{A}(u) = \sum_{m=0}^{n-1} b_m \phi^m,
\]
where, for all $m\in \{0, \ldots, n-1\}$, $a_m := \langle \varphi, \phi^m\rangle_{\mathcal{V}}$ and $b_m := \langle u-\mathcal{A}(u), \phi^m\rangle_{\mathcal{V}}$.
Hence
$
 \sum_{m=0}^{n-1} a_m \lambda_m \phi^m = \sum_{m=0}^{n-1} b_m \phi^m
$
and therefore, for any $l \in \{0, \ldots, n-1\}$,
\[
a_l \lambda_l = \left\langle \sum_{m=0}^{n-1} a_m \lambda_m \phi^m, \phi^l\right\rangle_{\mathcal{V}} = \left\langle \sum_{m=0}^{n-1} b_m \phi^m, \phi^l \right\rangle_{\mathcal{V}} = b_l.
\]
In particular, if $m\geq 1$, then $a_m = \lambda_m^{-1} b_m$. Because $\lambda_0=0$, the identity above does not constrain $a_0$.
Because, for $m\geq 1$, $\langle \phi^0, \phi^m\rangle=0$, it follows that, for $m\in \{1, \ldots, n-1\}$,
\begin{align}
b_m = \langle u-\mathcal{A}(u), \phi^m\rangle_{\mathcal{V}} &= \langle u, \phi^m\rangle_{\mathcal{V}} - \frac{\mathcal{M}(u)}{\vol{V}} \langle \chi_V, \phi^m\rangle = \langle u, \phi^m\rangle_{\mathcal{V}} - \frac{\mathcal{M}(u)}{\vol{V}} (\vol{V})^{1/2} \langle \phi^0, \phi^m\rangle \notag \\ &= \langle u, \phi^m\rangle_{\mathcal{V}} \label{eq:Aphimiszero}.
\end{align}
and therefore, for all $m\in \{1, \ldots, n-1\}$,
$
a_m = \lambda_m^{-1} \langle u, \phi^m\rangle_{\mathcal{V}}.
$ 
Furthermore
\[
a_0 = \langle \varphi, \phi^0\rangle_{\mathcal{V}} = (\vol{V})^{-1/2} \langle \varphi, \chi_V\rangle_{\mathcal{V}} = (\vol{V})^{-1/2} \mathcal{M}(\varphi).
\]
Substituting these expressions for $a_0$ and $a_m$ into the expansion of $\varphi$, we find that $\varphi$ is as in \eqref{eq:varphieigenexpansion}.

Conversely, if $\varphi$ is as in \eqref{eq:varphieigenexpansion}, a direct computation shows that $\Delta \varphi = u-\mathcal{A}(u)$.
\end{proof}

\begin{remark}\label{rem:Moore-Penrose}
From Lemma~\ref{lem:varphi} we see that we can write $\varphi-\mathcal{A}(\varphi) = \Delta^\dagger (u-\mathcal{A}(u))$, where $\Delta^\dagger$ is the Moore-Penrose pseudoinverse of $\Delta$ \cite{dresden1920,bjerhammer1951application,penrose_1955}.
\end{remark}

\begin{lemma}\label{lem:TVrewrite}
Let $q\in [1/2,1]$, $S\subset V$, and let $\kappa_S^{q,r}$, $\kappa_S$ be the graph curvatures from Definition~\ref{def:graphcurvature}, then
\[
\TV(\chi_S) = \sum_{m=1}^{n-1} \langle \kappa_S^{q,r}, \phi^m\rangle_{\mathcal{V}}\, \langle \chi_S, \phi^m\rangle_{\mathcal{V}}.
\]

Furthermore, if $q=1$, then
\begin{align}
\TV(\chi_S) &= \sum_{m=1}^{n-1} \lambda_m \langle \chi_s, \phi^m\rangle_{\mathcal{V}}^2\label{eq:TVchiphi}\\
&= \sum_{m=1}^{n-1} \lambda_m^{-1} \langle \kappa_s, \phi^m\rangle_{\mathcal{V}}^2.\label{eq:TVcurvphi}
\end{align}
\end{lemma}
\begin{proof}
Using an expansion as in \eqref{eq:expansion} for $\chi_S$ together with \eqref{eq:TVintermsofkappa}, we find
\[
\TV(\chi_S) =  \langle\kappa_S^{q,r}, \sum_{m=0}^{n-1} \langle \chi_S, \phi^m\rangle_{\mathcal{V}}\, \phi^m \rangle_{\mathcal{V}} = 
\sum_{m=0}^{n-1} \langle \kappa_S^{q,r}, \phi^m\rangle_{\mathcal{V}}\, \langle \chi_S, \phi^m\rangle_{\mathcal{V}} = \sum_{m=1}^{n-1} \langle \kappa_S^{q,r}, \phi^m\rangle_{\mathcal{V}}\, \langle \chi_S, \phi^m\rangle_{\mathcal{V}},
\]
where the last equality follows from
\[
\langle \kappa_S^{q,r}, \phi^0\rangle_{\mathcal{V}} = (\vol{V})^{-1/2} \left(\sum_{i\in S} \sum_{j\in S^c} \omega_{ij}^q (\chi_V)_i - \sum_{i\in S^c} \sum_{j\in S} \omega_{ij}^q (\chi_V)_i \right) = 0.
\]

Moreover, we use \eqref{eq:Laplcurv} to find
\[
 \langle \chi_S, \lambda_m \phi^m\rangle_{\mathcal{V}} =  \langle \chi_S, \Delta \phi^m\rangle_{\mathcal{V}} =  \langle \Delta \chi_S, \phi^m\rangle_{\mathcal{V}} = \langle \kappa_s, \phi^m\rangle_{\mathcal{V}},
\]
hence
\begin{equation}\label{eq:curvtochi}
\langle\chi_S, \phi^m\rangle_{\mathcal{V}} = \lambda_m^{-1} \langle \kappa_s, \phi^m\rangle_{\mathcal{V}},
\end{equation}
If $q=1$, such that $\kappa_S^{q,r} = \kappa_S$, then \eqref{eq:TVchiphi} and \eqref{eq:TVcurvphi} follow.
\end{proof}

\begin{lemma}\label{lem:H-1rewrite}
Let $q\in [1/2,1]$, $S\subset V$, and let $\kappa_S$ be the graph curvature (with $q=1$) from Definition~\ref{def:graphcurvature}, then
\begin{align*}
\|\chi_S-\mathcal{A}(\chi_S)\|_{H^{-1}}^2 &= \sum_{m=1}^{n-1} \lambda_m^{-1} \langle \chi_S, \phi^m\rangle_{\mathcal{V}}^2\\
&= \sum_{m=1}^{n-1} \lambda_m^{-3} \langle \kappa_S, \phi^m\rangle_{\mathcal{V}}^2.
\end{align*}
\end{lemma}
\begin{proof}
Let $k\in V$ and let $\varphi \in \mathcal{V}$ solve
\[
\begin{cases}
\Delta \varphi = \chi_S-\mathcal{A}(\chi_S),\\
\varphi_k = 0.
\end{cases}
\]
Using Lemma~\ref{lem:varphi}, we have
$
\varphi - \mathcal{A}(\varphi) = \sum_{m=1}^{n-1}\lambda_m^{-1} \langle \chi_S, \phi^m\rangle_{\mathcal{V}}\, \phi^m.
$ 
Because
$
\langle \mathcal{A}(\varphi), \chi_S - \mathcal{A}(\chi_S)\rangle_{\mathcal{V}} = 0,
$ 
we have
\[
\|\chi_S-\mathcal{A}(\chi_S)\|_{H^{-1}}^2 = \langle \varphi - \mathcal{A}(\varphi), \chi_S-\mathcal{A}(\chi_S)\rangle_{\mathcal{V}} = \sum_{m=1}^{n-1} \lambda_m^{-1} \langle \chi_S, \phi^m\rangle_{\mathcal{V}}\,  \langle \phi^m, \chi_S-\mathcal{A}(\chi_S)\rangle_{\mathcal{V}}.
\]
As in \eqref{eq:Aphimiszero} (with $u$ replaced by $\chi_S$), we have, for $m\geq 1$, $\langle \phi^m, \mathcal{A}(\chi_S)\rangle_{\mathcal{V}} = 0$, and thus
\[
\|\chi_S-\mathcal{A}(\chi_S)\|_{H^{-1}}^2 = \sum_{m=1}^{n-1} \lambda_m^{-1} \langle \chi_S, \phi^m\rangle_{\mathcal{V}}^2.
\]

We use \eqref{eq:curvtochi} to write
$
\langle \chi_S, \phi^m\rangle_{\mathcal{V}}^2 = \lambda_m^{-2} \langle \kappa_S, \phi^m\rangle_{\mathcal{V}}^2,
$ 
and therefore
\[
\|\chi_S-\mathcal{A}(\chi_S)\|_{H^{-1}}^2 = \sum_{m=1}^{n-1} \lambda_m^{-3} \langle \kappa_S, \phi^m\rangle_{\mathcal{V}}^2.
\]
\end{proof}

\begin{remark}
Note that $\|\chi_S-\mathcal{A}(\chi_S)\|_{H^{-1}}^2$ is independent of $q$ and thus the results from Lemma~\ref{lem:H-1rewrite} hold for all $q\in [1/2,1]$. However, the formulation involving the graph curvature relies on \eqref{eq:curvtochi} and thus on the identiy \eqref{eq:Laplcurv} which holds for $\kappa_S$ only, not for any $\kappa_S^{q,r}$. If $q\neq 1$ this leads to the somewhat unnatural situation of using $\kappa_S$ (which corresponds to the case $q=1$) in a situation where $q\neq 1$. Hence the curvature formulation in Lemma~\ref{lem:H-1rewrite} is more natural, in this sense, when $q=1$.
\end{remark}

\begin{corol}\label{cor:OKexpressions}
Let $q=1$, $S\subset V$, and let $F_0$ be the limit Ohta-Kawasaki functional from \eqref{eq:limitOK}, then
\begin{align}
F_0(\chi_S) &= \sum_{m=1}^{n-1} \left(\lambda_m + \gamma \lambda_m^{-1}\right) \langle \chi_s, \phi^m\rangle_{\mathcal{V}}^2 \label{eq:OKexpression2}\\
&=\sum_{m=1}^{n-1} \left(\lambda_m^{-1} + \gamma \lambda_m^{-3}\right) \langle \kappa_s, \phi^m\rangle_{\mathcal{V}}^2.\notag
\end{align}
\end{corol}
\begin{proof}
This follows directly from the definition in \eqref{eq:limitOK} and Lemmas~\ref{lem:TVrewrite} and~\ref{lem:H-1rewrite}.
\end{proof}

Corollary~\ref{cor:OKexpressions} allows us to explicitly give the Ohta-Kawasaki functional for our star graph example from Definition~\ref{def:bipartstar}.
\begin{lemma}\label{lem:starOKexpress}
Let $G=(V,E,\omega)$ be an unweighted star graph as in Definition~\ref{def:bipartstar} with $n\geq 3$ and let $q=1$. Let $S\subset V$. For $l\in \N$ define $N_l(S) := \left|\{i\in S: i\geq l\}\right|$. Then
\begin{align*}
F_0(\chi_S) &= \sum_{l=2}^{n-1} \frac{1+\gamma}{(n-l)(n-l+1)} \left[(n-l) \left(\chi_S\right)_l - N_{l+1}(S)\right]^2\\
& \hspace{0.1cm} + \frac1{n-1} \left(1+\frac\gamma{\left((n-1)^{1-r}+1\right)^2}\right) \left[(n-1)\left(\chi_S\right)_1 - N_2(S)\right]^2.
\end{align*}
\end{lemma}
\begin{proof}
This follows by combining the eigenfunctions and eigenvalues we found in Lemma~\ref{lem:starspectrum} with \eqref{eq:OKexpression2}. Define, for $k\in \N$, $I(k) := \{i\in \N: k\leq i\leq n\}$. Then we compute, for $m\in \{1, \ldots, n-2\}$,
\begin{align*}
\langle \chi_S, \phi^m\rangle_{\mathcal{V}}^2 &= \frac1{(n-m-1)(n-m)} \sum_{i,j\in S} \left[(n-m-1)^2 \delta_{i,m+1}\delta_{j,m+1} + \left(\chi_{I(m+2)}\right)_i \left(\chi_{I(m+2)}\right)_j\right. \\&\hspace{4cm} \left. - (n-m-1) \left(\delta_{i,m+1} \left(\chi_{I(m+2)}\right)_j + \delta_{j,m+1} \left(\chi_{I(m+2)}\right)_i\right)\right]\\
&= \frac{n-m-1}{n-m} \left(\chi_S\right)_{m+1} + \frac1{(n-m-1)(n-m)} \left(\sum_{i=m+2}^n \left(\chi_S\right)_i\right)^2\\ &\hspace{4cm} - \frac2{n-m} \left(\chi_S\right)_{m+1} \sum_{i=m+2}^n \left(\chi_S\right)_i,\\
\langle \chi_S, \phi^{n-1}\rangle_{\mathcal{V}}^2 &= \frac1{(n-1)\left((n-1)^{1-r}+1\right)} \sum_{i,j\in S} d_i^rd_j^r \left[(n-1)^{2-2r} \delta_{i1}\delta_{j1} + \left(\chi_{I(2)}\right)_i \left(\chi_{I(2)}\right)_j \right.\\ &\hspace{4cm} \left.- (n-1)^{1-r} \left(\delta_{i1} \left(\chi_{I(2)}\right)_j + \delta_{j1} \left(\chi_{I(2)}\right)_i\right)\right]\\
&= \frac{n-1}{(n-1)^{1-r}+1} \left(\chi_S\right)_1 + \frac1{(n-1)\left((n-1)^{1-r}+1\right)} \left(\sum_{i=2}^n \left(\chi_S\right)_i\right)^2\\
&\hspace{4cm} - \frac2{(n-1)^{1-r}+1} \left(\chi_S\right)_1 \sum_{i=2}^n \left(\chi_S\right)_i.
\end{align*}
Substituting these into \eqref{eq:OKexpression2} and noting that $N_l(S) = \sum_{i=l}^n \left(\chi_S\right)_i$ gives the desired result.
\end{proof}

\begin{corol}\label{cor:OKstarminimizers}
Let $G=(V,E,\omega)$ be an unweighted star graph as in Definition~\ref{def:bipartstar} with $n\geq 3$ nodes, with spectrum as in Lemma~\ref{lem:starspectrum}, and let $q=1$. Let $\mathfrak{M}$ be the set of admissable masses as in \eqref{eq:admissmass}.
Let $M\in \mathfrak{M}$ be such that there are $u, \tilde u\in \mathcal{V}_M^b$ with $u_1=0$ and $\tilde u_1 = 1$. Consider the minimization problem
from \eqref{eq:minimprobsF0}. 
We have
\begin{itemize}
\item if $M=\frac12 \vol{V}$ or $\gamma = \lambda_{n-1}$, then all $u\in \mathcal{V}_M^b$ are minimizers of \eqref{eq:minimprobsF0}; 
\item if $(\vol{V}-2M)(\gamma-\lambda_{n-1}) <0$, then $u\in \mathcal{V}_M^b$ is a minimizer of 
\eqref{eq:minimprobsF0} if and only if $u_1=0$;
\item if $(\vol{V}-2M)(\gamma-\lambda_{n-1}) >0$, then $u\in \mathcal{V}_M^b$ is a minimizer of 
\eqref{eq:minimprobsF0} if and only if $u_1=1$.,
\end{itemize}
\end{corol}
\begin{proof}
For $w\in \mathcal{V}$, define $\ell(w):= |\{i\in \{2, \ldots, n\}: w_i=1\}|$, i.e. $\ell(w)$ is the number of leave nodes on which $w$ takes the value $1$. By Lemma~\ref{lem:starmass2} we know that $F_0(w)=F_0(u)$ if $w_1=0$ and $F_0(w)=F_0(\tilde u)$ if $w_1=1$. Thus, for each $w\in \mathcal{V}$ there is a $\hat w\in \mathcal{V}$ such that $F_0(w)=F_0(\hat w)$, $\ell(w) = \ell(\hat w)$, for all $i\in \{2, \ldots, \ell(w)+1\}$ $\hat w_i = 1$, and (if $\ell(w)+2\leq n$) for all $i\in \{\ell(w)+2, \ldots, n\}$ $\hat w_i=0$. Hence, we assume without loss of generality that $w$ satisfies the properties prescribed for $\hat w$ above. In particular, in the notation of Lemma~\ref{lem:starOKexpress}, if $S=\{i\in V: w_i=1\}$, then for $2\leq l\leq n$, $N_l = \max(0, \ell(w)-(l-2))$. Substituting this in the expression for $F_0$ in Lemma~\ref{lem:starOKexpress} we find
\begin{align*}
F_0(w) &= \sum_{l=2}^{\ell(w)+1} \frac{1+\gamma}{(n-l)(n-l+1)} \left((n-l) - (\ell(w)-(l-1))\right)^2\\
&\hspace{2cm} + \frac1{n-1} \left(1+\frac\gamma{\left((n-1)^{1-r}+1\right)^2}\right) \left((n-1) w_1 - \ell(w)\right)^2\\
&= \frac1{n-1} \left[(1+\gamma) \ell(w) (n-1-L) + \left(1+\frac\gamma{\left((n-1)^{1-r}+1\right)^2}\right) \left((n-1)w_1-\ell(w)\right)^2\right],
\end{align*} 
where for the second equality we used that
\begin{equation}\label{eq:moresums}
\sum_{l=2}^{\ell(w)+1} \frac{1+\gamma}{(n-l)(n-l+1)} = \frac{\ell(w)}{(n-1)(n-\ell(w)-1)},
\end{equation}
which in turn follows from Corollary~\ref{cor:moresums}. Note that
\begin{equation}\label{eq:volumeell}
M=(n-1)^r w_1 + \ell(w),
\end{equation}
hence 
\begin{align*}
F_0(w) &= \frac1{n-1}\left[(1+\gamma) \left(n-1+(n-1)^r M w_1 - M\right) \left(M-(n-1)^r w_1\right)\right.\\
&\hspace{2cm} \left. + \left(1+\frac\gamma{\left((n-1)^{1-r}+1\right)^2}\right) \left((n-1)w_1 + (n-1)^r w_1 - M\right)^2\right].
\end{align*}
If $w_1 = 0$ we compute
\[
(n-1)F_0(u)=(n-1) F_0(w) = M^2 \left(-(1+\gamma) + 1 + \frac\gamma{\left((n-1)^{1-r}+1\right)^2}\right) + (1+\gamma)(n-1)M.
\]
If $w_1=1$ on the other hand, then
\begin{align*}
&(n-1)F_0(\tilde u) = (n-1)F_0(w) = \\
 &\hspace{0.3cm} M^2 \left(-(1+\gamma)+1+\frac\gamma{\left((n-1)^{1-r}+1\right)^2}\right)\\
&\hspace{0.3cm} + M\left[-n+1 + \gamma \left((n-1)^r \left(1+(n-1)^{1-r}\right)\left(1-\frac2{\left((n-1)^{1-r}+1\right)^2}\right) + (n-1)^r\right)\right]\\
&\hspace{0.3cm} + (1-\gamma)(n-1)^{1+r} + (n-1)^2.
\end{align*}
A short computation then shows that
\[
F_0(u)-F_0(\tilde u) = \left(2M-\left((n-1)^r+n-1\right)\right) \left(1-\frac\gamma{(n-1)^{1-r}+1}\right).
\]
Since $\vol{V}=(n-1)^r+n-1$ and $\lambda_{n-1} = (n-1)^{1-r}+1 >0$ the results follow.
\end{proof}

\begin{remark}
We can easily understand the critical role that the value $M=\frac12 \vol{V}$ plays in Corollary~\ref{cor:OKstarminimizers}. For any $S\subset V$ we have $\mathcal{M}(\chi_S) = \vol{V} - \mathcal{M}(\chi_{V\setminus S})$ and $F_0(\chi_S) = F_0(\chi_{V\setminus S})$, thus $\chi_S$ is a minimizer of \eqref{eq:minimprobsF0} for a given $M$, if and only if $\chi_{V\setminus S}$ is a minimizer for $\tilde M = \vol{V}-M$. We have $\tilde M = M$ if and only if $M=\frac12\vol{V}$.

Furthermore, in Corollary~\ref{cor:OKstarminimizers} we found that $\gamma=\lambda_{n-1}$ is a critical value for the star graph at which the minimizer of $F_0$ changes its value at the internal node $1$. This can heuristically be understood as the value of $\gamma$ for which, for $m\in \{1, \ldots, n-2\}$,  $\lambda_m+\frac\gamma{\lambda_m} = \lambda_{n-1}+\frac\gamma{\lambda_{n-1}}$, and so the influence of $\phi^{n-1}$ ---which is the eigenfunction that distinguishes node $1$ from the other nodes--- in \eqref{eq:OKexpression2} becomes noticeable. It is not clear to what degree this heuristic can be applied to other graphs as well.
\end{remark}

\begin{remark}
Note that in the star graph setting of Corollary~\ref{cor:OKstarminimizers} we assume that $M\in \mathfrak{M}$ is such that $\mathcal{V}_M^b$ contains both functions which take the value $0$ on node $1$ and functions which take the value $1$ on node $1$. If $M$ were such that all functions in $\mathcal{V}_M^b$ took the same value on node $1$, then minimizers of \eqref{eq:minimprobsF0} would be necessarily restricted to that class of functions and the `if and only if' statements in the corollary would have to be weakened.

Notice, however, that this assumption can be quite restrictive. For example, when $r=1$ we have that, if $M>n-1$, then all $u\in \mathcal{V}_M^b$ satisfy $u_1=1$, and if $M<n-1$, then all $u\in \mathcal{V}_M^b$ satisfy $u_1=0$. Hence, if $r=1$, then the assumption from the corollary is satisfied if and only if $M=n-1=\frac12 \vol{V}$, in which case the corollary tells us that all $u\in \mathcal{V}_M^b$ are minimizers of \eqref{eq:minimprobsF0}.

In order to obtain a larger set of admissable masses with interesting behaviour, one could consider minimising $F_0(\chi_S)$ over all $\chi_S \in \mathcal{V}^b$ for which $|S|=M$, for a given $M\in (0,n)\cap \N$. Note that this problem is equivalent to \eqref{eq:minimprobsF0} if $r=0$, but even if $r\neq 0$, any choice of $M\in (0,n)\cap \N$ will allow for admissible $u\in\mathcal{V}^b$ with $u_1=0$ and admissible $u\in \mathcal{V}^b$ with $u_1=1$. Of course it is a somewhat unnatural mixture of conditions to set $r=0$ in the mass condition, but not in the functional $F_0$. If we repeat the computation from the proof of Corollary~\ref{cor:OKstarminimizers} in this case, i.e. with $M=w_1+L$ instead of \eqref{eq:volumeell}, and define
\[
g(\gamma) := (n-2M) \left[\gamma \left(1-\frac{n}{\lambda_{n-1}^2}\right) -  (n-1)\right],
\]
we find that if $g(\gamma)=0$ all admissible $u$ are minimizers; if $g(\gamma)<0$ any admissible $u$ is a minimizer if and only if $u_1=1$; and if $g(\gamma)>0$ any admissible $u$ is a minimizer if and only if $u_1=0$.
\end{remark}

\section{Graph MBO schemes}\label{sec:MBO}

\subsection{The graph Ohta-Kawasaki MBO scheme}\label{sec:OKMBOscheme}

One way in which we can attempt to solve the $F_\e$ minimization problem in \eqref{eq:minimprobsFe} is via a gradient flow. In Appendix~\ref{sec:gradflows} we derive gradient flows with respect to the $\mathcal{V}$ inner product (which, if $r=0$ and each $u\in \mathcal{V}$ is identified with a vector in $\R^n$, is just the Euclidean inner product on $\R^n$) and with respect to the $H^{-1}$ inner product which leads to the graph Allen-Cahn and graph Cahn-Hilliard type systems of equations, respectively. In our simulations later in the paper, however, we do not use these gradient flows, but we use the MBO approximation. 

Heuristically, graph MBO type schemes (originally introduced in the continuum setting in \cite{MBO1992,MBO1993}) can be seen as approximations to graph Allen-Cahn type equations (as in \eqref{eq:AllenCahn}), obtained by replacing the double well potential term in that equation by a hard thresholding step. This leads to the algorithm~\ref{alg:OKMBO}. In the algorithm we have used the set $\mathcal{V}_\infty$, which we define to be the set of all functions $u: [0, \infty) \times V \to \R$ which are continuously differentiable in their first argument (which we will typically denote by $t$). For such functions, we will use the notation $u_i(t) := u(t,i)$. We note that where before $u$ and $\varphi$ denoted functions in $\mathcal{V}$, here these same symbols are used to denote functions in $\mathcal{V}_\infty$.

For reasons that will be explored below in Remark~\ref{rem:whynewvarphi}, in the algorithm we use a variation of \eqref{eq:varphiinH-1}: for given $u\in \mathcal{V}$, if $\varphi\in \mathcal{V}$ satisfies
\begin{equation}\label{eq:zeromassvarphiequation}
\begin{cases}
\Delta \varphi = u-\mathcal{A}(u),\\
\mathcal{M}(\varphi) = 0,
\end{cases}
\end{equation}
we say $\varphi$ \textit{solves \eqref{eq:zeromassvarphiequation} for $u$}.

If $\varphi \in \mathcal{V}$ solves \eqref{eq:zeromassvarphiequation} for a given $u\in \mathcal{V}$ and $\tilde \varphi \in \mathcal{V}$ solves \eqref{eq:varphiinH-1} for the same $u$, then $\Delta(\varphi-\tilde \varphi)=0$, hence there exists a $C\in \R$, such that $\varphi = \tilde \varphi + C\chi_V$. Because $\tilde \varphi_k=0$, we have $C= \varphi_k$. In particular, because \eqref{eq:varphiinH-1} has a unique solution, so does \eqref{eq:zeromassvarphiequation}.

For a given $\gamma\geq 0$, we define the operator $L: \mathcal{V} \to \mathcal{V}$ as follows. For $u\in \mathcal{V}$, let
\begin{equation}\label{eq:defofL}
Lu := \Delta u + \gamma \varphi,
\end{equation}
where $\varphi\in \mathcal{V}$ is the solution to \eqref{eq:zeromassvarphiequation}.

\begin{asm}{(OKMBO)}
\KwData{An initial node subset $S^0 \subset V$, a parameter $\gamma\geq 0$, a parameter $r\in [0,1]$, 
a time step $\tau > 0$,
and the number of iterations $N\in \N\cup\{\infty\}$. }
\KwOut{A sequence of node sets $\{S^k\}_{k=1}^N$, which is the \ref{alg:OKMBO} evolution of $S^0$. }
\For{$k = 1 \ \KwTo \  N$,}{
{\bf ODE step.} Compute $u\in \mathcal{V}_\infty$ by solving
\begin{equation}\label{eq:MBOstepa}
\begin{cases}
\frac{\partial u}{\partial t} = -Lu, &\text{on } (0,\infty) \times V,\\
u(0) = u_0, &\text{on } V,\\
\end{cases}
\end{equation}
where $u_0=\chi_{S^{k-1}}$ and $L$ is as in \eqref{eq:defofL} with $\varphi\in \mathcal{V}_\infty$ being such that, for all $t\in [0,\infty)$, $\varphi(t)$ solves \eqref{eq:zeromassvarphiequation} for $u(t)$.\\ \medskip

{\bf Threshold step.} Define the subset $S^k \subset V$ to be
\[
S^k := \left\{ i\in V \colon u(\tau) \geq \frac12 \right\}.
\]
}
\caption{\label{alg:OKMBO} The graph Ohta-Kawasaki Merriman-Bence-Osher algorithm}
\end{asm}

\begin{remark}\label{rem:uniquesolution} 
Since $L$, as defined in \eqref{eq:defofL}, is a continuous linear operator from $\mathcal{V}$ to $\mathcal{V}$ (see \eqref{eq:Leigexp}), by standard ODE theory (\cite[Chapter 1]{Hale2009}, \cite[Chapter 1]{coddington1984theory}) there exists a unique, continuously-differentiable-in-$t$, solution $u$ of \eqref{eq:MBOstepa} on $(0,\infty)\times V$. In the threshold step of \ref{alg:OKMBO}, however, we only require $u(\tau)$, hence it suffices to compute the solution $u$ on $(0,\tau]$.
\end{remark}

\begin{remark}\label{rem:whynewvarphi} Note that we cannot always find a $k\in V$ such that the solution to \eqref{eq:varphiinH-1} is also a solution to \eqref{eq:zeromassvarphiequation}. In other words, the solution to \eqref{eq:zeromassvarphiequation} with $\mathcal{M}(\varphi) = 0$ may have nonzero value at every node in $V$.
We could keep definition \eqref{eq:varphiinH-1} for $\varphi$, instead of \eqref{eq:zeromassvarphiequation}, but then we would need to replace the term $-\gamma \varphi$ in \eqref{eq:MBOstepa} (with \eqref{eq:defofL}) by $-\gamma (\varphi - \mathcal{A}(\varphi))$. For simplicity we choose the formulation as laid out in \eqref{eq:MBOstepa} with \eqref{eq:zeromassvarphiequation}, but this has as consequence that $\varphi$ from \eqref{eq:zeromassvarphiequation} cannot necessarily always be obtained via the Green's function approach outlined in \eqref{eq:Greensexpansion2},
\eqref{eq:GreenPoisson}. In general such $\varphi$ will have the form, for all $i\in V$,
$
\varphi_i  := \sum_{j\in S} d_j^r G_{ij} (u_j-\mathcal{A}(u)) + c,
$ 
where $G$ is as in 
\eqref{eq:GreenPoisson} and $c$ is a suitably chosen constant such that $\mathcal{M}(\varphi) = 0$. Of course, as remarked before, the value of $c$ will not  influence the value of $\|u-\mathcal{A}(u)\|_{H^{-1}}$.

Lemma~\ref{lem:sufficientnodezero} below gives a sufficient condition for $c$ to be zero.
\end{remark}

\begin{lemma}\label{lem:sufficientnodezero}
Let $k\in V$ and let $\varphi$ solve \eqref{eq:varphiinH-1}. If there exists an $l\in V$ such that, for all $j\in V$,
\[
\nu_l^{V\setminus\{k\}} - \nu_l^{V\setminus\{j\}} = \frac1{\vol{V}} \sum_{i\in V} d_i^r \left(\nu_i^{V\setminus\{k\}} - \nu_i^{V\setminus\{j\}}\right),
\]
then $\tilde \varphi := \varphi - \varphi_l$ satisfies both \eqref{eq:zeromassvarphiequation} and
$
\begin{cases}
\Delta \tilde\varphi = u-\mathcal{A}(u),\\
\tilde\varphi_l = 0.
\end{cases}
$
\end{lemma}
\begin{proof}
Clearly $\tilde \varphi_l=0$ and, since $\varphi$ satisfies \eqref{eq:varphiinH-1}, $\Delta \tilde \varphi = u-\mathcal{A}(u)$.  Let $G$ be the Green's function from \eqref{eq:GreenPoisson}. If the conditions from the lemma hold, then, for all $j\in V$,
\begin{align*}
G_{lj} &= \frac1{\vol{V}} \left(\nu_l^{V\setminus\{k\}}+\nu_k^{V\setminus\{j\}} - \nu_l^{V\setminus\{j\}}\right)
= \frac1{\vol{V}} \left[ \nu_k^{V\setminus\{j\}} + \frac1{\vol{V}} \sum_{i\in V} d_i^r \left(\nu_i^{V\setminus\{k\}} - \nu_i^{V\setminus\{j\}}\right)\right]\\
&= \frac1{\vol{V}^2} \sum_{i\in V} d_i^r \left(\nu_i^{V\setminus\{k\}}+\nu_k^{V\setminus\{j\}} - \nu_i^{V\setminus\{j\}}\right) (\chi_V)_i
= \frac1{\vol{V}} \langle G^j, \chi_V\rangle_{\mathcal{V}} = \mathcal{A}(G^j).
\end{align*}
Therefore\footnote{The index $\bigcdot$ in $G_{l\bigcdot}$ and $\mathcal{A}(G^{\bigcdot})$ indicates the index over which is summed in the inner products; thus in the second inner product the summation in the mass $\mathcal{M}(G^j)$ is over the lower index of $G_i^j$ (for fixed $j$) and the summation in the inner product is over the upper index.},
$
\varphi_l = \langle G_{l\bigcdot}, u-\mathcal{A}(u)\rangle_{\mathcal{V}} = \langle \mathcal{A}(G^{\bigcdot}), u-\mathcal{A}(u)\rangle_{\mathcal{V}}.
$ 
Moreover,
\[
\mathcal{M}(\varphi) = \sum_{i\in V} \varphi_i d_i^r = \sum_{i,j\in V} G_{ij} (u_j-\mathcal{A}(u)) d_i^r d_j^r = \sum_{j\in V} \mathcal{M}(G_{\bigcdot j}) (u_j-\mathcal{A}(u)) d_j^r = \langle \mathcal{M}(G^{\bigcdot}), u-\mathcal{A}(u)\rangle_{\mathcal{V}}.
\]
Hence we conclude that $\varphi_l = \mathcal{A}(\varphi)$ and thus $\varphi-\varphi_l$ satisfies \eqref{eq:zeromassvarphiequation}.
\end{proof}

The next lemma will come in handy in various proofs later in the paper.

\begin{lemma}\label{lem:decreasingfunction}
Let $G=(V,E,\omega)\in \mathcal{G}$ and $u\in \mathcal{V}$, then the function
\begin{equation}\label{eq:decreasingfunction}
[0,\infty) \to \R, t\mapsto \left\langle e^{-t L}u, u\right\rangle_{\mathcal{V}}
\end{equation}
is decreasing. Moreover, if $u$ is not constant on $V$, then the function in \eqref{eq:decreasingfunction} is strictly decreasing.

\end{lemma}
\begin{proof}
Using the expansion in \eqref{eq:expansion} for $u$, we have
\begin{align}
\left\langle e^{-t L}u, u\right\rangle_{\mathcal{V}} &=  \left\langle \sum_{m=0}^{n-1} e^{-t \Lambda_m} \langle u,\phi^m\rangle_{\mathcal{V}}\, \phi^m, \sum_{l=0}^{n-1} \langle u,\phi^l\rangle_{\mathcal{V}} \phi^l\right\rangle_{\mathcal{V}}\notag\\
&= \sum_{m,l=0}^{n-1} e^{-t\Lambda_m} \langle u,\phi^m\rangle_{\mathcal{V}} \langle u,\phi^l\rangle_{\mathcal{V}}\, \delta_{ml} = \sum_{m=0}^{n-1} e^{-t\Lambda_m} \langle u,\phi^m\rangle_{\mathcal{V}}^2.\label{eq:etauLuu}
\end{align}
Since, for each $m\in \{0, \ldots, n-1\}$, the function $t\mapsto e^{-t\Lambda_m}$ is decreasing, the function in \eqref{eq:decreasingfunction} is decreasing. Moreoever, for each $m\in \{1, \ldots, n-1\}$, the function $t\mapsto e^{-t\Lambda_m}$ is strictly decreasing; thus the function in \eqref{eq:decreasingfunction} is strictly decreasing unless for all  $m\in \{1, \ldots, n-1\}$, $ \langle u,\phi^m\rangle_{\mathcal{V}}=0$.

Assume that for all  $m\in \{1, \ldots, n-1\}$, $ \langle u,\phi^m\rangle_{\mathcal{V}}=0$. Then, by the expansion in \eqref{eq:expansion} and the expression in \eqref{eq:zerotheigvec}, we have $u= \langle u, \phi^0\rangle_{\mathcal{V}}\, \phi^0 = \vol{V}^{-1} \mathcal{M}(u) \chi_V$. Hence $u$ is constant. Thus, if $u$ is not constant, then the function in \eqref{eq:decreasingfunction} is strictly decreasing.
\end{proof}

The following lemma introduces a Lyapunov functional for the \ref{alg:OKMBO} scheme.

\begin{lemma}\label{lem:Lyapunov}
Let $G=(V,E,\omega) \in \mathcal{G}$, $\gamma\geq 0$, and $\tau>0$. Define $J_\tau: \mathcal{V} \to \R$ by
\begin{equation}\label{eq:Jtau}
J_\tau(u) := \langle \chi_V-u, e^{-\tau L} u \rangle_\mathcal{V}.
\end{equation}
Then the functional $J_\tau$ is strictly concave and Fr\'echet differentiable, with directional derivative at $u\in \mathcal{V}$ in the direction $v\in \mathcal{V}$ given by
\begin{equation}\label{eq:dJtau}
dJ_\tau^u(v) := \langle \chi_V-2 e^{-\tau L} u, v\rangle_{\mathcal{V}}.
\end{equation}
Furthermore, if $S^0\subset V$ and $\{S^k\}_{k=1}^N$ is a sequence generated by \ref{alg:OKMBO}, then for all $k\in \{1, \ldots, N\}$,
\begin{equation}\label{eq:MBOminimiz}
\chi_{S^k} \in \underset{v\in \mathcal{K}}\argmin\, dJ_\tau^{\chi_{S^{k-1}}}(v),
\end{equation}
where $\mathcal{K}$ is as defined in \eqref{eq:setK}. Moreover, $J_\tau$ is a Lyapunov functional for the \ref{alg:OKMBO} scheme in the sense that, for all $k\in \{1,\ldots, N\}$, $J_\tau(\chi_{S^k}) \leq J_\tau(\chi_{S^{k-1}})$, with equality if and only if $S^k=S^{k-1}$.
\end{lemma}
\begin{proof}
This follows immediately from the proofs of \cite[Lemma 4.5, Proposition 4.6]{vanGennipGuillenOstingBertozzi14} (which in turn were based on the continuum case established in \cite{CPA:CPA21527}), as replacing $\Delta$ in those proofs by $L$ does not invalidate any of the statements. It is useful, however, to reproduce the proof here, especially with an eye to incorporating a mass constraint into the \ref{alg:OKMBO} scheme in Section~\ref{sec:OKMBOmass}.

First let $u,v \in \mathcal{V}$ and $s\in \R$, then we compute
\[
\left.\frac{d J_\tau(u+sv)}{ds}\right|_{s=0}  = \langle \chi_V - u, e^{-\tau L} v\rangle_{\mathcal{V}} - \langle v, e^{-\tau L} u\rangle_{\mathcal{V}} = \langle \chi_V - 2 e^{-\tau L} u, v\rangle_{\mathcal{V}},
\]
where we used that $e^{-\tau L}$ is a self-adjoint operator and $e^{-\tau L} \chi_V = \chi_V$. Moreover, if $v\in \mathcal{V}\setminus\{0\}$, then
\[
\left.\frac{d^2 J_\tau(u+sv)}{ds^2}\right|_{s=0} = -2\langle v, e^{-\tau L} v\rangle_{\mathcal{V}} < 0,
\]
where the inequality follows for example from the spectral expansion in \eqref{eq:etauLuu}. 
Hence $J_\tau$ is strictly concave.

To construct a minimizer $v$ for the linear functional $dJ_\tau^{\chi_{S^{k-1}}}$ over $\mathcal{K}$, we set $v_i = 1$ whenever $1-2\left(e^{-\tau L} \chi_{S^{k-1}}\right)_i \leq 0$ and $v_i = 0$ for those $i\in V$ for which $1-2\left(e^{-\tau L} \chi_{S^{k-1}}\right)_i > 0$\footnote{Note that the arbitrary choice for those $i$ for which $1-2\left(e^{-\tau L} \chi_{S^{k-1}}\right)_i = 0$ introduces non-uniqueness into the minimzation problem \eqref{eq:MBOminimiz}.}. The sequence $\{S^k\}_{k=1}^N$ generated in this way by setting $S^k = \{i\in V: v_i=1\}$ corresponds exactly to the sequence generated by \ref{alg:OKMBO}.

Finally we note that, since $J_\tau$ is strictly concave and $dJ_\tau^{\chi_{S^{k-1}}}$ is linear, we have, if $\chi_{S^{k+1}} \neq \chi_{S^k}$, then
\[
J_\tau\left(\chi_{S^{k+1}}\right) - J_\tau\left(\chi_{S^k}\right) < dJ_\tau^{\chi_{S^k}}\left(\chi_{S^{k+1}}-\chi_{S^k}\right) = dJ_\tau^{\chi_{S^k}}\left(\chi_{S^{k+1}}\right) -dJ_\tau^{\chi_{S^k}}\left(\chi_{S^k}\right) \leq 0,
\]
where the last inequality follows because of \eqref{eq:MBOminimiz}. Clearly,  if $\chi_{S^{k+1}} = \chi_{S^k}$, then $J_\tau\left(\chi_{S^{k+1}}\right) - J_\tau\left(\chi_{S^k}\right) = 0$.
\end{proof}

\begin{remark}\label{rem:sequentiallinearprogramming}
It is worth elaborating briefly on the underlying reason why \eqref{eq:MBOminimiz} is the right minimization problem to consider in the setting of Lemma~\ref{lem:Lyapunov}. As is standard in sequential linear programming the minimization of $J_\tau$ over $\mathcal{K}$ is attempted by approximating $J_\tau$ by its linearization,
\[
J_\tau(u) \approx J_\tau\left(\chi_{S^k}\right) + d J_\tau^{\chi_{S^{k-1}}}\left(u-\chi_{S^k}\right) = J_\tau\left(\chi_{S^k}\right) + d J_\tau^{\chi_{S^{k-1}}}\left(u\right)-d J_\tau^{\chi_{S^{k-1}}}\left(\chi_{S^k}\right),
\]
and minimizing this linear approximation over all admissible $u\in \mathcal{K}$.
\end{remark}

We can use Lemma~\ref{lem:Lyapunov} to prove that the \ref{alg:OKMBO} scheme converges in a finite number of steps to stationary state in sense of the following corollary.

\begin{corol}\label{cor:finiteconvergence}
Let $G=(V,E,\omega) \in\mathcal{G}$, $\gamma\geq 0$, and $\tau>0$. If $S^0\subset V$ and $\{S^k\}_{k=1}^N$ is a sequence generated by \ref{alg:OKMBO}, then there is a $K\geq 0$ such that, for all $k\geq K$, $S^k = S^K$.
\end{corol}
\begin{proof}
If $N\in \N$ the statement is trivially true, so now assume $N=\infty$. Because $|V|<\infty$, there are only finitely many different possible subsets of $V$, hence there exists $K, k'\in \N$ such that $k' > K'$ and $S^K=S^{k'}$. Hence the set in $l:=\min\{l'\in \N: S^K=S^{K+l'}\}$\footnote{Remember that we use the convention $0\not\in \N$.} is not empty and thus $l\geq 1$. If $l \geq 2$, then by Lemma~\ref{lem:Lyapunov} we know that
\[
J_\tau(\chi_{S^{K+l}}) < J_\tau(\chi_{S^{K+l-1}}) < \ldots < J_\tau(\chi_{S^K}) = J_\tau(\chi_{S^{K+l}}).
\]
This is a contradiction, hence $l=1$ and thus $S^K=S^{K+1}$. Because equation \eqref{eq:MBOstepa} has a unique solution (as noted in Remark~\ref{rem:uniquesolution}), we have, for all $k \geq K$,  $S^k = S^K$.
\end{proof}

\begin{remark}
For given $\tau>0$, the minimization problem
\begin{equation}\label{eq:Jtauminimization}
u \in \underset{v\in \mathcal{K}}\argmin\, J_\tau(v)
\end{equation}
has a solution $u\in \mathcal{V}^b$, because $J_\tau$ is strictly concave and $\mathcal{K}$ is compact and convex. This solution is not unique; for example, if $\tilde u = \chi_V-u$, then, since $e^{-\tau L}$ is self-adjoint, we have
\[
J_\tau(u) = \langle \tilde u, e^{-tL}(\chi_V-u)\rangle_{\mathcal{V}} = \langle \chi_V-\tilde u, e^{-\tau L} \tilde u\rangle_{\mathcal{V}} = J_\tau(\tilde u).
\]
Lemma~\ref{lem:Lyapunov} shows that $J_\tau$ does not increase in value along a sequence $\{S^k\}_{k=1}^N$ of sets generated by the \ref{alg:OKMBO} algorithm, but this does not guarantee that \ref{alg:OKMBO} converges to the solution of the minimization problem in \eqref{eq:Jtauminimization}. In fact, we will see in Lemma~\ref{lem:dynamicsbounds} and Remark~\ref{rem:butwhatabouttheemptyandfullset} that for every $S^0\subset V$ there is a value $\tau_\rho(S^0)$ such that $S^1=S$ if $\tau<\tau_\rho(S^0)$. Hence, unless $S^0$ happens to be a solution to \eqref{eq:Jtauminimization}, if  $\tau<\tau_\rho(S^0)$ the \ref{alg:OKMBO} algorithm will not converge to a solution. This observation and related issues concerning the minimization of $J_\tau$ will become important in Section~\ref{sec:Gammaconvergence}, see for example Remarks~\ref{rem:trivialminimizers} and~\ref{rem:whataboutmass?}.
\end{remark}

\subsection{The spectrum of $L$}\label{sec:spectrumofL}

In this section we will have a closer look at the spectrum of the operator $L$ from \eqref{eq:defofL}, which will play a role in our further study of \ref{alg:OKMBO}.

\begin{remark}\label{rem:varphiandL}
Remembering from Remark~\ref{rem:Moore-Penrose} the Moore-Penrose pseudoinverse of $\Delta$, which we denote by $\Delta^\dagger$, we see that the condition $\mathcal{M}(\varphi) = 0$ in \eqref{eq:zeromassvarphiequation} allows us to write $\varphi = \Delta^\dagger (u-\mathcal{A}(u))$. In particular, if $\varphi$ satisfies \eqref{eq:zeromassvarphiequation}, then
\begin{equation}\label{eq:varphispectral}
\varphi  = \sum_{m=1}^{n-1}\lambda_m^{-1} \langle u, \phi^m\rangle_{\mathcal{V}} \ \phi^m,
\end{equation}
where $\lambda_m$ and $\phi^m$ are the eigenvalues of $\Delta$ and corresponding eigenfunctions, respectively, as in \eqref{eq:eigenvalues}, \eqref{eq:eigenfunctions}.
Hence, if we expand $u$ as in \eqref{eq:expansion}  
and $L$ is the operator defined in \eqref{eq:defofL}, then
\begin{equation}\label{eq:Leigexp}
L(u) = \sum_{m=1}^{n-1} \left(\lambda_m+ \frac\gamma{\lambda_m}\right)  \langle u, \phi^m\rangle_{\mathcal{V}} \ \phi^m.
\end{equation}
In particular, $L:\mathcal{V}\to \mathcal{V}$ is a continuous, linear, bounded, self-adjoint, operator and for every $c\in \R$, $L(c\chi_V)=0$. If, given a $u_0\in \mathcal{V}$, $u\in \mathcal{V}_\infty$ solves \eqref{eq:MBOstepa}, 
then we have that $u(t) = e^{-tL}u_0$. Note that the operator $e^{-tL}$ is self-adjoint, because $L$ is self-adjoint.
\end{remark}

\begin{lemma}\label{lem:Lspectrum}
Let $G=(V,E,\omega)\in \mathcal{G}$, $\gamma\geq 0$, and let $L: \mathcal{V} \to \mathcal{V}$ be the operator defined in \eqref{eq:defofL}, then $L$ has $n$ eigenvalues $\Lambda_m$ ($m\in \{0, \ldots, n-1\}$), given by
\begin{equation}\label{eq:eigvalsofL}
\Lambda_m = \begin{cases}
0, & \text{if } m=0,\\
\lambda_m + \frac{\gamma}{\lambda_m}, & \text{if } m\geq 1,
\end{cases}
\end{equation}
where the $\lambda_m$ are the eigenvalues of $\Delta$ as 
in \eqref{eq:eigenvalues}. The set $\{\phi^m\}_{m=0}^{n-1}$ from \eqref{eq:eigenfunctions} is a set of corresponding eigenfunctions. In particular, $L$ is positive semidefinite.
\end{lemma}
\begin{proof}
This follows from \eqref{eq:Leigexp} and the fact that $\lambda_0=0$ and, for all $m\geq 1$, $\lambda_m>0$.
\end{proof}

In the remainder of this paper we use the notation $\lambda_m$ for the eigenvalues of $\Delta$ and $\Lambda_m$ for the eigenvalues of $L$, with corresponding eigenfunctions $\phi^m$, as in \eqref{eq:eigenvalues}, \eqref{eq:eigenfunctions}, and Lemma~\ref{lem:Lspectrum}.

\begin{remark}\label{rem:eigorder}
Note that in the notation of Lemma~\ref{lem:Lspectrum}, the eigenvalues $\Lambda_m$ are not necessarily labelled in non-decreasing order. In fact, the function $f: (0,\infty)\to (0,\infty), x\mapsto x+\frac\gamma{x}$ achieves its unique minimum on $(0,\infty)$ at $x=\sqrt\gamma$ and is decreasing for $0<x<\sqrt\gamma$ and increasing for $x>\sqrt\gamma$. Hence, if $\lambda_{n-1} \leq \sqrt\gamma$, then the $\Lambda_m$ are in non-increasing order, except for $\Lambda_0$, which is always the smallest eigenvalue. On the other hand, if $\lambda_1 \geq \sqrt\gamma$, then the $\Lambda_m$ are in non-decreasing order. If neither of these conditions on $\lambda_1$ or $\lambda_{n-1}$ is met, the order is not guaranteed to be monotone.
\end{remark}

\begin{definition}\label{def:spectralbounds}
Let $G=(V,E,\omega)\in \mathcal{G}$, $\gamma\geq 0$, and 
let $\Lambda_m$ ($m\in \{0, \ldots, n-1\}$) be the eigenvalues of $L$ as in \eqref{eq:eigvalsofL}. 
The \textit{smallest nonzero eigenvalue of $L$} is
$\displaystyle
\Lambda_- := \underset{1 \leq m \leq n-1}\min \Lambda_m,
$ 
and the \textit{largest eigenvalue (or \textit{spectral radius}) of $L$} is
$\displaystyle
\Lambda_+ := \underset{0\leq m \leq n-1}\max |\Lambda_m|.
$
\end{definition}

 Lemma~\ref{lem:minLambda} characterizes the smallest nonzero eigenvalue and the largest eigenvalue of $L$. These eigenvalues will be of importance in Section~\ref{sec:pinningandspreading}.

\begin{lemma}\label{lem:minLambda}
Let $G=(V,E,\omega) \in \mathcal{G}$ 
and let $\lambda_m$ ($m\in \{0, \ldots, n-1\}$) be the eigenvalues of $\Delta$ as in \eqref{eq:eigenvalues}. 
If $\lambda_1 \leq \sqrt\gamma \leq \lambda_{n-1}$, we define $\lambda_* := \max\{\lambda_m: 1\leq m \leq n-1 \text{ and } \lambda_m \leq \sqrt\gamma\}$ and $\gamma^* := \min\{\lambda_m: 1\leq m \leq n-1 \text{ and }  \lambda_m \geq \sqrt\gamma\}$. Then the value of $\Lambda_- $ from Definition~\ref{def:spectralbounds} is given by
\[
\Lambda_- = \begin{cases}
\lambda_1 + \frac\gamma{\lambda_1}, &\text{if } \lambda_1 > \sqrt\gamma,\\
\lambda_*+ \frac\gamma{\lambda_*}, &\text{if } \lambda_1 \leq \sqrt\gamma < \sqrt{\lambda_*\lambda^*} \leq \lambda_{n-1},\\
\lambda^* + \frac\gamma{\lambda^*}, &\text{if } \lambda_1 \leq \sqrt{\lambda_*\lambda^*} \leq \sqrt\gamma \leq \lambda_{n-1},\\
\lambda_{n-1} + \frac\gamma{\lambda_{n-1}}, &\text{if } \lambda_{n-1} < \sqrt\gamma.
\end{cases}
\]
Moreover, $\gamma \mapsto \Lambda_-$ is continuous.
\end{lemma}
\begin{proof}
First note that, since $G$ is connected, for all $m\in \{1, \ldots, n-1\}$, $\Lambda_m >0$, hence $\Lambda_->0$. Furthermore, if $\lambda_1 \leq \sqrt\gamma\leq \lambda_{n-1}$, then the sets in the definitions of $\lambda_*$ and $\lambda^*$ are nonempty and thus $\lambda_*$ and $\lambda^*$ are well-defined. 
Following from the discussion in Remark~\ref{rem:eigorder} we know that $x\mapsto x+\frac\gamma{x}$ is either non-increasing, non-decreasing, or it achieves its unique minimum on $(0,\infty)$ at $x=\sqrt\gamma$, depending on the value of $\gamma$. Hence the minimum value of $\Lambda_m$ ($m\geq 1$) is achieved when either $m=1$, $m=n-1$, or $m$ is such that $\lambda_m = \lambda_*$ or $\lambda_m=\lambda^*$. By the argument in Remark~\ref{rem:eigorder} we know that the first two cases occur when $\lambda_1 \geq \sqrt\gamma$ or $\lambda_{n-1} \leq \sqrt\gamma$, respectively. The other two cases follow from
\[
\lambda_* + \frac\gamma{\lambda^*} < \lambda^* + \frac\gamma{\lambda^*} \Leftrightarrow \gamma < \frac{\lambda_*-\lambda^*}{\frac1{\lambda^*}-\frac1{\lambda_*}} = \lambda_*\lambda^*.
\]
Note that if $\lambda_1=\sqrt\gamma$, then $\sqrt{\lambda_*\lambda^*} = \lambda_1$; if $\lambda_{n-1} = \sqrt\gamma$, then $\sqrt{\lambda_*\lambda^*} = \lambda_{n-1}$; if $\lambda_*\lambda^* = \gamma$, then $\lambda_*+\frac\gamma{\lambda_*} = \lambda^* + \frac\gamma{\lambda_*}$. Thus $\gamma\mapsto \Lambda_-$ is continuous.
\end{proof}

\begin{lemma}\label{lem:specrad}
Let $G=(V,E,\omega) \in \mathcal{G}$ 
and let $\lambda_m$ ($m\in \{0, \ldots, n-1\}$) be the eigenvalues of $\Delta$ as in \eqref{eq:eigenvalues}. 
Then the spectral radius of $L$ as defined in Definition~\ref{def:spectralbounds} is
\[
\Lambda_+ = \begin{cases}
\lambda_1 + \frac\gamma{\lambda_1}, &\text{if } \lambda_1\lambda_{n-1} < \gamma,\\
\lambda_{n-1} + \frac\gamma{\lambda_{n-1}}, &\text{if } \lambda_1\lambda_{n-1} \geq \gamma.
\end{cases}
\]
Moreover, $\gamma\mapsto \Lambda_+$ is continuous.
\end{lemma}
\begin{proof}
First we note that, by Lemma~\ref{lem:Lspectrum} all eigenvalues $\Lambda_m$ of $L$ are nonnegative. Since $G$ is connected and $n\geq 2$, there is at least one positive eigenvalue, so $\Lambda_+>0$. By the computation in Remark~\ref{rem:eigorder} the function $x\mapsto x+\frac\gamma{x}$ is either non-increasing, non-decreasing, or strictly convex on the domain $(0, \lambda_{n-1})$, depending on the value of $\gamma$. Thus, by the expression for $\Lambda_m$ in Lemma~\ref{lem:Lspectrum}, we see that the maximum value of $\Lambda_m$ is achieved when either $m=1$ or $m=n-1$, depending on the value of $\gamma$. If $\lambda_1=\lambda_{n-1}$, then $\Lambda_1=\Lambda_{n-1}$ and the result follows. If $\lambda_1\neq \lambda_{n-1}$, then $\frac1{\lambda_{n-1}} - \frac1{\lambda_1} < 0$, hence we compute
\[
\lambda_1 + \frac\gamma{\lambda_1} > \lambda_{n-1} + \frac\gamma{\lambda_{n-1}} \Leftrightarrow \gamma > \frac{\lambda_1-\lambda_{n-1}}{\frac1{\lambda_{n-1}}-\frac1{\lambda_1}} = \lambda_1 \lambda_{n-1}.
\]
Replacing the inequality by an equality, shows continuity of $\gamma\mapsto \Lambda_+$.
\end{proof}

\begin{lemma}\label{lem:massandsuch}
Let $G\in\mathcal{G}$, $\gamma \geq 0$, and $u_0\in \mathcal{V}$. If $u\in \mathcal{V}_\infty$ is a solution of \eqref{eq:MBOstepa}, 
with corresponding $\varphi\in \mathcal{V}_\infty$,  
then, for all $t>0$,
\[
\frac{d}{dt} \mathcal{M}(u(t)) = 0.
\]
Furthermore, for all $t>0$,
\[
\frac{d}{dt} \|u(t)\|_{\mathcal{V}}^2 = - 2 \left(\|\nabla u(t)\|_{\mathcal{E}}^2 + \gamma \|\nabla \varphi\|_{\mathcal{E}}^2 \right) \leq 0.
\]
In particular, for all $t\geq 0$, $\|u(t)\|_{\mathcal{V}} \leq \|u_0\|_{\mathcal{V}}$.

Moreover, if $\eta>0$ and
$
t'> \Lambda_-^{-1} \log\left(\eta^{-1} d_-^{-\frac{r}2} \|u_0 - \mathcal{A}(u_0)\|_{\mathcal{V}}\right),
$ 
where $\Lambda_-$ is as in Lemma~\ref{lem:minLambda}, then for all $t>t'$,
$
\|u(t)-\mathcal{A}(u(t))\|_{\mathcal{V},\infty} < \eta.
$
\end{lemma}
\begin{proof}
This proof follows very closely the proof of \cite[Lemma 2.6(a), (b), and (c)]{vanGennipGuillenOstingBertozzi14}.

Using \eqref{eq:Laplacemass} and \eqref{eq:zeromassvarphiequation}, we find
\[
\frac{d}{dt} \mathcal{M}(u(t)) = \mathcal{M}\left(\frac{\partial}{\partial t} u(t)\right) = -\mathcal{M}(\Delta u(t)) - \gamma \mathcal{M}(\varphi) = 0.
\]

Next we compute
\[
\frac{d}{dt} \|u(t)\|_{\mathcal{V}}^2 = 2 \langle u(t), \frac{\partial}{\partial t} u(t)\rangle_{\mathcal{V}} = -2 \langle u(t), L\left(u(t)\right)\rangle_{\mathcal{V}} = -2 \left( \langle u(t), \Delta u(t)\rangle_{\mathcal{V}} + \gamma \langle u(t), \varphi(t)\rangle_{\mathcal{V}} \right).
\]
Since $\mathcal{M}(\varphi)=0$, we have $\langle \mathcal{A}(u(t)), \varphi\rangle_{\mathcal{V}}=0$ and thus 
$\langle u(t), \Delta u(t)\rangle_{\mathcal{V}} =  \langle \nabla u(t), \nabla u(t)\rangle_{\mathcal{V}}$ and
\[
\langle u(t), \varphi(t)\rangle_{\mathcal{V}} = \langle u(t)-\mathcal{A}(u(t)), \varphi(t)\rangle_{\mathcal{V}} = \langle \Delta\varphi(t), \varphi(t)\rangle_{\mathcal{V}} = \langle \nabla \varphi(t), \nabla \varphi(t)\rangle_{\mathcal{V}},
\]
from which the expression for $\frac{d}{dt} \|u(t)\|_{\mathcal{V}}^2$ follows.

To prove the final statement we expand 
$\displaystyle 
u(t)-\mathcal{A}(u(t)) = \sum_{m=1}^{n-1} e^{-t\Lambda_m} \langle u_0, \phi^m\rangle_{\mathcal{V}}\, \phi^m.
$ 
Let $t>0$. Recall that the eigenfunctions $\phi^m$ are pairwise $\mathcal{V}$-orthogonal, hence
\begin{align*}
\left\|\sum_{m=1}^{n-1} e^{-t\Lambda_m} \langle u_0, \phi^m\rangle_{\mathcal{V}}\, \phi^m \right\|_{\mathcal{V}}^2 &= \sum_{m=1}^{n-1} \left\|e^{-t\Lambda_m} \langle u_0, \phi^m\rangle_{\mathcal{V}}\, \phi^m \right\|_{\mathcal{V}}^2\\ 
&\leq e^{-2t \Lambda_-} \sum_{m=1}^{n-1} \left\| \langle u_0, \phi^m\rangle_{\mathcal{V}}\, \phi^m \right\|_{\mathcal{V}}^2
= e^{-2t \Lambda_-} \left\|\sum_{m=1}^{n-1}  \langle u_0, \phi^m\rangle_{\mathcal{V}}\, \phi^m \right\|_{\mathcal{V}}^2.
\end{align*}
Therefore
$\displaystyle
\|u(t) - \mathcal{A}(u(t))\|_{\mathcal{V}} \leq e^{-t \Lambda_-} \left\|\sum_{m=1}^{n-1} \langle u_0, \phi^m\rangle_{\mathcal{V}}\, \phi^m\right\|_{\mathcal{V}} = e^{-t\Lambda_-} \|u_0 - \mathcal{A}(u_0)\|_{\mathcal{V}}.
$
By \eqref{eq:normineq} we conclude that
$\displaystyle
\|u(t) - \mathcal{A}(u(t))\|_{\mathcal{V},\infty} \leq d_-^{-\frac{r}2} \|u(t) - \mathcal{A}(u(t))\|_{\mathcal{V}} \leq d_-^{-\frac{r}2} e^{-t\Lambda_-} \|u_0 - \mathcal{A}(u_0)\|_{\mathcal{V}} < \eta.
$
\end{proof}

\subsection{Pinning and spreading}\label{sec:pinningandspreading}

The following lemma and its proof use some of the results above and follow very closely \cite[Theorems 4.2, 4.3, 4.4]{vanGennipGuillenOstingBertozzi14}. The lemma gives sufficient bounds on the parameter $\tau$ for the \ref{alg:OKMBO} dynamics to be `uninteresting', i.e. for the evolution to be either pinned (i.e. each iteration gives back the initial set) or for the dynamics in \eqref{eq:MBOstepa} to spread the mass so widely that \ref{alg:OKMBO} arrives at a trivial (constant) stationary state in one iteration.  In the lemma's proof, we need an operator norm, which, for a linear operator $O: \mathcal{V}\to \mathcal{V}$, we define as
\[
\|O\|_o := \underset{\mathcal{V}\setminus\{0\}}\max \frac{\|Ou\|_{\mathcal{V}}}{\|u\|_{\mathcal{V}}}.
\]
A property of this norm is that, for all $u\in \mathcal{V}$, $\|Ou\|_{\mathcal{V}} \leq \|O\|_o \|u\|_{\mathcal{V}}$.

Since $L$ is self-adjoint, it follows from the Rayleigh quotient formulation of $L's$ eigenvalues, that $\|L\|_o=\Lambda_+$, where $\Lambda_+$ is the spectral radius of $L$ as in Lemma~\ref{lem:specrad} \cite[Theorem. VI.6]{ReedSimon1980}.

\begin{lemma}\label{lem:dynamicsbounds}
Let $G=(V,E,\omega) \in \mathcal{G}$. Let $\gamma\geq 0$ and let $\Lambda_-$ and $\Lambda_+$ be as in Lemma~\ref{lem:minLambda} and Lemma~\ref{lem:specrad}, respectively. Let $S\subset V$. If $S\neq \emptyset$, define
\[
\tau_\rho(S) := \Lambda_+^{-1} \log\left(1+\frac12 d_-^{\frac{r}2} (\vol{S})^{-\frac12}\right).
\]
If in addition $S \neq V$ and $\frac{\vol{S}}{\vol{V}} \neq \frac12$, also define
\[
\tau_t(S) := \Lambda_-^{-1} \log\left(\frac{\left(\vol{S}\right)^{\frac12} \left(\vol{S^c}\right)^{\frac12}}{\left(\vol{V}\right)^{\frac12} \left|\frac{\vol{S}}{\vol{V}}-\frac12\right| d_-^{\frac{r}2}}\right).
\]
Let $\gamma\geq 0$, $\tau>0$, and $S^1$ be the first set in the corresponding \ref{alg:OKMBO} evolution of the initial set $S^0=S$.
\begin{enumerate}
\item\label{item:bound1} If $\tau < \tau_\rho(S)$, then $S^1=S$. In particular, if $\tau < \Lambda_+^{-1} \log\frac32 \approx 0.4\Lambda_+^{-1}$, then $S^1=S$.
\item\label{item:bound3} If $\emptyset \neq S \neq V$, $\frac{\vol{S}}{\vol{V}} \neq \frac12$, and $\tau > \tau_t(S)$, 
then
$\displaystyle
S^1 = \begin{cases}
\emptyset, &\text{if } \frac{\vol{S}}{\vol{V}} < \frac12,\\
V, &\text{if } \frac{\vol{S}}{\vol{V}} > \frac12.
\end{cases}
$
\end{enumerate}
If $S=\emptyset$ or $S=V$, then $S^1=S$.

Moreover, if $\emptyset \neq S \neq V$, $\frac{\vol{S}}{\vol{V}} \neq \frac12$, and $\frac{\Lambda_-}{\Lambda_+} < \frac{\log\sqrt2}{\log\frac32} \approx 0.85$, then $\tau_\rho(S) < \tau_t(S)$.
\end{lemma}
\begin{proof}
The proof follows very closely the proofs of \cite[Theorems 4.2, 4.3, 4.4]{vanGennipGuillenOstingBertozzi14}, but we include it here for completeness.

To prove \ref{item:bound1}, first let $\tau<\tau_\rho(S)$. Let $\mathrm{Id}:\mathcal{V}\to\mathcal{V}$ be the identity operator, then we compute the operator norm
$\displaystyle
\|e^{-\tau L} - \mathrm{Id}\|_o \leq \sum_{j=1}^\infty \frac1{j!} \left(\tau \|L\|_o\right)^j = e^{\rho t}-1 < \frac12 d_-^{\frac{r}2} (\vol{S})^{-\frac12},
$
where we used the triangle inequality and submultiplicative property of $\|\cdot\|_o$ 
\cite{HornJohnson1990} for the first inequality. Hence, by \eqref{eq:normineq},
\begin{align*}
\|e^{-\tau L} \chi_S - \chi_S\|_{\mathcal{V},\infty} &\leq d_-^{-\frac{r}2} \|e^{-\tau L} \chi_S - \chi_S\|_{\mathcal{V}} \leq  d_-^{-\frac{r}2} \|e^{-\tau L} - \mathrm{Id} \|_{\mathcal{V}} \|\chi_S\|_{\mathcal{V}}\\
&= d_-^{-\frac{r}2} \|e^{-\tau L} - \mathrm{Id} \|_{\mathcal{V}} \sqrt{\vol{S}}
< \frac12.
\end{align*}
It follows from the thresholding step in \ref{alg:OKMBO} that $S^1=S$.

To prove \ref{item:bound3} (for any $r\in [0,1]$), we use Lemma~\ref{lem:massandsuch} with $\eta:=\left|\frac{\vol{S}}{\vol{V}}-\frac12\right|$ to find
\[
\left| \|u(\tau)\|_{\mathcal{V},\infty} - \mathcal{A}(\chi_S)\right| \leq \left\|u(\tau)-\mathcal{A}(\chi_S)\right\|_{\mathcal{V},\infty} < \left| \frac{\vol{S}}{\vol{V}}-\frac12\right|.
\]
If $\frac{\vol{S}}{\vol{V}}<\frac12$ this implies
\[
\|u(\tau)\|_{\mathcal{V},\infty} \leq \|u(\tau)-\mathcal{A}(\chi_S)\|_{\mathcal{V},\infty} + \|\mathcal{A}(\chi_S)\|_{\mathcal{V},\infty} < \left|\frac{\vol{S}}{\vol{V}}-\frac12\right| + \frac{\vol{S}}{\vol{V}} = \frac12.
\]
Alternatively, if $\frac{\vol{S}}{\vol{V}}> \frac12$, then
\[
\frac{\vol{S}}{\vol{V}} = \left\|\mathcal{A}(\chi_S)\right\|_{\mathcal{V},\infty} \leq \left\|u(\tau)-\mathcal{A}(\chi_S)\right\|_{\mathcal{V},\infty} + \|u(\tau)\|_{\mathcal{V},\infty}  
< \frac{\vol{S}}{\vol{V}}-\frac12 + \|u(\tau)\|_{\mathcal{V},\infty},
\]
and thus $\|u(\tau)\|_{\mathcal{V},\infty} > \frac12$. The result then follows from the thresholding step in \ref{alg:OKMBO}.

Since $L\chi_\emptyset = \chi_\emptyset$ and $L\chi_V = \chi_V$, the subsets $S=\emptyset$ and $S=V$ are stationary states of the ODE step in \ref{alg:OKMBO} and thus $S^1=S$.

To prove the final statement, we first note that, since $\emptyset\neq S \neq V$, we have $d_-^r \leq \vol{S} \leq \vol{V} - d_-^r$. Since $(\vol{S})(\vol{S^c}) = \vol{S}(\vol{V}-\vol{S})$ is concave as a function of $\vol{S}$, we find
\begin{align*}
(\vol{S})(\vol{S^c}) &\geq \min\{d_-^r \left(\vol{V}-d_-^r\right), \left(\vol{V}-d_-^r\right) \left(\vol{V} - \left(\vol{V}-d_-^r\right)\right)\}\\
&= d_-^r \left(\vol{V}-d_-^r\right).
\end{align*}
We also note that $\left|\frac{\vol{S}}{\vol{V}}-\frac12\right| \leq \frac12$. Then we find that
$
\tau_\rho(S) \leq \Lambda_+^{-1} \log\left(\frac32\right),
$ and 
$
\tau_t(S) \geq \Lambda_- \log\left(2 \sqrt{1-\frac{d_-r}{\vol{S}}}\right) \leq \Lambda^{-1} \log\sqrt2,
$ 
where the last inequality follows from $\vol{V} \geq n d_-^r \geq 2 d_-^r$. 
\end{proof}

\begin{remark}\label{rem:butwhatabouttheemptyandfullset}
The exclusion of the case $\vol{S^0} = \frac12 \vol{V}$ for the establishment of $\tau_t$ in Lemma~\ref{lem:dynamicsbounds} is a necessary one, as in this case symmetry could lead to pinning in \ref{alg:OKMBO}, such that $S^1=S^0$ (and thus $\emptyset\neq S^1 \neq V$). For example, consider an unweighted, completely connected graph and an initial set $S^0$ such that $\vol{S^0} = \frac12 \vol{V}$. By symmetry, no nontrivial dynamics can occur, regardless of the value of $\tau$; hence $e^{-\tau L} \chi_{S^0} = \chi_{S^0}$ and thus $S^1=S^0$.
\end{remark}

In Lemma~\ref{cor:OKstarminimizers} we saw that for the unweighted star graph the value of $u_1$ determines if $u\in \mathcal{V}$ is a minimizer of $F_0$ (with $q=1$) or not (unless $M=\frac12 \vol{V}$ or $\gamma = \lambda_{n-1}$). It is therefore interesting to investigate the pinning behaviour of \ref{alg:OKMBO} at the centre node of the star graph in more detail. In particular we are interested in the case where $1\in S$ and $\left(e^{-\tau L} \chi_S\right)_1 \geq \frac12$ and the case where $1\not\in S$ and $\left(e^{-\tau L} \chi_S\right)_1 < \frac12$, as those are the cases in which the status of node $1$ does not change after one iteration of \ref{alg:OKMBO} (i.e. if $S^0 = S$, then either $1\in S^0\cap S^1$ or $1\not\in S^0\cap S^1$). The following lemma gives explicit conditions on $\tau$ for these cases to occur.

\begin{lemma}\label{lem:node1changeinstar}
Let $G=(V,E,\omega) \in \mathcal{G}$ be an unweighted star graph as in Definition~\ref{def:bipartstar} with $n\geq 3$ nodes. 
Let $\gamma\geq 0$, let $\Lambda_{n-1}=\lambda_{n-1} + \frac\gamma{\lambda_{n-1}}$ be the eigenvalue of $L$ as in Lemma~\ref{lem:starspectrum} and \eqref{eq:eigvalsofL}, and let $S\subset V$. 
If $1\in S$, then $\left(e^{-\tau L} \chi_S\right)_1 \geq \frac12$ if and only if $\tau \geq 0$ is such that
\[
e^{-\Lambda_{n-1}\tau} \geq \frac12 \frac{\vol{V} - 2\mathcal{M}\left(\chi_S\right)}{\vol{V}-\mathcal{M}\left(\chi_S\right)}.
\]
Alternatively, if $1\not\in S$ then $\left(e^{-\tau L} \chi_S\right)_1 < \frac12$ if and only if $\tau \geq 0$ is such that 
\[
e^{-\Lambda_{n-1}\tau} > 1 - \frac12 \frac{\vol{V}}{\mathcal{M}\left(\chi_S\right)}.
\]
\end{lemma}
It is worth remembering that in the setting of Lemma~\ref{lem:node1changeinstar} we have $\vol{V} = (n-1)^r + n-1$ and $\mathcal{M}\left(\chi_S\right) = (n-1)^r \left(\chi_S\right)_1 + \left|S\setminus\{1\}\right|$.
\begin{proof}[Proof of Lemma~\ref{lem:node1changeinstar}]
The proof is a direct computation using an expansion as in \eqref{eq:expansion} along the lines of what was done in \cite{vanGennipGuillenOstingBertozzi14,vanGennip17note}. Using the spectrum in Lemma~\ref{lem:starspectrum}, we compute
\[
\chi_S = \sum_{m=0}^{n-1} \langle \chi_S, \phi^m\rangle_{\mathcal{V}}\, \phi^m
= \mathcal{A}\left(\chi_S\right) + \sum_{m=1}^{n-2} \langle \chi_S, \phi^m\rangle_{\mathcal{V}}\, \phi^m + \frac{(n-1)^{1-r} d_1^r \left(\chi_S\right)_1 - \sum_{i=2}^n d_i^r \left(\chi_S\right)_i}{(n-1)^{2-r} + n - 1} \phi^{n-1}.
\]
Since $\Lambda_0 = 0$, for $m\in \{1, \ldots, n-2\}$, $\phi^m_1=0$, $\phi^{n-1}_1 = (n-1)^{1-r}$, $d_1^r = (n-1)^r$, for $i\in \{2, \ldots, n\}$, $d_i^r=1$, and $\vol{V} = (n-1)^r + n-1$, we compute
\begin{align*}
\left(e^{-\tau L} \chi_S\right)_1 &= \sum_{m=0}^{n-1} e^{-\tau \Lambda_m}  \langle \chi_S, \phi^m\rangle_{\mathcal{V}}\, \phi^m\\
&= \frac{\mathcal{M}\left(\chi_S\right)}{\vol{V}} + e^{-\tau \Lambda_{n-1}}  \frac{(n-1)^{1-r}}{(n-1)^{2-r}+n-1} \left((n-1) \left(\chi_S\right)_1 - \sum_{i=2}^n \left(\chi_S\right)_i\right)\\
&= \frac1{\vol{V}} \left[ \mathcal{M}\left(\chi_S\right) + e^{-\tau \Lambda_{n-1}} \left( \vol{V}\left(\chi_S\right)_1 - \mathcal{M}\left(\chi_S\right)\right)\right].
\end{align*}
The results in the lemma now follow by considering the cases $1\in S$ and $1\not\in S$, hence $(\chi_S)_1=1$ and $\left(\chi_S\right)_1=0$, respectively.
\end{proof}

\begin{remark}
Interpreting the results from Lemma~\ref{lem:node1changeinstar}, we see that, for the unweighted star graph, if $1\in S$, pinning at node $1$ occurs for any value of $\tau\geq 0$ if $\mathcal{M}\left(\chi_S\right) \geq \frac12\vol{V}$. If instead $1\not\in S$, then pinning at node $1$ occurs, independent of the value of $\tau$, if $\mathcal{M}\left(\chi_S\right) \leq \frac12\vol{V}$. In particular, pinning at node $1$ always occurs if $r=1$, independent of the choice of $\tau$ or $S$.
\end{remark}

\subsection{$\Gamma$-convergence of the Lyapunov functional}\label{sec:Gammaconvergence}

In this section we prove that the functionals $\tilde J_\tau: \mathcal{K}\to \overline\R$, defined by
\begin{equation}\label{eq:tildeJ}
\tilde J_\tau (u) := \frac1\tau \langle \chi_V-u, e^{-\tau L} u \rangle_\mathcal{V},
\end{equation}
for $\tau>0$, $\Gamma$-converge to $\tilde F_0: \mathcal{K}\to \overline\R$ as $\tau \to 0$, where $\tilde  F_0$ is defined by
\begin{equation}\label{eq:tildeF0}
\tilde F_0(u) := \begin{cases}
F_0(u), &\text{if } u\in \mathcal{K}\cap \mathcal{V}^b,\\
+\infty, &\text{otherwise},
\end{cases}
\end{equation}
where $F_0$ is as in \eqref{eq:limitOK} with $q=1$\footnote{Note that $\mathcal{K}\cap \mathcal{V}^b = \mathcal{V}^b$. We included $\mathcal{K}$ explicitly in the intersection here to emphasize that the domain of $\tilde F_0$ is $\mathcal{K}$, not $\mathcal{V}$.}. We use the notation $\overline\R := \R \cup \{-\infty,+\infty\}$ for the extended real line.

Remember that the set $\mathcal{K}$ was defined in \eqref{eq:setK} as the subset of all $[0,1]$-valued functions in $\mathcal{V}$. We note that $\tilde J_\tau = \frac1\tau \left.J_\tau\right|_{\mathcal{K}}$, where $J_\tau$ is the Lyapunov functional from Lemma~\ref{lem:Lyapunov}. 

In this section we will encounter different variants of the functional $\frac1\tau J_\tau$, such as $\tilde J_\tau$, $\overline J_\tau$, and $\overline{\overline J}_\tau$, and similar variants of $F_0$. The differences between these functionals are the domains on which they are defined and the parts of their domains on which they take finite values: $\tilde J_\tau$ is defined on all of $\mathcal{K}$, while $\overline J_\tau$ and $\overline{\overline J}_\tau$ (which will be defined later in this section) incorporate a mass constraint and relaxed mass constraint in their domains, respectively. For technical reasons we thought it prudent to distinguish these functionals through their notation, but intuitively they can be thought of as the same functional with different types of constraints (or lack thereof).

For sequences in $\mathcal{V}$ we use convergence in the $\mathcal{V}$-norm, i.e. if $\{u_k\}_{k\in\N}\subset \mathcal{V}$, then we say $u_k\to u$ as $k\to\infty$ if $\|u_k-u\|_{\mathcal{V}} \to 0$ as $k\to \infty$. Note, however, that all norms on the finite space $\mathcal{V}$ induce equivalent topologies, so different norms can be used without this affecting the convergence results in this section.

\begin{lemma}\label{lem:nonbinary}
Let $G=(V,E,\omega)\in \mathcal{G}$. Let $\{u_k\}_{k\in\N}\subset \mathcal{V}$ and $u\in \mathcal{V}\setminus\mathcal{V}^b$ be such that $u_k\to u$ as $k\to\infty$. Then there exists an $i\in V$, an $\eta>0$, and a $K>0$ such that for all $k\geq K$ we have $(u_k)_i \in \R\setminus\big([-\eta,\eta]\cup [1-\eta,1+\eta]\big)$.
\end{lemma}
\begin{proof}
Because $u\in \mathcal{V}\setminus\mathcal{V}^b$, there is an $i\in V$ such that $u_i\not\in \{0,1\}$. Since $G\in \mathcal{G}$, we know that $d_i^r>0$. Thus, since $u_k\to u$ as $k\to \infty$, we know that for every $\hat\eta>0$ there exists a $K(\hat\eta)>0$ such that for all $k\geq K(\hat\eta)$ we have $\left|(u_k)_i - u_i\right| < \hat\eta$. Define
$\displaystyle
\eta := \frac12 \min\{\left|u_i\right|, \left|u_i-1\right|\}>0.
$
Then, for all $k\geq K(\eta)$, we have
$\displaystyle 
\left|(u_k)_i\right| \geq \big| \left|(u_k)_i - u_i\right| - \left|u_i\right|\big | > \frac12 |u_i| \geq \eta
$
and similarly
$\displaystyle
\left|(u_k)_i-1\right| > \eta.
$
\end{proof}

\begin{lemma}\label{lem:masssquared}
Let $G=(V,E,\omega)\in\mathcal{G}$, $u\in \mathcal{V}$, and let $\{\phi^m\}_{m=0}^{n-1}$ be $\mathcal{V}$-orthonormal Laplacian eigenfunctions as in \eqref{eq:eigenfunctions}. Then
\[
\sum_{m=0}^{n-1} \langle u,\phi^m\rangle_{\mathcal{V}}^2 = \mathcal{M}(u^2).
\]
\end{lemma}
\begin{proof}
Let $j\in V$ and define $f\in \mathcal{V}$ by, for all $i\in V$, $f^j_i:=d_i^{-r} \delta_{ij}$, where $\delta$ denotes the Kronecker delta. Using the expansion in \eqref{eq:expansion}, we find, for all $i\in V$,
\[
f^j_i = \sum_{m=0}^{n-1} \langle f^j, \phi^m\rangle_{\mathcal{V}}\, \phi^m_i = \sum_{m=0}^{n-1} \sum_{k\in V} d_i^{-r} \delta_{kj} d_k^r \phi^m_k \phi^m_i = \sum_{m=0}^{n-1} \phi^m_j \phi^m_i.
\]
Hence
\[
\sum_{m=0}^{n-1} \langle u,\phi^m\rangle_{\mathcal{V}}^2 = \sum_{m=0}^{n-1} \sum_{i,j\in V} u_i u_j d_i^r d_j^r \phi^m_i \phi^m_j = \sum_{i,j\in V} u_i u_j d_i^r d_j^r f^j_i = \langle u^2, \chi_V\rangle_{\mathcal{V}} =  \mathcal{M}(u^2).
\]
\end{proof}

\begin{theorem}[$\Gamma$-convergence]\label{thm:gammaconvergence}
Let $G=(V,E,\omega)\in \mathcal{G}$, $q=1$, and $\gamma\geq 0$. Let $\{\tau_k\}_{k \in \N}$ be a sequence of positive real numbers such that $\tau_k\to 0$ as $k\to \infty$. Let $u\in \mathcal{K}$. Then the following lower bound and upper bound hold:
\begin{itemize}
\item[(LB)] for every sequence $\{u_k\}_{k\in \N} \subset \mathcal{K}$ such that $u_k \to u$ as $k\to \infty$, $\tilde F_0(u) \leq \underset{k\to\infty}\liminf\, \tilde J_{\tau_k}(u_k)$, and
\item[(UB)] there exists a sequence $\{u_k\}_{k\in \N} \subset \mathcal{K}$ such that $u_k \to u$ as $k\to \infty$ and $\underset{k\to\infty}\limsup\, \tilde J_{\tau_k}(u_k) \leq \tilde F_0(u)$.
\end{itemize}
\end{theorem}
\begin{proof}
With $J_\tau$ the Lyapunov functional from Lemma~\ref{lem:Lyapunov}, we compute, for $\tau>0$ and $u\in \mathcal{V}$,
\[
J_\tau(u) = \langle \chi_V-u, e^{-\tau L} u \rangle_{\mathcal{V}} = \mathcal{M}\left( e^{-\tau L} u\right) -  \langle u, e^{-\tau L} u \rangle_{\mathcal{V}} = \mathcal{M}(u) -  \langle u, e^{-\tau L} u \rangle_{\mathcal{V}},
\]
where we used the mass conservation property from Lemma~\ref{lem:massandsuch}. Using the expansion in \eqref{eq:expansion} and Lemma~\ref{lem:masssquared}, we find
\begin{align}
\frac1\tau J(u) &= \frac1\tau \mathcal{M}(u) - \sum_{m=0}^{n-1} \frac{e^{-\tau \Lambda_m}-1}\tau \langle u, \phi^m\rangle_{\mathcal{V}}^2 -\frac1\tau \sum_{m=0}^{n-1} \langle u, \phi^m\rangle_{\mathcal{V}}^2\notag\\
&=  - \sum_{m=0}^{n-1} \frac{e^{-\tau \Lambda_m}-1}\tau \langle u, \phi^m\rangle_{\mathcal{V}}^2 +\frac1\tau \left(\mathcal{M}(u)-\mathcal{M}(u^2)\right).\label{eq:fractauJ}
\end{align}

Now we prove (LB). Let $\{\tau_k\}_{k \in \N}$ and $\{u_k\}_{n\in \N} \subset \mathcal{K}$ be as stated in the theorem. Then, for all $m\in \{0, \ldots, n-1\}$, we have that $\displaystyle -\underset{k \to \infty}\lim\, \frac{e^{-{\tau_k} \Lambda_m}-1}{\tau_k} = \Lambda_m$ and $\displaystyle \underset{k \to \infty}\lim\, \langle u_k, \phi^m\rangle_{\mathcal{V}}^2 = \langle u, \phi^m\rangle_{\mathcal{V}}^2$, hence 
\begin{equation}\label{eq:positivityoflimit}
\underset{k\to \infty}\lim\, - \sum_{m=0}^{n-1} \frac{e^{-\tau_k \Lambda_m}-1}{\tau_k} \langle u_k, \phi^m\rangle_{\mathcal{V}}^2 = \sum_{m=0}^{n-1} \Lambda_m \langle u, \phi^m\rangle_{\mathcal{V}}^2 \geq 0.
\end{equation}
Moreover, if $u\in \mathcal{K} \cap \mathcal{V}^b$, then, combining the above with \eqref{eq:OKexpression2} (remember that $q=1$) and Lemma~\ref{lem:Lspectrum}, we find
\begin{equation}\label{eq:limitisF0}
\underset{k\to \infty}\lim\, - \sum_{m=0}^{n-1} \frac{e^{-\tau_k \Lambda_m}-1}{\tau_k} \langle u_k, \phi^m\rangle_{\mathcal{V}}^2 = F_0(u).
\end{equation}
Furthermore, since, for every $k\in \N$, $u_k\in \mathcal{K}$, we have that, for all $i\in V$, $u_i^2 \leq u_i$ and thus $\mathcal{M}(u_k)-\mathcal{M}(u_k^2) \geq 0$. Hence
\[
\underset{k\to\infty}\liminf\, \tilde J_{\tau_k}(u_k) \geq -\underset{k\to\infty}\liminf\,  \sum_{m=0}^{n-1} \frac{e^{-\tau_k \Lambda_m}-1}{\tau_k} \langle u_k, \phi^m\rangle_{\mathcal{V}}^2= F_0(u).
\]

Assume now that $u\in \mathcal{K}\setminus\mathcal{V}_b$ instead, then by Lemma~\ref{lem:nonbinary} it follows that there are an $i\in V$ and an $\eta>0$, such that, for all $k$ large enough, $(u_k)_i  \in (\eta, 1-\eta)$. Thus, for all $k$ large enough,
\[
\mathcal{M}(u_k) - \mathcal{M}(u_k^2) \geq d_i^r (u_k)_i (1-(u_k)_i) > d_i^r \eta^2 > 0.
\]
Combining this with \eqref{eq:positivityoflimit} we deduce
\[
\underset{k\to\infty}\liminf\, \tilde J_{\tau_k}(u_k) \geq \underset{k\to\infty}\liminf\, \frac1{\tau_k} \big(\mathcal{M}(u_k) - \mathcal{M}(u_k^2) \big) = +\infty = F_0(u),
\]
which completes the proof of (LB).

To prove (UB), first we note that, if $u\in \mathcal{K}\setminus\mathcal{V}_b$, then $F_0(u)=+\infty$ and the upper bound inequality is trivially satisfied. If instead $u\in \mathcal{K} \cap \mathcal{V}^b$, then we define a so-called recovery sequence as follows: for all $k\in \N$, $u_k:=u$. We trivially have that $u_k\to u$ as $k\to\infty$. Moreover, since $u=u^2$, we find, for all $k\in\N$, $\mathcal{M}(u_k)-\mathcal{M}(u_k^2) = 0$. Finally we find
\[
\underset{k\to\infty}\limsup\, \tilde J_{\tau_k}(u_k) =- \underset{k\to\infty}\lim\,  \sum_{m=0}^{n-1} \frac{e^{-\tau_k \Lambda_m}-1}{\tau_k} \langle u, \phi^m\rangle_{\mathcal{V}}^2 = F_0(u),
\]
where we used a similar calculation as in \eqref{eq:limitisF0}.
\end{proof}

\begin{theorem}[Equi-coercivity]\label{thm:equicoercivity}
Let $G=(V,E,\omega)\in \mathcal{G}$ and $\gamma\geq 0$.   Let $\{\tau_k\}_{k \in \N}$ be a sequence of positive real numbers such that $\tau_k\to 0$ as $k\to \infty$ and let $\{u_k\}_{k\in \N} \subset \mathcal{K}$ be a sequence for which there exists a $C>0$ such that, for all $k\in\N$, $\tilde J_\tau(u_k)\leq C$. Then there is a subsequence $\{u_{k_l}\}_{l \in \N} \subset \{u_k\}_{k\in \N}$ and a $u\in \mathcal{V}^b$ such that $u_{k_l} \to u$ as $l\to \infty$.
\end{theorem}
\begin{proof}
Since, for all $k\in \N$, we have $u_k\in \mathcal{K}$, it follows that, for all $k\in \N$, $0\leq \|u_k\|_2 \leq \sqrt n$, where $\|\cdot\|_2$ denotes the usual Euclidean norm on $\R^n$ pulled back to $\mathcal{V}$ via the natural identification of each function in $\mathcal{V}$ with one and only one vector in $\R^n$ (thus, it is the norm $\|\cdot\|_{\mathcal{V}}$ if $r=0$). By the Bolzano-Weierstrass theorem it follows that there is a subsequence $\{u_{k_l}\}_{l \in \N} \subset \{u_k\}_{k\in \N}$ and a $u\in \mathcal{V}$ such that $u_{k_l} \to u$ with respect to the norm $\|\cdot\|_2$ as $l\to \infty$. Because the $\mathcal{V}$-norm is topologically equivalent to the $\|\cdot\|_2$ norm (explicitly, $d_-^{\frac{r}2} \|\cdot\|_2 \leq \|\cdot\|_{\mathcal{V}} \leq d_+^{\frac{r}2} \|\cdot\|_2$), we also have $u_{k_l} \to u$ as $l\to \infty$. Moreover, since convergence with respect to $\|\cdot\|_2$ implies convergence of each component of $(u_{k_l})_i$ ($i\in V$) in $\R$ we have $u\in \mathcal{K}$.

Next we compute
\begin{equation}\label{eq:Jpos}
\tau_{k_l} \tilde J_{\tau_{k_l}}(u_{k_l}) = \langle \chi_V-u_{k_l}, e^{-\tau_{k_l} L} u_{k_l}\rangle_{\mathcal{V}} = \langle \chi_V, u_{k_l}\rangle_{\mathcal{V}}  - \langle u_{k_l}, e^{-\tau_{k_l} L} u_{k_l}\rangle_{\mathcal{V}} \geq \langle \chi_V - u_{k_l}, u_{k_l}\rangle_{\mathcal{V}},
\end{equation}
where we used $\langle \chi_V, e^{-\tau_{k_l} L} u_{k_l}\rangle_{\mathcal{V}}  = \langle \chi_V, u_{k_l}\rangle_{\mathcal{V}}$ from Lemma~\ref{lem:massandsuch} and $\langle u_{k_l}, e^{-\tau_{k_l} L} u_{k_l}\rangle_{\mathcal{V}} \leq \langle u_{k_l}, u_{k_l}\rangle_{\mathcal{V}}$ from Lemma~\ref{lem:decreasingfunction}. Thus, for all $l\in \N$, we have
\begin{equation}\label{eq:contradingred1}
0 \leq \langle \chi_V - u_{k_l}, u_{k_l}\rangle_{\mathcal{V}} \leq C \tau_{k_l}.
\end{equation}

Assume that $u\in \mathcal{K}\setminus\mathcal{V}^b$, then there is an $i\in V$ such that $0<u_i<1$. Hence, by Lemma~\ref{lem:nonbinary}, there is a $\delta>0$  such that for all $l$ large enough, $\delta < (u_{k_l})_i < 1-\delta$ and thus
\begin{equation}\label{eq:contradingred2}
\langle \chi_V - u_{k_l}, u_{k_l}\rangle_{\mathcal{V}} \geq d_i^r \big(1-(u_{k_l})_i\big) (u_{k_l})_i \geq d_i^r \delta^2.
\end{equation}

Let $l$ be large enough such that $C \tau_{k_l} < d_i^r \delta^2$ and large enough such that \eqref{eq:contradingred2} holds. Then we have arrived at a contradiction with \eqref{eq:contradingred1} and thus we conclude that $u\in \mathcal{V}^b$.
\end{proof}

\begin{remark}\label{rem:trivialminimizers}
The computation in \eqref{eq:Jpos} shows that, for all $\tau>0$ and for all $u\in \mathcal{K}$, we have $\tau \tilde J_\tau(u) \geq \langle \chi_V-u, u\rangle_{\mathcal{V}} \geq 0$. Moreover, we have $\tilde J_\tau(0) = \tilde J_\tau(\chi_V) = 0$. Furthermore, since each term of the sum in the inner product is nonnegative, we have  $\langle \chi_V-u, u\rangle_{\mathcal{V}}=0$ if and only if $u=0$ or $u=\chi_V$. Hence we also have $\tilde J_\tau(u) = 0$ if and only if $u=0$ or $u=\chi_V$. The minimization of $\tilde J_\tau$ over $\mathcal{K}$ is thus not a very interesting problem. Therefore we now extend our $\Gamma$-convergence and equi-coercivity results from above to incorporate a mass constraint.

As an aside, note that Lemma~\ref{lem:dynamicsbounds} and Remark~\ref{rem:butwhatabouttheemptyandfullset} guarantee that for $\tau$ large enough and $S^0$ such that $\vol{S^0} \neq \frac12\vol{V}$, the \ref{alg:OKMBO} algorithm converges in at most one step to the minimzer $\emptyset$ or the minimizer $V$.
\end{remark}

Let $M\in \mathfrak{M}$, where $\mathfrak{M}$ is the set of admissible masses as defined in \eqref{eq:admissmass}. Remember from \eqref{eq:setKM} that $\mathcal{K}_M$ is the set of $[0,1]$-valued functions in $\mathcal{V}$ with mass equal to $M$. For $\tau>0$ we define the following functionals with restricted domain. Define $\overline J_\tau: \mathcal{K}_M\to \overline\R$ by $\overline J_\tau := \left. \tilde J_\tau\right|_{\mathcal{K}_M}$, where $\tilde J_\tau$ is as defined above in \eqref{eq:tildeJ}. Also define $\overline F_0: \mathcal{K}_M\to \overline \R$ by
$
\overline F_0(u) := \left.\tilde F_0\right|_{\mathcal{K}_M},
$
where $\tilde F_0$ is as in \eqref{eq:tildeF0}, with $q=1$. 
Note that by definition, $\tilde F_0$, and thus $\overline F_0$, do not assign a finite value to functions $u$ that are not in $\mathcal{V}^b$.

\begin{theorem}\label{thm:gammaconvergencemass}
Let $G=(V,E,\omega)\in \mathcal{G}$, $q=1$, and $\gamma\geq 0$. Let $\{\tau_k\}_{k \in \N}$ be a sequence of positive real numbers such that $\tau_k\to 0$ as $k\to \infty$. Let $u\in \mathcal{K}_M$. Then the following lower bound and upper bound hold:
\begin{itemize}
\item[(LB)] for every sequence $\{u_k\}_{k\in \N} \subset \mathcal{K}_M$ such that $u_k \to u$ as $k\to\infty$, $\overline F_0(u) \leq \underset{k\to\infty}\liminf\, \overline J_{\tau_k}(u_k)$, and
\item[(UB)] there exists a sequence $\{u_k\}_{k\in \N} \subset \mathcal{K}_M$ such that $u_k \to u$ as $k\to \infty$ and $\underset{k\to\infty}\limsup\, \overline J_{\tau_k}(u_k) \leq \overline F_0(u)$.
\end{itemize}

Furthermore, if $\{v_k\}_{k\in \N} \subset \mathcal{K}_M$ is a sequence for which there exists a $C>0$ such that, for all $k\in\N$, $\overline J_\tau(v_k)\leq C$, then there is a subsequence $\{v_{k_l}\}_{l \in \N} \subset \{v_k\}_{k\in \N}$ and a $v\in \mathcal{K}_M \cap \mathcal{V}^b$ such that $v_{k_l} \to v$ as $l\to \infty$.
\end{theorem}
\begin{proof}
We note that any converging sequence in $\mathcal{K}_M$ with limit $u$ is also a converging sequence in $\mathcal{K}$ with limit $u$. Moreover, on $\mathcal{K}_M$ we have $\overline J_{\tau_k} = \tilde J_{\tau_k}$ and $\overline F_0 = \tilde F_0$. Hence (LB) follows directly from (LB) in Theorem~\ref{thm:gammaconvergence}.

For (UB) we note that if we define, for all $k\in \N$, $u_k:=u$, then trivially the mass constraint on $u_k$ is satisfied for all $k\in \N$ and the result follows by a proof analogous to that of (UB) in Theorem~\ref{thm:gammaconvergence}.

Finally, for the equi-coervicity result, we first note that by Theorem~\ref{thm:equicoercivity} we immediately get the existence of a subsequence $\{v_{k_l}\}_{l \in \N} \subset \{v_k\}_{k\in \N}$ which converges to some $v\in \mathcal{K}$. Since the functional $\mathcal{M}$ is continuous with respect to $\mathcal{V}$-convergence, we conclude that in fact $v\in \mathcal{K}_M$.
\end{proof}

\begin{remark}
Note that for $\tau>0$, $M\in \mathfrak{M}$, and $u\in \mathcal{K}_M$, we have
\[
\tau \overline J_\tau(u) = \mathcal{M}(u) - \langle u, e^{\tau L}u\rangle_{\mathcal{V}} = M - \sum_{m=0}^{n-1} e^{-\tau \Lambda_m} \langle u, \phi^m\rangle_{\mathcal{V}}^2 = M\left(1-\frac{M}{\vol{V}}\right) - \sum_{m=1}^{n-1} e^{-\tau \Lambda_m} \langle u, \phi^m\rangle_{\mathcal{V}}^2.
\]
Hence finding the minimizer of $\overline J_\tau$ in $\mathcal{K}_M$ is equivalent to finding the maximizer of $\displaystyle u\mapsto \sum_{m=1}^{n-1} e^{-\tau \Lambda_m} \langle u, \phi^m\rangle_{\mathcal{V}}^2$ in $\mathcal{K}_M$.
\end{remark}

The following result shows that the $\Gamma$-convergence and equi-coercivity results still hold, even if the mass conditions are not strictly satisfied along the sequence.

\begin{corol}
Let $G=(V,E,\omega)\in \mathcal{G}$, $q=1$, and $\gamma\geq 0$ and let $\mathcal{C} \subset \mathfrak{M}$ be a set of admissible masses. For each $k\in \N$, let $\mathcal{C}_k \subset [0,\infty)$ be such that $\displaystyle\underset{k\in \N}\bigcap \mathcal{C}_k = \mathcal{C}$ and define, for all $k\in\N$,
\[
\mathcal{K}_M^k := \{u \in \mathcal{K}: \mathcal{M}(u) \in \mathcal{C}_k\}.
\]
Let $\{\tau_k\}_{k \in \N}$ be a sequence of positive real numbers such that $\tau_k\to 0$ as $k\to \infty$. Define $\overline{\overline J}_{\tau_k}: \mathcal{K} \to \overline\R$ by
\[
\overline{\overline J}_{\tau_k}(u) := \begin{cases}
\tilde J_{\tau_k}(u), &\text{if  } u\in \mathcal{K}_M^k,\\
+\infty, &\text{otherwise.} 
\end{cases}
\]
Furthermore, define $\overline{\overline F}_0: \mathcal{K} \to \overline\R$ by
\[
\overline{\overline F}_0(u) := \begin{cases}
\tilde F_0(u), &\text{if  } u\in \mathcal{K}_M,\\
+\infty, &\text{otherwise.} 
\end{cases}
\]
Then the results of Theorem~\ref{thm:gammaconvergencemass} hold with $\overline J_{\tau_k}$ and $\overline F_0$ replaced by $\overline{\overline J}_{\tau_k}$ and $\overline{\overline F}_0$, respectively, and with the sequences $\{u_k\}_{k\in \N}$ and $\{v_k\}_{k\in \N}$ in (LB), (UB), and the equi-coercivity result taken in $\mathcal{K}$ instead of $\mathcal{K}_M$, such that, for each $k\in \N$, $u_k, v_k \in \mathcal{K}_M^k$.
\end{corol}
\begin{proof}
The proof is a slightly tweaked version of the proof of Theorem~\ref{thm:gammaconvergencemass}. 
On $\mathcal{K}_M^k$ we have that $\overline{\overline J}_{\tau_k} = \tilde J_{\tau_k}$ and $\overline{\overline F}_0 = \tilde F_0$. Hence (LB) follows from (LB) in Theorem~\ref{thm:gammaconvergence}.
For (UB) we note that, since $\displaystyle\underset{k\in \N}\bigcap \mathcal{C}_k \supset \mathcal{C}$,  the recovery sequence defined by, for all $k\in \N$, $u_k:=u$, is admissable and the proof follows as in in the proof of Theorem~\ref{thm:gammaconvergence}.

Finally, for the equi-coercivity result, we obtain a converging subsequence $\{v_{k_l}\}_{l \in \N} \subset \{v_k\}_{k\in \N}$ with limit $v\in \mathcal{K}$ by Theorem~\ref{thm:equicoercivity}. By continuity of $\mathcal{M}$ it follows that $\mathcal{M}(v) \in \overline{\underset{k\in \N}\bigcap\mathcal{C}_k}$, where $\overline{\hspace{0.3cm}\cdot\hspace{0.3cm}}$ denotes the topological closure in $[0,\infty) \subset \R$. Because $\mathfrak{M}$ is a set of finite cardinality in $\R$, we know $\displaystyle\underset{k\in \N}\bigcap\mathcal{C}_k \subset \mathcal{C} \subset \mathfrak{M}$ is closed, hence $\displaystyle\mathcal{M}(v)\in \underset{k\in \N}\bigcap\mathcal{C}_k \subset \mathcal{C}$ and thus $v\in \mathcal{K}_M$.
\end{proof}

\begin{remark}\label{rem:whataboutmass?}
By a standard $\Gamma$-convergence result (\cite[Chapter 7]{DalMaso93},\cite[Section 1.5]{Braides02}) we conclude from Theorem~\ref{thm:gammaconvergencemass} that (for fixed $M\in\mathfrak{M}$) minimizers of $\overline J_\tau$ converge (up to a subsequence) to a minimizer of $\overline F_0$ (with $q=1$) when $\tau\to 0$.

By Lemma~\ref{lem:Lyapunov} we know that iterates of \ref{alg:OKMBO} solve \eqref{eq:MBOminimiz} and decrease the value of $J_\tau$, for fixed $\tau>0$ (and thus of $\tilde J_\tau$). By Lemma~\ref{lem:dynamicsbounds}, however, we know that when $\tau$ is sufficiently small, the \ref{alg:OKMBO} dynamics is pinned, in the sense that each iterate is equal to the initial condition. Hence, unless the initial condition is a minimizer of $\overline J_\tau$, for small enough $\tau$ the \ref{alg:OKMBO} algorithm does not generate minimizers of $\overline J_\tau$ and thus we cannot use Theorem~\ref{thm:gammaconvergencemass} to conclude that solutions of \ref{alg:OKMBO} approximate minimizers of $\overline F_0$ when $\tau\to 0$. 

As an interesting aside that can be an interesting angle for future work, we note that it is not uncommon in sequential linear programming for the contraints (such as the constraint that the domain of $\tilde J_\tau$ consists of $[0,1]$-valued functions only) to be an obstacle to convergence; compare for example the Zoutendijk method with the Topkis and Veinott method \cite[Chapter 10]{BazaraaSheraliShetty1993}. An analogous relaxation of the constraints might be a worthwhile direction for alternative MBO type methods for minimization of functionals like $\tilde J_\tau$.

We will not follow that route in this paper. Instead, in the next section we will look at a variant of \ref{alg:OKMBO} which conserves mass in each iteration.
\end{remark}

\subsection{A mass conserving graph Ohta-Kawasaki MBO scheme}\label{sec:OKMBOmass}

In Section~\ref{sec:Gammaconvergence} we saw that, for given $M\in \mathfrak{M}$, any solution to
\begin{equation}\label{eq:minimizeJtauoverKM}
u \in \underset{u\in \mathcal{K}_M}\argmin\, J_\tau(u),
\end{equation}
where $J_\tau$ is as in \eqref{eq:Jtau}\footnote{In Section~\ref{sec:Gammaconvergence} we required various rescaled versions of $J_\tau$ defined on different domains, for technical reasons related to the $\Gamma$-convergence proof. Any of those functionals could be substituted in \eqref{eq:minimizeJtauoverKM} for $J_\tau$, as long as their domain contains $\mathcal{K}_M$.} is an approximate solution to the $F_0$ minimization problem in \eqref{eq:minimprobsF0} (with $q=1$); see for example the discussion in Remark~\ref{rem:whataboutmass?}.

We propose the \ref{alg:massOKMBO} scheme described below to include the mass condition into the \ref{alg:OKMBO} scheme. As part of the algorithm we need a node relabelling function. For $u\in \mathcal{V}$, let $R_u: V \to \{1, \ldots, n\}$ be a bijection such that, for all $i, j \in V$, $R_u(i) < R_u(j)$ if and only if $u_i \geq u_j$. Note that such a function need not be unique, as it is possible that $u_i=u_j$ while $i\neq j$. Given a relabelling function $R_u$, we will define the relabeled version of $u$ denoted by $u^R \in \mathcal{V}$, by, for all $i\in V$,
\begin{equation}\label{eq:relabelling}
u^R_i := u_{R_u^{-1}(i)}.
\end{equation}
In other words, $R_u$ relabels the nodes in $V$ with labels in $\{1, \dots, n\}$, such that in the new labelling we have $u^R_1 \geq u^R_2 \geq \ldots \geq u^R_n$.

Because this will be of importance later in the paper, we introduce the new set of {\it almost binary functions} with prescribed mass $M\geq 0$:
\[
\mathcal{V}^{ab}_M := \left\{ u\in \mathcal{K}_M: \exists i \in V\,\, \forall j\in V\setminus\{i\}\,\, u_j \in \{0,1\}\right\}.
\]

\begin{asm}{(mcOKMBO)}
\KwData{A prescribed mass value $M\in\mathfrak{M}$, an initial function $v^0\in \mathcal{V}^{ab}_M$\footnotemark, a parameter $r\in [0,1]$, a parameter $\gamma\geq 0$, a time step $\tau > 0$, and the number of iterations $N\in \N\cup\{\infty\}$. }
\KwOut{A sequence of functions $\{v^k\}_{k=1}^N \subset \mathcal{V}_M^{ab}$, which is the \ref{alg:massOKMBO} evolution of $v^0$. }
\For{$k = 1 \ \KwTo \  N$,}{
{\bf ODE step.} Compute $u\in \mathcal{V}_\infty$ by solving \eqref{eq:MBOstepa}, where $u_0 = v^{k-1}$.

\medskip

{\bf Mass conserving threshold step.} Let $R_{u}$ be a relabelling function and $u^R$ the relabelled version of $u$ as in \eqref{eq:relabelling}. Let $i^*$ be the unique $i\in V$ such that 
\[
\sum_{i=1}^{i^*} d_i^r u_i \leq M \qquad \text{ and } \sum_{i=1}^{i^*+1} d_i^r u_i > M.
\]
Define $v^k \in \mathcal{V}$ by, for all $i\in V$,
\[
v^k_i := \begin{cases}
1, &\text{if } 1 \leq i \leq i^*,\\
d_i^{-r} \left(M - \sum_{i=1}^{i^*} d_i^r u_i\right), &\text{if } i=i^*+1,\\
0, &\text{if } i^*+2 \leq i \leq n.
\end{cases}
\]
}
\caption{\label{alg:massOKMBO} The mass conserving graph Ohta-Kawasaki Merriman-Bence-Osher algorithm}
\end{asm}
\footnotetext{As the algorithm enforces the mass condition in each iteration, it is not necessary for the initial condition to satisfy the mass condition (or even to be almost binary, it could be any function in $\mathcal{K}$), but for a cleaner presentation we assume it does (and is).}

We see that the ODE step in \ref{alg:massOKMBO} is as the ODE step in \ref{alg:OKMBO}, using the outcome of the previous iteration as initial condition. However, the threshold step is significantly different. In creating the function $v^k$, it assigns the available mass to the nodes $\{1, \ldots, i^*\}$ on which $u$ has the highest value. Note that if $r=0$, there is exactly enough mass to assign the value $1$ to each node in $\{1, \ldots, i^*\}$, since we assumed that $M\in\mathfrak{M}$ and each node contributes the same value to the mass via the factor $d_i^r=1$. In this case we see that $v^k_{i^*+1} = 0$. However, if $r\in (0,1]$, this is not necessarily the case and it is possible to end up with a value in $(0,1)$ being assigned to $v^k_{i^*+1}$ (even if $v^{k-1}\in \mathcal{V}_M^b$). Hence, in general $v^k \in \mathcal{V}_M^{ab}$, but not necessarily $v^k\in \mathcal{V}_M^b$.  

Of course there is no issue in evaluating $F_0(v^k)$ for almost binary functions $v^k$, but strictly speaking an almost binary $v^N$ cannot serve as approximate solution to the $F_0$ minimization problem in \eqref{eq:minimprobsF0} as it is not admissible. 
We can either accept that the qualifier ``approximate'' refers not only to approximate minimization, but also to the fact that $v^N$ is binary when restricted to $V\setminus\{i^*+1\}$, but not necessarily on all of $V$, or we can apply a final thresholding step to $v^N$ and set the value at node $i^*+1$ to either $0$ or $1$ depending on which choice leads to the lowest value of $F_0$ and/or the smallest deviation of the mass from the prescribed mass $M$. In the latter case the function will be binary, but the adherence to the mass constraint will be ``approximate''. We emphasize again that this is not an issue when $r=0$ (or on a regular graph; i.e. a graph in which each node has the same degree). This case is the most interesting case, as the mass condition can be very restrictive when $r\in (0,1]$, especially on (weighted) graphs in which most nodes each have a different degree. Whenever we consider examples for $r\in (0,1]$ in this paper, we will use the first interpretation of ``approximate'', i.e. we will use $v^N$ as is and accept that its value at node $i^*+1$ may be in $(0,1)$.

Note that the sequence $\{v^k\}_{k=1}^N$ generated by the \ref{alg:massOKMBO} scheme is not necessarily unique, as the relabelling function $R_u$ in the mass conserving threshold step is not uniquely determined if there are two different nodes $i,j\in V$ such that $u_i=u_j$. This non-uniqueness of $R_u$ can lead to non-uniqueness in $v^k$ if exchanging the labels $R_u(i)$ and $R_u(j)$ of those nodes leads to a different `threshold node' $i^*$. In the practice of our examples in Section~\ref{sec:numerical} we used the MATLAB function \texttt{sort($\cdot$, `descend')} to order the nodes.

Lemma~\ref{lem:massLyapunov} shows that some of the important properties of \ref{alg:OKMBO} from Lemma~\ref{lem:Lyapunov} and Corollary~\ref{cor:finiteconvergence} also hold for \ref{alg:massOKMBO}. First we prove an intermediate lemma.

\begin{lemma}\label{lem:wminimization}
Let $G=(V,E,\omega)\in \mathcal{G}$, $M\geq 0$ and $z\in V$. Consider the minimization problem 
\begin{equation}\label{eq:minprobforw}
\max_{w\in \mathcal{V}} \sum_{l\in V} w_l z_l, \quad \text{subject to} \quad \sum_{l\in V} w_l = M \quad \text{and} \quad \forall l\in V\,\,\, 0\leq w_l \leq d_l^r.
\end{equation}
Let $w^*\in V$ satisfy the constraints in \eqref{eq:minprobforw}. Then $w^*$ is a minimizer for \eqref{eq:minprobforw} if and only if for all $i,j\in V$, if $z_i<z_j$, then $w^*_i =d_i^r$ or $w^*_j  = 0$.
\end{lemma}
\begin{proof}
First consider the case where $z$ is constant, i.e. for all $l\in V$, $z_l = z_1$. Then, for any $w\in \mathcal{V}$ which satisfies the constraints in \eqref{eq:minprobforw}, we have $\sum_{l\in V} w_l z_l = z_1 M$; hence any such $w$ is trivially a minimizer. Moreover, the condition $z_i<z_j$ is never satisfied. Hence the result of the lemma holds. In the rest of the proof we assume $z$ is not constant.

Next we note that, if $M=0$, only $w=0$ is admissible, in which case again the result of the lemma trivially holds. Hence we now assume $M>0$.

Furthermore, if we define $\tilde z := z-z_n$, then, for all $l\in V$, $\tilde z_l \geq 0$. Moreover, for all $w\in \mathcal{V}$ which satisfy the constraints in \eqref{eq:minprobforw}, we have
\[
\sum_{l\in V} w_l z_l = \sum_{l\in V} w_l \tilde z_l + z_n \sum_{l\in V} w_l = \sum_{l\in V} w_l \tilde z_l + z_n M.
\]
Hence we can assume, without loss of generality, that, for all $l\in V$, $z_l\geq 0$.

To prove the ``only if'' statement, let $w^*$ be a minimizer for \eqref{eq:minprobforw} which satisfies the constraints. Assume for a proof by contradiction that there are $i,j\in V$ and $\e \in \left(0, \min(d_i^r,d_j^r)\right)$ such that $z_i < z_j$, $0\leq w^*_i \leq d_i^r-\e$, and $\e \leq w^*_j \leq d_j^r$. Define $w^{**}\in \R^n$ by, for all $l\in V$,
\[
w^{**}_l = \begin{cases}
w^*_l, &\text{if } l\not\in\{i,j\},\\
w^*_i + \e, &\text{if } l = i,\\
w^*_j - \e, &\text{if } l=j.
\end{cases}
\]
Then $\sum_{l\in V} w^{**}_l = \sum_{l\in V} w^*_l = M$, for all $l\in V$, $0\leq w^{**}_k \leq d_l^r$, and
\[
\sum_{l\in V} w^{**}_l z_l = \sum_{l\in V} w^*_l z_l + \e (z_i-z_j) < \sum_{l\in V} w^*_l z_l.
\]
This contradicts the fact that $w^*$ is a minimizer. Hence, for all $i,j\in V$ and for all $\e\in \left(0, \min(d_i^r,d_j^r)\right)$, if $z_i < z_j$, then $w^*_i > d_i^r - \e$ or $w^*_j < \e$. Thus, if $z_i<z_j$, then $w^*_i = d_i^r$ or $w^*_j=0$.

To prove the ``if'' statement in the lemma, assume that 
$w^*\in \mathcal{V}$ satisfies the constraints in \eqref{eq:minprobforw} and that for all $i,j\in V$, if $z_i<z_j$, then $w^*_i =d_i^r$ or $w^*_j  = 0$. Let $R^z$ be a relabelling function and let $z^R$ and ${w^*}^R$ be the corresponding relabelled versions of $z$ and $w^*$, respectively, as in \eqref{eq:relabelling}. For notational simplicity, we will drop the superscript $R$ from $z^R$ and ${w^*}^R$ in the rest of this proof. Define $L_1 := \min\{l\in V: w^*_l > 0 \}$. Since $M>0$, $w^*\neq 0$ and thus $L_1 \leq n$ exists. 

Assume first that, for all $l>L_1$, $z_l = z_{L_1}$. Because, for all $l<L_1$, we have $w^*_l=0$, we compute
\[
\sum_{l\in V} w^*_l z_l = \sum_{l=L_1}^n w^*_l z_l = z_{L_1} \sum_{l=L_1}^n w^*_l = z_{L_1} \sum_{l\in V} w^*_l = z_{L_1} M.
\]
Moreover we note that, by assumption $z_{L_1} = z_n = \min\{z_l\in \R: l\in V\}$, hence for all $w\in \mathcal{V}$ which satisfy the constraints in \eqref{eq:minprobforw}, we have
\[
\sum_{l\in V} w_l z_l \geq z_{L_1} \sum_{l\in V} = z_{L_1} M = \sum_{l\in V} w^*_l z_l.
\]
Hence $w^*$ is a minimizer in \eqref{eq:minprobforw}. 

Next we assume instead that there is an $l>L_1$, such that $z_l < z_{L_1}$. Define $L_2 := \min\{l\in V: z_l < z_{L_1}\}$ and let $w\in \mathcal{V}$ satisfy the constraints in \eqref{eq:minprobforw}. Per definition $L_2>L_1$. By construction we have that, for all $l\geq L_2$, $z_l < z_{L_1}$ and for all $l\in [L_1, L_2)$, $z_l=Z_{L_1}$. By definition $z_{L_1} \neq 0$, thus, by our assumption on $w^*$ it follows that, for all $l\geq L_2$, $w^*_l = d^r_l \geq w_l$. We compute
\begin{align*}
\sum_{l=1}^n z_l w_l  &\geq \sum_{l=L_1}^n z_l w_l = \sum_{l=L_1^n} \left(w_l^*-w_l\right) \left(z_{L_1} - z_l\right) - z_{L_1} \sum_{l=L_1}^n w_l^* + z_{L_1} \sum_{l=L_1}^n w_l + \sum_{l=L_1}^n w^*_l z_l\\
&= \sum_{l=L_2}^n \left(w_l^*-w_l\right) \left(z_{L_1} - z_l\right) - z_{L_1} M + z_{L_1} M + \sum_{l\in V} w^*_l z_l\\
&\geq \sum_{l\in V} w^*_l z_l.
\end{align*}
Hence also in this case $w^*$ is a minimizer.
\end{proof}

\begin{lemma}\label{lem:massLyapunov}
Let $G=(V,E\omega) \in \mathcal{G}$, $\gamma\geq 0$, $\tau>0$, and $M\geq 0$. Let $J_\tau: \mathcal{V}\to \R$ be as in \eqref{eq:Jtau}, $v^0\in \mathcal{V}_M^{ab}$, and let $\{v^k\}_{k=1}^N \subset \mathcal{V}_M^{ab}$ be a sequence generated by \ref{alg:massOKMBO}. Then, for all  $k\in \{1, \ldots, N\}$,
\begin{equation}\label{eq:massMBOminimiz}
v^k \in \underset{v\in \mathcal{K}_M}\argmin\, dJ_\tau^{v^{k-1}}(v),
\end{equation}
where $dJ_\tau$ is given in \eqref{eq:dJtau}.
Moreover, for all $k\in \{1,\ldots, N\}$, $J_\tau(v^k) \leq J_\tau(v^{k-1})$, with equality if and only if $v^k=v^{k-1}$. Finally, there is a $K\geq 0$ such that for all $k\geq K$, $v^k = v^K$.
\end{lemma}
\begin{proof}
For all $i\in V$, define $w_i:=d_i^r v_i$ and $z_i := \left(\chi_V-2 e^{-\tau L} v^{k-1}\right)_i$. Then the minimization problem \eqref{eq:massMBOminimiz} turns into \eqref{eq:minprobforw}. Hence, by Lemma~\ref{lem:wminimization}, $v^*$ is a solution of \eqref{eq:massMBOminimiz} if and only if $v^*$ satisfies the constraints in \eqref{eq:massMBOminimiz} and for all $i, j\in V$, if $\left(e^{-\tau L} v^{k-1}\right)_i > \left(e^{-\tau L} v^{k-1}\right)_j$, then $v^*_i = d^r_i$ or $v^*_j = 0$. It is easily checked that $v^k$ generated from $v^{k-1}$ by one iteration of the \ref{alg:massOKMBO} algorithm satisfies these properties.

We note that \eqref{eq:massMBOminimiz} differs from \eqref{eq:MBOminimiz} only in the set of admissible functions over which the minimization takes place. This difference does not necessitate any change in the proof of the second part of the lemma compared to the proof of the equivalent statements at the end of Lemma~\ref{lem:Lyapunov}.

The final part of the lemma is trivially true if $N\in\N$. Now assume $N=\infty$. The proof is similar to that of Corollary~\ref{cor:finiteconvergence}. In the current case, however, our functions $v^k$ are not necessarily binary. We note that for each $k$, there is at most one node $i(k)\in V$ at which $v^k$ can take a value in $(0,1)$. For fixed $k$ and $i(k)$, there are only finitely many different possible functions that $\left.v^k\right|_{V\setminus\{i(k)\}}$ can be. Because $\mathcal{M}\left(v^k\right) = \sum_{i\in V\setminus\{i(k)\}} \left(\left.v^k\right|_{V\setminus\{i(k)\}}\right)_i + d_{i(k)}^r v_{i(k)} =  M$, this leads to finitely many possible values $v^k_{i(k)}$ can have. Since $i(k)$ can be only one of finitely many ($n$) nodes, there are only finitely many possible functions that $v^k$ can be. Hence the proof now follows as in Corollary~\ref{cor:finiteconvergence}.
\end{proof}

\begin{remark}
Similar to what we saw in Remark~\ref{rem:sequentiallinearprogramming} about \ref{alg:OKMBO}, we note that \eqref{eq:massMBOminimiz} is a sequential linear programming approach to minimizing $J_\tau$ over $\mathcal{K}_M$; the linear approximation of $J_\tau$ over $\mathcal{K}_M$ is minimized instead.
\end{remark}

\begin{remark}\label{rem:pinningmassOKMBO}
In Lemma~\ref{lem:dynamicsbounds} and Remark~\ref{rem:butwhatabouttheemptyandfullset} we saw that if $\tau$ is chosen too small in \ref{alg:OKMBO} pinning occurs, while if $\tau$ is chosen too large, a constant stationary state will be achieved in one iteration of \ref{alg:OKMBO}. The choice of $\tau$ is also critically important for \ref{alg:massOKMBO}, yet the details of the situation are somewhat different in this case.

Using an expansion as in \eqref{eq:expansion} and the eigenfunctions and eigenvalues as in Lemma~\ref{lem:Lspectrum}, we can write the solution to \eqref{eq:MBOstepa} as
\begin{equation}\label{eq:solutionu}
u(t) = \sum_{m=0}^{n-1} e^{-t \Lambda_m} \langle u_0 \phi^m\rangle_{\mathcal{V}}\, \phi^m,
\end{equation}
hence $u(t) \to \mathcal{A}(u_0)$ as $t\to \infty$. Thus for large $\tau$, the function $u(\tau) \in\mathcal{V}$ will be approximately constant. It will typically not be exactly constant though, and hence the mass conserving threshold step of \ref{alg:massOKMBO} could still be able to produce a non-arbitrary result, in the sense that the result is based on an actual ordering inherent in $u(\tau)$ instead of on an arbitrary ordering of nodes on all of which $u(\tau)$ has the same value. However, for those nodes $i,j\in V$ for which $u_i(\tau) \neq u_j(\tau)$, the differences in value of $u(\tau)$ are likely very small when $\tau$ is large. In a numerical implementation they might even be below machine precision, which renders the resulting output meaningless, determined by the particularities of the sorting algorithm used, instead of the mathematical problem. From Section~\ref{sec:Gammaconvergence} we know that, for $F_0$ minimization purposes, we are mainly interested in small $\tau$, so we will avoid choosing $\tau$ too large in our implementations in Section~\ref{sec:numerical}.

When $\tau$ is small, pinning can occur in the \ref{alg:massOKMBO} algorithm\footnote{We say that pinning occurs in the $k^{\text{th}}$ step if $v^k = v^{k-1}$.}, as it did in \ref{alg:OKMBO}, but the underlying reasons are different in both cases. In \ref{alg:OKMBO} pinning at a node occurs when $\tau$ is so small that the value of $u$ at that node changes by an amount less than (or equal to) $1/2$, whereas pinning in \ref{alg:massOKMBO} occurs in step $k$, if, for all $i,j\in V$ for which $v^{k-1}_i=1\neq v^{k-1}_j$ we have $\left(e^{-\tau L} v^{k-1}\right)_i > \left(e^{-\tau L} v^{k-1}\right)_j$, and for all $i,j \in V$ for which $v^{k-1}_i=0\neq v^{k-1}_j$ we have $\left(e^{-\tau L} v^{k-1}\right)_i < \left(e^{-\tau L} v^{k-1}\right)_j$\footnote{Pinning definitely occurs if these two strict inequalities hold. Depending on which choices the ordering process makes when there are $i,j\in V$ for which $v^{k-1}_i=v^{k-1}_j$, pinning might also occur if non-strict inequalities hold instead. For simplicity of the discussion we assume that the ordering process is such that if in step $k-1$ node $i$ is ranked before node $j$, then these nodes retain their relative ordering in step $k$ unless $v^k_i < v^k_j$ (in particular, we assume their relative ordering does not change if $v^k_i=v^k_j$).}. We need both these conditions to guarantee that $v^k = v^{k-1}$, because of the possibility that $v^{k-1}$ or $v^k$ take values in $(0,1)$ at a single node. In a situation where $v^{k-1}$ and $v^k$ are guaranteed to be in $\mathcal{V}^b_M$, e.g. when $r=0$, and $v^{k-1}$ is not constant we can simplify these pinning conditions: pinning will not occur if
\begin{equation}\label{eq:choiceoftau}
\underset{i\in \{i\in V: v^{k-1}_i = 1\}}\min\, \left(e^{-\tau L} v^{k-1}\right)_i < \underset{j\in \{j\in V: v^{k-1}_j = 0\}}\max\, \left(e^{-\tau L} v^{k-1}\right)_j.
\end{equation}
When $v^{k-1}$ and $v^k$ are not guaranteed to be in $\mathcal{V}^b_M$, the condition above is still sufficient, but might not be necessary, for the absence of pinning.
\end{remark}

\section{Special classes of graphs}\label{sec:specialclasses}

There are certain classes of graphs on which the dynamics of equation \eqref{eq:MBOstepa}, can be directly related to graph diffusion equations, in a way which we will make precise in Section~\ref{sec:graphtransformation}. The tools which we develop in that section will again be used in Section~\ref{sec:comppinning} to prove additional comparison principles.

\subsection{Graph transformation}\label{sec:graphtransformation}

\begin{definition}\label{def:graphclasses}
Let $G=(V,E,\omega)\in \mathcal{G}$. For all $j\in V$, let $\nu^{V\setminus\{j\}}$ be the equilibrium measure which solves \eqref{eq:equilibrium} for $S=V\setminus\{j\}$ and define the functions $f^j\in \mathcal{V}$ as
\begin{equation}\label{eq:fij}
f^j:= \nu^{V\setminus\{j\}} - \mathcal{A}\left(\nu^{V\setminus\{j\}}\right).
\end{equation}

Now we introduce the following classes of graphs:
\begin{enumerate}
\item $\mathcal{C} := \left\{G\in \mathcal{G}: \forall j\in V\,\, \forall i\in V\setminus\{j\}\,\,  f^j_i \geq 0\right\}$,
\item $\mathcal{C}^0 := \left\{G\in \mathcal{G}: \forall j\in V\,\, \forall i\in V\setminus\{j\}\,\, \omega_{ij}>0 \text{ or }  f^j_i \geq 0 \right\}$,
\item $\mathcal{C}_\gamma := \left\{G\in \mathcal{C}^0: \forall j\in V\,\, \forall i\in V\setminus\{j\}\,\,  \omega_{ij}=0 \text{ or } d_i^{-r} \omega_{ij} +  \gamma \frac{d_j^r}{\vol V} f^j_i > 0\right\}$, \text{ for } $\gamma>0$.
\end{enumerate}
For $\gamma=0$, we define $\mathcal{C}_0 := \mathcal{G}$\footnote{The definition of $\mathcal{C}_0$  is purely for notational convenience, so that we do not always need to treat the case $\gamma=0$ separately.}.
\end{definition}

\begin{remark}
Let us have a closer look at the properties of graphs in $\mathcal{C}_\gamma$. Let $\gamma\geq 0$. If $G\in \mathcal{C}_\gamma$, then per definition $G\in \mathcal{C}^0$. Let $i,j\in V$. If $\omega_{ij}=0$, then $f^j_i\geq 0$ and thus, per definition of $\mathcal{C}^0$, $d_i^{-r} \omega_{ij} +  \gamma \frac{d_j^r}{\vol V} f^j_i \geq 0$. On the other hand, if $\omega_{ij} > 0 $, then per definition of $\mathcal{C}_\gamma$, $d_i^{-r} \omega_{ij} +  \gamma \frac{d_j^r}{\vol V} f^j_i > 0$.
\end{remark}

\begin{lemma}\label{lem:Csubsets}
Let the setting and notation be as in Definition~\ref{def:graphclasses}. Then, $\mathcal{C} \subset \mathcal{C}^0$ and, for all $\gamma\geq 0$, $\mathcal{C} \subset \mathcal{C}_\gamma$. Moreover, if $G\in \mathcal{C}^0\setminus \mathcal{C}$, there is a $\gamma_*(G)>0$ such that, for all $\gamma\in [0,\gamma_*(G))$, $G\in \mathcal{C}_\gamma$.
\end{lemma}
\begin{proof}
The first two inclusions stated in the lemma follow immediately from the definitions of the sets involved. To prove the final statement, let $G\in \mathcal{C}^0$ and let $j\in V$, $i\in V\setminus\{j\}$. If $f^j_i \geq 0$, then, $\omega_{ij}=0$ or, for all $\gamma\geq 0$, $ d_i^{-r} \omega_{ij} +   \gamma \frac{d_j^r}{\vol V} f^j_i \geq d_i^{-r} \omega_{ij} \geq 0$. If $f^j_i < 0$ (and, by the assumption that $G\not\in \mathcal{C}$, there are $j\in V$, $i\in V\setminus\{j\}$ for which this is the case), then by definition of $\mathcal{C}^0$ we have $\omega_{ij} > 0$. Define
\[
\gamma_*(G) := \vol V \min\left\{d_i^{-r} d_j^{-r} \omega_{ij} \left|f^j_i\right|^{-1}: j \in V, i\in V\setminus\{j\} \text{ such that } f^j_i < 0 \right\}
\]
and let $\gamma \in [0, \gamma_*(G))$, then $ d_i^{-r} \omega_{ij} +  \gamma\frac{d_j^r}{\vol V} f^j_i > d_i^{-r} \omega_{ij} - \gamma_*(G)\frac{d_j^r}{\vol V} |f_j^i| \geq 0$.
\end{proof}

The following lemma shows that $\mathcal{C}$ is not empty (and thus, by Lemma~\ref{lem:Csubsets}, so are $\mathcal{C}^0$ and, for all $\gamma\geq 0$, $\mathcal{C}_\gamma$).

\begin{lemma}\label{lem:stargraphinC}
Let $G=(V,E,\omega) \in \mathcal{G}$ be an unweighted star graph as in Definition~\ref{def:bipartstar} with $n\geq 3$ nodes. Then $G\in\mathcal{C}$.
\end{lemma}
\begin{proof}
Using Lemma~\ref{lem:starequil}, we compute
\begin{align*}
\mathcal{M}\left(\nu^{V\setminus\{1\}}\right) &= \sum_{k\in V\setminus\{1\}} d_k^r = \vol{V} - d_1^r,\\
\mathcal{M}\left(\nu^{V\setminus\{j\}}\right) &= (\vol{V}-1) d_1^r + \vol{V} \sum_{k\in V\setminus\{1,j\}} d_k^r = \vol{V} (\vol{V} - d_j^r) - d^r_1\\ &= \vol{V} (\vol{V} - 1) - d^r_1,
\end{align*}
where $j\in V\setminus\{1\}$ in the second line. Hence, if $i\in V\setminus\{1\}$,
\[
f^1_i = 1 - \frac{\vol{V}-d_1^r}{\vol{V}} = \frac{d_1^r}{\vol V} > 0.
\]
Furthermore, if $j\in V\setminus\{1\}$ and $i\in V\setminus\{j\}$,
\[
f^j_i \geq \vol{V}-1 - \frac{\vol{V} (\vol{V} - 1) - d^r_1}{\vol{V}} = \frac{d^r_1}{\vol{V}} > 0.
\]
We conclude that $G\in \mathcal{C}$.
\end{proof}

\begin{remark}
The following is an example of a graph $G\in \mathcal{C}^0\setminus\mathcal{C}$. Let $G$ be the unweighted graph with $V=\{1, 2, 3, 4\}$ and $E=\{(1,2), (2,3), (3,4)\}$. A quick computation verifies that the equilibrium measures $\nu^{V\setminus\{1\}}$ and $\nu^{V\setminus\{2\}}$ are given by, for $i\in V$,
\[
\nu^{V\setminus\{1\}}_i := \begin{cases}
0, &\text{if } i=1,\\
2^{1+r}+1, &\text{if } i =2,\\
2^{1+r}+2^r+2, &\text{if } i=3,\\
2^{1+r}+2^r+3, &\text{if } i=4,
\end{cases}
\quad \text{ and } \quad
\nu^{V\setminus\{2\}}_i := \begin{cases}
1, &\text{if } i=1,\\
0, &\text{if } i =2,\\
2^r + 1, &\text{if } i=3,\\
2^r + 2, &\text{if } i=4.
\end{cases}
\]
We also compute that
\[
\vol V = 2^{r+1} + 2, \quad \mathcal{M}\left(\nu^{V\setminus\{1\}}\right) = 2^{2r+2} + 2^{r+2}  + 2^{r+1} +  2^{2r} + 3, \quad \mathcal{M}\left(\nu^{V\setminus\{2\}}\right) = 2^{r+1} + 2^{2r} + 3.
\]
Hence
\begin{align*}
\mathcal{A}\left(\nu^{V\setminus\{1\}}\right) _2 &= 2^{r+1} + 1 + \frac{2^{2r}+1}{2^{r+1}+2} > \nu^{V\setminus\{1\}}_2,\\
\vol V \nu^{V\setminus\{2\}}_1 &= 2^r+2 < 3+2^{r+1} < \mathcal{M}\left(\nu^{V\setminus\{2\}}\right),
\end{align*}
so $f^1_2<0$ and $f^2_1<0$ and thus $G\not\in \mathcal{C}$. However, $(2^r+1)(2^{r+1}+2) = 2(2^{2r}+2^{r+1}+1) > 2^{2r}+1$, and thus
\begin{align*}
\mathcal{A}\left(\nu^{V\setminus\{1\}}\right)_3- 2^{r+1} - 1 &= \frac{2^{2r}+1}{2^{r+1}+2}  < \nu^{V\setminus\{1\}}_3 - 2^{r+1} - 1,\\
\vol V \nu^{V\setminus\{2\}}_3 &= 2^{r+1} + 2^{2r} + 2^r + 2 > 2^{r+1} + 2^{2r} + 1 + 2 =  \mathcal{M}\left(\nu^{V\setminus\{2\}}\right).
\end{align*}
Therefore $f^1_3>0$ and $f^2_3>0$. Since $\nu^{V\setminus\{1\}}_4 > \nu^{V\setminus\{1\}}_3$ and $\nu^{V\setminus\{2\}}_3>\nu^{V\setminus\{2\}}_3$, it follows that also $f^1_4>0$ and $f^2_4>0$. The corresponding inequalities for $f^3$ and $f^4$ follow by symmetry. We conclude that, if $\omega_{ij}=0$, then $f^j_i>0$. Hence $G\in \mathcal{C}^0$.
\end{remark}

\begin{lemma}\label{lem:someconditionsforCC0}
Let $G=(V,E,\omega)\in \mathcal{G}$. For all $j\in V$, let $\nu^{V\setminus\{j\}}$ be the equilibrium measure which solves \eqref{eq:equilibrium} for $S=V\setminus\{j\}$ and define,
for $i,j\in V$, 
\begin{equation}\label{eq:sharedneighbours}
\mathcal{N}_s(i,j) := \sum_{k\in \mathcal{N}(j)} \omega_{ik},
\end{equation}
where $\mathcal{N}(j)$ is the set of neighbours of node $j$, as in \eqref{eq:neighbours}. Then following statements are true.
\begin{enumerate}
\item\label{item:GinC} If, for all $j\in V$ and for all $i\in \mathcal{N}(j)$, $\omega_{ij} \mathcal{M}\left(\nu^{V\setminus\{j\}}\right) \leq \vol{V} d_i^r$, then $G\in \mathcal{C}$. 
\item\label{item:GinC0} If, for all $j\in V$ and for all $i\in V\setminus \left(\{j\} \cup \mathcal{N}(j)\right)$, $\mathcal{N}_s(i,j) \mathcal{M}\left(\nu^{V\setminus\{j\}}\right) \leq \vol{V} d_i^r$, then $G\in \mathcal{C}^0$.
\end{enumerate}
\end{lemma}
\begin{proof}
To prove these statements, we fix $j\in V$ and use a similar approach as in the proof of Lemma~\ref{lem:nuandkappa} with $x:= \frac{\mathcal{M}\left(\nu^{V\setminus\{j\}}\right)}{\vol{V}}$. Note, for $i\in V\setminus\{j\}$, that $\nu^{V\setminus\{j\}}_i \geq x$ implies that $f^j_i \geq 0$, where $f^j_i$ is as in \eqref{eq:fij}.

To prove \ref{item:GinC}, set $S:=V\setminus\{j\}$. We compute, for $i\in S$, $\left(\kappa_S\right)_i = d_i^r \omega_{ij}$. Note that the inequality in the assumption in \ref{item:GinC} trivially holds for all $i\in V\setminus\mathcal{N}(j)$. Hence, for all $i\in S$, $x \left(\kappa_S\right)_i \leq 1$. Repeating the argument in the proof of Lemma~\ref{lem:nuandkappa}, we find that, for all $i\in S$, $\nu^{V\setminus\{j\}}_i \geq x$, hence $G\in \mathcal{C}$.

To prove \ref{item:GinC0}, set $S:= V\setminus \left(\{j\} \cup \mathcal{N}(j)\right)$. Then we have, for $i\in S$, $\left(\kappa_S\right)_i = d_i^{-r} \sum_{k\in \{j\} \cup \mathcal{N}(j)} \omega_{ij} = d_i^{-r} \left(\mathcal{N}_s(i,j) + \omega_{ij}\right) = d_i^{-r} \mathcal{N}_s(i,j)$. Hence, by assumption, for all $i\in S$, $x \left(\kappa_S\right)_i \leq 1$. Repeating again the argument in the proof of Lemma~\ref{lem:nuandkappa}, we find that, for all $i\in S$, $\nu^S_i \geq x$. By statement \ref{item:subsetsnu} in Lemma~\ref{lem:equimeasprops}, we also know that $\nu^{V\setminus\{j\}} \geq \nu^S$. Hence $G\in \mathcal{C}$.
\end{proof}

\begin{remark}
Note that, for the quantity in \eqref{eq:sharedneighbours} we have
$\displaystyle
\mathcal{N}_s(i,j) = \sum_{k\in \mathcal{N}(j) \cap \mathcal{N}(i)}\omega_{ik}.
$
Hence, if $G\in \mathcal{G}$ is an unweighted graph, then $\mathcal{N}_s(i,j) = \mathcal{N}_s(j,i)$ is the number of \textit{shared neighbours} of nodes $i$ and $j$, i.e. the number of nodes $k$ for which both edges $(i,k)$ and $(j,k)$ exist in $E$. We also see that, for all $G\in\mathcal{G}$ and $i\in V$, $\mathcal{N}_s(i,i) = d_i$.
\end{remark}

\begin{corol}\label{cor:completegraphinC0}
Let $G=(V,E,\omega)\in \mathcal{G}$ be complete (see Definition~\ref{def:bipartstar}). Then $G\in\mathcal{C^0}$.
\end{corol}
\begin{proof}
Let $j\in V$. Because $G$ is complete, $\mathcal{N}(j) = V\setminus\{j\}$ and thus $V\setminus \left(\{j\} \cup \mathcal{N}(j)\right)=\emptyset$. It follows, either directly from the definition of $\mathcal{C}^0$, or from condition~\ref{item:GinC0} in Lemma~\ref{lem:someconditionsforCC0}, that $G\in\mathcal{C}^0$.
\end{proof}

\begin{remark}
Let us consider a simple example to illustrate the conditions from Lemma~\ref{lem:someconditionsforCC0}. Let $G\in \mathcal{G}$ be the graph with node set $V=\{1,2,3\}$, edge set $E=\{(1,2), (2,3), (3,1)\}$ and edge weights $\omega_{12}=\omega_{23} = 1$ and $\omega_{13}=\omega>0$. We note right away that, since $G$ is a complete graph, condition~\ref{item:GinC0} in Lemma~\ref{lem:someconditionsforCC0} is trivially satisfied (see Corollary~\ref{cor:completegraphinC0}).

We compute $d_1 = d_3=1+\omega$, $d_2=2$, $\vol V = 4 + 2\omega$. We can confirm via direct computation that the equilibrium measure $\nu^{V\setminus\{1\}}$ is given by, for $i\in V$,
\[
\nu^{V\setminus\{1\}}_i = \frac1{3+2\omega}\begin{cases}
0, &\text{if } i=1,\\
2^r (1+\omega) + (1+\omega)^2, &\text{if } i=2,\\
2^r + 2(1+\omega)^r, &\text{if } i=3.
\end{cases}
\]
Therefore
\begin{align*}
\mathcal{M}\left(\nu^{V\setminus\{1\}}\right) &= \frac{2^r\left(2^r(1+\omega)+(1+\omega)^r\right) + (1+\omega)^r \left(2^r + 2(1+\omega)^r\right)}{3+2\omega}\\ &= \frac{2^{2r} (1+\omega) + 2^{r+1} (1+\omega)^r + 2(1+\omega)^{2r}}{3+2\omega}.
\end{align*}
If we choose $\omega=1$, it is a matter of straightforward computation to check that condition~\ref{item:GinC} from Lemma~\ref{lem:someconditionsforCC0} is satisfied for $j=1$. Since the graph is fully symmetric when $\omega=1$, it then follows that the condition is also satisfied for $j\in \{2,3\}$.

On the other hand, a computation with $\omega=7$ and $r=0$, shows that $\omega \mathcal{M}\left(\nu^{V\setminus\{1\}}\right) > \vol{V} d_3^r$, thus this provides an example of a graph for which condition~\ref{item:GinC0} is satisfied, but condition~\ref{item:GinC} is not. Note that by continuity of $(r, \omega) \mapsto \omega \mathcal{M}\left(\nu^{V\setminus\{1\}}\right) - \vol{V} d_3^r$, the same is true for values of $(r,\omega)$ close to $(0,7)$.
\end{remark}

The following lemma hints at the reason for our interest in the functions $f^j$ from \eqref{eq:fij}.

\begin{lemma}\label{lem:varphijintroduced}
Let $G=(V,E,\omega)\in \mathcal{G}$. Let $j\in V$ and let $f^j\in \mathcal{V}$ be as in \eqref{eq:fij}. Then the function $\varphi^j \in \mathcal{V}$, defined by
\begin{equation}\label{eq:defvarphij}
\varphi^j := - \frac{d_j^r}{\vol{V}} f^j,
\end{equation}
solves \eqref{eq:zeromassvarphiequation} for $\chi_{\{j\}}$.
\end{lemma}
\begin{proof}
From \eqref{eq:fij} and \eqref{eq:equilibrium} it follows immediately that, for all $i\in V\setminus\{j\}$, $\left(\Delta f^j\right)_i = \left(\Delta \nu^{V\setminus\{j\}}\right)_i = 1$. Thus, for all $i\in V\setminus\{j\}$, $\left(\Delta \varphi^j\right)_i = -\frac{d_j^r}{\vol V} = \left(\chi_{\{j\}}\right)_i - \mathcal{A}\left(\chi_{\{j\}}\right)$. Moreover, by \eqref{eq:massu-Au} we have $\displaystyle 0=\mathcal{M}\left(\Delta \varphi^j\right)= d_j^r \left(\Delta \varphi^j\right)_j + \sum_{i\in V\setminus\{j\}} d_i^r \left(\Delta \varphi^j\right)_i$, and thus
\begin{align*}
\left(\Delta \varphi^j\right)_j &= -d_j^{-r} \sum_{i\in V\setminus\{j\}} d_i^r \left(\Delta \varphi^j \right)_i = d_j^{-r} \sum_{i\in V\setminus\{j\}} d_i^r \frac{d_j^r}{\vol{V}} =  \frac{\vol{V}-d_j^r}{\vol{V}}\\
&= 1-\frac{d_j^r}{\vol{V}} = \left(\chi_{\{j\}}\right)_j-\left(\mathcal{A}\left(\chi_{\{j\}}\right)\right)_j.
\end{align*}
Finally, by \eqref{eq:fij}, $\mathcal{M}\left(f^j\right)=0$, thus $\mathcal{M}\left(\varphi^j\right)=0$.
\end{proof}

\begin{corol}\label{cor:spectralvarphij}
Let $G=(V,E,\omega)\in \mathcal{G}$.  Let $\lambda_m$ and $\phi^m$ be the eigenvalues and corresponding eigenfunctions of the graph Laplacian $\Delta$ (with parameter $r$), as in \eqref{eq:eigenvalues}, \eqref{eq:eigenfunctions}. Let $j\in V$. If $\varphi^j \in \mathcal{V}$ is as defined in \eqref{eq:defvarphij}, then, for all $i\in V$,
\begin{equation}\label{eq:spectralvarphij}
\varphi^j_i = \sum_{m=1}^{n-1} \lambda_m^{-1} d_j^r \phi_i^m \phi_j^m.
\end{equation}
In particular, if $f^j$ is as in \eqref{eq:fij} and $i\in V$, then $f^j_i \geq 0$ if and only if $\sum_{m=1}^{n-1} \lambda_m^{-1} d_j^r \phi_i^m \phi_j^m \leq 0$.
\end{corol}
\begin{proof}
Let $j\in V$. By Lemma~\ref{lem:varphijintroduced} we know that $\varphi^j$ solves \eqref{eq:zeromassvarphiequation} for $\chi_{\{j\}}$. Then by \eqref{eq:varphispectral} we can write, for all $i\in V$,
\[
\varphi^j_i =  \sum_{m=1}^{n-1}\lambda_m^{-1} \langle \chi_{\{j\}}, \phi^m\rangle_{\mathcal{V}} \ \phi^m = \sum_{m=1}^{n-1} \lambda_m^{-1} d_j^r \phi_i^m \phi_j^m,
\]
where we used that,
\begin{equation}\label{eq:innerwithchi}
\langle \chi_{\{j\}}, \phi^m\rangle_{\mathcal{V}} = \sum_{k\in V} d_k^r \delta_{jk} \phi^m_k = d_j^r \phi_j^m,
\end{equation}
where $\delta_{jk}$ is the Kronecker delta.

The final statement follows from the definition of $\varphi^j$ in Lemma~\ref{lem:varphijintroduced}, which shows that, for all $i\in V$, $f_i^j \geq 0$ if and only $\varphi^j_i \leq 0$.
\end{proof}

\begin{corol}\label{cor:symmetries}
Let $G=(V,E,\omega)\in \mathcal{G}$. For all $j\in V$, let $\varphi^j$ be as in \eqref{eq:defvarphij}, let $f^j$ be as in \eqref{eq:fij}, and let $\nu^{V\setminus\{j\}}$ be the equilibrium measure for $V\setminus\{j\}$ as in \eqref{eq:equilibrium}. If $r=0$, then, for all $i,j\in V$, $\varphi^j_i = \varphi^i_j$, $f^j_i = f^i_j$, and
\[
\nu^{V\setminus\{j\}}_i - \mathcal{A}\left(\nu^{V\setminus\{j\}}\right)_i = \nu^{V\setminus\{i\}}_j- \mathcal{A}\left(\nu^{V\setminus\{i\}}\right)_j.
\]
\end{corol}
\begin{proof}
This follows immediately from \eqref{eq:spectralvarphij}, \eqref{eq:defvarphij}, and \eqref{eq:fij}.
\end{proof}

\begin{remark}
The result of Corollary~\ref{cor:spectralvarphij} is not only an ingredient in the proof of Theorem~\ref{thm:newgraph} below, but can also be useful when testing numerically whether or not a graph is in $\mathcal{C}$ or in $\mathcal{C}^0$. 
\end{remark}

\begin{lemma}\label{lem:newedgeweights}
Let $\gamma \geq 0$ and let $G=(V,E,\omega) \in \mathcal{C}_\gamma$. Let $L$ be as defined in \eqref{eq:defofL}. Let $\lambda_m$ and $\phi^m$ be the eigenvalues and corresponding eigenfunctions of the graph Laplacian $\Delta$ (with parameter $r$), as in \eqref{eq:eigenvalues}, \eqref{eq:eigenfunctions} and define 
\begin{equation}\label{eq:tildeomega}
\tilde \omega_{ij} := \begin{cases}
-d_j^r\sum_{m=1}^{n-1} \Lambda_m \phi_i^m \phi_j^m, &\text{if } i\neq j,\\
0, & \text{if } i=j,
\end{cases}
\end{equation}
where $\Lambda_m$ is defined in \eqref{eq:eigvalsofL}. Then, for all $i,j\in V$, $\tilde \omega_{ij} \geq 0$. Moreover, if $\omega_{ij}>0$, then $\tilde \omega_{ij} > 0$. If, additionally, $G\in \mathcal{C}$, then $\tilde \omega_{ij} \geq d_i^{-r}\omega_{ij}$.
\end{lemma}
\begin{proof}
Expanding $\chi_{\{j\}}$ as in \eqref{eq:expansion} and using \eqref{eq:innerwithchi}, we find, for $i,j\in V$,
\begin{equation}\label{eq:Lchi}
(L\chi_{\{j\}})_i = \sum_{m=1}^{n-1} \langle \chi_{\{j\}}, \phi^m\rangle_{\mathcal{V}} \left(L \phi^m\right)_i = d_j^r \sum_{m=1}^{n-1}\Lambda_m \phi^m_j \phi^m_i.
\end{equation}
Note in particular that, if 
$i\neq j$, then $\tilde \omega_{ij} = -(L\chi_{\{j\}})_i$.

For $i,j \in V$ we also compute
\begin{equation}\label{eq:Deltadelta}
(\Delta \chi_{\{j\}})_i = d_i^{-r} \sum_{k\in V} \omega_{ik} (\delta_{ji} - \delta_{jk}) = d_i^{-r} (d_i \delta_{ji} - \omega_{ij}),
\end{equation}
hence, if 
$i\neq j$, then $\omega_{ij} = - d_i^r \left(\Delta \chi_{\{j\}}\right)_i$. Combining the above with \eqref{eq:defvarphij}, we find for  
$i\neq j$,
\begin{equation}\label{eq:tildeomegaisomegapluspos}
\tilde \omega_{ij} = -(L\chi_{\{j\}})_i = -\left(\Delta \chi_{\{j\}}\right)_i - \gamma \varphi^j_i = d_i^{-r} \omega_{ij}  + \gamma\frac{d_j^r}{\vol{V}} f^j_i \geq 0,
\end{equation}
where the inequality follows since $G\in \mathcal{C}_\gamma$ (note that for $\gamma=0$ the inequality follows from the nonnegativity of $\omega$). Moreover, if $\omega_{ij}>0$, then, by definition of $\mathcal{C}_\gamma$, the inequality is strict, and thus $\tilde \omega_{ij} > 0$\footnote{This property is used to prove connectedness of $\tilde G$ in Theorem~\ref{thm:newgraph}, which is the reason why we did not define $\mathcal{C}_\gamma$ to simply be $\left\{G\in \mathcal{G}: \forall j\in V\,\, \forall i\in V\setminus\{j\}\,\,  d_i^{-r} \omega_{ij} +  \gamma \frac{d_j^r}{\vol V} f^j_i \geq 0\right\}$.}.

If additionally $G\in \mathcal{C}$, then, for $i\neq j$, $f^j_i \geq 0$ and thus by \eqref{eq:tildeomegaisomegapluspos}, $\tilde \omega_{ij} \geq \omega_{ij}$.
\end{proof}

Lemma~\ref{lem:newedgeweights} suggests that, given a graph $G\in \mathcal{C}_\gamma$ with edge weights $\omega$, we can construct a new graph $\tilde G$ with edge weights $\tilde \omega$ as in \eqref{eq:tildeomega}, that are also nonnegative. The next theorem shows that, in fact, if $r=0$, then this new graph is in $\mathcal{G}$ and the graph Laplacian $\tilde \Delta$ on $\tilde G$ is related to $L$.

\begin{theorem}\label{thm:newgraph}
Let $\gamma \geq 0$ and let $G=(V,E,\omega) \in \mathcal{C}_\gamma$. Let $L$ be as defined in \eqref{eq:defofL}. Let $\lambda_m$ and $\phi^m$ be the eigenvalues and corresponding eigenfunctions of the graph Laplacian $\Delta$ (with parameter $r$), as in \eqref{eq:eigenvalues}, \eqref{eq:eigenfunctions}. Assume $r=0$ and let $\tilde \omega$ be as in \eqref{eq:tildeomega}. Let $\tilde E \subset V^2$ contain an undirected edge $(i,j)$ between $i\in V$ and $j\in V$ if and only if $\tilde \omega_{ij} >0$. Then $\tilde G=(V, \tilde E, \tilde \omega) \in \mathcal{G}$. Let $\tilde \Delta$ be the graph Laplacian (with parameter $\tilde r$) on $\tilde G$. If $\tilde r=0$, then $\tilde \Delta = L$.
\end{theorem}
\begin{proof}
In the following it is instructive to keep $r, \tilde r\in [0,1]$ as unspecified parameters in the proof and point out explicitly where the assumptions $r=0$ and $\tilde r = 0$ are used.

From the definition of $\tilde \omega_{ij}$ in \eqref{eq:tildeomega} it follows directly that $\tilde G$ has no self-loops ($\tilde \omega_{ii}=0$). Moreover, using $r=0$ in \eqref{eq:tildeomega}, we see that $\tilde \omega_{ij} = \tilde \omega_{ji}$ and thus $\tilde G$ is undirected. Furthermore, by 
Lemma~\ref{lem:newedgeweights} we know that, for all $i,j\in V$, if $\omega_{ij} > 0$, then $\tilde \omega_{ij}>0$. Thus $\tilde G$ is connected, because $G$ is connected. Hence $\tilde G\in \mathcal{G}$.

Repeating the computation from \eqref{eq:Deltadelta} for $\tilde \Delta$ instead of $\Delta$, we find, for $i,j\in V$,
\begin{equation}\label{eq:tildeDeltachi}
\left(\tilde \Delta \chi_{\{j\}}\right)_i = \tilde d_i^{-\tilde r} \left(\tilde d_i \delta_{ji} - \tilde \omega_{ij}\right),
\end{equation}
where $\tilde d_i:=\sum_{j\in V} \tilde \omega_{ij}$.
Combining this with \eqref{eq:Lchi}, we find that, if $j\in V$ and $i\in V\setminus\{j\}$, then
\begin{equation}\label{eq:DeltaLineqj}
\left(\tilde \Delta \chi_{\{j\}}\right)_i = - \tilde d_i^{-\tilde r} \tilde \omega_{ij} =   \tilde d_i^{-\tilde r} d^r_j \sum_{m=1}^{n-1} \Lambda_m \phi^m_j \phi^m_i =  \tilde d_i^{-\tilde r}\left(L\chi_{\{j\}}\right)_i.
\end{equation}
Since we have
$\displaystyle
0=\langle \phi^m, \chi_V\rangle_{\mathcal{V}} = \sum_{j\in V} d_j^r \phi^m_j,
$
it follows that, for all $i\in V$, $d_i^r \phi^m_i = -\sum_{j\in V\setminus\{i\}} d_j^r \phi^m_j$. It follows that, for $i\in V$,
\[
\tilde d_i = \sum_{j\in V} \tilde \omega_{ij} = \sum_{j\in V\setminus\{i\}} \tilde \omega_{ij} =  -\sum_{m=1}^{n-1} \Lambda_m \sum_{j\in V\setminus\{i\}}d_j^r \phi^m_j \phi^m_i = d_i^r \sum_{m=1}^{n-1} \Lambda_m \left(\phi^m_i\right)^2.
\]
By \eqref{eq:Lchi} and \eqref{eq:tildeDeltachi} with $i=j$, we then have
\begin{equation}\label{eq:DeltaLi=j}
\left(\tilde \Delta \chi_{\{i\}}\right)_i = \tilde d_i^{1-\tilde r} = \left(d_i^r \sum_{m=1}^{n-1} \Lambda_m \left(\phi^m_i\right)^2\right)^{1-\tilde r} = \left((L\chi_{\{i\}})_i\right)^{1-\tilde r}.
\end{equation}
Now we use $\tilde r=0$ in \eqref{eq:DeltaLineqj} and \eqref{eq:DeltaLi=j} to deduce that, for all $j\in V$, $\tilde \Delta \chi_{\{j\}}= L \chi_{\{j\}}$. Since $\{\chi_{\{i\}} \in \mathcal{V}: i\in V\}$ is a basis for the vector space $\mathcal{V}$, we conclude $\tilde\Delta  = L$.
\end{proof}

\begin{remark}\label{rem:therolesofrandtilder}
In the proof of Theorem~\ref{thm:newgraph} we can trace the roles that $r$ and $\tilde r$ play. We only used the assumption $r=0$ in order te deduce that $\tilde G$ is undirected. The assumption $\tilde r = 0$ is necessary to obtain equality between $\tilde \Delta$ and $L$ in equations \eqref{eq:DeltaLineqj} and \eqref{eq:DeltaLi=j}. 

These assumptions on $r$ and $\tilde r$ have a futher interesting consequence. Since the graphs $G$ and $\tilde G$ have the same node set, both graphs also have the same associated set of node functions $\mathcal{V}$. Moreover, since $r=\tilde r = 0$, the $\mathcal{V}$-inner product is the same for both graphs. Hence we can view $\mathcal{V}$ corresponding to $G$ as the same inner product space as $\mathcal{V}$ corresponding to $\tilde G$. In this setting the operator equality $\tilde \Delta = L$ from Theorem~\ref{thm:newgraph} holds not only between operators on the vector space $\mathcal{V}$, but also between operators on the inner product space $\mathcal{V}$.
\end{remark}

\begin{lemma}\label{lem:F0onnewgraph}
Let $\gamma \geq 0$, $q=1$, and let $G=(V,E,\omega) \in \mathcal{C}_\gamma$. Assume $r=0$. Let $\tilde \omega$ be as in \eqref{eq:tildeomega} and $\tilde E$ as in Theorem~\ref{eq:tildeDeltachi}. Let $\tilde r$ be the $r$-parameter corresponding to the graph $\tilde G=(V,\tilde E, \tilde \omega)$. Suppose $S\subset V$, $F_0$ is as in \eqref{eq:limitOK}, for all $i\in V$ $\tilde d_i := \sum_{j\in V} \tilde\omega_{ij}$, and $\tilde \kappa_S$ is the graph curvature of $S$ as in Definition~\ref{def:graphcurvature} corresponding to $\tilde \omega$. Then $F_0(\chi_S) = \sum_{i,j\in S} \tilde\omega_{ij}$.  Moreover, if $\tilde r = 0$, $F_0(\chi_S) = \sum_{i\in S} \left(\tilde d_i - \left(\tilde \kappa_S\right)_i\right)$.
\end{lemma}
\begin{proof}
From Corollary~\ref{cor:OKexpressions} and \eqref{eq:eigvalsofL} we find
\[
F_0(\chi_S) = \sum_{m=1}^{n-1} \Lambda_m \langle \chi_S, \phi^m\rangle_{\mathcal{V}}^2 = \sum_{m=1}^{n-1} \Lambda_m \sum_{i,j\in V} \left(\chi_S\right)i \left(\chi_S\right)_j d_i^r d_j^r \phi_i^m \phi_j^m = \sum_{i,j\in V}\left(\chi_S\right)_i \left(\chi_S\right)_j \tilde\omega_{ij},
\]
where we used that $r=0$. Moreover, if $\tilde r=0$,
\[
\sum_{i,j\in S}\tilde \omega_{ij} = \sum_{i\in S} \left(\sum_{j\in V} \tilde \omega_{ij} - \sum_{j\in V\setminus S} \tilde \omega_{ij}\right) = \sum_{i\in S} \left(\tilde d_i - \left(\tilde \kappa_S\right)_i\right).
\]
\end{proof}

\begin{lemma}\label{lem:increaseinweights}
Let $\gamma\geq 0$ and $G=(V,E,\omega) \in \mathcal{C}_\gamma$. Let $L$ be as defined in \eqref{eq:defofL} for $G$. Let $\lambda_m$ and $\phi^m$ be the eigenvalues and corresponding eigenfunctions of the graph Laplacian $\Delta$, as in \eqref{eq:eigenvalues}, \eqref{eq:eigenfunctions}, with parameter $r=0$. Let $\tilde \omega$ be as in \eqref{eq:tildeomega}. Then, for all $i, j\in V$ for which $i\neq j$, we have
\[
\gamma\left(\frac12 \sum_{m=1}^{n-1} \frac{\left(\phi_i^m - \phi_j^m\right)^2}{\lambda_m} - \frac{1-n^{-1}}{\lambda_1}\right) \leq \tilde \omega_{ij} - \omega_{ij}  \leq \gamma\left(\frac12 \sum_{m=1}^{n-1} \frac{\left(\phi_i^m - \phi_j^m\right)^2}{\lambda_m} - \frac{1-n^{-1}}{\lambda_{n-1}}\right).
\]
\end{lemma}
\begin{proof}
Let $i,j\in V$ be such that $i\neq j$. Combining \eqref{eq:tildeomegaisomegapluspos} with \eqref{eq:defvarphij} and \eqref{eq:spectralvarphij}, we obtain
\begin{equation}\label{eq:omegadiff}
\tilde \omega_{ij} - \omega_{ij} = \frac\gamma{n} f^j_i = -\varphi^j_i = -\gamma \sum_{m=1}^{n-1} \lambda_m^{-1} \phi_i^m \phi_j^m.
\end{equation}
Note that the matrix which has (the vector representations of)  $\phi^m$ ($m=0, \ldots, n-1$) as orthonormal columns also has orthonormal rows, hence (since $r=0$) we have that $\displaystyle \sum_{m=0}^{n-1} \left(\phi_i^m\right)^2 = 1$. 
Thus
\[
\sum_{m=1}^{n-1} \lambda_m^{-1} \left(\phi_i^m\right)^2 \leq \lambda_1^{-1} \sum_{m=1}^{n-1} \left(\phi_i^m\right)^2 = \lambda_1^{-1} \left(\sum_{m=0}^{n-1} \left(\phi_i^m\right)^2 - n^{-1}\right) = \lambda_1^{-1} \left(1-n^{-1}\right)
\]
and similarly
\[
\sum_{m=1}^{n-1} \lambda_m^{-1} \left(\phi_i^m\right)^2 \geq \lambda_{n-1}^{-1} \left(1-n^{-1}\right).
\]
Since
\[
-\sum_{m=1}^{n-1} \lambda_m^{-1} \phi_i^m \phi_j^m = \frac12 \sum_{m=1}^{n-1} \lambda_m^{-1} \left(\phi_i^m - \phi_j^m\right)^2 - \frac12 \sum_{m=1}^{n-1} \lambda_m^{-1} \left[\left(\phi_i^m\right)^2 + \left(\phi_j^m\right)^2\right],
\]
the result follows.
\end{proof}

\begin{remark}\label{rem:diffineigenfunc}
If $r=0$ and $\gamma \geq 0$, Theorem~\ref{thm:newgraph} tells us that the dynamics of \eqref{eq:MBOstepa} on a graph $G=(V,E,\omega)\in \mathcal{C}_\gamma$ correspond to diffusion dynamics on a different graph $\tilde G=(V,\tilde E, \tilde \omega)\in \mathcal{G}$ with the same node set $V$, but a different edge set and different edge weights. Furthermore, from Lemma~\ref{lem:newedgeweights} it follows that $E \subset \tilde E$, so $\tilde G$ can gain edges compared to $G$, but not lose any. By the same lemma we know that, if $G\in \mathcal{C}$, any edges that already existed in $G$ cannot have a lower weight in $\tilde G$ than they had in $G$. Equation \eqref{eq:tildeomegaisomegapluspos} quantifies the change in edge weight. Lemma~\ref{lem:increaseinweights} suggests (but does not prove) that the largest increase (in the case when $G\in\mathcal{C}$) in edge weight (including potentially creation of a new edge where there was none in $G$) occurs between nodes $i\in V$ and $j\in V$ for which $\sum_{m=1}^{n-1} \frac1{\lambda_m}\left(\phi_i^m - \phi_j^m\right)^2$ is large. If this suggestion is accurate and $G$ is such that the eigenvalues $\lambda_m$ rapidly increase with increasing $m$, then the main addition of edge weight going from $G$ to $\tilde G$ happens between those nodes $i$ and $j$ for which $\left(\phi_i^1-\phi_j^1\right)^2$ is large (or for which $\sum_{m=1}^k \left(\phi_i^m-\phi_j^m\right)^2$ is large, if the eigenvalue $\lambda_1$ has multiplicity $k$). 

In this context it is interesting to note that the second smallest eigenvalue (when $r=0$), i.e. the smallest nonzero eigenvalue for a connected graph, is called the \textit{algebraic connectivity} of the graph and the corresponding eigenfunction (or eigenvector) is called the\footnote{Assuming the eigenvalue is simple.} \textit{Fiedler vector} \cite{fiedler1973}. In
\cite{4177113,pmlr-v28-osting13,osting2014optimal} it is argued that a good strategy when attempting to add an edge to a graph such as to maximize the algebraic connectivity of the resulting graph, is to add the edge between those nodes whose corresponding values in the Fiedler vector have a large (absolute) difference. In other words, adding an edge between those nodes $i$ and $j$ for which $(\phi^1_i - \phi^1_j)^2$ is largest, is a good heuristic for maximizing the algebraic connectivity of a graph (if the addition of one edge is allowed). Our discussion above thus suggests that in going from graph $G$ to graph $\tilde G$, most weight is added to those edges which make the largest contribution to the algebraic connectivity of the graph.
\end{remark}

\begin{remark}
The discussion in Remark~\ref{rem:diffineigenfunc} can give a some high level intuition about the dynamics of the \ref{alg:massOKMBO} algorithm on graphs in $\mathcal{C}_\gamma$. These dynamics can be seen as a diffusion process on a new graph which differs from the original graph by having additional (or more highly weighted) edges (approximately) between those nodes whose values in the eigenfunctions corresponding to the smallest nonzero eigenvalues of $\Delta$ differ by a large amount. 
The mass conserving thresholding step in \ref{alg:massOKMBO} distributes all the available mass over those nodes which, in the ODE step, acquired the most mass through this diffusion process on the new graph. Thus, the available mass from the initial function $v^0$  is most likely\footnote{This should currently be interpreted as an imprecise, nonrigorous, statement, but might be turned into a precise conjecture for future research.} to end up (after one iteration) at those nodes that are more strongly connected (in the new graph) to the nodes in the support of $v^0$, than to nodes in the support's complement. These connections could have been present already in the original graph, or they could have been created (or strengthened) via the newly created edges determined in large part by the eigenfunctions (corresponding to the smallest nonzero eigenvalues) of $\Delta$. The relative influence of both these effects is controlled by the parameter $\gamma$.

Lemma~\ref{lem:F0onnewgraph} shows that sets $S\subset V$ which minimize $F_0(\chi_S)$ have to balance their `volume', $\sum_{i\in S} \tilde d_i$, and curvature, $\sum_{i\in S} \left(\tilde \kappa_S\right)_i$, on the new graph $\tilde G$. We have put `volume' in scare quotes here, because $\tilde r=0$ in Lemma~\ref{lem:F0onnewgraph}, thus $\sum_{j\in S} \tilde d_j$ is not equal to $\vol{S}$ on $\tilde G$.
\end{remark}

In the following, we use the unweighted star graph from Definition~\ref{def:bipartstar} to illustrate some of the concepts discussed so far in this section. Remember from Lemma~\ref{lem:stargraphinC} that this star graph is in $\mathcal{C}$, so it is a suitable example.

\begin{lemma}\label{lem:starvarphi}
Let $G\in\mathcal{G}$ be an unweighted star graph as in Definition~\ref{def:bipartstar} with $n\geq 3$ nodes. Then the functions $\varphi^1, \varphi^j: \mathcal{V}\to \R$, for $j\in V\setminus\{1\}$, as defined in \eqref{eq:defvarphij}, are given by, for $i\in V$,
\begin{align*}
\varphi^1_i &= \left(\vol{V}\right)^{-2}\begin{cases}
(n-1)^{r+1}, &\text{if } i=1,\\
-(n-1)^{2r}, &\text{if } i \neq 1,
\end{cases}\\
\varphi^j_i &= \left(\vol{V}\right)^{-2}\begin{cases}
-(n-1)^r, &\text{if } i=1,\\
\left( (n-1)^r + n-1\right)^2 - 2(n-1)^r - (n-1), &\text{if } i=j,\\
-2(n-1)^r - (n-1), &\text{if } 1\neq i\neq j.
\end{cases}
\end{align*}
Assume $r=0$ and let $\gamma\geq 0$. Let $\tilde \omega$ be as in \eqref{eq:tildeomega}, then, for $i, j \in V$,
\[
\tilde \omega_{ij} = \begin{cases}
0, &\text{if } i=j,\\
1+ \frac\gamma{n^2}, &\text{if } i=1\neq j \text{ or } j=1\neq i,\\
\frac{\gamma (n+1)}{n^2}, &\text{if } i\neq 1 \neq j \neq i.
\end{cases}
\]
\end{lemma}
\begin{proof}
A direct computation can be performed to validate that, for all $j\in V$, $\varphi^j$ indeed solves \eqref{eq:zeromassvarphiequation} for $\chi_{\{j\}}$, but we will give a different derivation here based directly on \eqref{eq:defvarphij} and \eqref{eq:fij}. Noting that $d_1=n-1$ and, for $i\in V\setminus\{1\}$, $d_i = 1$, and using Lemma~\ref{lem:starequil}, we compute
$
\mathcal{M}\left(\nu^{V\setminus\{1\}}\right) = \sum_{i=2}^n d_i^r = n-1
$ 
and, for $j\in V\setminus\{1\}$,
\[
\mathcal{M}\left(\nu^{V\setminus\{j\}}\right) = d_1^r (\vol{V}-1) + \sum_{i\in V\setminus\{1, j\}} d_i^r  \vol{V}  = \left(\vol{V}\right)^2 - 2\vol{V} + n-1.
\]
Furthermore,
$
\frac{d_1^r}{\vol{V}} = 1-\frac{n-1}{\vol{V}}$ and, for $j\in V\setminus\{1\}$,
$\frac{d_j^r}{\vol{V}} =\frac1{\vol{V}}$. 
Combining these results with the expressions for $\nu^{V\setminus\{1\}}$ and $\nu^{V\setminus\{j\}}$ in  Lemma~\ref{lem:starequil}, we find that $\varphi^1$ and $\varphi^j$ are as defined in this lemma.

Now assume that $r=0$ and let $i,j\in V$. Per definition, if $i=j$, then $\tilde \omega_{ij} = 0$. If $i\neq j$, we know, by \eqref{eq:tildeomegaisomegapluspos} and \eqref{eq:defvarphij}, that $\tilde \omega_{ij} = \omega_{ij} + \frac\gamma{n} f^j_i = \omega_{ij} - \varphi^j_i$. A direct computation for $r=0$ shows that, for $j\in V\setminus\{1\}$, $i\in V$, 
\begin{equation}\label{eq:starPhi}
\varphi^1_i := \begin{cases}
\frac{n-1}{n^2}, &\text{if } i=1,\\
-\frac1{n^2}, &\text{if } i \neq 1,
\end{cases} \qquad \qquad
\varphi^j_i := \begin{cases}
-\frac1{n^2}, &\text{if } i=1,\\
\frac{n^2-n-1}{n^2}, &\text{if } i=j,\\
-\frac{n+1}{n^2}, &\text{if } 1\neq i\neq j.
\end{cases}
\end{equation}
Note that for all $i,j\in V$, $\varphi^j_i= \varphi^i_j$, as should be the case per Corollary~\ref{cor:symmetries}. The result now follows from the fact that for all $j\neq 1$, $\omega_{1j} = 1$ and all other $\omega_{ij}$ are $0$.
\end{proof}

\begin{remark}
In the proof of Lemma~\ref{lem:starvarphi} we computed the functions $\varphi^j$, for $j\in V$, using the equilibrium measures from Lemma~\ref{lem:starequil}. It is instructive to compute $\varphi^j$ directly from the eigenvalues and eigenfunctions as well, for the case $r=0$. Using Lemma~\ref{lem:starspectrum}, we see that, for $i,j\in V$,
\begin{equation}\label{eq:starPhidirect}
\varphi^j_i = \sum_{m=1}^{n-2} \phi^m_i \phi^m_j + \frac1n \phi^{n-1}_i \phi^{n-1}_j.
\end{equation}
In Appendix~\ref{sec:starPhidirect} we give the details showing that this computation leads to the same expression for $\varphi^j$ as given above in \eqref{eq:starPhi}.
\end{remark}

\begin{remark}
If we want to apply the observation from Remark~\ref{rem:diffineigenfunc} to the star graph discussed above, we see from Lemma~\ref{lem:starspectrum} that the smallest nonzero eigenvalue is $1$ with multiplicity $n-2$. Hence, from Remark~\ref{rem:diffineigenfunc}, we expect that $\tilde\omega_{ij}-\omega_{ij}$ is largest for those nodes $i,j$ for which $\sum_{m=1}^{n-2} \left(\phi_i^m-\phi_j^m\right)^2$ is large\footnote{Unfortunately, while the star graph has served us very well in previous examples, it is not the cleanest case to illustrate our heuristic from Remark~\ref{rem:diffineigenfunc}. The symmetry of the star graph, which has simplified some of the calculations in earlier examples, now means that our heuristic requires some more calculation, since we cannot suffice with checking $\left(\phi_i^1-\phi_j^1\right)^2$ only.}. 
From Lemma~\ref{lem:starspectrum} we have, for $m\in \{1, \ldots, n-2\}$ and $i\in V$,
\begin{align*}
\left(\phi_i^m\right)^2 &= \frac{(n-i)^2}{(n-i)^2+n-i} (1-\delta_{i1})(1-\delta_{in}) + (1-\delta_{i1}) (1-\delta_{i2})\sum_{m=1}^{i-2} \frac1{(n-m-1)^2+n-m-1},\\
&= \frac{n-i}{n-i-1} (1-\delta_{i1})(1-\delta_{in}) + \frac{i-2}{(n-i+1)(n-1)} (1-\delta_{i1}) (1-\delta_{i2}),
\end{align*}
where we used the Kronecker delta and \eqref{eq:subtractsums}. Furthermore, if $m\in \{1, \ldots, n-2\}$ and $i,j\in V$ with $j<i$, then
\begin{align*}
\phi_i^m \phi_j^m &= -\frac{n-j}{(n-j)^2+ n-j} + (1-\delta_{j1}) (1-\delta_{j2})\sum_{m=1}^{j-2} \frac1{(n-m-1)^2+n-m-1}\\
&= \frac1{n-j+1} + (1-\delta_{j1}) (1-\delta_{j2}) \frac{j-2}{(n-j+1)(n-1)}.
\end{align*}
Hence, for $j<i$,
\begin{align*}
&\hspace{0.6cm}\sum_{m=1}^{n-2} \left(\phi_i^m-\phi_j^m\right)^2 = \sum_{m=1}^{n-2} \left(\left(\phi_i\right)^2+\left(\phi_j\right)^2-2\phi_i\phi_j\right)\\ &=\begin{cases}
\frac{n-i}{n-i+1} + \frac{i-2}{(n-i+1)(n-1)} + \frac2n, &\text{if } j=1 \text{ and } i\neq n,\\
\frac{n-2}{n-1} + \frac2n, &\text{if } j=1 \text{ and } i=n,\\
\frac{n-i}{n-i+1} + \frac{i-2}{(n-i+1)(n-1)} + \frac{n-j}{n-j+1} + \frac{j-2}{(n-j+1)(n-1)} + \frac2{n-j+1} - \frac{2(j-2)}{(n-j+1)(n-1)}, &\text{if } j\neq 1 \text{ and } i\neq n,\\
\frac{n-2}{n-1} + \frac{n-j}{n-j+1} + \frac{j-2}{(n-j+1)(n-1)} + \frac2{n-j+1} - \frac{2(j-2)}{(n-j+1)(n-1)}, &\text{if } j\neq 1 \text{ and } i=n,
\end{cases}\\
&= \begin{cases}
\frac{n^2-2}{n(n-1)}, &\text{if } j=1 \text{ and } i\neq n,\\
\frac{n^2-2}{n(n-1)}, &\text{if } j=1 \text{ and } i=n,\\
2, &\text{if } j\neq 1 \text{ and } i\neq n,\\
2, &\text{if } j\neq 1 \text{ and } i=n.
\end{cases}
\end{align*}
The second equality follows by straightforward simplification of the fractions. The case where $j>i$ follows by symmetry. 

There are two things we can immediately note. First, despite there being an apparent difference in computation of the cases $1\neq i\neq n$ and $1\neq i=n$, there is of course no difference in eventual outcome. As we expect by symmetry of the star graph, each of the nodes $\{2, \ldots, n\}$ is interchangable without affecting the outcome. Most importantly for our present purposes, we have that $\frac{n^2-2}{n(n-1)} < 2$. In fact, a direct calculation shows that $\frac{n^2-2}{n(n-1)}$ has a maximum value of $\frac76$ for $n\in \{n\in \N: n\geq 2\}$, which is attained at $n=3$ and $n=4$. Hence, according to our heuristic, the increase $\tilde \omega - \omega$ between the leaves (i.e. nodes $\{2, \ldots, n\}$) of the star graph should be larger than the incease between the leaves and the centre node $1$. This is indeed what we saw in Lemma~\ref{lem:starvarphi}.
\end{remark}

\subsection{More comparison principles}\label{sec:comppinning}

Theorem~\ref{thm:newgraph} tells us that, if $\gamma\geq 0$ is such that $G\in\mathcal{C}_\gamma$ and if $r=0$, then the dynamics in \eqref{eq:MBOstepa} can be viewed as graph diffusion on a new graph with the same node set, but a different edge set and weights, as the original graph $G$. We can use this to prove that properties of $\Delta$ also hold for $L$ on such graphs. Note that, when $\gamma=0$, $L=\Delta$, so this can be viewed as a generalization of results for $\Delta$ to $L$.

In this section we prove a generalization of Lemma~\ref{lem:compell} and a generalization of the comparison principle in \cite[Lemma 2.6(d)]{vanGennipGuillenOstingBertozzi14}. In fact, despite the new graph construction in Theorem~\ref{thm:newgraph} requiring $r=0$ for symmetry reasons (see Remark~\ref{rem:therolesofrandtilder}), the crucial ingredient that will allow these generalizations is that $G\in \mathcal{C}_\gamma$; the assumption on $r$ is not required. We will also see a counterexample illustrating that this generalization does not extent (at least not without further assumptions) to graphs that are not in $\mathcal{C}_\gamma$.

Lemma~\ref{lem:varphiistar} gives a result which we need to prove the comparison principles in Lemmas ~\ref{lem:compellgeneralization} and~\ref{lem:varphiistar}.

\begin{lemma}\label{lem:varphiistar}
Let $\gamma\geq 0$, $G=(V,E,\omega)\in \mathcal{C}_\gamma$, $w\in \mathcal{V}$, and let $i^*\in V$ be such that $w_{i^*}=\min_{i\in V} w_i$. Let $\varphi\in V$ solve \eqref{eq:zeromassvarphiequation} for $w$. Then $\varphi_{i^*} \leq 0$.
\end{lemma}
\begin{proof}
Let $j\in V$ and let $\varphi^j \in \mathcal{V}$ be as in \eqref{eq:defvarphij}.  Then, by Lemma~\ref{lem:varphijintroduced}, we have that $\displaystyle \Delta \varphi^j = \chi_{\{j\}} - \mathcal{A}\left(\chi_{\{j\}}\right)$ and $\displaystyle \mathcal{M}\left(\varphi^j\right)=0$. Furthermore, by Definition~\ref{def:graphclasses} and \eqref{eq:defvarphij}, it follows that, for all $i\in V\setminus\{j\}$, $\varphi^j_i \leq 0$. Because $\displaystyle w = \sum_{j\in V} w_j \chi_{\{j\}}$, we have $\displaystyle \mathcal{A}(w) = \sum_{j\in V} w_j \mathcal{A}\left(\chi_{\{j\}}\right)$ and thus $\displaystyle \Delta \varphi = \sum_{j\in V} w_j \left(\chi_{\{j\}}-\mathcal{A}\left(\chi_{\{j\}}\right)\right)$. Since also $\displaystyle \mathcal{M}\left(\sum_{j\in V} w_j \varphi^j\right) = \sum_{j\in V} \mathcal{M}\left(w_j \varphi^j\right) = 0$, we find that $\displaystyle \varphi = \sum_{j\in V} \varphi^j$. Hence
$\displaystyle 
\varphi_{i^*} = \sum_{j\in V} w_j \varphi^j_{i^*} = w_{i^*} \varphi^{i^*}_{i^*}  + \sum_{j\in V\setminus\{i^*\}} w_j \varphi^j_{i^*}.
$
For $j\neq i^*$, we know that $w_{i^*} \leq w_j$ and $\varphi^j_{i^*} \leq 0$, hence $w_j \varphi^j_{i^*}  \leq w_{i^*} \varphi^j_{i^*}$. Therefore
$\displaystyle
\varphi_{i^*} \leq w_{i^*} \varphi^{i^*}_{i^*}  + \sum_{j\in V\setminus\{i^*\}} w_{i^*} \varphi^j_{i^*} = w_{i^*} \sum_{j\in V} \varphi^j_{i^*}.
$
If we define $\displaystyle \tilde \varphi := \sum_{j\in V} \varphi^j = \sum_{j\in V} \left(\chi_V\right)_j \varphi^j$, then by a similar argument as above, $\displaystyle \Delta \tilde \varphi = \chi_V - \mathcal{A}\left(\chi_V\right) = 0$ and $\displaystyle \mathcal{M}\left(\tilde \varphi\right)=0$. Thus $\tilde \varphi = 0$ and we conclude that $\varphi_{i^*} \leq 0$.
\end{proof}

\begin{lemma}[Generalization of comparison principle I]\label{lem:compellgeneralization}
Let $\gamma\geq 0$, $G=(V,E,\omega) \in \mathcal{C}_\gamma$, and let $V'$ be a proper subset of $V$. Assume that $u, v \in \mathcal{V}$ are such that, for all $i\in V'$,
$(L u)_i \geq (L v)_i$ and, for all $i\in V\setminus V'$, $u_i\geq v_i$. Then, for all $i\in V$, $u_i \geq v_i$.
\end{lemma}
\begin{proof}
When $r=0$, we know that $L=\tilde \Delta$, where $\tilde\Delta$ is the graph Laplacian on the graph $\tilde G$, in the notation from Theorem~\ref{thm:newgraph}. Because $G$ and $\tilde G$ have the same node set $V$, the result follows immediately by applying Lemma~\ref{lem:compell} to $\tilde \Delta$. We will, however, prove the generalization for any $r\in [0,1]$.

Let the situation and notation be as in the proof of Lemma~\ref{lem:compell}, with the exception that now, for all $i\in V'$, $(Lw)_i\geq 0$ (instead of $(\Delta w)_i\geq 0$). Let $\varphi\in \mathcal{V}$ by such that $\Delta \varphi = w - \mathcal{A}(w)$ and $\mathcal{M}(w)=0$. Proceed with the proof in the same way as the proof of Lemma~\ref{lem:compell}, up to and incluing the assumption that $\min_{j\in V} w_j <0$ and the subsequent construction of the path from $U$ to $i^*$ and the special nodes $j^*$ and $k^*$ on this path. Then, as in that proof, we know that $(\Delta w)_j^* < 0$. Moreover, since $w_{j^*} = \min_{i\in V}w_i$, we know by Lemma~\ref{lem:varphiistar} that $\varphi_{j^*} \leq 0$. Hence, for all $\gamma \geq 0$, $(Lw)_j^* < 0$. This contradicts the assumption that, for all $i\in V'$, $(Lw)_i\geq 0$, hence $\min_{i\in V} w_i \geq 0$ and the result is proven.
\end{proof}

The following corrollary of Lemma~\ref{lem:varphiistar} will be useful in proving Lemma~\ref{lem:comprincII}

\begin{corol}\label{cor:LtildeuLu}
Let $\gamma\geq 0$, $G=(V,E,\omega)\in \mathcal{C}_\gamma$. Assume that $u, \tilde u \in \mathcal{V}$ satisfy, for all $i\in V$, $u_i \leq \tilde u_i$, and let there be an $i^* \in V$ such that $u_{i^*}=\tilde u_{i^*}$. Then $(Lu)_{i^*} \geq (L\tilde u)_{i^*}$.
\end{corol}
\begin{proof}
Define $w:= \tilde u - u$, then, for all $i\in V$, $w\geq 0$ and $w_{i^*} = 0$. We compute
\[
d_{i^*}^r (\Delta w)_{i^*} = d_{i^*} w_{i^*} - \sum_{j\in V} \omega_{i^*j} w_j = - \sum_{j\in V} \omega_{i^*j} w_j  \leq 0.
\]
Let $\varphi\in \mathcal{V}$ solve \eqref{eq:zeromassvarphiequation} for $w$. Since $w_{i^*} = \min_{i\in V} w_i$, we have by Lemma~\ref{lem:varphiistar} that $\varphi_{i^*} \leq 0$. Hence
\[
(L\tilde u)_{i^*} - (Lu)_{i^*} = (Lw)_{i^*} = (\Delta w)_{i^*} + \gamma \varphi_{i^*} \leq 0.
\]
\end{proof}

\begin{lemma}[Comparison principle II]\label{lem:comprincII}
Let $\gamma\geq 0$, $G=(V,E,\omega) \in \mathcal{C}_\gamma$, $u_0, v_0\in \mathcal{V}$, and let $u, v \in \mathcal{V}_\infty$ be solutions to \eqref{eq:MBOstepa}, with initial conditions $u_0$ and $v_0$, respectively. If, for all $i\in V$, $(u_0)_i \leq (v_0)_i$, then, for all $t\geq 0$ and for all $i\in V$, $u_i(t)\leq v_i(t)$.
\end{lemma}
\begin{proof}
If $r=0$ we note that, by Theorem~\ref{thm:newgraph}, $L$ can be rewritten as a graph Laplacian on a new graph $\tilde G$ with the same node set $V$. The result in \cite[Lemma 2.6(d)]{vanGennipGuillenOstingBertozzi14} shows the desired conclusion holds for graph Laplacians (i.e. when $\gamma=0$) and thus we can apply it to the graph Laplacian on $\tilde G$ to obtain to result for $L$ on $G$.

In the general case when $r\in [0,1]$, Corollary~\ref{cor:LtildeuLu} tells us that $L$ satisfies the condition which is called $W_+$ in \cite[Section 4]{Szarski1965}\footnote{The property of $L$ in Corollary~\ref{cor:LtildeuLu} is sometimes also called quasimonotonicity, or, more properly it can be seen as a consequence of quasimonotonicity in the sense of \cite{Volkmann1972,Chaljub-Simon1992,herzog2004characterization}.}. Since, for a given initial condition, the solution to \eqref{eq:MBOstepa} is unique, the result now follows by applying \cite[Theorem 9.3]{Szarski1965} (or \cite[Theorem 9.4]{Szarski1965}.
\end{proof}

\begin{corol}\label{cor:betweenconstants}
Let $\gamma\geq 0$, $G=(V,E,\omega) \in \mathcal{C}_\gamma$, and let $w\in \mathcal{V}_\infty$ be a solution to \eqref{eq:MBOstepa} with initial condition $w_0 \in \mathcal{V}$. Let $c_1, c_2\in \R$ be such that, for all $i\in V$, $c_1 \leq (w_0)_i \leq c_2$. Then, for all $t\geq 0$ and for all $i\in V$, $c_1 \leq w_i(t) \leq c_2$.

In particular, for all $t\geq 0$, $\|w(t)\|_{\mathcal{V}, \infty} \leq \|w_0\|_{\mathcal{V},\infty}$.
\end{corol}
\begin{proof}
First note that $c_1$ and $c_2$ always exist, since $V$ is finite.

If $u\in \mathcal{V}_\infty$ solves \eqref{eq:MBOstepa} with initial condition $u_0 = c_1\chi_V \in \mathcal{V}$, then, for all $t\geq 0$, $u(t)=c_1 \chi_V$. Applying Lemma~\ref{lem:comprincII} with $v_0 = w_0$ and $v=w$, we obtain that, for all $t\geq 0$ and for all $i\in V$, $w_i(t)\geq c_1$. Similarly, if $v\in \mathcal{V}_\infty$ solves \eqref{eq:MBOstepa} with initial condition $v_0 = c_2 \chi_V \in \mathcal{V}$, then, for all $t\geq 0$, $u(t)=c_2 \chi_V$. Hence, Lemma~\ref{lem:comprincII} with $u_0=w_0$ and $u=w$ tells us that, for all $t\geq 0$ and for all $i\in V$, $w_i(t) \leq c_2$.

The final statement follows by noting that, for all $i\in V$, $-\|w_0\|_{\mathcal{V},\infty} \leq (w_0)_i \leq  \|w_0\|_{\mathcal{V},\infty}$.
\end{proof}

\begin{remark}
Numerical simulations show that when $G\not\in \mathcal{C}_\gamma$, the results from Corollary~\ref{cor:betweenconstants} do not necessarily hold for all $t>0$. For example, consider an unweighted 4-regular graph (in the notation of Section~\ref{sec:examplegraphs} we take the graph $G_{\text{torus}}(900)$) with $r=0$ and $\gamma=0.7$. We compute $\min_{i,j\in V} (d_i^{-r} \omega_{ij} +  \gamma \frac{d_j^r}{\vol V} f^j_i) \approx -0.1906$ in MATLAB using \eqref{eq:defvarphij}, \eqref{eq:spectralvarphij}, so the graph is not in $\mathcal{C}_{0.7}$. Computing $v(0.01) = e^{-0.01 L} v^0$, where $v^0$ is a $\{0,1\}$-valued initial condition\footnote{To be precise, here we choose $v^0$ based on the eigenvector corresponding to option (c) explained in Section~\ref{sec:initialcondition}, with $M=450$.}, we find $\min_{i\in V} v_i(0.01) \approx -0.0033<0$ and $\max_{i\in V} v_i(0.01) \approx 1.0033>1$. Hence the conclusions of  Corollary~\ref{cor:betweenconstants} do not hold in this case.
\end{remark}

We can use the result from Corollary~\ref{cor:betweenconstants} to prove a second pinning bound, in the vein of Lemma~\ref{lem:dynamicsbounds}, for graphs in $\mathcal{C}_\gamma$.

\begin{lemma}
Let $\gamma\geq 0$ and let $G=(V,E,\omega) \in \mathcal{C}_\gamma$. Let $S\subset V$ be nonempty and define
\[
\tau_\kappa(S) := \frac12 \|L\chi_S\|_{\mathcal{V},\infty}^{-1}.
\]
Let $S^1$ be the first set in the corresponding \ref{alg:OKMBO} evolution of the initial set $S^0=S$. If $0\leq \tau<\tau_\kappa(S)$, then $S^1=S$.
\end{lemma}
\begin{proof}
The proof is based on (parts of the) proof of \cite[Theorem 4.2]{vanGennipGuillenOstingBertozzi14}.

Writing the solution $u(t)=e^{-t L} \chi_S$ to \eqref{eq:MBOstepa} at $t=\tau$ as
$\displaystyle
u(\tau) = \chi_S - \int_0^\tau L(u(t)) \,dt,
$
we find
\begin{align*}
\|u(\tau)-\chi_S\|_{\mathcal{V},\infty} &\leq \int_0^\tau \|L(u(t))\|_{\mathcal{V},\infty} \,dt \leq \int_0^\tau \|e^{-tL} L\chi_S\|_{\mathcal{V},\infty}\, dt \leq \int_0^\tau \|L \chi_S\|_{\mathcal{V},\infty} \, dt\\ 
&= \tau \|L\chi_S\|_{\mathcal{V},\infty} < \frac12,
\end{align*}
where we used that $L$ and $e^{-tL}$ commute for the second inequality, and  
Corollary~\ref{cor:betweenconstants} for the third inequality. 
We conclude that $S^1=S$.
\end{proof}

\section{Numerical implementations}\label{sec:numerical}

In this section we discuss how we implemented \ref{alg:massOKMBO} (in MATLAB version 9.2.0.538062 (R2017a)) and show some results.

\subsection{Spectral expansion method}

We use the spectral expansion \eqref{eq:solutionu} to solve \eqref{eq:MBOstepa}. This is similar in spirit to the spectral expansion methods used in, for example, \cite{BertozziFlenner12,calatroni2017graph}. However, in those papers an iterative method is used to deal with additional terms in the equation. Here, we can deal with the operator $L$ in \eqref{eq:MBOstepa} in one go.

Note that in other applications of spectral expansion methods, such as those in \cite{bertozzi2016diffuse}, sometimes only a subset of the eigenvalues and corresponding eigenfunctions is used. When $n$ is very large, computation time can be saved, often without a great loss of accuracy, by using a truncated version of \eqref{eq:expansion} wich only uses the $K \ll n$ smallest eigenvalues $\Lambda_m$ with corresponding eigenfunctions. The examples we show in this paper are small enough that such an approximation was not necessary, but it might be considered if the method is to be run on large graphs.

\subsection{Example graphs}\label{sec:examplegraphs}

For the purpose of having visually appealing results, in the experiments we present here we have mostly used graphs whose structure resembles a discretization of the plane ---such as the graphs $G_{\text{torus}}$, $G_{\text{stitched}}$, and even, to a degree, $G_{\text{moons}}$, which are introduced below--- as they allow us to see pattern formation similar to what we expect based on the continuum case \cite[Chapter 2]{vanGennip08}. For example, spherical droplets (Figure~\ref{fig:02bmini}) or lamellar patterns (Figure~\ref{fig:01bmini}). However, the algorithm is not restricted to such examples; in the visually less interesting examples we will still display the evolution of the value of $F_0$ along the sequence of \ref{alg:massOKMBO} iterates, to illustrate that the algorithm does (mostly) decrease the value of $F_0$ also in these cases. In this paper we present results obtained on the following graphs:
\begin{itemize}
\item An unweighted $4$-regular graph (i.e. each node has degree $4$) which can be graphically represented as the grid obtained by tessellating a square with periodic boundary conditions (i.e. the square two-dimensional flat torus) with square tiles, see for example Figure~\ref{fig:initialcond}\footnote{In order to increase the visibility of the patterns in the function values on the nodes, the size of the nodes as depicted was chosen to be large. As a consequence, in the figure the nodes cover the edges and the edges are no longer visible; for each node edges are present between it and each of the four nodes placed immediately adjacent to it, taking into account periodic boundary conditions.}. We denote this graph by $G_{\text{torus}}(n)$, where $n$ is the number of nodes (and thus $\sqrt{n}$ is the number of nodes along each direction of the square in the tesselation representation.

\item An unweighted graph obtained by adjoining a square lattice graph (this time without periodic boundary conditions) and a triangular lattice graph, as in Figure~\ref{fig:13b} (see also \cite{vanGennipGuillenOstingBertozzi14}. We will denote these `stitched together' graphs by $G_{\text{stitched}}(n)$ where $n$ is the total number of nodes in the graph.

\item A two moon graph constructed as in \cite{Buhler2009}. This graph is constructed by sampling points from two half-circles in $\R^2$, embedding these into a high-dimensional space, adding Gaussian noise, and constructing a weighted $K$-nearest neighbour graph with the sample points represented by the nodes. We will denote this graph by $G_{\text{moons}}$. It has $600$ nodes. See Figure~\ref{fig:18b19b20b}.

\item To illustrate that the method can also be applied to more complex networks, we use a symmetrized version of the weighted ``neural network'' graph obtained from \cite{Newmanwebsite} and based on \cite{white1986structure,WattsStrogatz98}. It represents the neural network of C. Elegans and has 297 nodes. Since the original network with weight matrix $A$ is directed, we use the symmetrized weight matrix $\frac12(A+A^T)$.  We will denote the resulting undirected, weighted, graph by $G_{\text{neural}}$. See Figure~\ref{fig:21b}.
\end{itemize}

\begin{figure}
\begin{center}
\begin{subfigure}[b]{0.45\textwidth}
\includegraphics[width=1.1\textwidth]{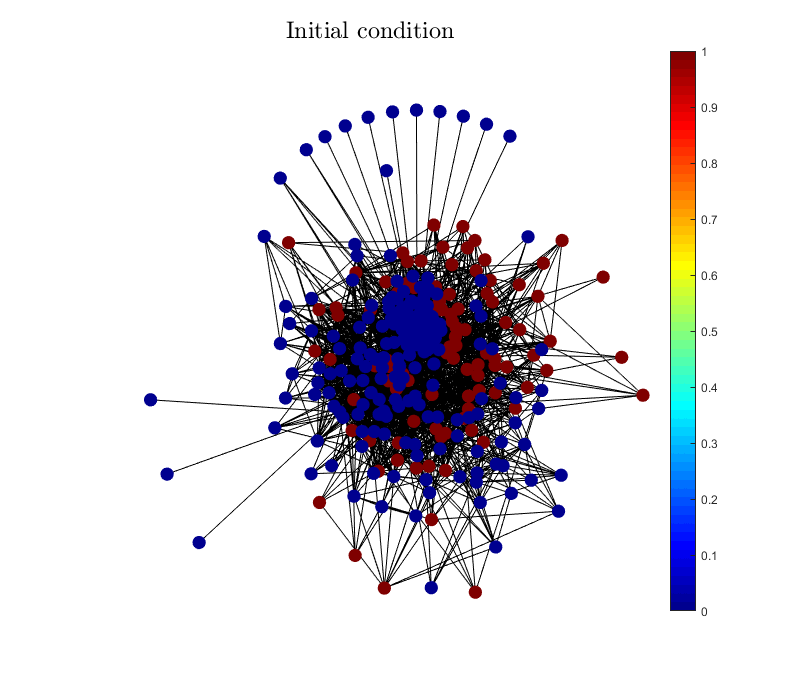}\vspace{-0.7cm}
\caption{Initial condition} \label{fig:21binit}
\end{subfigure}
\begin{subfigure}[b]{0.45\textwidth}
\includegraphics[width=1.1\textwidth]{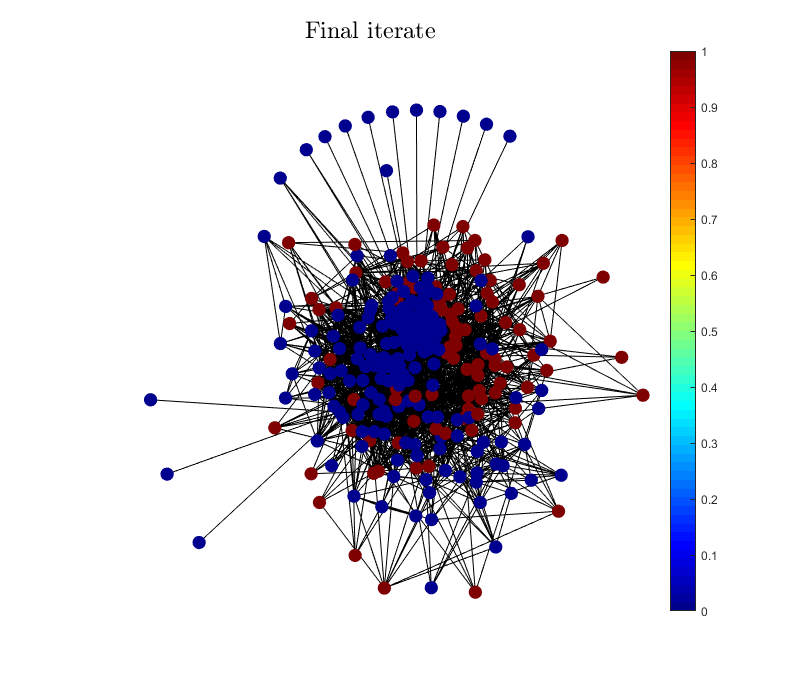}\vspace{-0.7cm}
\caption{Final iterate ($k=2$)} \label{fig:21bmini}
\end{subfigure}\\ \vspace{0.3cm}
\begin{subfigure}[b]{0.45\textwidth}
\includegraphics[width=1.1\textwidth]{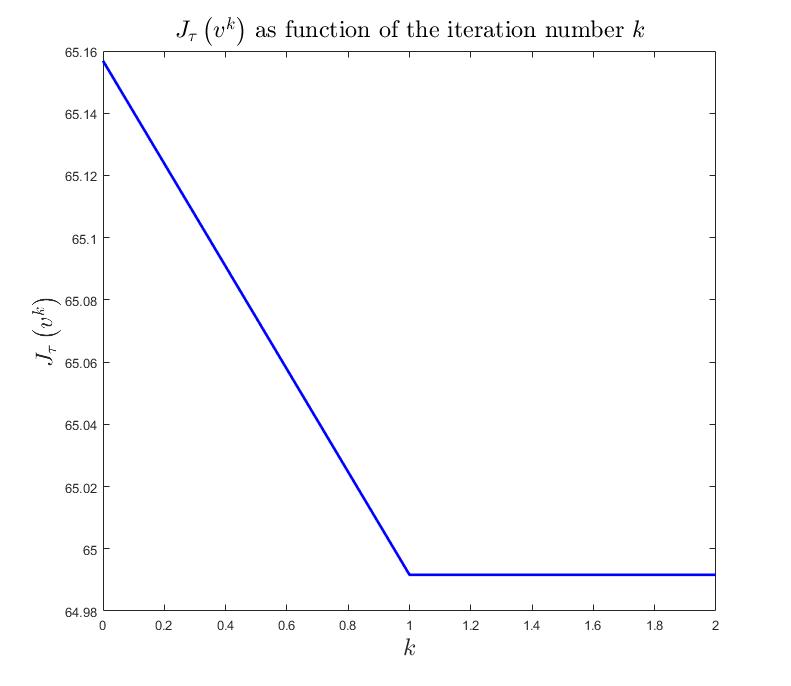}\vspace{-0.4cm}
\caption{Plot of $J_{0.75}\left(v^k\right)$} \label{fig:21bJtau}
\end{subfigure}
\begin{subfigure}[b]{0.45\textwidth}
\includegraphics[width=1.1\textwidth]{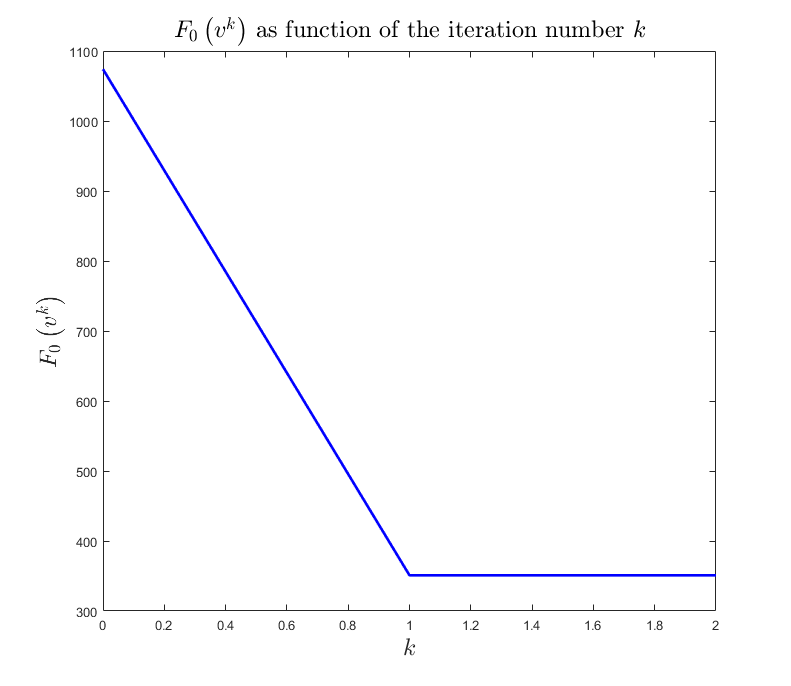}\vspace{-0.4cm}
\caption{Plot of $F_0\left(v^k\right)$} \label{fig:21bOKTV}
\end{subfigure}\\ \vspace{0.3cm}
\begin{subfigure}[b]{0.45\textwidth}
\includegraphics[width=1.1\textwidth]{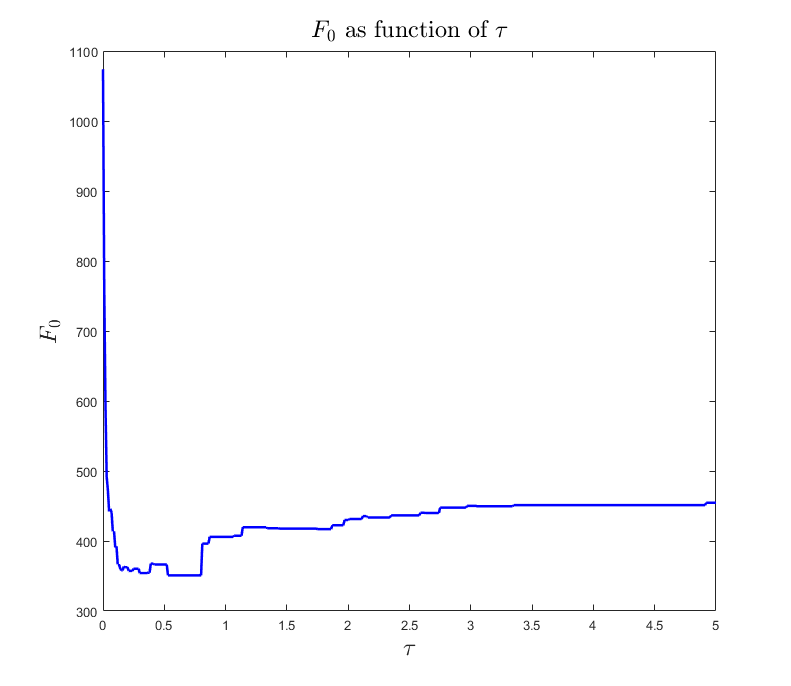}\vspace{-0.5cm}
\caption{$F_0$ at the final iterate as a function of $\tau$} \label{fig:multtau07}
\end{subfigure}
\vspace{-0.2cm} \caption{Results from Algorithm~\ref{alg:massOKMBO} on $G_{\text{neural}}$ with $r=0$, $\gamma=1$, and $M=100$. The initial condition in Figure~\ref{fig:21binit} was constructed using option (b) in Section~\ref{sec:initialcondition} and was used to obtain the other results displayed here. The results in Figures~\ref{fig:21bmini}, \ref{fig:21bJtau}, and \ref{fig:21bOKTV} were obtained for $\tau=0.75$, in which case the value of $F_0$ at the final iterate is approximately $350.95$. The graph in Figure~\ref{fig:multtau07} shows the values of $F_0$ at the final iterates for a ranges of values of $\tau$, with a resolution on the $\tau$ axis (step size) of $0.01$.}
\label{fig:21b}
\end{center}
\end{figure}

\subsection{Choice of $\tau$}

The choice of $\tau$ is an important one.  We know from the discussion in Remark~\ref{rem:pinningmassOKMBO} that $\tau$ should not be chosen too small or too large, but it is not easy to decide a priori what a good choice would be.  The discussion in Remark~\ref{rem:whataboutmass?} suggests that, if minimizing $F_0$ in \eqref{eq:minimprobsF0} (with $q=1$) is our goal, then we should choose $\tau$ small, but the potential for pinning prevents us from choosing $\tau$ too small. It is also worthwhile to note that, while the $\Gamma$-convergence results in Section~\ref{sec:Gammaconvergence} guarantee convergence of minimizers of $J_\tau$ over $\mathcal{K}_M$ to a minimizer of $F_0$ over $\mathcal{V}_M^b$, there is no monotonicity result in the sense that we do not know if minimizers of $J_\tau$ for smaller $\tau$ are better approximations.

One might think that the condition in \eqref{eq:choiceoftau} (for $v^0$) gives us some guidance in choosing $\tau$. After all, we do not want the algorithm to pin straight away in the first iteration. There are, however, some difficulties with this approach. The condition does not give us a way to determine $\tau$ a priori, before actually computing $e^{-\tau L} v^0$, and so while it might serve as a condition to reject or accept a given $\tau$ a posteriori (which boils down to being a glorified trial and error approach), it does not directly help in deciding on $\tau$ beforehand. We experimented with replacing the exponentials in \eqref{eq:choiceoftau} with their linear or quadratic Taylor approximations at $\tau=0$. While such approximations allow us to find a value of $\tau$ which satisfies the approximated version of the inequality \eqref{eq:choiceoftau}, in our experiments these $\tau$ did not satisfy the exact inequality.

Even if we do manage to find a $\tau$ which satisfies the condition in \eqref{eq:choiceoftau} for $v^0$, this same $\tau$ might not satisfy the condition for $v^1$ or some other $v^{k-1}$ down the line. In fact, we know that \ref{alg:massOKMBO} does terminate, hence there is a $k$ for which $\tau$ violates the condition for $v^{k-1}$. It is not at all clear which $k$ is the preferred final iteration number, even if we could somehow design a way to choose $\tau$ at the start in such a way to have the algorithm terminate after this preferred $k^{\text{th}}$ iteration. Lemma~\ref{lem:massLyapunov} tells us that $J_\tau$ decreases along iterates of the \ref{alg:massOKMBO} algorithm, but it does not tell us how close each iterate is to minimizing $J_\tau$.

We did consider (and experimented with) updating $\tau$ in each iteration of \ref{alg:massOKMBO} such that it satisfies \eqref{eq:choiceoftau} with $v^{k-1}$ in the $k^{\text{th}}$ iteration. This might seem a good approach, but it does not actually address the problems described above and introduces some new ones. First of all, we are still posed with the same difficulties we had in choosing a good $\tau$ based on $v^0$, only now at each iteration. Second, this introduces the question of when to stop updating $\tau$. If we update $\tau$ after each iteration such that it satisfies \eqref{eq:choiceoftau} in each new iteration, the algorithm will only terminate once it reaches a state in which \eqref{eq:choiceoftau} has no solutions, which is not necessarily guaranteed to be a preferred state. One possible choice could be to terminate when the only possible choices lead to a new value of $\tau$ that is higher than the previous value of $\tau$ (with some possible leeway in the first few iterations, to allow the scheme to move away from the initial condition). Third, such iterative updating of $\tau$ introduces a new layer of difficulty in the theoretical interpretation of the algorithm. If we run \ref{alg:massOKMBO} at a fixed $\tau$, we know that we do so in order to minimize $J_\tau$ (even though we do not know how well the final iterate approximates a minimzer), which in turn we do because such minimizers approximate minimizers of $J_\tau$ (by Theorem~\ref{thm:gammaconvergencemass}). Updating $\tau$ in each step complicates the first part of that interpretation. 

In our experiments the results obtained by updating $\tau$ did not outperform results obtained with fixed $\tau$ (measured by the value of $F_0$ at the final iterate). It might be that significant improvements can be obtained with the right update rule (we tried various ad hoc update techniques that would allow the algorithm to progress through a number of iterations before terminating), but since we did not discover such rule if it even exists, in this paper all the results we present are obtained with fixed $\tau$, chosen by trial and error. The graphs and results in Figures~\ref{fig:12b11b}, \ref{fig:16b17b}, and~\ref{fig:differentgamma} show how the final value of $F_0$ obtained by the algorithm can vary greatly depending on the choice of $\tau$. Note that large values of $\tau$ can lead to spurious patterns as explained in more detail in Section~\ref{sec:spurious}.

In Figure~\ref{fig:12b11b} we see detailed results obtained at two different values of $\tau$ with all the other parameter values kept the same. The resulting final iterates in Figures~\ref{fig:12bmini} and~\ref{fig:11bmini} are very different. The latter has a significantly lower value of $F_0$ than the former ($228.42$ versus $454.96$) and is thus a better approximation of a minimizer of $F_0$. It is however not an exact minimizer, as in this case we can obtain even lower values of $F_0$ by choosing a different initial condition (namely the one in Figure~\ref{fig:initialcondb}, as is explained in more detail in Section~\ref{sec:initialcondition} and can be seen in Figures~\ref{fig:02bmini} and~\ref{fig:02bOKTV}.

\begin{figure}
\begin{center}
\begin{subfigure}[b]{0.45\textwidth}
\includegraphics[width=1.1\textwidth]{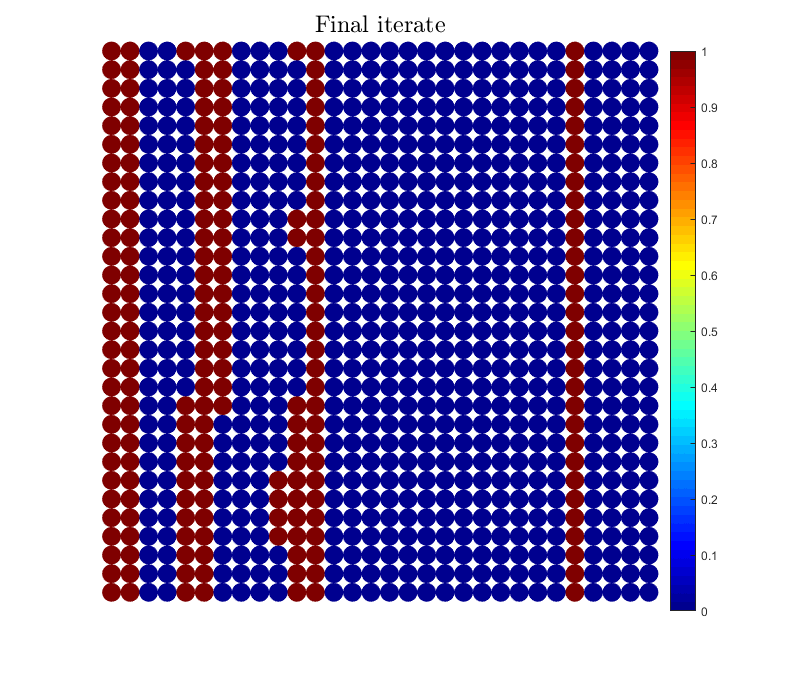}\vspace{-0.7cm}
\caption{ Final iterate ($k=3$) for $\tau=1$} \label{fig:12bmini}
\end{subfigure}
\begin{subfigure}[b]{0.45\textwidth}
\includegraphics[width=1.1\textwidth]{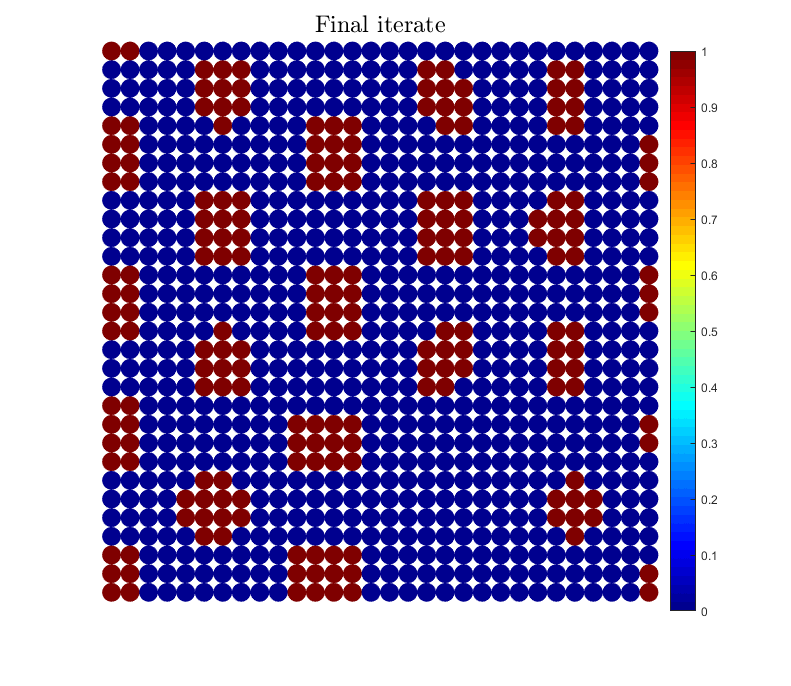}\vspace{-0.7cm}
\caption{Final iterate ($k=13$) for $\tau=5$} \label{fig:11bmini}
\end{subfigure}\\ \vspace{0.3cm}
\begin{subfigure}[b]{0.45\textwidth}
\includegraphics[width=1.1\textwidth]{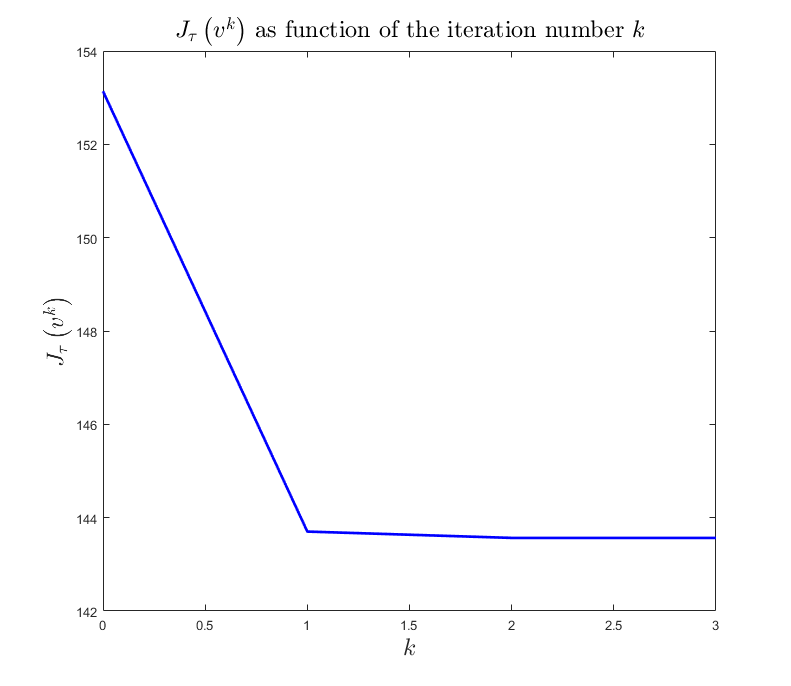}\vspace{-0.4cm}
\caption{Plot of $J_1\left(v^k\right)$} \label{fig:12bJtau}
\end{subfigure}
\begin{subfigure}[b]{0.45\textwidth}
\includegraphics[width=1.1\textwidth]{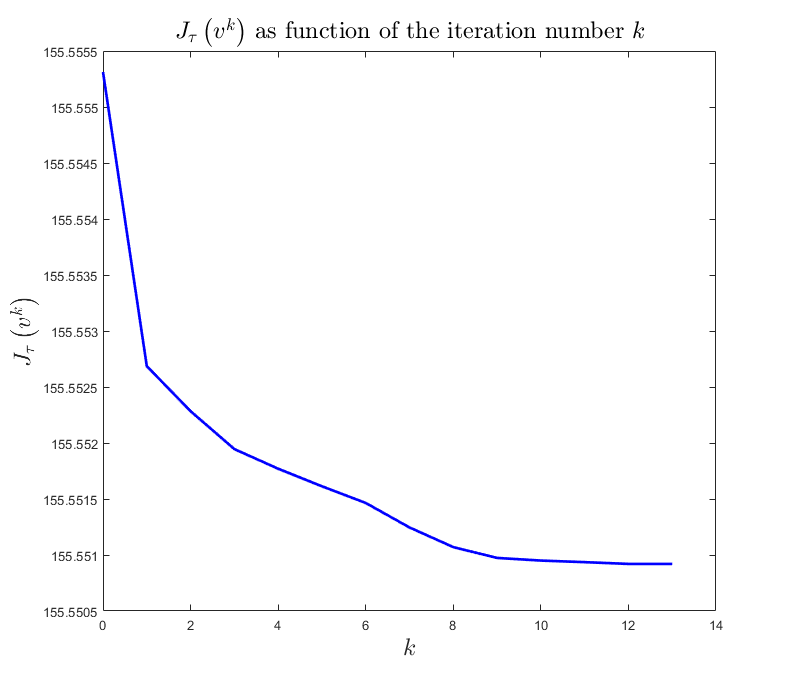}\vspace{-0.4cm}
\caption{Plot of $J_5\left(v^k\right)$} \label{fig:11bJtau}
\end{subfigure}\\ \vspace{0.3cm}
\begin{subfigure}[b]{0.45\textwidth}
\includegraphics[width=1.1\textwidth]{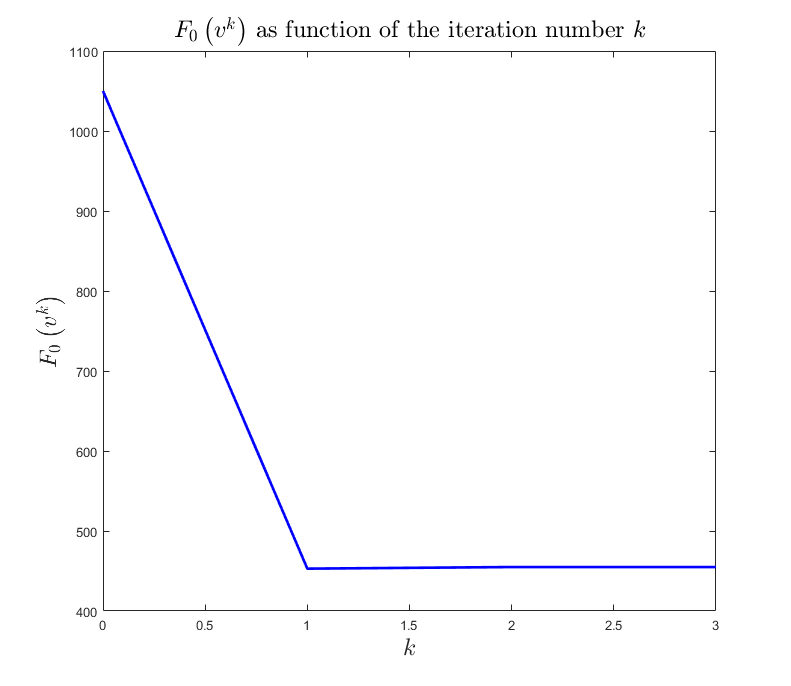}\vspace{-0.4cm}
\caption{Plot of $F_0\left(v^k\right)$ for $\tau=1$} \label{fig:12bOKTV}
\end{subfigure}
\begin{subfigure}[b]{0.45\textwidth}
\includegraphics[width=1.1\textwidth]{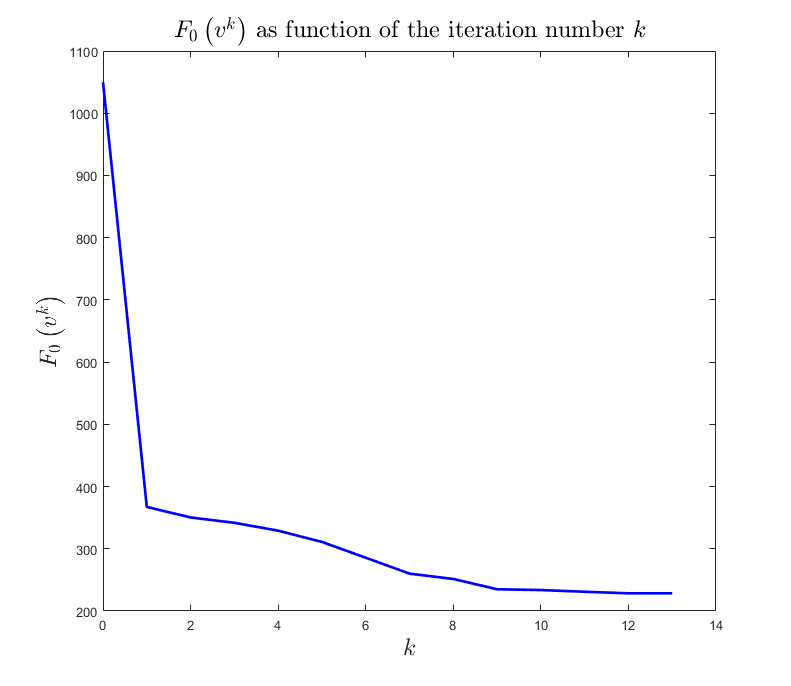}\vspace{-0.4cm}
\caption{Plot of $F_0\left(v^k\right)$ for $\tau=5$} \label{fig:11OKTV}
\end{subfigure}
\caption{Results from Algorithm~\ref{alg:massOKMBO} on $G_{\text{torus}}(900)$ with $r=0$, $\gamma=1$, and $M=200$. The initial condition from Figure~\ref{fig:initialconda} was used. The figures in the left column were obtained with $\tau=1$, the ones on the right with $\tau=5$. The value of $F_0$ at the final iterate is approximately $454.96$ for $\tau=1$ and $228.42$ for $\tau=5$.}
\label{fig:12b11b}
\end{center}
\end{figure}

In most of the numerical results we present here, $F_0$ decreases monotonically along the sequence of \ref{alg:massOKMBO} iterates (until the penultimate iterate after which it stays constant; see the discussion of the stopping criterion in Section~\ref{sec:otherchoices}). Figure~\ref{fig:17bOKTV} shows that this is not always the case. In Figures~\ref{fig:13bOKTV} and~\ref{fig:19bOKTV} we even see cases in which the value of $F_0$ at the final iterate is higher than at some of the earlier iterates. If the required additional memory and computation time are available, at every iteration of \ref{alg:massOKMBO} one can store the state which has obtained the lowest value of $F_0$ so far and use that state as approximate minimizer of $F_0$ upon termination of the algorithm. Note, however, from Figures~\ref{fig:17bJtau}, \ref{fig:13bJtau}, and~\ref{fig:19bJtau} that also in those cases the value of $J_\tau$ does decrease along the sequence of iterates, as it is guaranteed to do by Lemma~\ref{lem:massLyapunov}.

\begin{figure}
\begin{center}
\begin{subfigure}[b]{0.45\textwidth}
\includegraphics[width=1.1\textwidth]{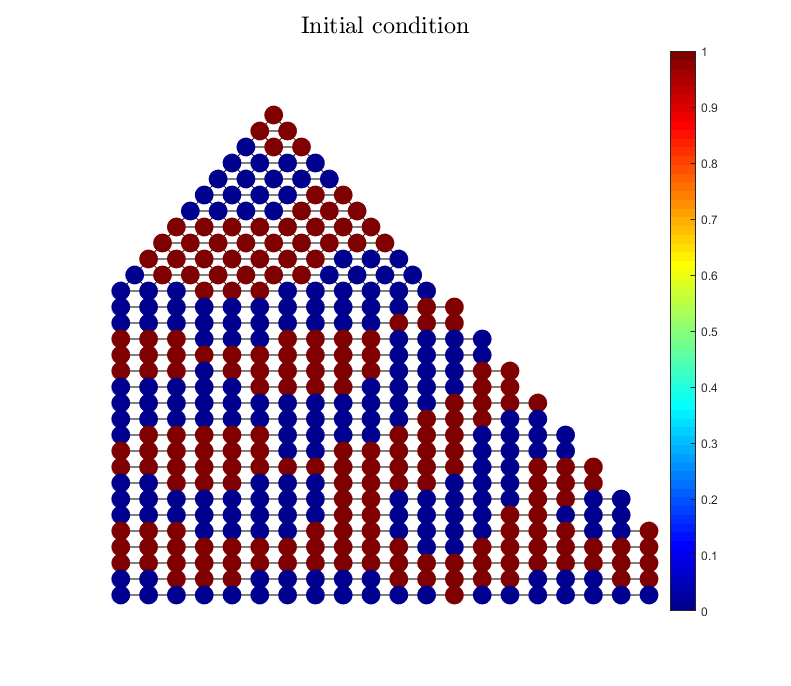}\vspace{-0.7cm}
\caption{Initial condition} \label{fig:13binit}
\end{subfigure}
\begin{subfigure}[b]{0.45\textwidth}
\includegraphics[width=1.1\textwidth]{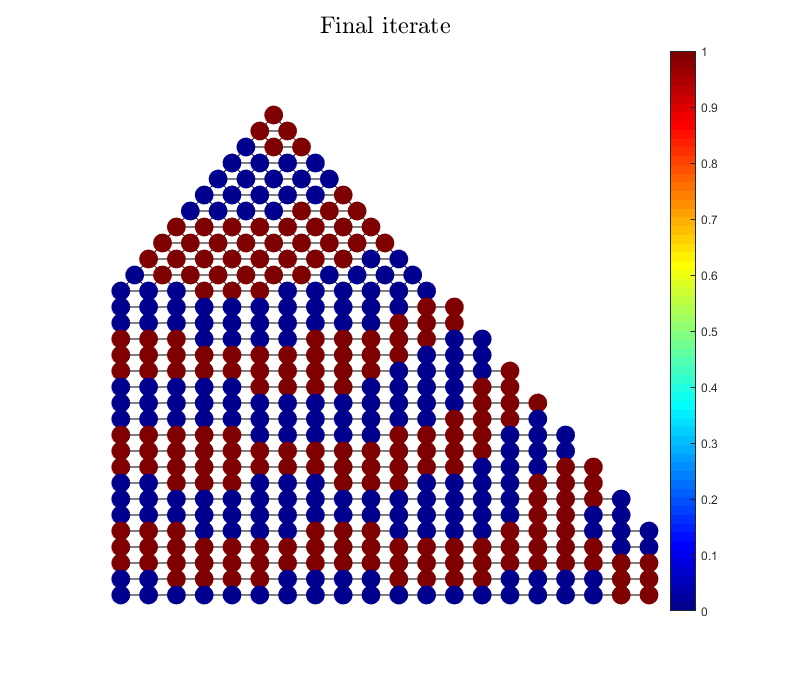}\vspace{-0.7cm}
\caption{Final iterate ($k=8$)} \label{fig:13bmini}
\end{subfigure}\\ \vspace{0.3cm}
\begin{subfigure}[b]{0.45\textwidth}
\includegraphics[width=1.1\textwidth]{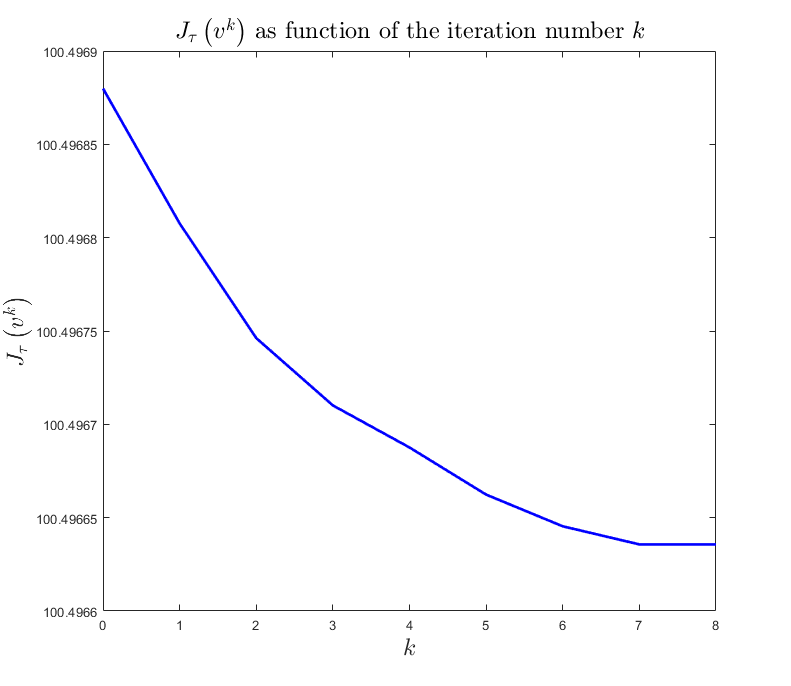}\vspace{-0.4cm}
\caption{Plot of $J_5\left(v^k\right)$} \label{fig:13bJtau}
\end{subfigure}
\begin{subfigure}[b]{0.45\textwidth}
\includegraphics[width=1.1\textwidth]{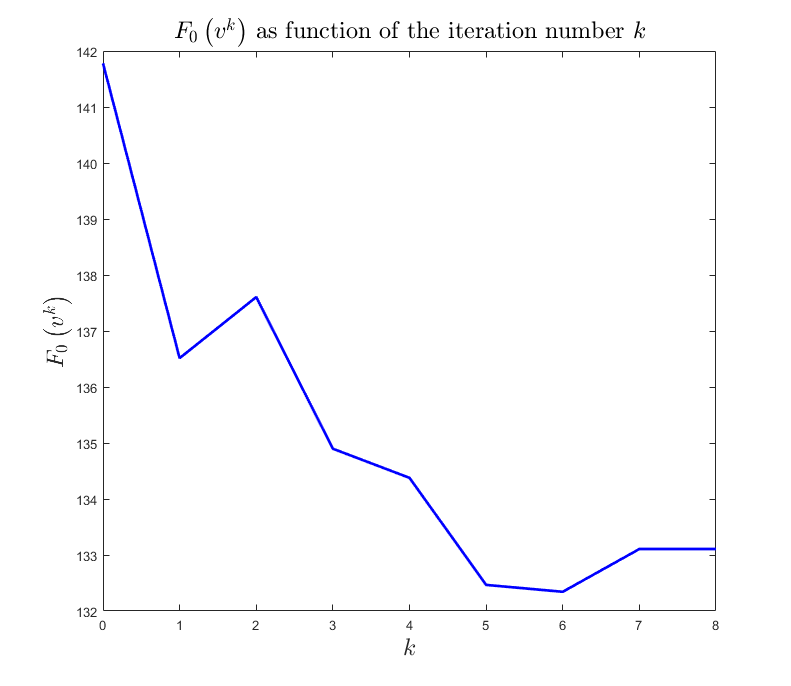}\vspace{-0.4cm}
\caption{Plot of $F_0\left(v^k\right)$} \label{fig:13bOKTV}
\end{subfigure}
\caption{Results from Algorithm~\ref{alg:massOKMBO} on $G_{\text{stitched}}(402)$ with $r=0$, $\gamma=1$, $M=201$, and $\tau=5$. The initial condition in Figure~\ref{fig:13binit} was constructed according to option (c) in Section~\ref{sec:initialcondition}. The value of $F_0$ at the final iterate is approximately $133.11$.}
\label{fig:13b}
\end{center}
\end{figure}

\subsection{Choice of initial condition}\label{sec:initialcondition}

Up until now, we have not paid much attention to the choice of initial condition, but in practice this choice has a big influence on the final state of \ref{alg:massOKMBO}; different initial conditions can lead to final states with significantly different values of $F_0$. In the experiments which we report on here\footnote{We also tried some other initial conditions in $\mathcal{K}_M\setminus\mathcal{V}_M^{ab}$, $\mathcal{K}\setminus\mathcal{V}^b$, and $\mathcal{V}^b\setminus\mathcal{V}^b_M$, such as states with the mass spread out evenly over all nodes or other constant functions, states with randomly spread mass (which differs from option (a) in the main text in that this state is usually not binary), and states constructed by changing a function $\mathcal{V}^{ab}_M$ at one node to make it binary. In our test these approaches were never optimal, so we will not spend more time on them here.} we used three different options for constructing initial conditions:
\begin{itemize}
\item[(a)] Assign the available mass to random nodes (by applying the mass conserving thresholding step of \ref{alg:massOKMBO} to a random vector generated by MATLAB's \texttt{rand} function). A realization of such a randomly constructed initial condition is given in Figure~\ref{fig:14binit}.
\item[(b)] Cluster the initial mass together in one part of the graph. This description is necessarily somewhat vague, as it is not a well-defined method in itself which is applicable across all choices of graphs. Instead, in this option we let the structure of the graph suggest the structure of the initial condition. It is best illustrated by specific examples, e.g. assigning all mass to the nodes in one strip of the square grid/discretized torus or one part of the stitched mess; see Figures~\ref{fig:initialconda} and~\ref{fig:15binit}. Figure~\ref{fig:21binit} shows another example where this option was used\footnote{In practice this is achieved by applying the mass conserving threshold step to the vector $(n, n-1, \ldots, 1)^T$, where the numbering of the nodes in $G_{\text{torus}}$ and $G_{\text{stitched}}$ is clear from the resulting initial conditions in Figures~\ref{fig:initialconda} and~\ref{fig:15binit}, respectively, and the node numbering in $G_{\text{neural}}$ is the one inherited from the dataset from \cite{Newmanwebsite}.}.
\item[(c)] Construct $v^0$ based on the eigenfunctions $\phi^m$ by applying the mass conserving thresholding step to an eigenvector $\phi^m$ corresponding to the smallest nonzero eigenvalue of $L$\footnote{Other variations we tried include using other eigenvalues, e.g. the smallest non-zero eigenvalue of $\Delta$, (since, for small $\gamma$, we can view $L$ as a perturbation of $\Delta$ in the sense of Theorem~\ref{thm:newgraph}), and applying the mass conserving threshold step to the vector with entries $\left|\pm \phi^m_i\right|$ (or $\pm \left|\sum_m \phi^m_i\right|$ in the case of a non-simple eigenvalue) to reflect (in crude approximation) the fact that the relevant quantity to minimize in \eqref{eq:OKexpression2} is $\left|\langle v^0, \phi^m\rangle_{\mathcal{V}}\right|$. None of those choices stood out from option (c)  mentioned in the main text.}. When this eigenvalue $\Lambda_m$ has multiplicity greater than $1$, the choice of $\phi_m$ is not unique (besides the `standard' non-uniqueness in sign when $m\geq 1$). In our experiments we choose $\phi_m$ to be the sum of all eigenfunctions (after normalization) that MATLAB's \texttt{eig} function returns corresponding to $\Lambda_m$\footnote{Note that this still could be machine dependent, as \texttt{eig} does not necessarily use a consistent order.}. Examples of initial conditions constructed in this way are given in Figures~\ref{fig:initialcondb}, \ref{fig:initialcondc}, \ref{fig:16binit}, \ref{fig:18binit}, \ref{fig:19binit}, and~\ref{fig:20binit}. It should be noted that, while some of these initial conditions are very close to the final iterate they lead to, they are not (in our experiments that are presented here) equal to the final iterates. Even in those cases in which the initial condition is closest to the final iterate (out of the cases we present here), i.e. those in Figure~\ref{fig:02b01b}, the algorithm goes through at least one iteration before arriving at the final state. That is not to say the algorithm cannot pin (it will of course, if $\tau$ is chosen small enough), but it shows that the eigenfunctions are not necessarily stationary states of the algorithm and \ref{alg:massOKMBO} can improve on those states.
\end{itemize}

Comparing the right column of Figure~\ref{fig:12b11b} with the left column of Figure~\ref{fig:02b01b} we see a case in which the eigenfunction based initial condition (option (c) above)  gives better results than the `structured' approach of option (b) with all other parameters kept the same. The latter (Figure~\ref{fig:12b11b}) gives a final value of $F_0$ of approximately $228.42$, whereas the former gives a value of approximately $206.59$. Option (c) is not always preferred though. In Figures~\ref{fig:15b14b} and~\ref{fig:16b17b} we see that both the value obtained with the initial condition in Figure~\ref{fig:14binit} (which is a particular realization of the random process of option (a)) and the value obtained with the initial condition from Figure~\ref{fig:15binit} (option (b)) are both lower than the value obtained with option (c) (Figure~\ref{fig:16binit}), with all other parameters kept the same: $102.01$ and $122.83$, respectively, versus $126.05$. We can improve the result obtained with option (c) by choosing a different $\tau$ ($\tau=7$ instead of $\tau=5$ in Figures~\ref{fig:17bmini} and~\ref{fig:17bOKTV}), but the resulting value $104.01$ is still higher than the lowest value in Figure~\ref{fig:15b14b} (at $\tau=7$ options (a) and (b) did perform worse than option (c) in our experiments; not pictured). We did not find any $\tau$ values that achieved lower $F_0$ values in this case. It should also be noted that of course not every realization of the random process that generates the initial condition in option (a) achieves the same low value for $F_0$. In a separate run of 10 experiments (which did not include the run pictured in the right column of Figure~\ref{fig:15b14b}, but used the exact same parameters) we obtained an average final value for $F_0$ of $118.77$ with a (corrected sample) standard deviation (obtained via \texttt{std} in MATLAB) of $26.84$, a maximum of $191.24$ and a minimum of $101.84$ (all numbers rounded to two decimals).

For some non-optimal initial conditions, we see patterns emerge that look like the intermediate-time phase ordering pictures in \cite[Figure 2]{ito1998domain}, as we see in Figure~\ref{fig:22b}. Figure~\ref{fig:23b} has been constructed using the same parameter choices as Figure~\ref{fig:22b}, but uses an initial condition constructed using option (b), instead of an eigenfunction based initial condition (option (c)). In this case a lower final value of $F_0$ is achieved.

\begin{figure}
\begin{center}
\begin{subfigure}[b]{0.45\textwidth}
\includegraphics[width=1.1\textwidth]{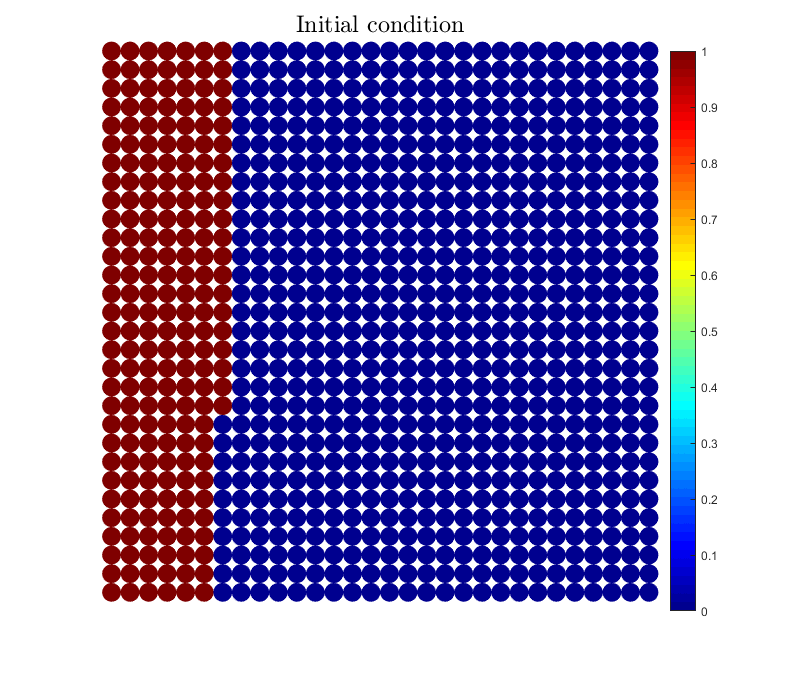}\vspace{-0.7cm}
\caption{`Structured' initial condition} \label{fig:initialconda}
\end{subfigure}
\hspace{.5cm}
\begin{subfigure}[b]{0.45\textwidth}
\includegraphics[width=1.1\textwidth]{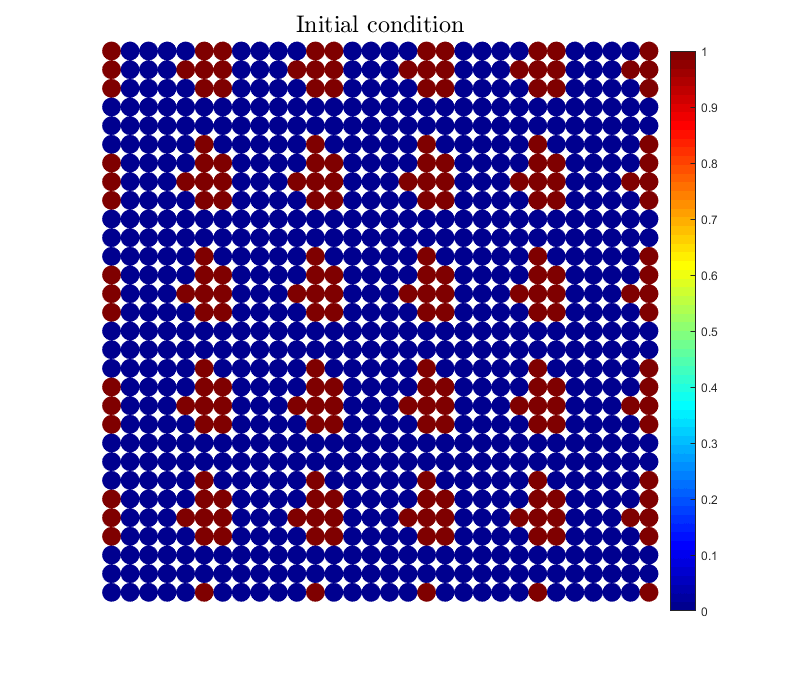}\vspace{-0.7cm}
\caption{Eigenfunction based initial condition} \label{fig:initialcondb}
\end{subfigure}\\ \vspace{0.3cm}
\begin{subfigure}[b]{0.45\textwidth}
\includegraphics[width=1.1\textwidth]{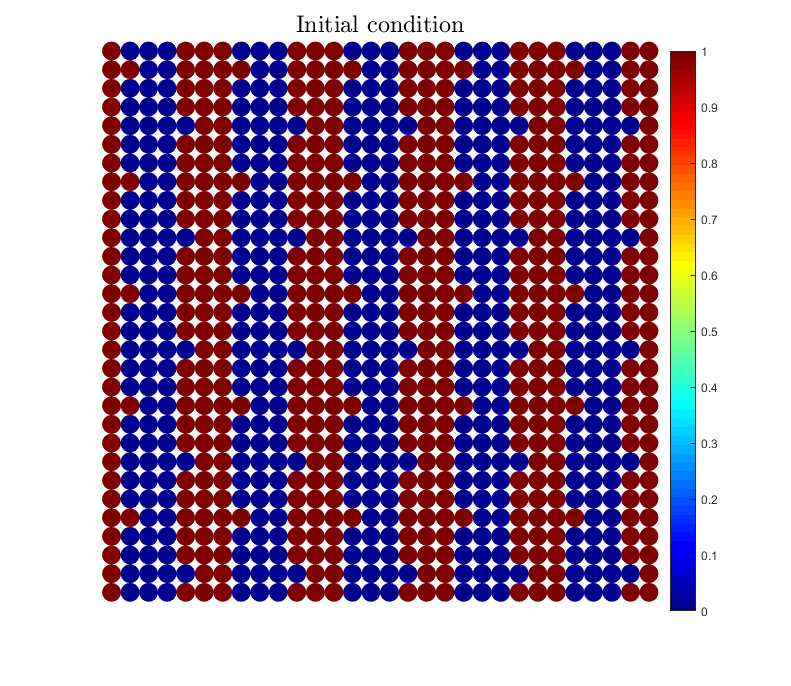}\vspace{-0.7cm}
\caption{Eigenfunction based initial condition} \label{fig:initialcondc}
\end{subfigure}
\hspace{.5cm}
\begin{subfigure}[b]{0.45\textwidth}
\includegraphics[width=1.1\textwidth]{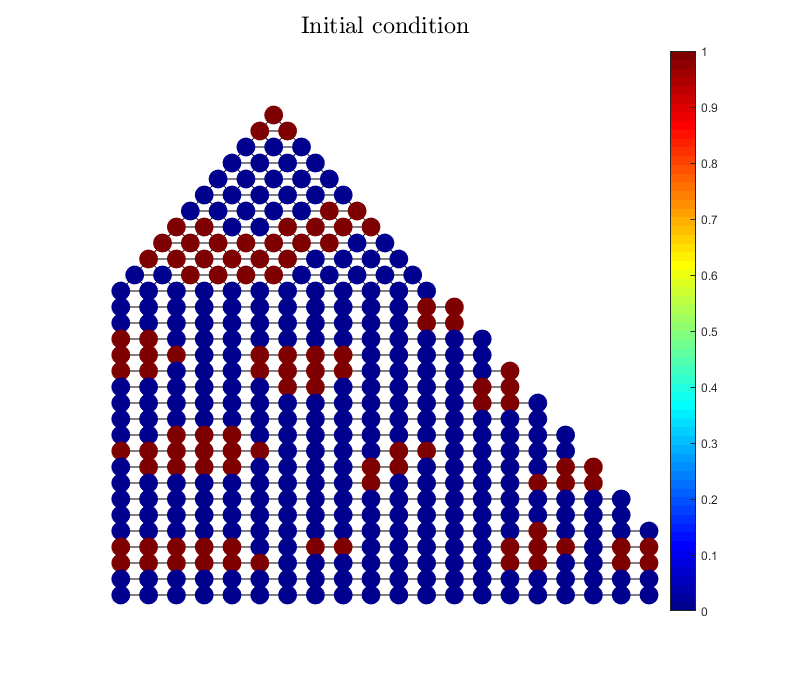}\vspace{-0.7cm}
\caption{Eigenfunction based initial condition} \label{fig:16binit}
\end{subfigure}
\caption{Three different initial conditions for \ref{alg:massOKMBO} on $G_{\text{torus}}(900)$ and one on $G_{\text{stitched}}(402)$, all with $r=0$. The top two figures have $M=200$, the bottom left figure has $M=450$, the bottom right one $M=201$. The initial condition in the top left figure is `structured' in the sense of option (b) in Section~\ref{sec:initialcondition}; the others are based on eigenfunctions as in option (c) in Section~\ref{sec:initialcondition}.}
\label{fig:initialcond}
\end{center}
\end{figure}

\begin{figure}
\begin{center}
\begin{subfigure}[b]{0.45\textwidth}
\includegraphics[width=1.1\textwidth]{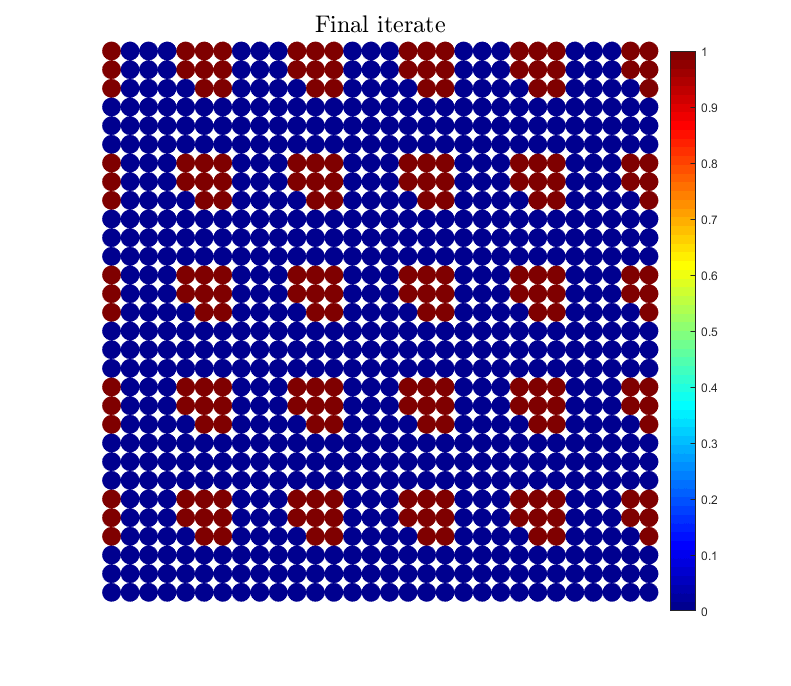}\vspace{-0.7cm}
\caption{ Final iterate ($k=3$) for $M=200$} \label{fig:02bmini}
\end{subfigure}
\begin{subfigure}[b]{0.45\textwidth}
\includegraphics[width=1.1\textwidth]{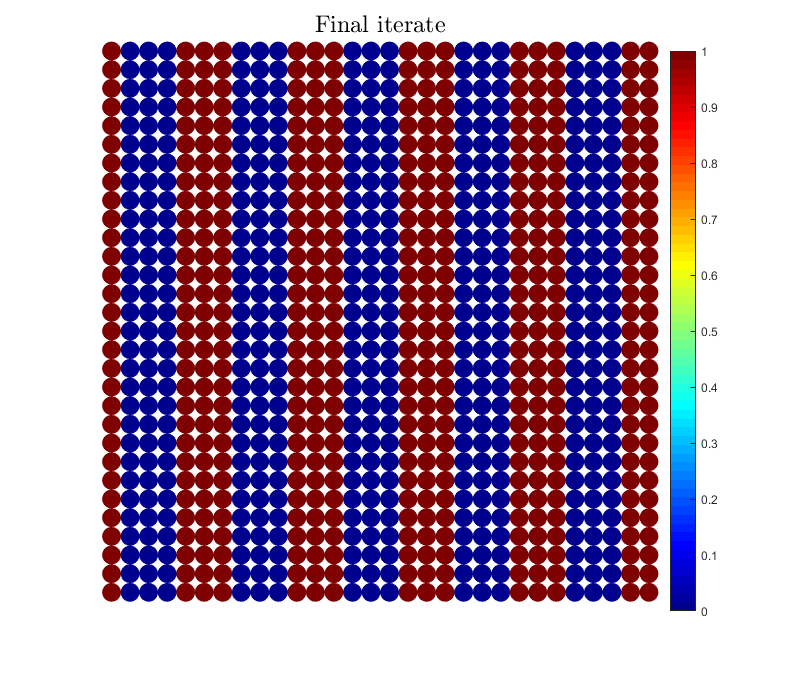}\vspace{-0.7cm}
\caption{Final iterate ($k=2$) for $M=450$} \label{fig:01bmini}
\end{subfigure}\\ \vspace{0.3cm}
\begin{subfigure}[b]{0.45\textwidth}
\includegraphics[width=1.1\textwidth]{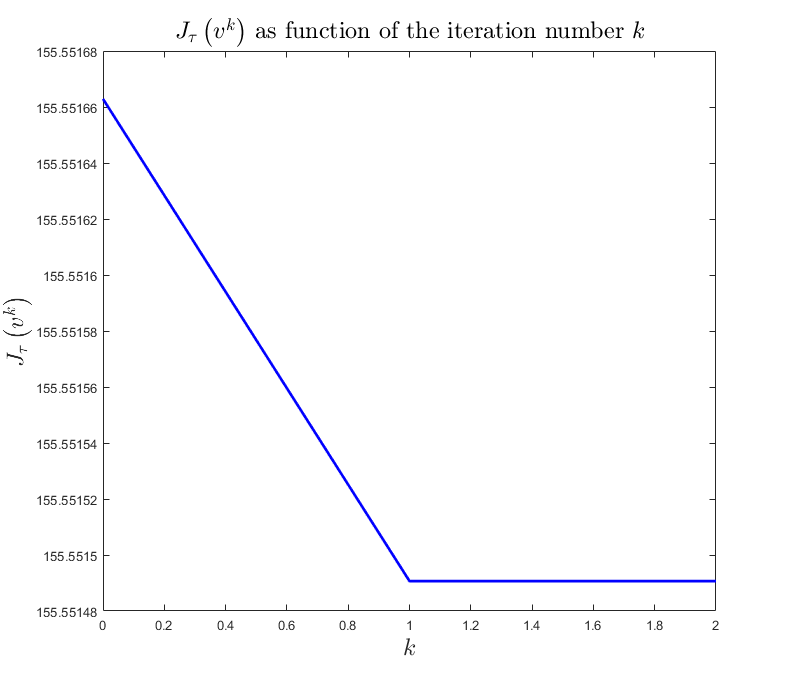}\vspace{-0.4cm}
\caption{Plot of $J_5\left(v^k\right)$} \label{fig:02bJtau}
\end{subfigure}
\begin{subfigure}[b]{0.45\textwidth}
\includegraphics[width=1.1\textwidth]{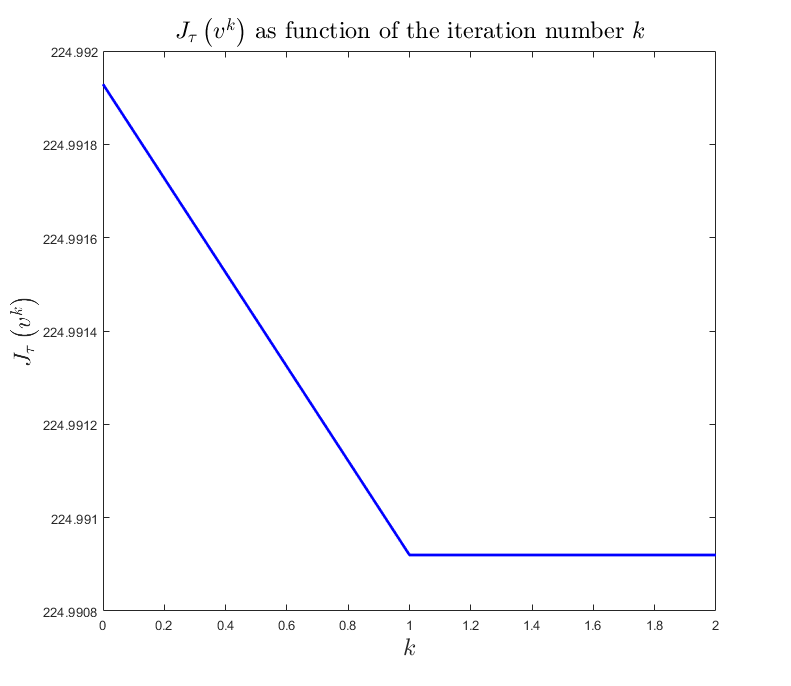}\vspace{-0.4cm}
\caption{Plot of $J_5\left(v^k\right)$} \label{fig:01bJtau}
\end{subfigure}\\ \vspace{0.3cm}
\begin{subfigure}[b]{0.45\textwidth}
\includegraphics[width=1.1\textwidth]{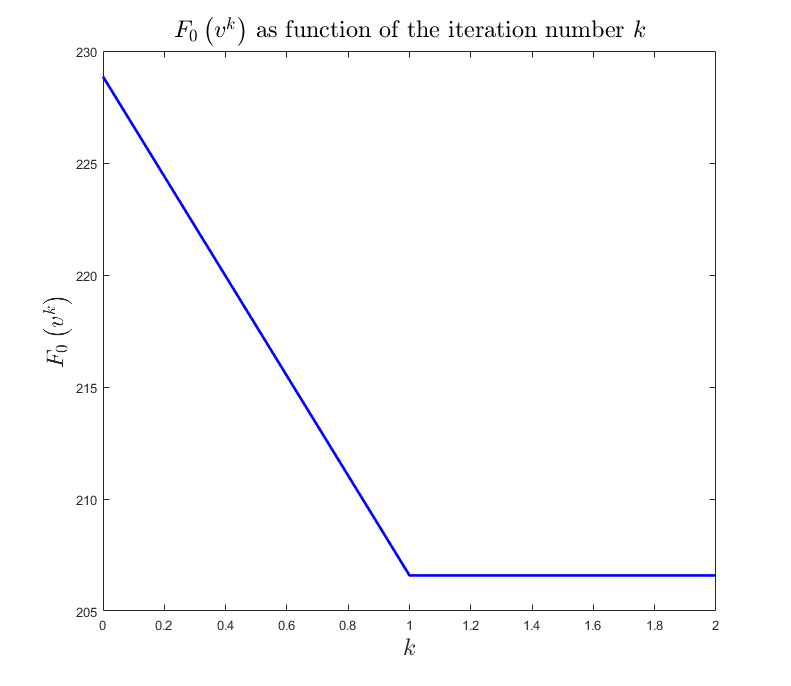}\vspace{-0.4cm}
\caption{Plot of $F_0\left(v^k\right)$ for $M=200$} \label{fig:02bOKTV}
\end{subfigure}
\begin{subfigure}[b]{0.45\textwidth}
\includegraphics[width=1.1\textwidth]{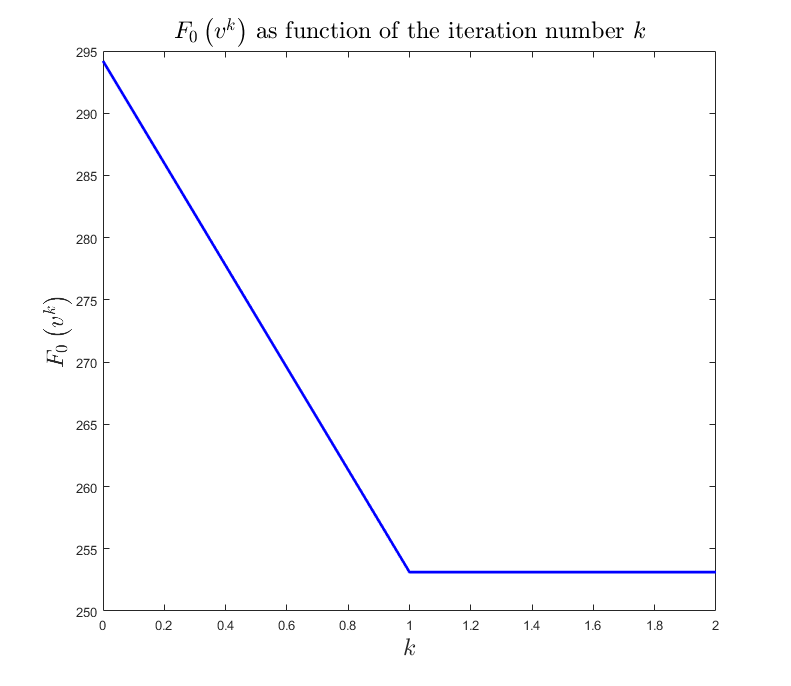}\vspace{-0.4cm}
\caption{Plot of $F_0\left(v^k\right)$ for $M=450$} \label{fig:01bOKTV}
\end{subfigure}
\caption{Results from Algorithm~\ref{alg:massOKMBO} on $G_{\text{torus}}(900)$ with $r=0$, $\gamma=1$, and $\tau=5$. The figures in the left column were obtained with $M=200$ and the initial condition from Figure~\ref{fig:initialcondb}, the ones on the right with $M=450$ and the initial condition from Figure~\ref{fig:initialcondc}. The value of $F_0$ at the final iterate is approximately $206.59$ for $M=200$ and $253.12$ for $M=450$.}
\label{fig:02b01b}
\end{center}
\end{figure}

\begin{figure}
\begin{center}
\begin{subfigure}[b]{0.45\textwidth}
\includegraphics[width=1.1\textwidth]{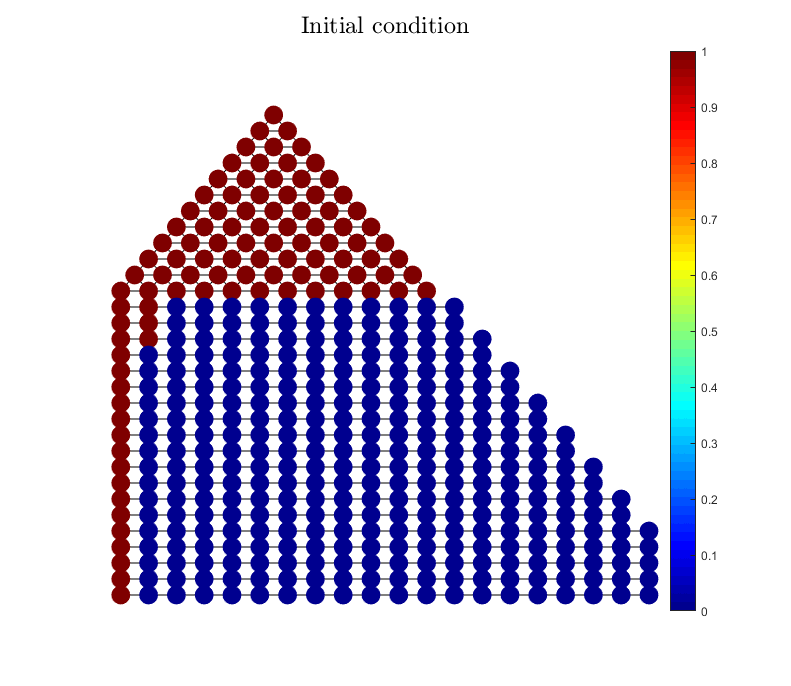}\vspace{-0.4cm}
\caption{Initial condition} \label{fig:15binit}
\end{subfigure}
\begin{subfigure}[b]{0.45\textwidth}
\includegraphics[width=1.1\textwidth]{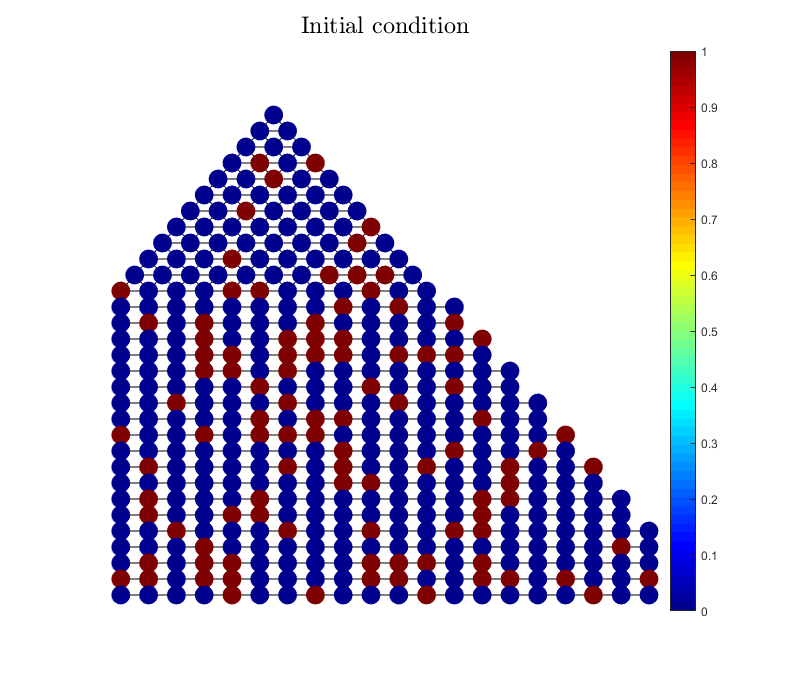}\vspace{-0.4cm}
\caption{Initial condition} \label{fig:14binit}
\end{subfigure}\\ \vspace{0.3cm}
\begin{subfigure}[b]{0.45\textwidth}
\includegraphics[width=1.1\textwidth]{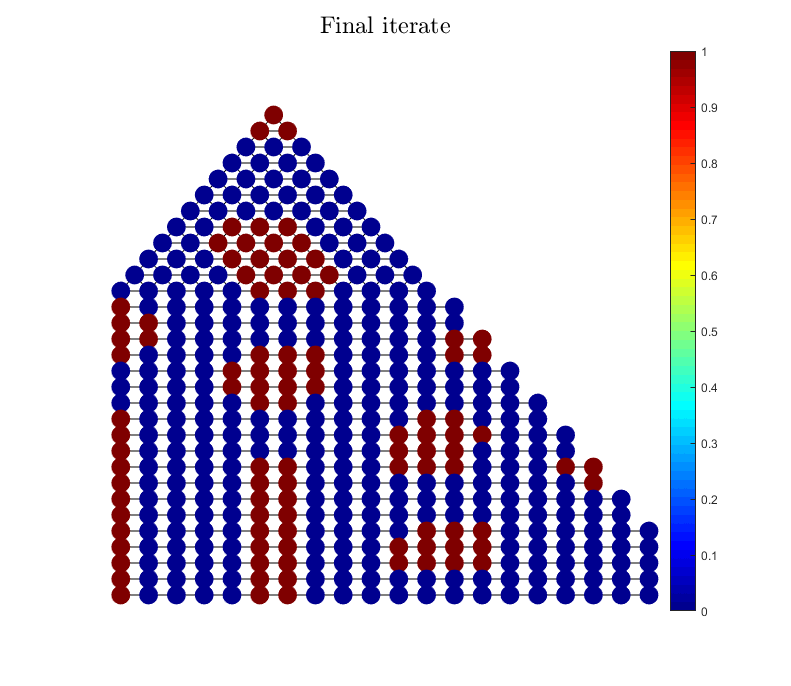}\vspace{-0.7cm}
\caption{ Final iterate ($k=13$) starting from Fig.~\ref{fig:15binit}} \label{fig:15bmini}
\end{subfigure}
\begin{subfigure}[b]{0.45\textwidth}
\includegraphics[width=1.1\textwidth]{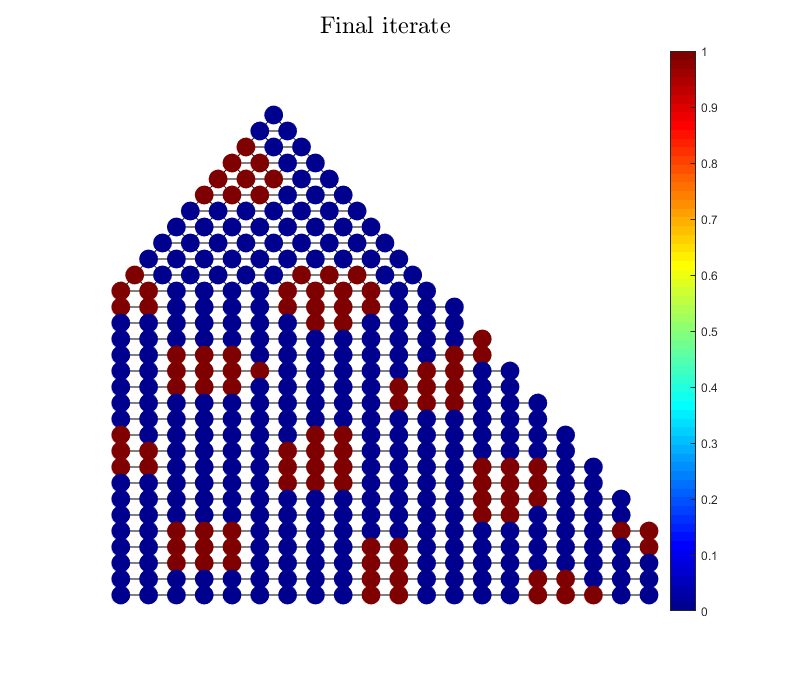}\vspace{-0.7cm}
\caption{Final iterate ($k=9$) starting from Fig.~\ref{fig:14binit}} \label{fig:14bmini}
\end{subfigure}\\ \vspace{0.3cm}
\begin{subfigure}[b]{0.45\textwidth}
\includegraphics[width=1.1\textwidth]{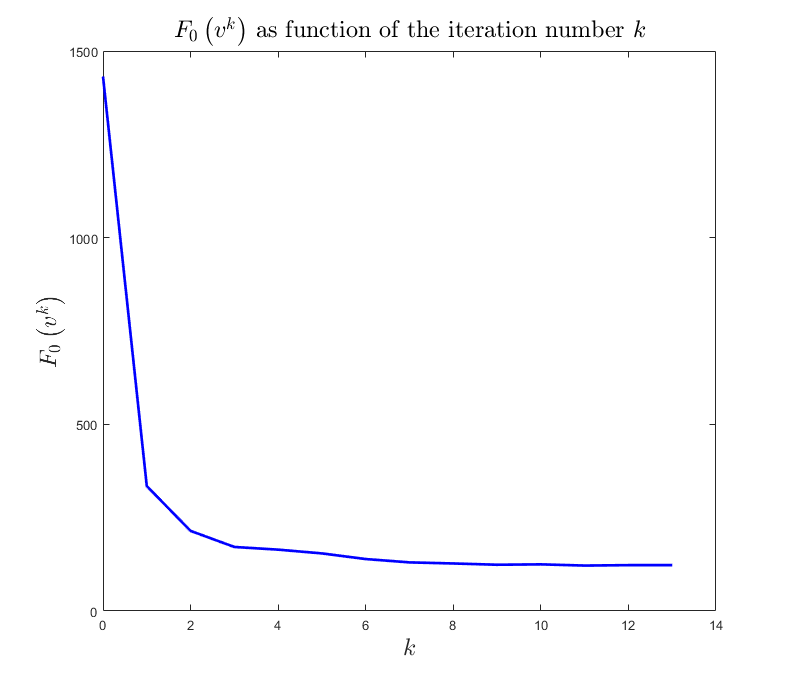}\vspace{-0.4cm}
\caption{Plot of $F_0\left(v^k\right)$ starting from Fig.~\ref{fig:15binit}} \label{fig:15bOKTV}
\end{subfigure}
\begin{subfigure}[b]{0.45\textwidth}
\includegraphics[width=1.1\textwidth]{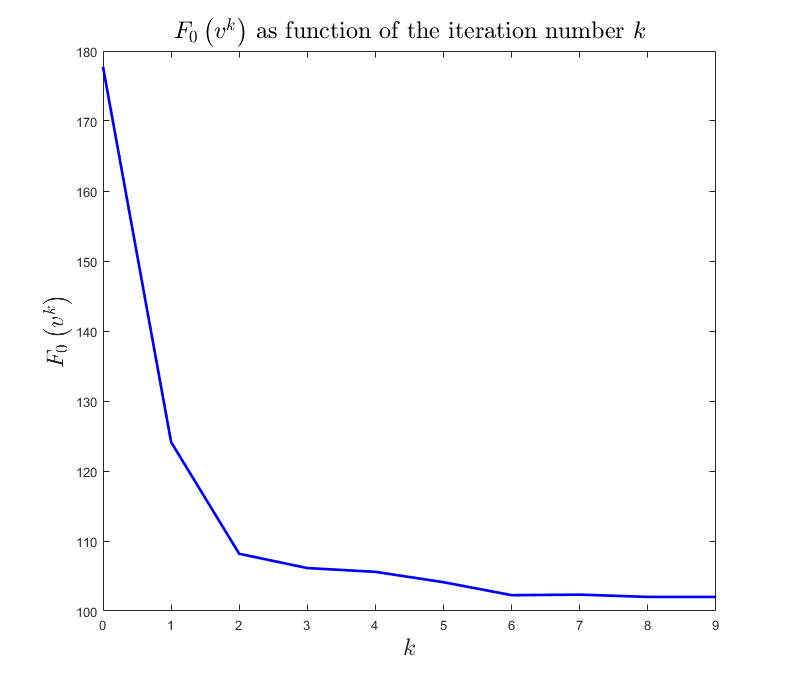}\vspace{-0.4cm}
\caption{Plot of $F_0\left(v^k\right)$ starting from Fig.~\ref{fig:14binit}} \label{fig:14bOKTV}
\end{subfigure}
\caption{Results from Algorithm~\ref{alg:massOKMBO} on $G_{\text{stitched}}(402)$ with $r=0$, $\gamma=1$, $M=100$, and $\tau=5$. The left and right columns correspond to the cases in which the initial conditions from Figures~\ref{fig:15binit} (option (b) from Section~\ref{sec:initialcondition}) and~\ref{fig:14binit} (option (a)), respectively, were used. The value of $F_0$ at the final iterate is approximately $122.83$ on the left and $102.01$ on the right.}
\label{fig:15b14b}
\end{center}
\end{figure}

\begin{figure}
\begin{center}
\begin{subfigure}[b]{0.45\textwidth}
\includegraphics[width=1.1\textwidth]{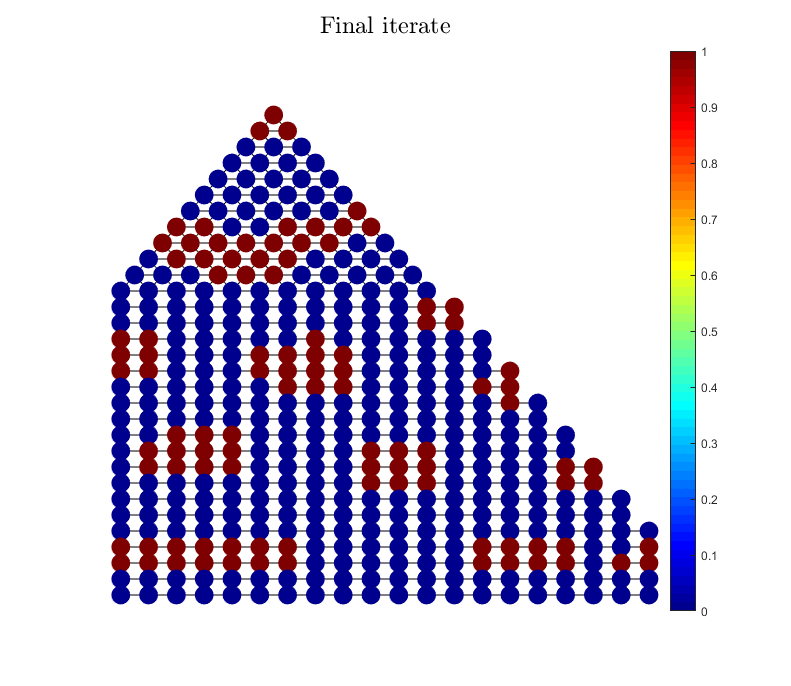}\vspace{-0.7cm}
\caption{ Final iterate ($k=4$) with $\tau=5$} \label{fig:16bmini}
\end{subfigure}
\begin{subfigure}[b]{0.45\textwidth}
\includegraphics[width=1.1\textwidth]{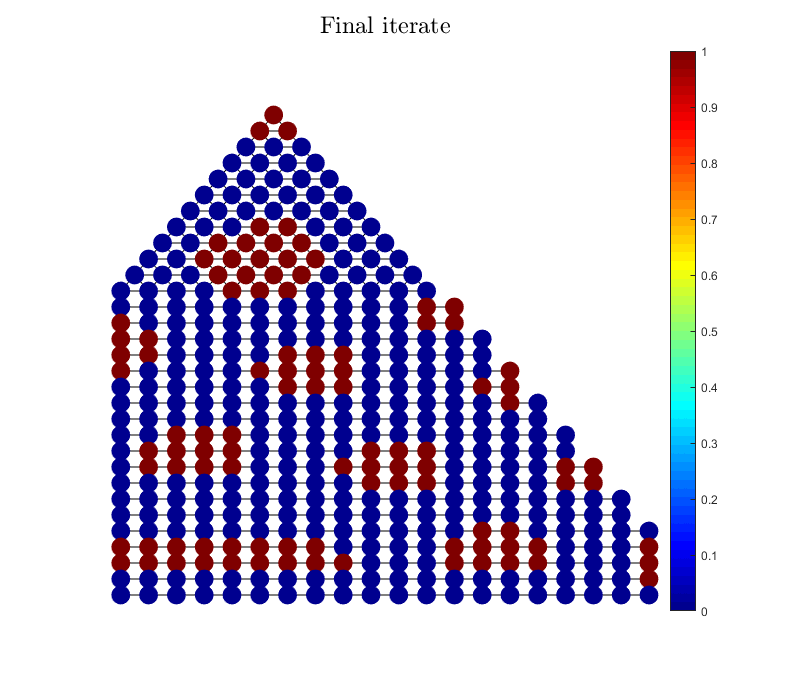}\vspace{-0.7cm}
\caption{Final iterate ($k=9$) with $\tau=7$} \label{fig:17bmini}
\end{subfigure}\\ \vspace{0.3cm}
\begin{subfigure}[b]{0.45\textwidth}
\includegraphics[width=1.1\textwidth]{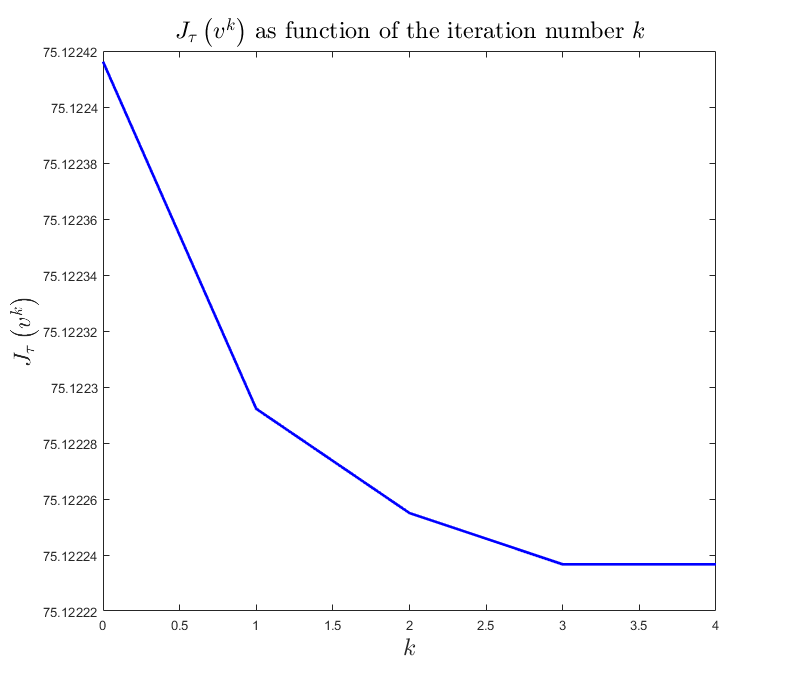}\vspace{-0.4cm}
\caption{Plot of $J_5\left(v^k\right)$ with $\tau=5$} \label{fig:16bJtau}
\end{subfigure}
\begin{subfigure}[b]{0.45\textwidth}
\includegraphics[width=1.1\textwidth]{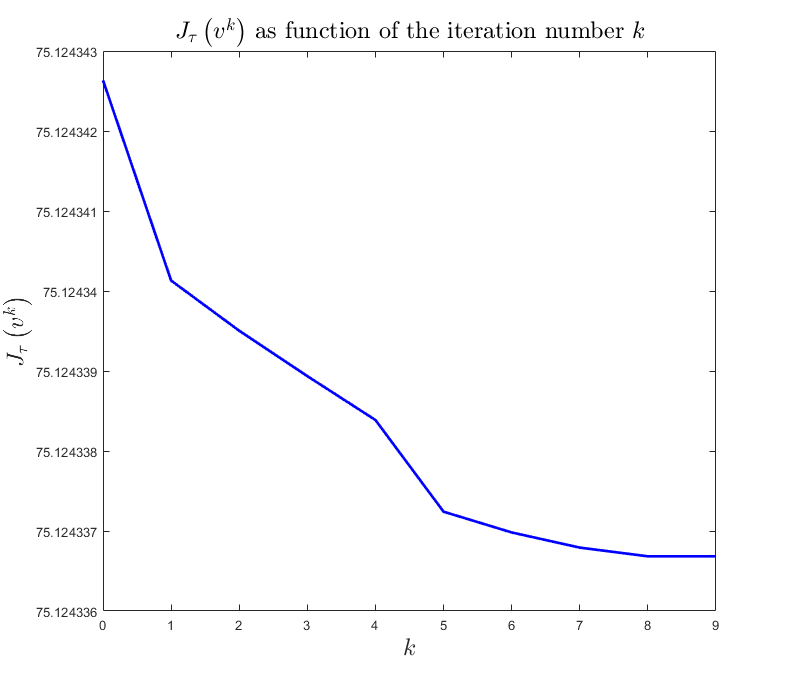}\vspace{-0.4cm}
\caption{Plot of $J_7\left(v^k\right)$ with $\tau=7$} \label{fig:17bJtau}
\end{subfigure}\\ \vspace{0.3cm}
\begin{subfigure}[b]{0.45\textwidth}
\includegraphics[width=1.1\textwidth]{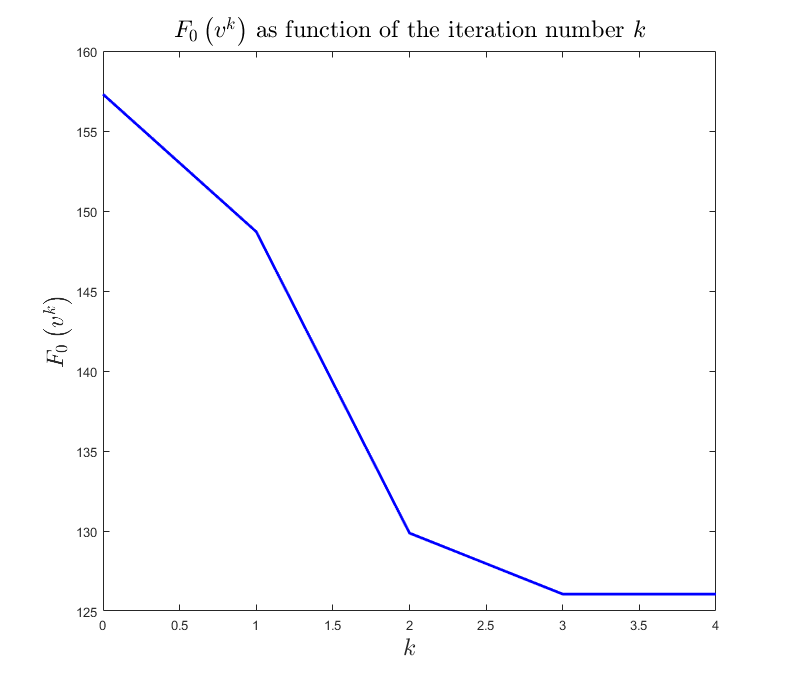}\vspace{-0.4cm}
\caption{Plot of $F_0\left(v^k\right)$ with $\tau=5$} \label{fig:16bOKTV}
\end{subfigure}
\begin{subfigure}[b]{0.45\textwidth}
\includegraphics[width=1.1\textwidth]{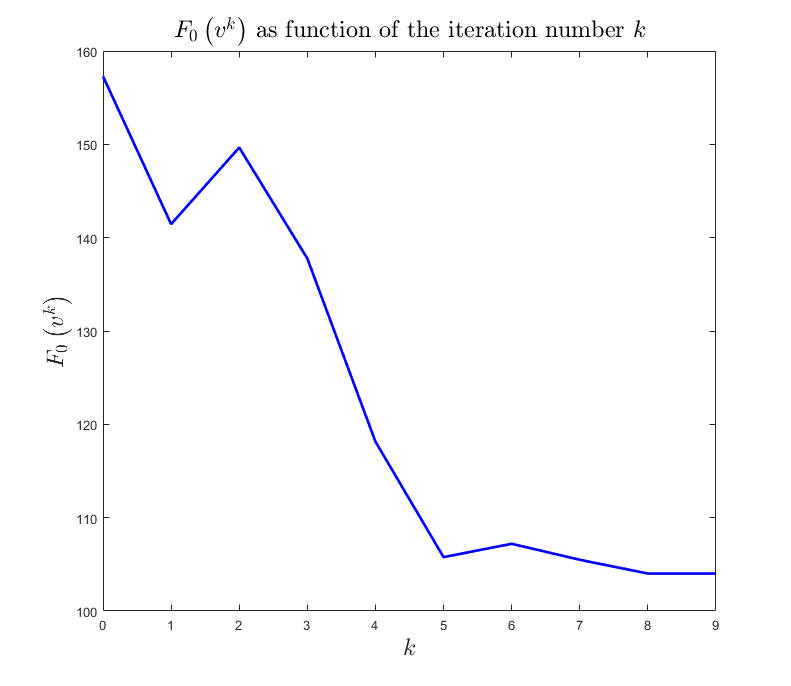}\vspace{-0.4cm}
\caption{Plot of $F_0\left(v^k\right)$ with $\tau=7$} \label{fig:17bOKTV}
\end{subfigure}
\caption{Results from Algorithm~\ref{alg:massOKMBO} on $G_{\text{stitched}}(402)$ with $r=0$, $\gamma=1$ and $M=100$ and starting from the initial condition in Figure~\ref{fig:16binit}. The left and right columns in the two lower rows correspond to the cases in which $\tau=5$ and $\tau=7$, respectively. The value of $F_0$ at the final iterate is approximately $126.05$ on the left and $104.01$ on the right.}
\label{fig:16b17b}
\end{center}
\end{figure}

\begin{figure}
\begin{center}
\begin{subfigure}[b]{0.45\textwidth}
\includegraphics[width=1.1\textwidth]{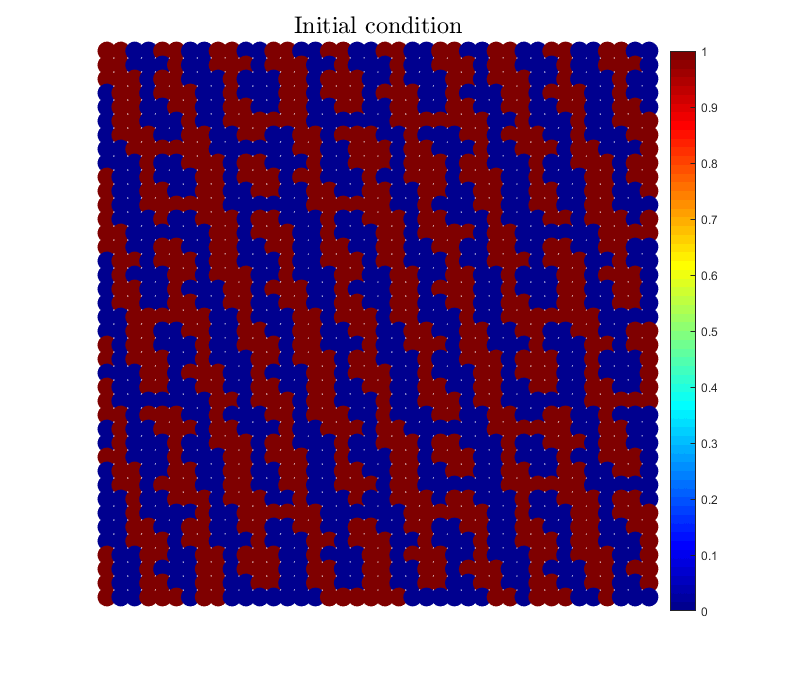}\vspace{-0.7cm}
\caption{Initial condition} \label{fig:22binit}
\end{subfigure}
\begin{subfigure}[b]{0.45\textwidth}
\includegraphics[width=1.1\textwidth]{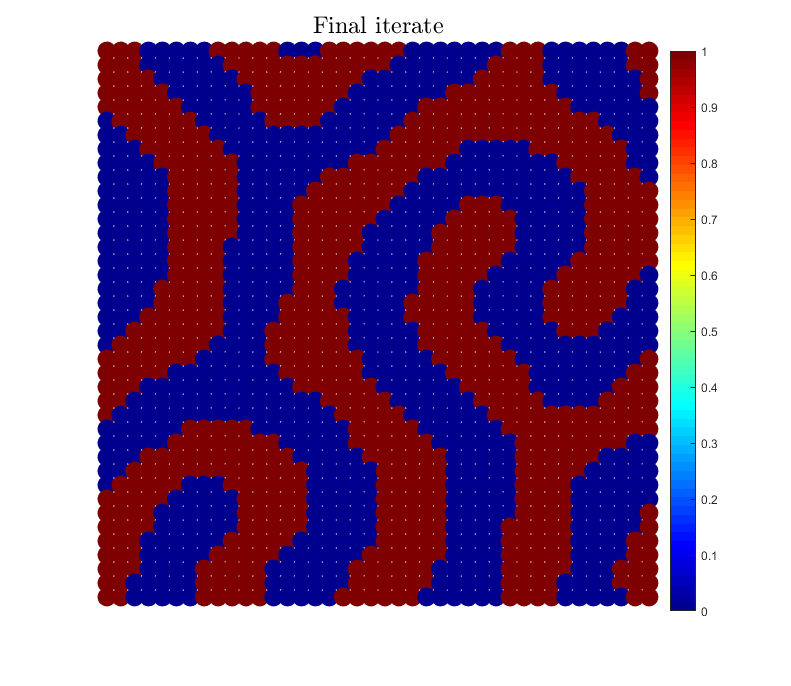}\vspace{-0.7cm}
\caption{Final iterate ($k=25$)} \label{fig:22bmini}
\end{subfigure}\\ \vspace{0.3cm}
\begin{subfigure}[b]{0.45\textwidth}
\includegraphics[width=1.1\textwidth]{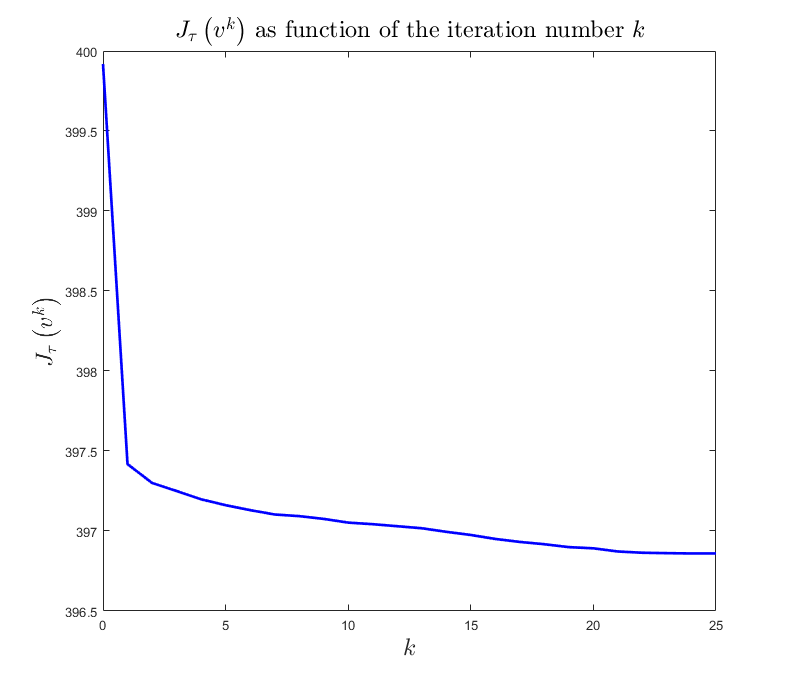}\vspace{-0.4cm}
\caption{Plot of $J_5\left(v^k\right)$} \label{fig:22bJtau}
\end{subfigure}
\begin{subfigure}[b]{0.45\textwidth}
\includegraphics[width=1.1\textwidth]{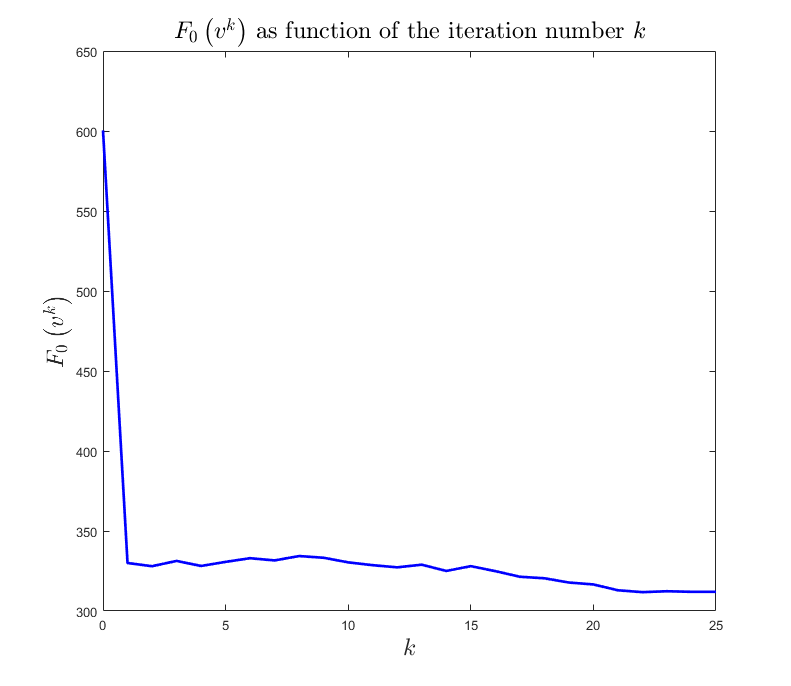}\vspace{-0.4cm}
\caption{Plot of $F_0\left(v^k\right)$} \label{fig:22bOKTV}
\end{subfigure}
\caption{Results from Algorithm~\ref{alg:massOKMBO} on $G_{\text{torus}}(1600)$ with $r=0$, $\gamma=0.2$, $M=800$, and $\tau=5$. The initial condition in Figure~\ref{fig:22binit} was constructed using option (c) in Section~\ref{sec:initialcondition} and was used to obtain the other results displayed here. The value of $F_0$ at the final iterate is approximately $311.99$.}
\label{fig:22b}
\end{center}
\end{figure}

\begin{figure}
\begin{center}
\begin{subfigure}[b]{0.45\textwidth}
\includegraphics[width=1.1\textwidth]{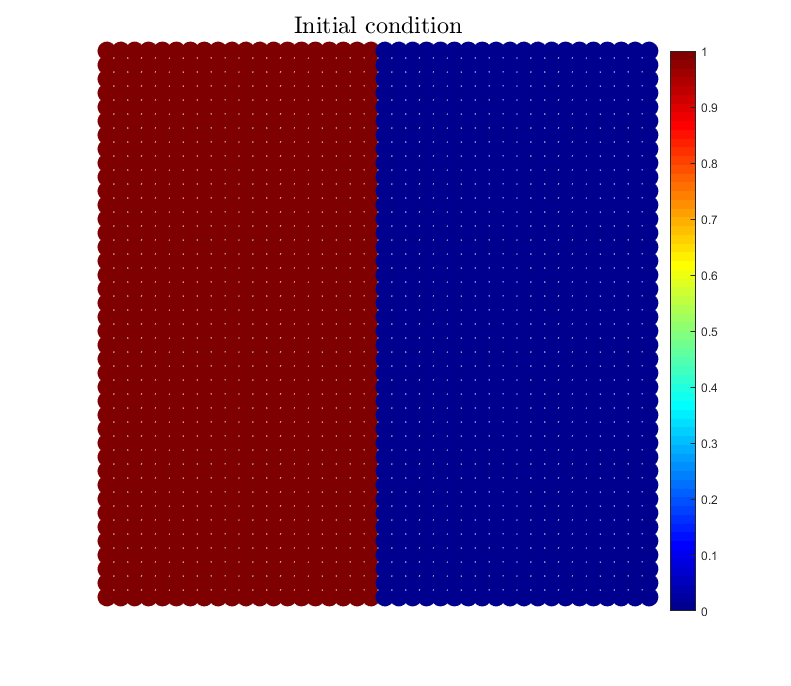}\vspace{-0.7cm}
\caption{Initial condition} \label{fig:23binit}
\end{subfigure}
\begin{subfigure}[b]{0.45\textwidth}
\includegraphics[width=1.1\textwidth]{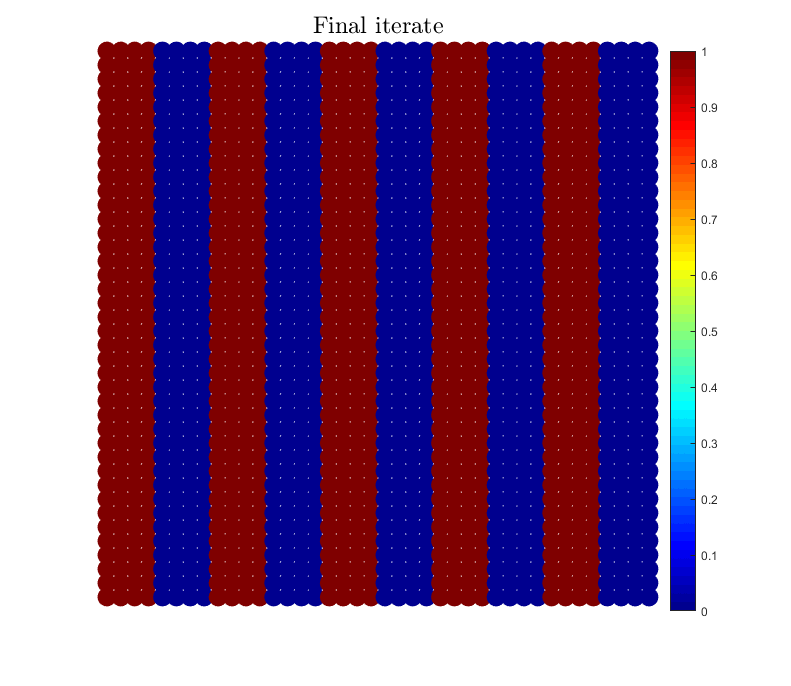}\vspace{-0.7cm}
\caption{Final iterate ($k=3$)} \label{fig:23bmini}
\end{subfigure}\\ \vspace{0.3cm}
\begin{subfigure}[b]{0.45\textwidth}
\includegraphics[width=1.1\textwidth]{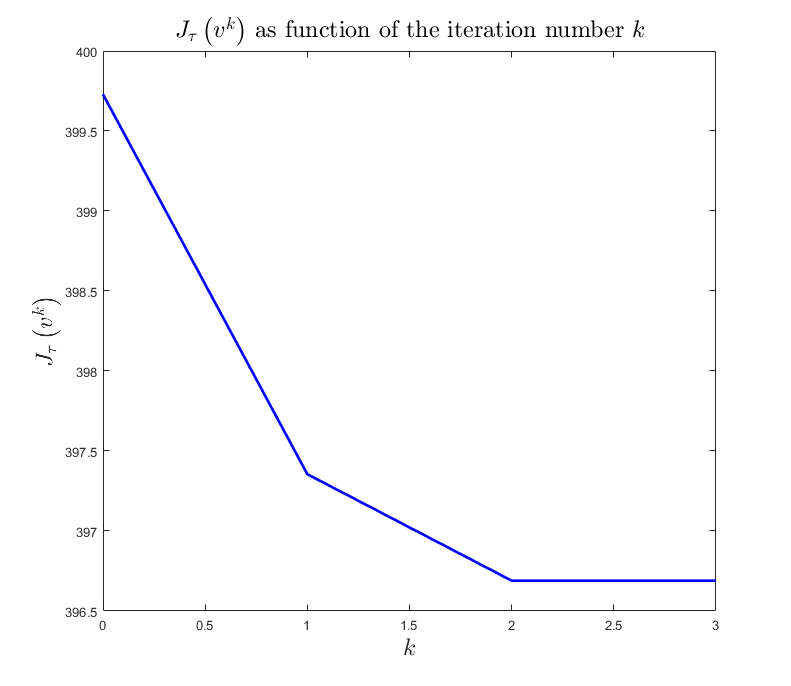}\vspace{-0.4cm}
\caption{Plot of $J_5\left(v^k\right)$} \label{fig:23bJtau}
\end{subfigure}
\begin{subfigure}[b]{0.45\textwidth}
\includegraphics[width=1.1\textwidth]{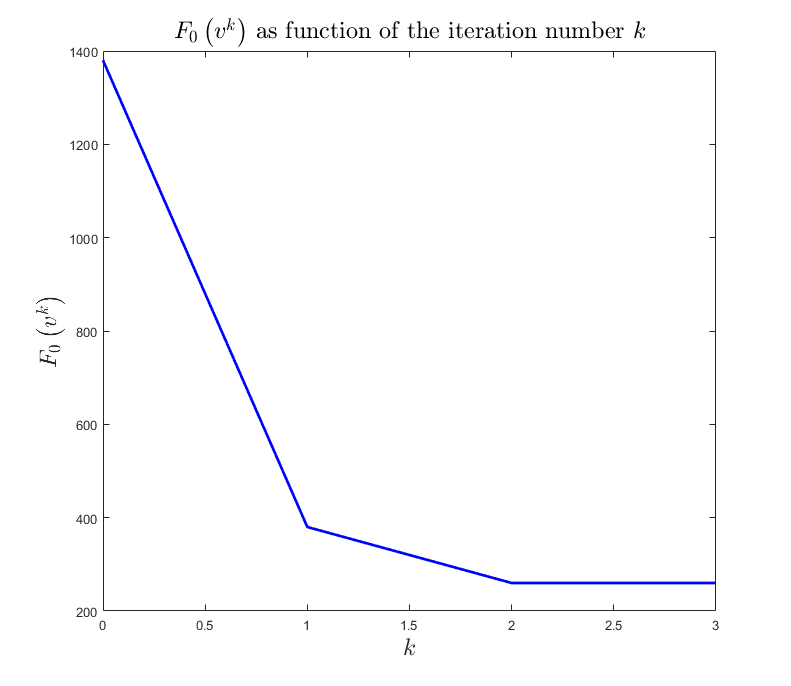}\vspace{-0.4cm}
\caption{Plot of $F_0\left(v^k\right)$} \label{fig:23bOKTV}
\end{subfigure}
\caption{Results from Algorithm~\ref{alg:massOKMBO} on $G_{\text{torus}}(1600)$ with $r=0$, $\gamma=0.2$, $M=800$, and $\tau=5$. The initial condition in Figure~\ref{fig:23binit} was constructed using option (b) in Section~\ref{sec:initialcondition} and was used to obtain the other results displayed here. The value of $F_0$ at the final iterate is  $260$.}
\label{fig:23b}
\end{center}
\end{figure}

\subsection{Other choices in the problem setting and the algorithm}\label{sec:otherchoices}

There are some other choices to make, besides the graph, $\tau$, and the initial condition, before running the \ref{alg:massOKMBO} algorithm, both in the set-up of the original problem \eqref{eq:minimprobsF0} as well as for the algorithm.

The parameter $\gamma$ is a parameter that is part of the original problem setting \eqref{eq:minimprobsF0}. Its value does not only influence the structure of the (approximate) solutions, but also influences what the appropriate choices of $\tau$  and $v^0$ are. As, for given $m\neq 0$, $\gamma\mapsto \Lambda_m$ is an increasing function and $\tau$ always appears in the combination $\tau \Lambda_m$ in the algorithm via \eqref{eq:solutionu}, increasing $\gamma$ decreases the values of $\tau$ at which good results are obtained (all other things being equal). We see an example of this in Figure~\ref{fig:differentgamma}. The choice of $\gamma$ also has an influence on the order of the eigenvalues $\Lambda_m$, as per Remark~\ref{rem:eigorder}, hence the eigenfunction based method for choosing $v^0$ described in Section~\ref{sec:initialcondition} (option (c)) is also influenced by the choice of $\gamma$.

\begin{figure}
\begin{center}
\begin{subfigure}[b]{0.45\textwidth}
\includegraphics[width=1.1\textwidth]{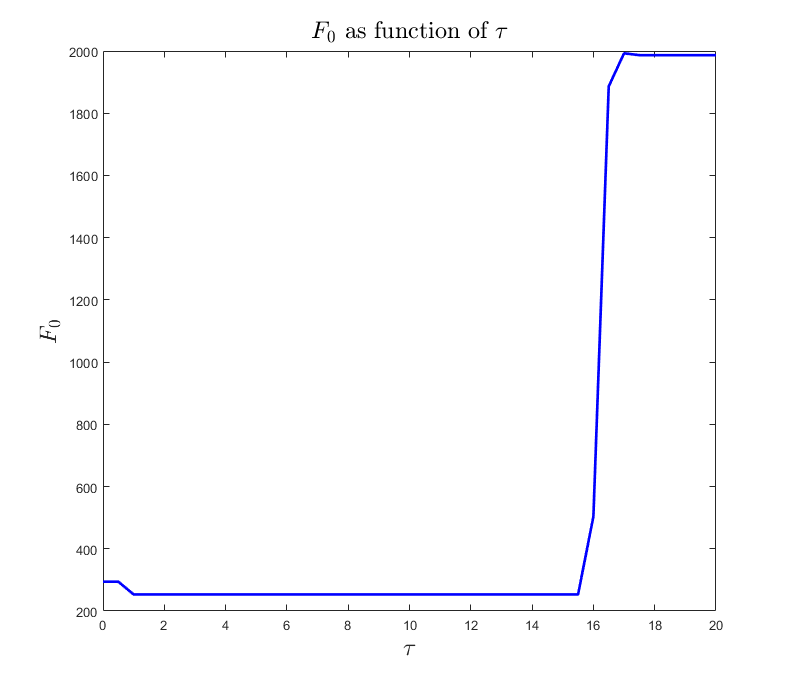}\vspace{-0.4cm}
\caption{$\gamma=1$} \label{fig:differentgammaa}
\end{subfigure}
\hspace{.5cm}
\begin{subfigure}[b]{0.45\textwidth}
\includegraphics[width=1.1\textwidth]{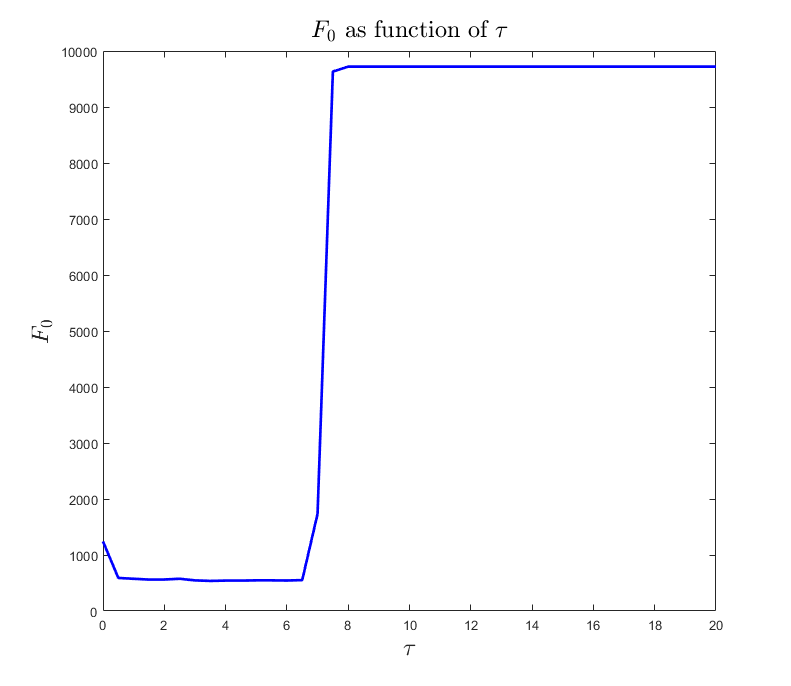}\vspace{-0.4cm}
\caption{$\gamma=5$} \label{fig:differentgammab}
\end{subfigure}
\caption{The value of $F_0\left(v^k\right)$, where $v^k$ is the final iterate of \ref{alg:massOKMBO}, as a function of $\tau$, for two different values of $\gamma$.  In both cases $G_{\text{torus}}(900)$ was used, with $r=0$, $M=450$, and the initial condition from Figure~\ref{fig:initialcondc}. The resolution on the $\tau$ axis (step size) is $0.5$ for both graphs.}
\label{fig:differentgamma}
\end{center}
\end{figure}

\begin{figure}
\begin{center}
\begin{subfigure}[b]{0.45\textwidth}
\includegraphics[width=1.1\textwidth]{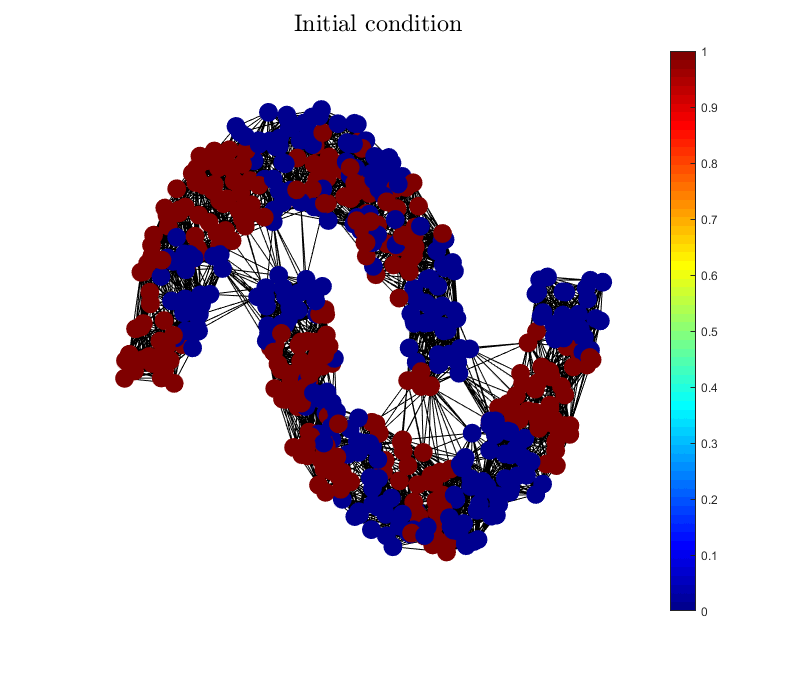}\vspace{-0.7cm}
\caption{Initial condition for $\gamma=0.1$} \label{fig:18binit}
\end{subfigure}
\hspace{.5cm}
\begin{subfigure}[b]{0.45\textwidth}
\includegraphics[width=1.1\textwidth]{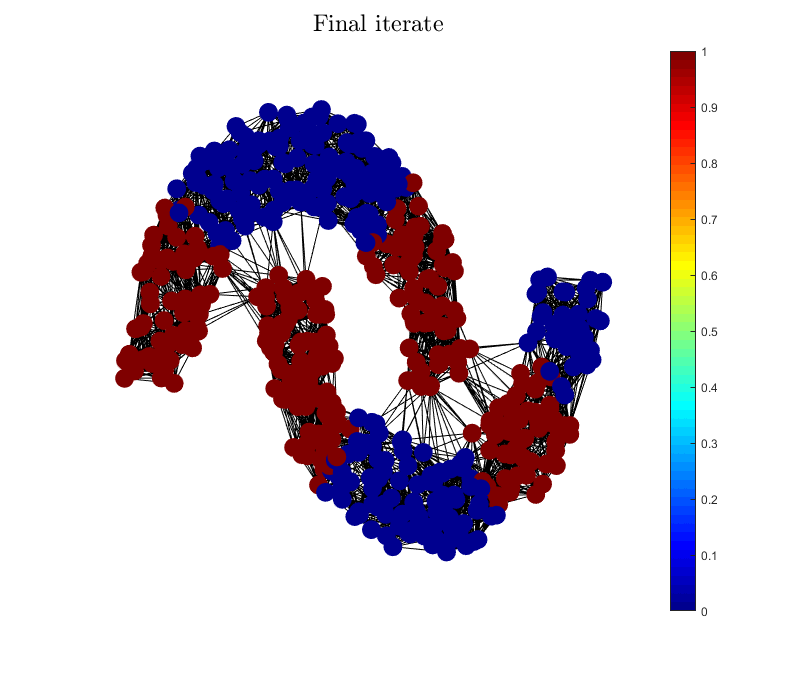}\vspace{-0.7cm}
\caption{Final iterate ($k=21$) for $\gamma=0.1$} \label{fig:18bmini}
\end{subfigure}\\ \vspace{0.3cm}
\begin{subfigure}[b]{0.45\textwidth}
\includegraphics[width=1.1\textwidth]{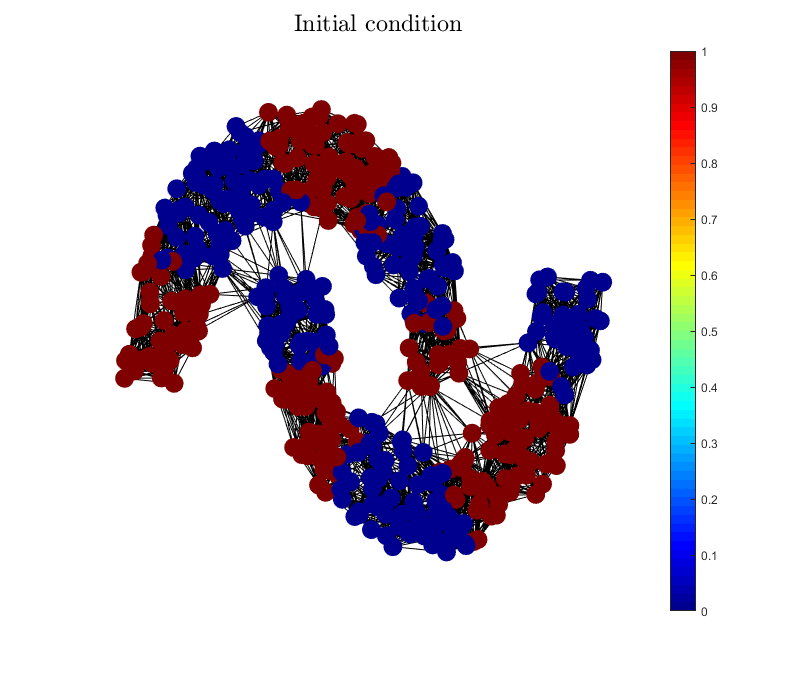}\vspace{-0.7cm}
\caption{Initial condition for $\gamma=1$} \label{fig:19binit}
\end{subfigure}
\hspace{.5cm}
\begin{subfigure}[b]{0.45\textwidth}
\includegraphics[width=1.1\textwidth]{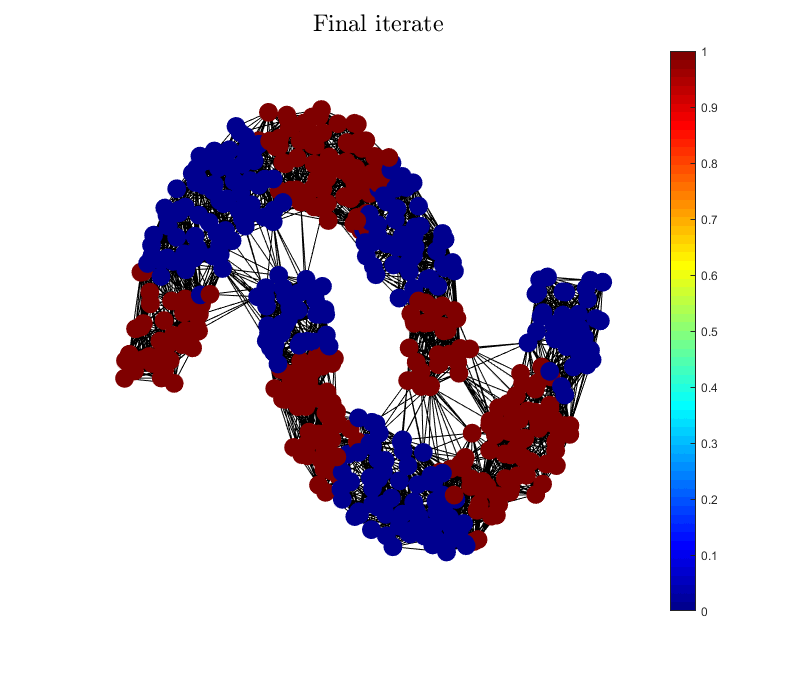}\vspace{-0.7cm}
\caption{Final iterate ($k=9$) for $\gamma=1$} \label{fig:19bmini}
\end{subfigure}\\ \vspace{0.3cm}
\begin{subfigure}[b]{0.45\textwidth}
\includegraphics[width=1.1\textwidth]{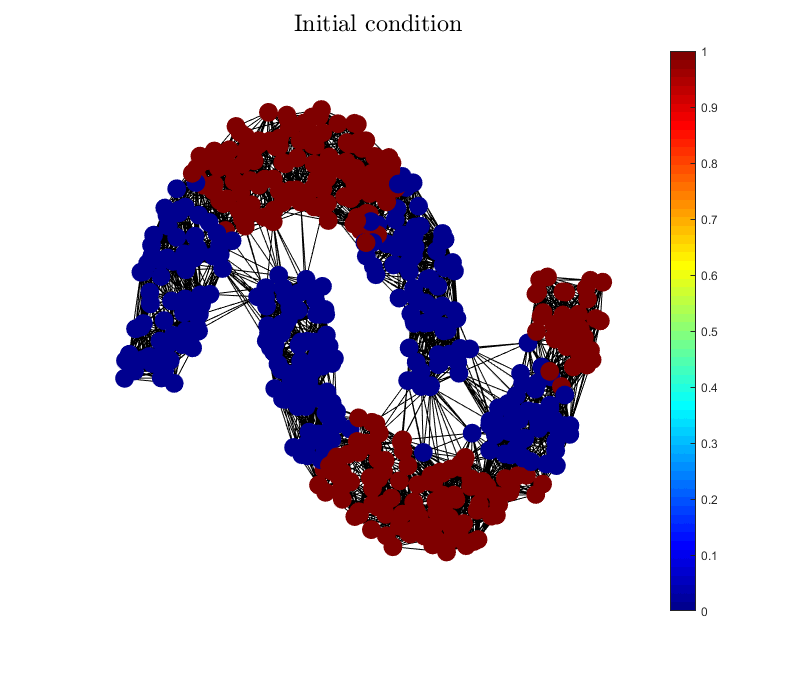}\vspace{-0.7cm}
\caption{Initial condition for $\gamma=10$} \label{fig:20binit}
\end{subfigure}
\hspace{.5cm}
\begin{subfigure}[b]{0.45\textwidth}
\includegraphics[width=1.1\textwidth]{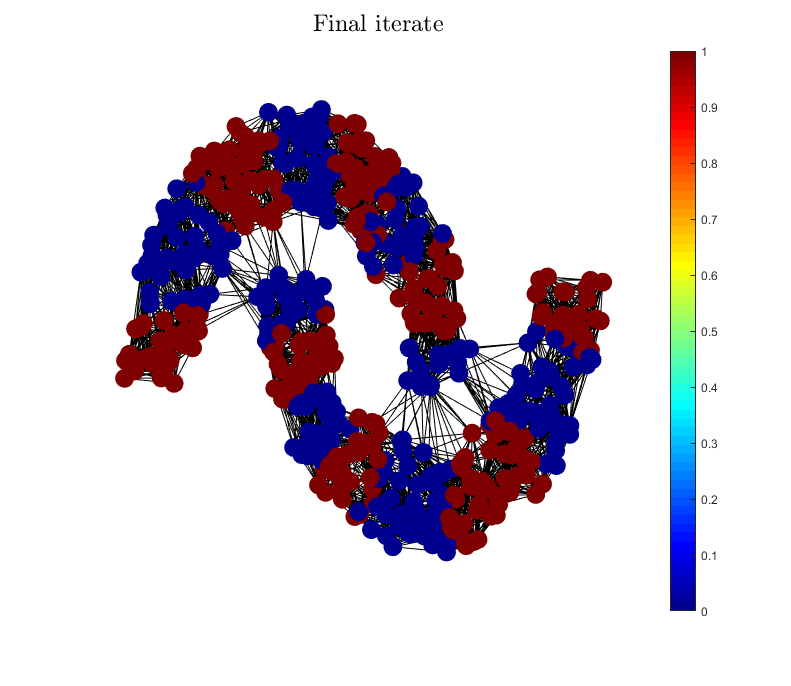}\vspace{-0.7cm}
\caption{Final iterate ($k=7$) for $\gamma=10$} \label{fig:20bmini}
\end{subfigure}
\caption{Initial (left column) and final (right column) states of Algorirthm \ref{alg:massOKMBO} applied to $G_{\text{moons}}$ with $r=0$, $M=300$, $\tau=1$ for a different value of $\gamma$ in each row. The initial conditions are eigenfunction based in the sense of option (c) in Section~\ref{sec:initialcondition}. The values of $F_0$ at the final iterates are approximately $109.48$ (top row), $230.48$ (middle row), and $626.89$ (bottom row).
}\label{fig:18b19b20b}
\end{center}
\end{figure}

\begin{figure}
\begin{center}
\begin{subfigure}[b]{0.45\textwidth}
\includegraphics[width=1.1\textwidth]{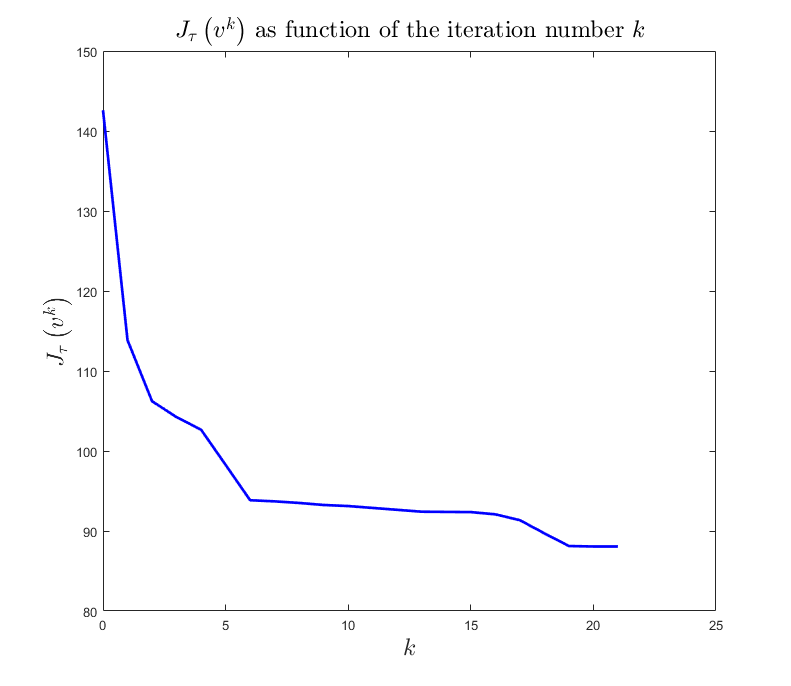}\vspace{-0.4cm}
\caption{Plot of $J_1\left(v^k\right)$ for $\gamma=0.1$} \label{fig:18bJtau}
\end{subfigure}
\hspace{.5cm}
\begin{subfigure}[b]{0.45\textwidth}
\includegraphics[width=1.1\textwidth]{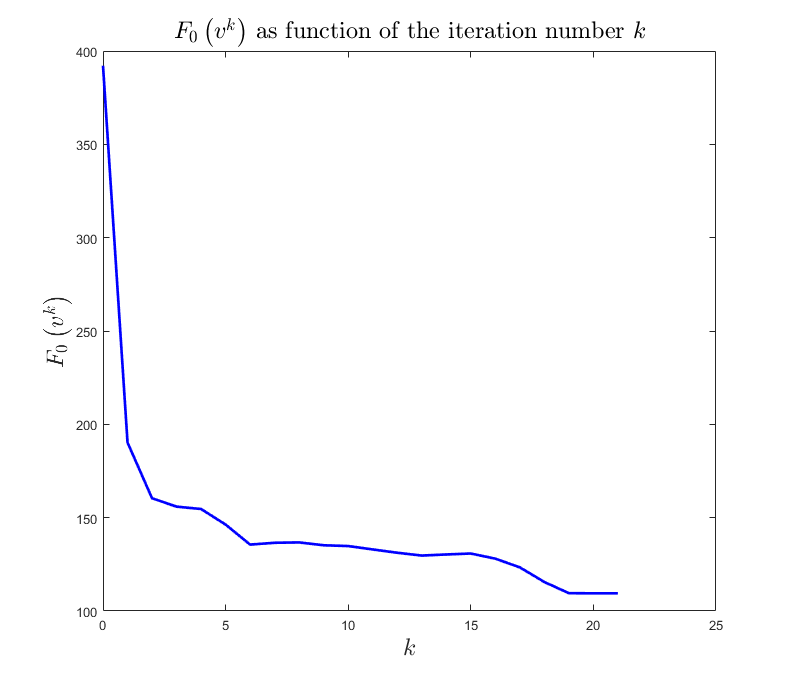}\vspace{-0.4cm}
\caption{Plot of $F_0\left(v^k\right)$ for $\gamma=0.1$} \label{fig:18bOKTV}
\end{subfigure}\\ \vspace{0.3cm}
\begin{subfigure}[b]{0.45\textwidth}
\includegraphics[width=1.1\textwidth]{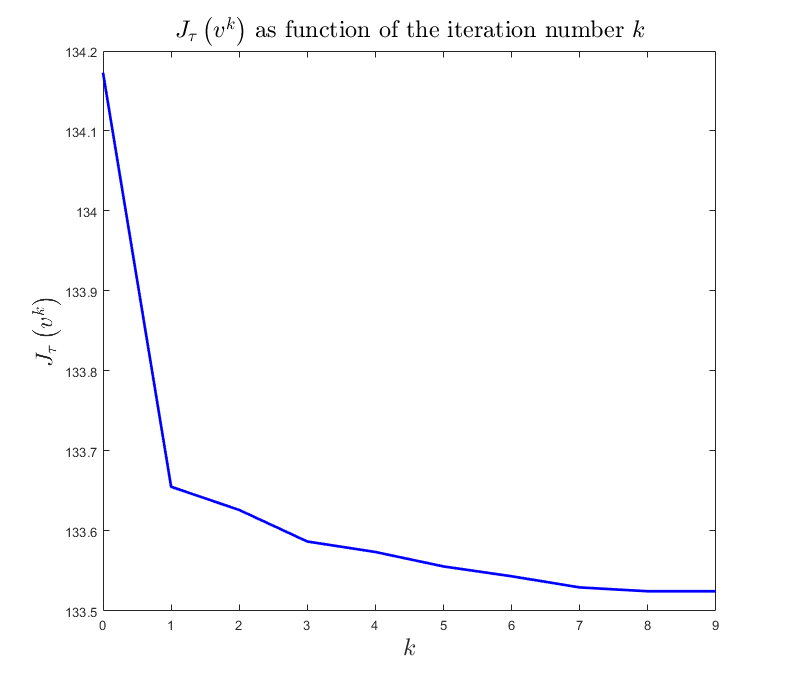}\vspace{-0.4cm}
\caption{Plot of $J_1\left(v^k\right)$ for $\gamma=1$} \label{fig:19bJtau}
\end{subfigure}
\hspace{0.5cm}
\begin{subfigure}[b]{0.45\textwidth}
\includegraphics[width=1.1\textwidth]{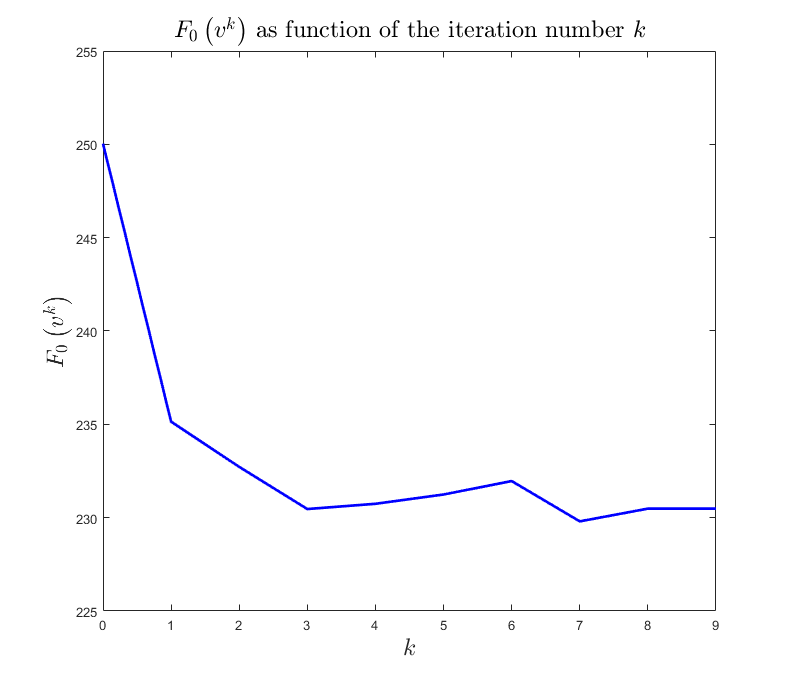}\vspace{-0.4cm}
\caption{Plot of $F_0\left(v^k\right)$ for $\gamma=1$} \label{fig:19bOKTV}
\end{subfigure}\\ \vspace{0.3cm}
\begin{subfigure}[b]{0.45\textwidth}
\includegraphics[width=1.1\textwidth]{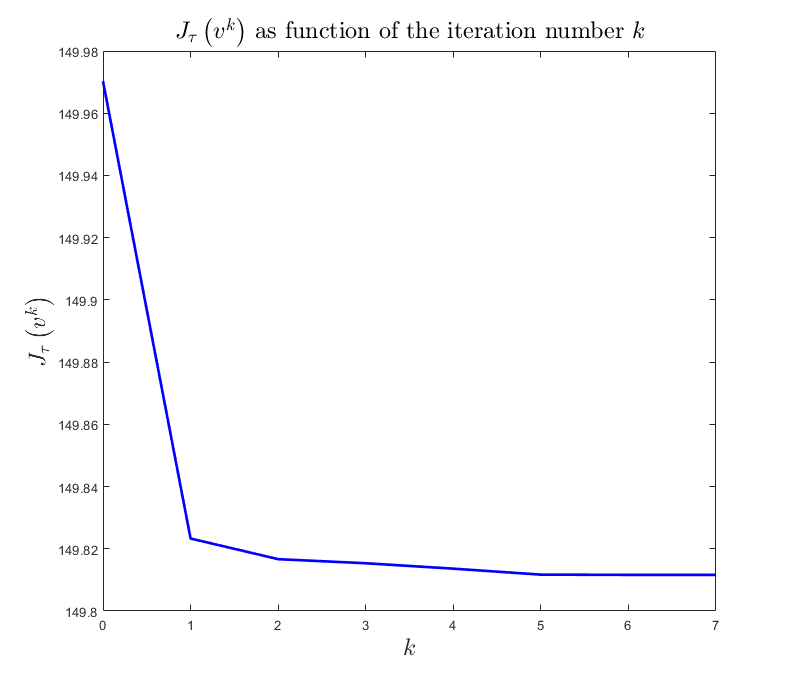}\vspace{-0.4cm}
\caption{Plot of $J_1\left(v^k\right)$ for $\gamma=10$} \label{fig:20bJtau}
\end{subfigure}
\hspace{0.5cm}
\begin{subfigure}[b]{0.45\textwidth}
\includegraphics[width=1.1\textwidth]{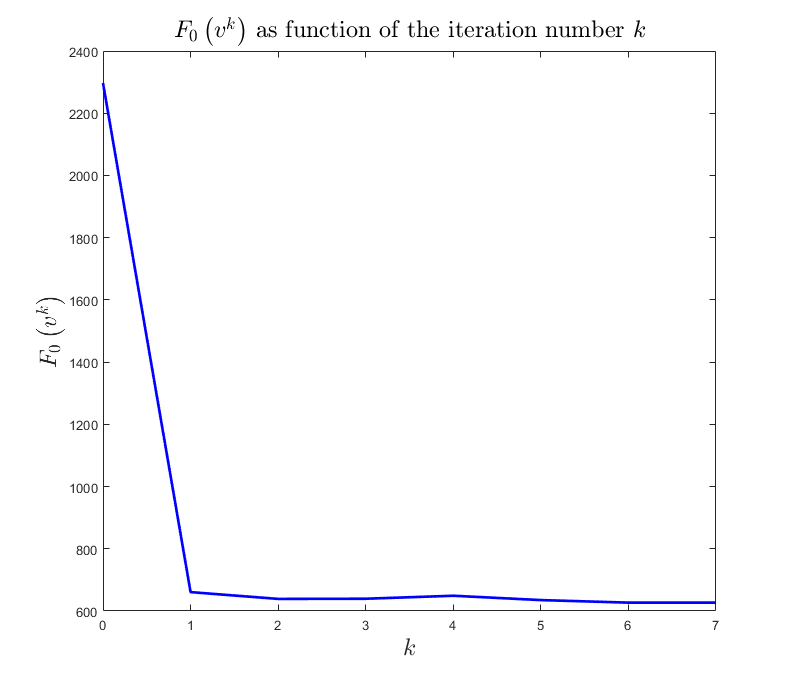}\vspace{-0.4cm}
\caption{Plot of $F_0\left(v^k\right)$ for $\gamma=10$} \label{fig:20bOKTV}
\end{subfigure}
\caption{Plots of $J_1\left(v^k\right)$ (left column) and $F_0\left(v^k\right)$ (right column) for the applications of \ref{alg:massOKMBO} corresponding to Figure~\ref{fig:18b19b20b}.}\label{fig:18b19b20bJtauOKTV}
\end{center}
\end{figure}

The parameters $r$ and $q$, that are part of the original setup of our function spaces $\mathcal{V}$ and $\mathcal{E}$ also play a role. The value of $q$ changes the value of $F_0$. Important results in this paper, such as Corollary~\ref{cor:OKexpressions} and Theorem~\ref{thm:gammaconvergencemass} have all been obtained under the assumption that $q=1$, hence that is also the choice we make when we compute the value of $F_0$ for our experiments. Note however that the choice of $q$ does not influence the actual algorithm \ref{alg:massOKMBO}.

The choice of $r$ does influence the problem setup in \eqref{eq:minimprobsF0} and the algorithm \ref{alg:massOKMBO}. The functional $F_0$ is independent of $r$, but the mass functional $\mathcal{M}$ is not. As noted a few times in this paper already (e.g. in Sections~\ref{sec:setup} and~\ref{sec:OKMBOmass}), when $r\neq 0$ the mass condition can be very restrictive in that the set $\mathcal{V}^b_M$ (or even $\mathcal{V}^{ab}_M$) can be very small. This is especially the case if the graph has a highly irregular degree distribution. Hence all the examples we show are for the case $r=0$. 
The parameter $r$ also influences \ref{alg:massOKMBO} through its effect on $\Delta$.

Finally we mention $N$, the number of iterations in \ref{alg:massOKMBO} (or \ref{alg:OKMBO}). Up until now we have assumed that the algorithm is run for a preset number of iterations, mostly for notational convenience; we know, however, that the algorithm converges in a finite number of steps, in the sense which was made precise in Lemma~\ref{lem:massLyapunov} (or Corollary~\ref{cor:finiteconvergence}). It thus makes sense to add a stopping criterion to the algorithm. In our experiments we set $N=500$ and add a stopping criterion which ends the algorithm's run if the Euclidean norm of the difference between (the vector representations of) $v^{k-1}$ and $v^k$ is less than $10^{-24}$. This tolerance in practice means that the algorithm stops before it has run through 500 iterations if and only if $v^k=v^{k-1}$. In fact, in our examples the algorithm runs for at most a few dozen iterations before the stopping criterion kicks in and never gets to the (arbitrarily chosen) maximum of 500 iterations. Note that as a consequence, in all our examples the states obtained in the final two iterations are the same. For example, in the left hand column of Figure~\ref{fig:12b11b} the final value of $k$ is $3$. Hence $v^3=v^2$ and in that case the algorithm only took two iterates to arrive at its final state. In the right hand side of that same figure the algorithm took twelve iterates to arrive at the final state $v^{12}=v^{13}$.

\subsection{Spurious patterns}\label{sec:spurious}

Because mass is conserved in \ref{alg:massOKMBO} and the iterates of the algorithm are forced to be in $\mathcal{V}_M^{ab}$, patterns are guaranteed to appear, in the sense that mass will be allocated to some nodes and not to others, giving the appearance of a pattern. We used MATLAB's \texttt{sort} function to produce the relabelling $R_u$ in the mass conserving threshold step of \ref{alg:massOKMBO}. This function will produce an output, even if $u$ has the same value on two or more nodes. This means that our choice of sorting method, effectively hides the non-uniqueness that is inherent in the choice of $R_u$ when $u$ takes the same value at different nodes. This is desirable when the non-uniqueness involves the relabelling of a small number of nodes only, since some choice of $R_u$ has to be made to continue the algorithm and the influence of that choice on the final state (and value of $F_0$) is presumably small in that case. However, when $u$ has the same value on many different nodes (within machine precision), for example when $\tau$ in the ODE step has been chosen too large, the resulting non-uniqueness in the choice of $R_u$ is very large (e.g. for constant $u$ all relabelling functions $R_u$ are equally admissible). Hence the resulting output of the mass conserving threshold step is (close to) arbitrary, yet it will still produce a pattern when visualized. Thus it is important to have a way to identify if this has occurred or if the resulting pattern is indeed meaningful in the context of the $F_0$ minimization problem of \eqref{eq:minimprobsF0}. 

One could inspect the function $u$ before the mass conserving threshold step and discard the result if $u$ is (too close to being) constant. The problem with this approach is that it is not a priori clear what ``too close to'' means. In our experiments sometimes the variation in node values of $u$ (as measured by the standard deviation, computed with MATLAB's \texttt{std} function) is on the order of $10^{-12}$ (or less) and yet still meaningful in the sense explained below.

Luckily we have an arbiter of meaning in this case. After all, our goal is to minimize $F_0$, hence as long as $F_0$ decreases along the iterates of \ref{alg:massOKMBO} the algorithm (and thus also the mass conserving threshold step) is performing a meaningful operation. A decrease in the values of the functional $J_\tau$ can also be used to justify confidence in the output of the algorithm. We include plots of the values of $F_0$ and $F_\tau$ as function of the iteration number $k$ with our results in this paper to validate the algorithm's ouput.

\section{Discussion and future work}\label{sec:discussthefuture}

In this paper we present three main results: the Lyapunov functionals associated with the (mass conserving) Ohta-Kawasaki MBO schemes $\Gamma$-converge to the sharp interface Ohta-Kawasaki functional; there exists a class of graphs on which this MBO scheme can be interpreted as a standard graph MBO scheme on a transformed graph (and for which additional comparison principles hold); the mass conserving Ohta-Kawasaki MBO scheme works well in practice when attempting to minimize the sharp interface graph Ohta-Kawasaki functional under a mass constraint. Along the way we have also further developed the theory of PDE inspired graph problems  and added to the theoretical underpinnings of this field.

Future research on the graph Ohta-Kawasaki functional can mirror the research on the continuum Ohta-Kawasaki functional and attempt to prove the existence of certain structures in minimizers on certain graphs, analogous to structures such as lamellae and droplets in the continuum case. The numerical methods presented in this paper might also prove useful for simulations of minimizers of the continuum functional.

The $\Gamma$-convergence results presented in this paper also fit in well with the ongoing programme, started in \cite{vanGennipGuillenOstingBertozzi14}, aimed at improving our understanding how various PDE inspired graph based processes, such as the graph MBO scheme, graph Allen-Cahn equation, and graph mean curvature flow, are connected.

\appendix

\section{The continuum Ohta-Kawasaki model}\label{sec:continuumOK}

In this section we give a brief introduction to the continuum Ohta-Kawasaki model which was introduced into the physics/chemistry literature to describe diblock copolymer melts. This model has been studied intensively in recent decades and this section is not aiming to be exhaustive or even extensive.

The continuum Ohta-Kawasaki functional \cite{OhtaKawasaki86,KawasakiOhtaKohrogui88} has a diffuse interface form, $F_\e: H^{-1}(\Omega; \R) \to \R$,
\[
\mathcal{F}_\e(u) := \frac12 \int_{\Omega} |\nabla u|^2 + \frac1\e \int_{\Omega} W(u) + \frac\gamma2 \left\|u-\frac1{|\Omega|} \int_\Omega u\right\|_{H^{-1}(\Omega)}^2,
\]
and a sharp interface form, $F_0: BV(\Omega; \{-1, 1\}) \to \R$, 
\[
\mathcal{F}_\e(u) := \sigma  \int_{\Omega} |\nabla u| + \frac\gamma2 \left\|u-\frac1{|\Omega|} \int_\Omega u\right\|_{H^{-1}(\Omega)}^2.
\]
Here $\Omega\subset \R^n$ is an open, bounded set, $\e$, $\gamma$, and $\sigma$ are positive parameters, and $W$ denotes a nonnegative double well potential with equal depth wells, for example the double well potential in \eqref{eq:doublewell} with wells at $x\in\{0,1\}$. The total variation \cite{Giusti84} is defined as
\[
\int_\Omega |\nabla u| := \sup\left\{\int_\Omega u, \dvg v: \, v\in C^{\infty}_0(\Omega; \R^n), \, \forall x \in \Omega\, |v(x)|\leq q\right\},
\]
and the negative Sobolev $H^{-1}$ norm as
\[
\left\|u-\frac1{|\Omega|} \int_\Omega u\right\|_{H^{-1}(\Omega)}^2 := \int_\Omega |\nabla \varphi|^2,
\]
where $\varphi \in H^1(\Omega)$ solves
\[
\Delta \varphi = u-\frac1{|\Omega|} \int_\Omega u
\]
with appropriate boundary conditions (which can vary depending on the context). The diffuse interface functional $\mathcal{F}_\e$ is an approximation of the sharp interface functional $\mathcal{F}_0$ in the sense of $\Gamma$-convergence: any sequence of functionals $\mathcal{F}_\e$ $\Gamma$-converges to $\mathcal{F}_0$ when $\e\to 0$ in the $L^1(\Omega)$ topology \cite{ModicaMortola77,Modica87,Modica87b}. Note in particular that $\mathcal{F}_0$ is defined on binary functions that take values $\pm 1$ only. For such functions $\frac12 \int_\Omega |\nabla u|$ computes the length of the (reduced) boundary \cite[Definition 3.54]{AmbrosioFuscoPallara00} between the set where $u=-1$ and the set where $u=1$. In this limit, the surface tension parameter $\sigma>0$ is determined by the specific choice of $W$, but its precise value is not of importance here.

When $\mathcal{F}_0$ (or $\mathcal{F}_\e$) is minimized under a mass constraint $\frac1{|\Omega|}\int_\Omega u = M$, the boundary minimizing effect of the total variation term competes with the mixing preference of the $H^{-1}$ norm, which leads to pattern formation on a scale determined by the parameter $\gamma$, which controls the relative influence of both terms.  The mass parameter $M$ has large impact on the type of patterns that appear. When $M$ is close to $-1$ or close to $1$, such that one phase is much more prevalent than the other, small droplets of the minority phase will form in a background formed by the majority phase; when $M\approx 0$ a lamellar phase forms; see for example \cite[Figure 3]{BatesFredrickson99} for a simplified theoretical sketch of some of the expected patterns in a physical diblock copolymer system. A goal in the mathematical literature has been to proof the existence of various patterns that appear as minimizers and study their stability, see for example \cite{RenWei00,RenWei02,RenWei03a,RenWei03b,ren2014double,ren2017spectrum}. Extensions of the model, for example including a third phase either through triblock copolymers or through adding a homopolymer, have also been considered, for example in \cite{UneyamaDoi05,vanGennip08,vanGennipPeletier08,vanGennipPeletier09}.

\section{Random walk interpretation of the Green's function}\label{sec:randomwalk} 

For more information about the general concepts discussed in this section, see e.g. \cite{Chung97,2000math......1057D,Sigman09}.

Consider a discrete time random walk on the graph $G=(V,E,\omega)\in \mathcal{G}$, with transition probabilities, for all $i,j\in V$,
$
p_{ij} := d_i^{-1} \omega_{ij},
$ 
i.e. the probability of moving from vertex $i$ to vertex $j$ in one time step is $d_i^{-1} \omega_{ij}$. Note that $\sum_{j\in V} p_{ij} = 1$.

For all $i\in V$, let $T_i$ be the earliest time step at which the random walk is at node $i$. By convention, let the random walk start at time 0. Now (remembering that $|V| \geq 2$) fix two different vertices $a, b \in V$ and define $h\in \mathcal{V}$, by, for all $i\in V$,
$
h_i := P[T_a < T_b | T_i = 0],
$ 
i.e. $h_i$ is the probability the random walk starting from node $i$ reaches $a$ before it reaches $b$. Clearly
$
h_a = 1 \quad \text{and} \quad h_b = 0.
$ 
Moreover, since the walk at each time step is independent, we have, for $i\in V\setminus \{a, b\}$ (and for any $r\in [0,1]$),
$
h_i = \sum_{j\in V} p_{ij} h_j$ or, equivalently, $(\Delta h)_i = 0.
$ 
If, for all $i\in V$, $\tilde h_i := P[T_b < T_a | T_i = 0]$ (note that the roles of $a$ and $b$ are exchanged, compared to $h$) and $h^+ := h + \tilde h$, then $\Delta h^+ = 0$ on $V\setminus\{a,b\}$ and $h^+_a=h^+_b=1$, hence $h^+_i=1$ for all $i\in V$, as expected (the probability that the random walk either reaches $a$ before $b$, or $b$ before $a$, is $1$).

Conversely to the computation for $h$ above, if $v\in \mathcal{V}$ solves
\[
\begin{cases}
(\Delta v)_i = 0, & \text{if } i\in V\setminus\{a, b\},\\
v_a = c, & \text{for some } c\in \R,\\
v_b = 0,
\end{cases}
\]
then, for each $i\in V$, $v_i$ is the expected payoff value in a game consisting of a random walk starting at $i\in V$, with payoff equal to $c$ if the walk reaches $a$ before $b$ and zero otherwise. (Since the same equation is satisfied by the voltage function on an electric network with voltage $c$ applied to node $a$ and voltage $0$ to node $b$, such a $v$ can also be interpreted as voltage on an electric network \cite{2000math......1057D}.)

Consider now the Green's function for the Poisson equation satisfying \eqref{eq:GreensproblemPoisson} with \eqref{eq:qchoice} and \eqref{eq:Cchoice}, for $j=a \in V$ and $k=b \in V$, then, 
\[
\begin{cases}
(\Delta G^a)_i = 0, & \text{if } i\in V\setminus\{a, b\},\\
G^a_a = \frac1{\vol{V}} \left(\nu_a^{V\setminus{\{b\}}} + \nu_b^{V\setminus{\{a\}}}\right),\\
G^a_b = 0,
\end{cases}
\]
where we used \eqref{eq:GreenPoisson} for the second line. Hence $G^a_i$ is the expected payoff of the game described above with the walk starting at $i\in V$ and $c = c_{ab} := \frac1{\vol{V}} \left(\nu_a^{V\setminus{\{b\}}} + \nu_b^{V\setminus{\{a\}}}\right) > 0$. Positivity of $c_{ab}$ follows from positivity of the equilibrium measures $\nu_a^{V\setminus{\{b\}}}$ and $\nu_b^{V\setminus{\{a\}}}$ (Definition~\ref{def:equilibrium}). Note that $c_{ab}$ implicitly depends on $r$.

\section{$\Gamma$-convergence of $F_\e$}\label{sec:gammaconvergence}

In this section we prove that the diffuse interface graph Ohta-Kawasaki functionals from \eqref{eq:epsOK} converge to the limit functional $F_0$ \eqref{eq:limitOK} in the sense of $\Gamma$-convergence. The upper bound and lower bound properties of Theorem~\ref{thm:gammaconvergencemass} (or Theorem~\ref{thm:gammaconvergencemass}) are the two defining conditions of $\Gamma$-convergence. We refer the reader to \cite{DalMaso93,Braides02} for a detailed definition and important properties of $\Gamma$-convergence.

Note that in the results below, we do not specify the topology under which the convergence of sequences in $\mathcal{V}$ are considered. We can use $\|\cdot\|_{\mathcal{V}}$, but any other norm based topology will be equivalent in this finite dimensional setting.

We remind the reader that $\overline\R$ denotes the extended real line $\R\cup \{-\infty+\infty\}$.

\begin{lemma}\label{lem:gammaconv}
Let $G=(V,E,\omega)\in\mathcal{G}$ and let $\{\e_k\}_{k\in\N}$ be a sequence such that, for all $k\in \N$, $\e_k>0$ and $\e_k\to 0$ as $k\to\infty$. For each $k\in \N$, define $F_{\e_k}: \mathcal{V} \to \overline\R$ by the expression in \eqref{eq:epsOK} and let $\hat F_0: \mathcal{V}\to \overline\R$ be defined as
\[
\hat F_0(u) := \begin{cases}
F_0(u), &\text{if } u\in \mathcal{V}^b,\\
+\infty, &\text{otherwise},
\end{cases}
\]
where $F_0$ is as in \eqref{eq:limitOK}.
Then $\{F_{\e_k}\}_{k\in \N}$ $\Gamma$-converges to $\hat F_0$ as $k\to \infty$.

Moreover, if $\{u_k\}_{k\in \N} \subset \mathcal{V}$ and there exists a $C>0$ such that, for all $k\in \N$, $F_{\e_k}(u_k) < C$, then there is a subsequence $\{u_{k_l}\}_{l\in \N} \subset \{u_k\}_{k\in \N}$ and a $u \in \mathcal{V}^b$ such that $u_{n_k} \to  u$ as $k\to \infty$.
\end{lemma}
\begin{proof}
A proof of the $\Gamma$-convergence of the terms $\frac12 \|\nabla u\|_{\mathcal{E}}^2 + \frac1{\e_n} \sum_{i\in V} W(u_i)$ in $F_{\e_n}$ is given  in \cite[Section 3.1]{vanGennipBertozzi12}. Since $\Gamma$-convergence is stable under continuous perturbations \cite[Proposition 6.21]{DalMaso93} and both the map $u\mapsto u-\mathcal{A}(u)$ and the $H^{-1}$ norm are continuous, the $\Gamma$-convergence statement follows. The compactness result in the second part of the lemma's statement follows directly from \cite[Section 3.1]{vanGennipBertozzi12}.
\end{proof}

\begin{lemma}[$\Gamma$-convergence with a mass constraint]
Let $G=(V,E,\omega)\in\mathcal{G}$ and let $\{\e_k\}_{k\in\N}$ be a sequence such that, for all $k\in \N$, $\e_k>0$ and $\e_k\to 0$ as $k\to\infty$. Let $M\in \mathfrak{M}$, where $\mathfrak{M}$ is the set of admissible masses as in \eqref{eq:admissmass}. For each $k\in \N$, define $\breve F_{\e_k}: \mathcal{V}_M \to \overline\R$ by $\breve F_{\e_k} := \left.F_{\e_k}\right|_{\mathcal{V}_M}$, where $F_{\e_k}$ is as in \eqref{eq:epsOK}, and let $\breve F_0: \mathcal{V}_M \to \overline\R$ be defined as $\breve F_0:= \left.\hat F_0\right|_{\mathcal{V}_M}$, where $\hat F_0$ is as in Lemma~\ref{lem:gammaconv}. Then $\{\breve F_{\e_k}\}_{k\in \N}$ $\Gamma$-converges to $\breve F_0$ as $k\to \infty$.

Moreover, if $\{u_k\}_{k\in \N} \subset \mathcal{V}$ and there exists a $C>0$ such that, for all $k\in \N$, $\breve F_{\e_k}(u_k) < C$, then there is a subsequence $\{u_{k_l}\}_{l\in \N} \subset \{u_k\}_{k\in \N}$ and a $u \in \mathcal{V}_0$ such that $u_{k_l} \to  u$ as $l\to \infty$.
\end{lemma}
\begin{proof}
The only difference between this result and that of Lemma~\ref{lem:gammaconv}, is that now the definitions of $\breve F_{\e_n}$ and $\breve F_0$ incorporate a mass constraint in their domains. Analogously to the argument in \cite[Section 3.2]{vanGennipBertozzi12}, we see that by continuity of $u\mapsto \mathcal{M}(u)$, the proof of the lower bound in the $\Gamma$-convergence proof and the proof of the compactness result remain unchanged from the case of Lemma~\ref{lem:gammaconv}. For the proof of the upper bound, we note, as in \cite[Section 3.2]{vanGennipBertozzi12}, that the recovery sequence used in this proof will satisfy the same mass constraint as its limit.
\end{proof}

\section{Direct computation of \eqref{eq:starPhidirect}}\label{sec:starPhidirect}

In this section we compute \eqref{eq:starPhidirect} using the eigenfunctions and eigenvalues as given in Lemma~\ref{lem:starspectrum}. For $j\in V$,
\[
\varphi^j_i = \varphi^1_j = 0 + \frac1n \left[(n-1)n\right]^{-1} (n-1) \left((n-1) \delta_{1j} - (1-\delta{j1})\right) = \begin{cases}
\frac{n-1}n, &\text{if } j=1,\\
-\frac1{n^2}, &\text{if } j\neq 1.
\end{cases}
\]
Next we assume $i\neq 1\neq j$. Let $I(k) := \{i\in \N: k\leq i\leq n\}$, for $k\in \N$. If $i=j$, we find
\begin{align*}
\varphi^i_i &= \sum_{m=1}^{n-2} \left[(n-m-1)^2 + (n-m-1)\right]^{-1} \left((n-m-1) \delta_{i,m+1} - \left(\chi_{I(m+2)}\right)_i\right)^2\\ &\hspace{1cm}+ \frac1n \left[(n-1)n\right]^{-1}\\
&= \begin{cases}
\frac{n-2}{n-1} + \frac1{n^2(n-1)}, &\text{if } i=2,\\
\frac{n-i}{n-i+1} + \sum_{m=1}^{i-2} \frac1{(n-m-1)(n-m)} + \frac1{n^2(n-1)}, &\text{if } i\in \{3, \ldots, n-1\},\\
\sum_{m=1}^{n-2} \frac1{(n-m-1)(n-m)} + \frac1{n^2(n-1)}, &\text{if } i=n,\\
\end{cases}\\
&= \frac{n^2-n-1}{n^2}.
\end{align*}
The final equality above is not immediately obvious and follows from the fact that we have
\begin{align}
\sum_{m=i-1}^{n-2} \frac1{(n-m-1)(n-m)} &= \frac{n-i}{n-i+1}, \quad i\in \{3, \ldots, n-1\},\label{eq:sum1}\\
\quad \sum_{m=1}^{n-2} \frac1{(n-m-1)(n-m)} &= \frac{n-2}{n-1}.\label{eq:sum2}
\end{align}
In Lemma~\ref{lem:proofsums} below we give a proof of the identities in \eqref{eq:sum1}, \eqref{eq:sum2}.

Finally, we consider the case $i\neq 1 \neq j \neq i$. Without loss of generality (because of symmetry under exchange of $i$ and $j$) we assume that $i\leq j-1$. Then
\begin{align*}
\varphi^j_i &= \sum_{m=1}^{n-2} \frac{\left( (n-m-1) \delta_{i,m+1} - \left(\chi_{I(m+2)}\right)_i\right) \left((n-m-1) \delta_{j,m+1} - \left(\chi_{I(m+2)}\right)_j\right)}{(n-m-1)^2 + (n-m-1)} + \frac1{n^2(n-1)}\\
&= \begin{cases}
-\frac1{n-1} + \frac1{n^2(n-1)}, &\text{if } i=2,\\
\sum_{m-1}^{i-2} \frac1{(n-m-1)(n-m)} - \frac1{n-(i-1)} + \frac1{n^2(n-1)}, &\text{if } i\geq 3,
\end{cases}\\
&= -\frac{n+1}{n^2},
\end{align*}
where for the last equality we have used that, for $i\geq 3$,
\begin{equation}\label{eq:subtractsums}
\sum_{m=1}^{i-2} \frac1{(n-m-1)(n-m)} = \frac{i-2}{(n-i+1)(n-1)},
\end{equation}
which is proven by subtracting \eqref{eq:sum1} from \eqref{eq:sum2}.

\begin{lemma}\label{lem:proofsums}
For $N\in \N$, $N\geq 1$, we have
\begin{equation}\label{eq:sum3}
\sum_{l=0}^{N-1} \frac1{(N-l)(N-l+1)} = \frac{N}{N+1}.
\end{equation}
In particular, the identities in \eqref{eq:sum1} and \eqref{eq:sum2} hold.
\end{lemma}
\begin{proof}
First we prove \eqref{eq:sum3} by induction. If $N=1$ we immediately find that both the left and right hand side of \eqref{eq:sum3} are equal to $\frac12$. Now assume \eqref{eq:sum3} is true for $N=k\in \N$, $k\geq 1$, then, for $N=k+1$ we compute
\begin{align*}
\sum_{l=0}^k \frac1{(k-l+1)(k-l+2)} &= \sum_{\tilde l = -1}^{k-1} \frac1{(k-\tilde l)(k-\tilde l+1)} = \sum_{\tilde l = 0}^{k-1} \frac1{(k-\tilde l)(k-\tilde l+1)} + \frac1{(k+1)(k+2)}\\
&= \frac{k}{k+1} + \frac1{(k+1)(k+2)} = \frac{k+1}{k+2}.
\end{align*}
This proves \eqref{eq:sum3}. Setting $m=l+i-1$ and $N=n-i$ proves \eqref{eq:sum1}. Setting $m=l+1$ and $N=n-2$ proves \eqref{eq:sum2}.
\end{proof}

The following corollary is used to prove \eqref{eq:moresums}.
\begin{corol}\label{cor:moresums}
Let $N,q\in \N$ such that $1\leq q+1\leq N$. Then
\[
\sum_{l=2}^{q+1} \frac1{(N-l)(N-l+1)} = \frac{q}{(N-1)(N-q-1)}.
\]
\end{corol}
\begin{proof}
Using \eqref{eq:subtractsums} we find
\[
\sum_{l=2}^{q+1} \frac1{(N-l)(N-l+1)} = \sum_{l=1}^q \frac1{(N-l)(N-l-1)} = \frac{q}{(N-q-1)(N-1)}.
\]
\end{proof}

\section{Gradient flows of $F_\e$}\label{sec:gradflows}

Let $u,v\in \mathcal{V}$ and $s\in \R$. Let $\varphi, \psi \in \mathcal{V}$ satisfy
$\Delta \varphi = u-\mathcal{A}(u)$ and $\Delta \psi = v-\mathcal{A}(v)$, respectively,
then $\Delta (\varphi+s\psi) = u+sv - \mathcal{A}(u+sv)$. Hence, using \eqref{eq:adjoint}, we find
\begin{align*}
\left.\frac{d}{ds} \|u+sv-\mathcal{A}(u+sv)\|_{H^{-1}}^2 \right|_{s=0} 
&= \left.\frac{d}{ds} \langle u-\mathcal{A}(u)+sv-s\mathcal{A}(v), \varphi+s\psi\rangle_{\mathcal{V}}  \right|_{s=0}\\
&= \langle u-\mathcal{A}(u), \psi\rangle_{\mathcal{V}} + \langle v-\mathcal{A}(v), \varphi\rangle_{\mathcal{V}}\\
&= \langle \Delta \varphi, \psi\rangle_{\mathcal{V}} + \langle v-\mathcal{A}(v), \varphi\rangle_{\mathcal{V}} = 2 \langle v-\mathcal{A}(v),\varphi\rangle_{\mathcal{V}}.
\end{align*}
We note that
\[
\langle \mathcal{A}(v),\varphi\rangle_{\mathcal{V}} = \frac1{\vol{V}} \sum_{i,j\in V} d_i^r v_i d_j^r \varphi_j = \left\langle v, \left(\frac1{\vol{V}} \sum_{j\in V} d_j^r \varphi_j\right) \chi_V\right\rangle_{\mathcal{V}} = \langle v, \mathcal{A}(\varphi)\rangle_{\mathcal{V}},
\]
hence
\[
\left.\frac{d}{ds} \|u+sv-\mathcal{A}(u+sv)\|_{H^{-1}}^2 \right|_{s=0} = 2 \langle v, \varphi - \mathcal{A}(\varphi)\rangle_{\mathcal{V}}.
\]
Using the gradient of the first terms in $F_\e$ as computed in \cite[Section 5]{vanGennipGuillenOstingBertozzi14}, we deduce that
\[
\left.\frac{d}{ds} F_\e(u+sv) \right|_{s=0} = \big\langle \Delta u + \frac1\e d^{-r} W'(u) + \gamma \big(\varphi-\mathcal{A}(\varphi)_i\big), v\big\rangle_{\mathcal{V}},
\]
where $d^{-r} W'(u)$ is to be interpreted as the function in $\mathcal{V}$ defined by $(d^{-r} W'(u))_i := d_i^{-r} W'(u_i)$, for all $i\in V$.
Using \eqref{eq:H-1innerrewrite} we can also write
\[
\left.\frac{d}{ds} F_\e(u+sv) \right|_{s=0} = \big\langle \Delta\big\{\Delta u + \frac1\e d^{-r} W'(u) + \gamma \big(\varphi-\mathcal{A}(\varphi)\big)\big\}, v\big\rangle_{H^{-1}}.
\]
We note that, as expected, the freedom to add an arbitrary constant to $\varphi$ has no influence on the final result.

Hence, the $\mathcal{V}$ gradient flow is the \textit{Allen-Cahn type system of equations}
\begin{equation}\label{eq:AllenCahn}
\frac{d}{dt} u_i = -(\Delta u)_i - \frac1\e d_i^{-r} W'(u_i) - \gamma (\varphi_i-\mathcal{A}(\varphi)_i), \quad \text{for } i\in V,
\end{equation}
while the $H^{-1}$ gradient flow leads to the \textit{Cahn-Hillard type system of equations}
\begin{equation}\label{eq:CahnHilliard}
\frac{d}{dt} u_i = -(\Delta(\Delta u))_i - \frac1\e \Delta(d^{-r} W'(u))_i - \gamma (u_i-\mathcal{A}(u)_i), \quad \text{for } i\in V.
\end{equation}
The functions $u$ and $\varphi$ are in $\mathcal{V}_\infty$ (which is defined near the start of Section~\ref{sec:OKMBOscheme}) as is usual for gradient flows\footnote{Note that Peano's existence theorem \cite[Theorem 1.1]{Hale2009} guarantees existence of a continuously-differentiable-in-$t$ solution $u$ of equations \eqref{eq:AllenCahn} \eqref{eq:CahnHilliard}, and \eqref{eq:constrainedAllenCahn}, because in each of these the right hand side can be written as $Ou$, where $O$ is a continuous operator from $\mathcal{V}$ to $\mathcal{V}$. Continuity of $O$ follows from continuity of $W'$ and \eqref{eq:Leigexp}.}. We did not write the explicit dependence on $t$ here.

Since we are interested in minimising $F_\e$ over the set $\mathcal{V}_M$ of node functions with mass $M$, as defined in \eqref{eq:VM}, we need to ensure that the mass of $u$ does not change along the gradient flow. Because the right hand side of \eqref{eq:CahnHilliard} is of the form $\Delta f(u(t))$, with $f: \mathcal{V} \to \mathcal{V}$ determined by \eqref{eq:CahnHilliard}, for any solution of the $H^{-1}$ gradient flow above we have, by \eqref{eq:Laplacemass},
\[
\frac{d}{dt} \mathcal{M}(u(t)) = \mathcal{M}\left(\frac{du(t)}{dt}\right)= \mathcal{M}\big(\Delta f(u(t))\big) = 0.
\] 

For the $\mathcal{V}$ gradient flow mass conservation is not guaranteed and we need to introduce a Langrange multiplier $\mu: [0,\infty) \to \R$ in the equation:
\[
\frac{d}{dt} u_i = -(\Delta u)_i - \frac1\e d_i^{-r} W'(u_i) - \gamma \big(\varphi_i - \mathcal{A}(\varphi)\big) - \mu,
\]
such that
\[
0=\frac{d}{dt} \mathcal{M}(u) = -\mathcal{M}\big(\Delta u + \gamma (\varphi_i-\mathcal{A}(\varphi))\big) - \frac1\e \sum_{i\in V}  W'(u_i) + \mu \vol{V}  = - \frac1\e \sum_{i\in V}  W'(u_i) + \mu \vol{V}.
\]
Hence $\frac{d}{dt} \mathcal{M}(u)=0$ if and only if
$\displaystyle
\mu = \frac1{\e\,\vol{V}} \sum_{i\in V}  W'(u_i).
$
Therefore the mass constrained Allen-Cahn equation becomes
\begin{align}\label{eq:constrainedAllenCahn}
\frac{d}{dt} u_i &= -(\Delta u)_i - \frac1\e \left(d_i^{-r} W'(u_i) - (\vol{V})^{-1} \sum_{j\in V}  W'(u_j)\right) - \gamma \big(\varphi_i - \mathcal{A}(\varphi)\big)\notag\\
&=  -(\Delta u)_i - \frac1\e \big(d_i^{-r} W'(u_i) - \mathcal{A}\left(d^{-r} W'(u)\right)_i\big) - \gamma \big(\varphi_i - \mathcal{A}(\varphi)\big).
\end{align}

\newcommand{\etalchar}[1]{$^{#1}$}
\def\cprime{$'$} \def\cprime{$'$}
\providecommand{\bysame}{\leavevmode\hbox to3em{\hrulefill}\thinspace}
\providecommand{\MR}{\relax\ifhmode\unskip\space\fi MR }
\providecommand{\MRhref}[2]{%
  \href{http://www.ams.org/mathscinet-getitem?mr=#1}{#2}
}
\providecommand{\href}[2]{#2}

\end{document}